\documentclass{article}
\usepackage[utf8]{inputenc}
\usepackage{hyperref}
\usepackage{url}
\usepackage{comment}
\usepackage{graphicx}
\usepackage{authblk}
\usepackage{amsmath,amssymb}
\allowdisplaybreaks
\usepackage{float}
\usepackage{graphicx}
\usepackage{appendix}
\usepackage{listings} 
\usepackage{subfigure}
\usepackage{cite}
\usepackage{bm}
\usepackage{makecell}
\usepackage{multirow}
\usepackage{algorithm}  
\usepackage{algorithmicx}
\usepackage{algpseudocode}  
\usepackage{stmaryrd}
\usepackage{multicol}
\allowdisplaybreaks[4]

\newtheorem{theorem}{Theorem}[section]

\newtheorem{remark}{Remark}[section]

\newtheorem{assumption}{Assumption}[section]
\newtheorem{proof}{Proof}[section]


\title{A Domain Decomposition Deep Neural Network Method with Multi-Activation Functions for Solving Elliptic and Parabolic Interface Problems}
\author[a]{Qijia Zhai \thanks{zhaiqijia@163.com}}
\affil[a]{School of Mathematics, Sichuan University, Chengdu 610064, China}

\begin{document}

\maketitle

\begin{abstract}
We present a domain decomposition-based deep learning method for solving elliptic and parabolic interface problems with discontinuous coefficients in two to ten dimensions. Our Multi-Activation Function (MAF) approach employs two independent neural networks, one for each subdomain, coupled through interface conditions in the loss function. The key innovation is a multi-activation mechanism within each subdomain network that adaptively blends multiple activation functions (e.g., $\tanh$ and Gaussian-type) with interface-aware weighting, enhancing learning efficiency near interfaces where coupling constraints are most demanding. We prove conditional error bounds relating solution accuracy to trained loss values and quadrature errors. Numerical experiments on elliptic and parabolic interface problems with various interface geometries (2D--10D) validate the effectiveness and accuracy of the proposed method.
\end{abstract}

\section{Introduction}

In recent years, great attention has been directed towards a new category of computational techniques known as deep neural networks (DNNs), especially physics-informed neural networks (PINNs) \cite{raissi2019physics, cuomo2022scientific}. These methods utilize feed-forward neural networks to approximate unknown field variables by minimizing loss functions derived from the residuals of governing field equations, initial and boundary conditions, as well as observed data. The loss functions are evaluated at randomly selected collocation points distributed throughout the computational domain where the field variables are to be determined. By incorporating the governing equations directly into the loss functions, PINNs effectively overcome a major limitation of traditional data-driven modeling approaches: the requirement for large volumes of high-quality training data, which is often difficult and costly to obtain. Since their development, PINNs have found application across a wide range of engineering problems, including solid mechanics \cite{haghighat2020deep, haghighat2021physics, bai2023physics}, fluid mechanics \cite{cai2021physics, mao2020physics, zhang2024priori}, flow in porous media \cite{almajid2022prediction, fraces2021physics, zhang2023physics}, and heat transfer \cite{cai2021physicsheat}.

Elliptic interface problems form a class of partial differential equations (PDEs) that are distinguished by the presence of discontinuities in both the solution and its derivatives across interfaces. These problems are particularly common in a wide range of scientific and engineering fields, such as fluid dynamics \cite{chern2007coupling}, material science \cite{wang2010embedded}, electromagnetics \cite{costabel1988strongly}, and bio-mechanics \cite{elhaddad2018multi}, where different physical properties are present in adjacent subdomains. The interfaces between these subdomains often correspond to sharp changes in material properties, leading to complex mathematical formulations that are challenging to solve accurately and efficiently.

In addition to elliptic interface problems, parabolic interface problems also arise in numerous real-world applications, for instance in the modeling of heat conduction in composite materials \cite{ahmadikia2012analytical} or diffusion-driven phenomena with discontinuities \cite{gennari2022phase}. Like their elliptic counterparts, parabolic interface problems involve sharp changes in material parameters across interfaces; however, these problems include an additional time dimension. This temporal component further complicates the numerical treatment, as one must accurately capture both the spatial discontinuities and the evolving solution behavior over time. Handling these coupled spatial-temporal complexities often demands refined meshes or special numerical techniques that can adapt to both the moving solution fronts and the interface geometry, significantly increasing computational cost and implementation challenges \cite{chen1998finite, astrakhantsev1971finite}.

Traditional mesh-based methods, such as the finite element method (FEM) \cite{costabel1990coupling, ciarlet2002finite, mu2016new, gong2008immersed, wang2018conforming, he2011immersed, chen1998finite, lin2013method, zhu2020stable}, finite difference method (FDM) \cite{li1990numerical,iliev2003generalized,feng2022high,feng2024sixth, cao2024finite, astrakhantsev1971finite}, and finite volume method (FVM) \cite{oevermann2009sharp, liu2019analysis, sevilla2018face, remmerswaal2022parabolic}, have been widely used to solve elliptic interface problems. These methods typically rely on mesh generation, where the domain is discretized into smaller subdomains (elements or grids). However, accurately capturing the behavior of solutions near interfaces, especially when the geometry of the interface is irregular or high-dimensional, poses significant difficulties. In particular, these conventional methods require fine meshes around the interfaces to resolve the discontinuities, dramatically increasing the computational cost. Moreover, mesh generation can be highly non-trivial, particularly in complex geometries, often leading to issues related to stability and convergence.

Unlike mesh-based methods, deep learning approaches, especially PINNs, can operate in a mesh-free manner, which significantly alleviates the difficulties associated with mesh generation and refinement \cite{cuomo2022scientific, raissi2019physics}. Additionally, neural networks possess inherent flexibility, allowing them to approximate complex solution spaces, even in high-dimensional domains. Their ability to learn from data also makes them well-suited for problems where analytical solutions are difficult to obtain or where experimental data is available for training. Furthermore, the parallelization capabilities of deep learning models, coupled with advances in computational hardware, such as GPUs and TPUs, contribute to their efficiency in solving large-scale and high-dimensional problems.

The application of deep learning to elliptic interface problems, however, is still in its early stages, and several challenges remain to be addressed \cite{cuomo2022scientific}. These include the design of neural network architectures that can effectively capture discontinuities, the development of efficient training strategies that ensure convergence, and the integration of physical constraints into the learning process to ensure that the solutions remain physically meaningful. Nevertheless, there have been many representative works that have successfully used deep learning methods to solve elliptic interface problems \cite{zhang2022multi, sarma2024interface, roy2024adaptive, he2022mesh, hu2022discontinuity, tseng2023cusp, jiang2024generalization, ying2024accurate, li2025local, li2025two, zeng2025high, fan2025novel, zhang2025immersed}.

We now review several representative neural network-based methods for solving PDEs with interfaces or domain decomposition. Extended Physics-Informed Neural Networks (XPINN) \cite{jagtap2020extended} is a domain decomposition framework developed primarily for parallel computing and accelerating computations, rather than specifically for interface problems with discontinuous coefficients. While XPINNs divide the computational domain into subdomains handled by separate neural networks with continuity conditions at boundaries, they do not explicitly address discontinuities in solution gradients or material coefficients.

Several specialized methods have been developed for interface problems with discontinuous coefficients. Multi-domain PINNs (M-PINN) \cite{zhang2022multi} employ separate networks for each subdomain with standard activation functions, coupled through interface loss terms. Interface PINNs (I-PINN) \cite{sarma2024interface} incorporate interface conditions directly into both network architecture and loss formulation. Adaptive Interface-PINNs (AdaI-PINN) \cite{roy2024adaptive} further enhance I-PINNs with adaptive sampling strategies near the interface. Mesh-free methods (MFM) \cite{he2022mesh} use scattered points with radial basis functions, offering flexibility for irregular domains but with limited application to parabolic problems. Discontinuity Capturing Shallow Neural Networks (DCSNN) \cite{hu2022discontinuity, tseng2023cusp} transform problems to higher-dimensional spaces to capture jump discontinuities with shallow architectures, though they focus primarily on elliptic problems.

In this paper, we propose a Multi-Activation Function (MAF) method for solving elliptic and parabolic interface problems. Our method employs two independent neural networks, one for each subdomain $\Omega_1$ and $\Omega_2$, coupled through interface conditions in the loss function. The key innovation lies in the multi-activation mechanism within each subdomain network: instead of using a single fixed activation function, we adaptively combine multiple activation functions (e.g., $\tanh$ and Gaussian-type) with interface-aware weighting. This design allows each network to efficiently capture sharp variations required to satisfy interface constraints near $\Gamma$, while maintaining smooth approximations in the bulk. Our approach offers: (i) unified treatment of elliptic and parabolic systems through space-time formulation; (ii) geometric flexibility for arbitrary interface configurations (static, moving, rotating, deforming) without remeshing; (iii) high-dimensional capability up to 10D; and (iv) conditional error estimates linking network parameters to solution accuracy.

Unlike XPINN, which was designed for general domain decomposition, our method explicitly handles discontinuous coefficients and flux jumps. Compared to M-PINN, I-PINN, and AdaI-PINN, which address interface problems through standard activation functions, architectural modifications, or adaptive sampling respectively, MAF introduces learnable, spatially-adaptive activation function blending within each subdomain network. Unlike DCSNN, which transforms to higher-dimensional spaces, our method works directly in the original space with interface-aware activation weighting. Importantly, while existing methods (MFM, DCSNN, M-PINN, I-PINN, AdaI-PINN) have not been demonstrated for parabolic interface problems with moving interfaces, our MAF method handles such problems effectively.

Our theoretical contributions include conditional error bounds adapted from the generalization error framework of Mishra and Molinaro \cite{mishra2023estimates} for interface problems with jump conditions. These bounds relate solution accuracy to trained loss values, quadrature errors, and sampling density, though practical performance depends on network architecture, optimization algorithms, and hyperparameters best validated through experiments.

The remainder of this paper is organized as follows: Section \ref{ModelProblem} formulates the model elliptic and parabolic interface problems. Section \ref{Multi-ActivationFunction} details our multi-activation DNN architecture and interface-aware training methodology. Section \ref{sec:error-estimation} presents the error analysis framework and convergence theorems. Section \ref{NumericalResults} demonstrates numerical performance across 2D-10D static/moving interface problems. Section \ref{Conclusion} discusses broader implications and future research directions.

\section{Model Problem}
\label{ModelProblem}

In this section, we present the model problems that motivate our multi-activation function approach. Let $\Omega \subset \mathbb{R}^d$ with $d \ge 2$ be a bounded open domain whose boundary $\partial \Omega$ is Lipschitz. Denote by $\Gamma$ a closed interface inside $\Omega$, which partitions $\Omega$ into two disjoint open subdomains $\Omega_1$ and $\Omega_2$ (see Fig.~\ref{computationaldomain} for an illustrative example). We denote by $\overline{\Omega}_i$ the closure of $\Omega_i$ for $i = 1, 2$, so that $\Gamma \subset \overline{\Omega}_1 \cap \overline{\Omega}_2$. Throughout this paper, we denote the unit normal vector $\boldsymbol{n}$ on $\Gamma$ pointing from $\Omega_1$ to $\Omega_2$. Although we assume that $\Gamma$ is Lipschitz, our method can accommodate more general interfaces in practice.

\begin{figure}[H]
    \centering
    \subfigure[The immersed case]{\includegraphics[width=6cm,height=4cm]{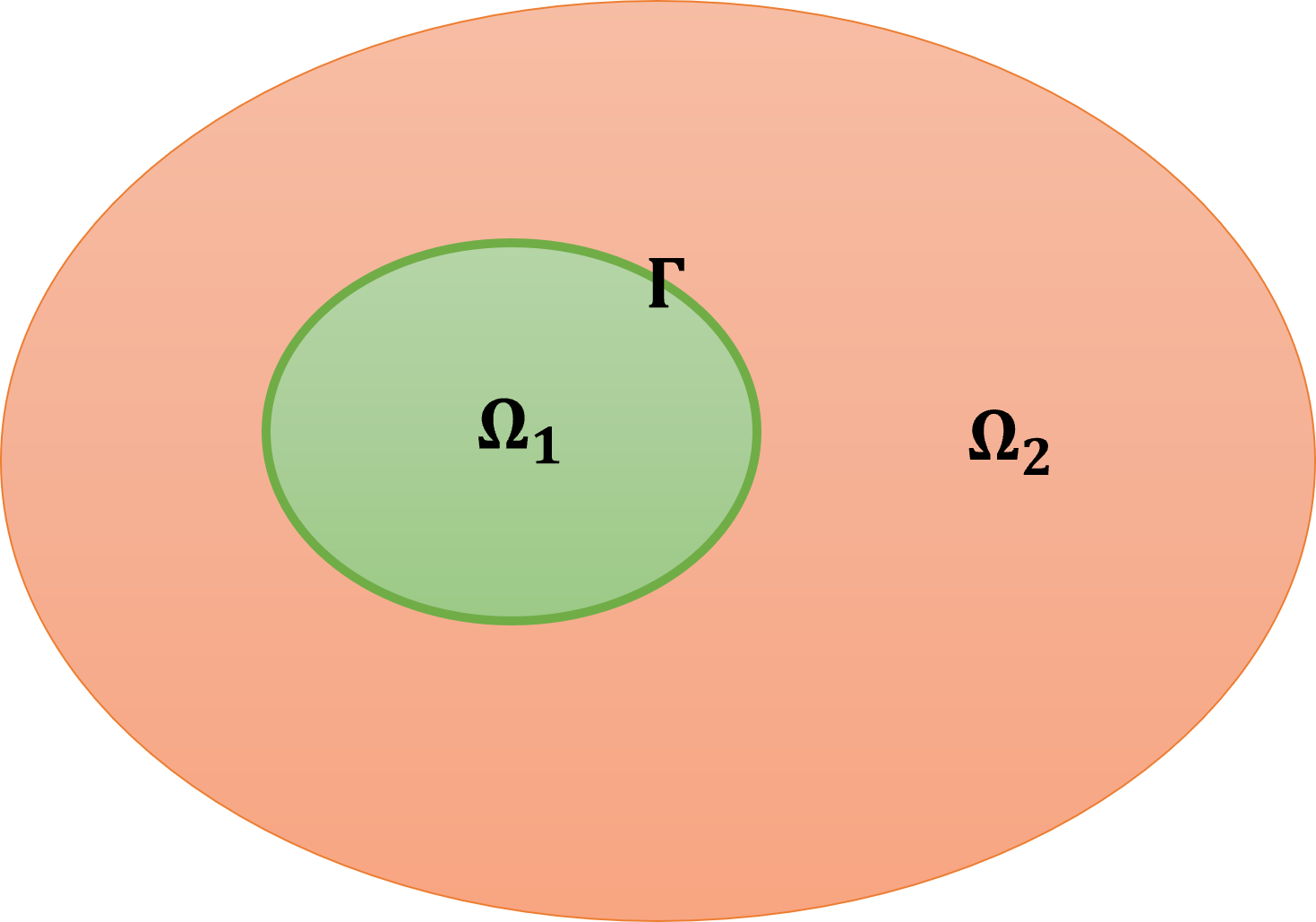}}
    \subfigure[The back-to-back case]{\includegraphics[width=4cm,height=4cm]{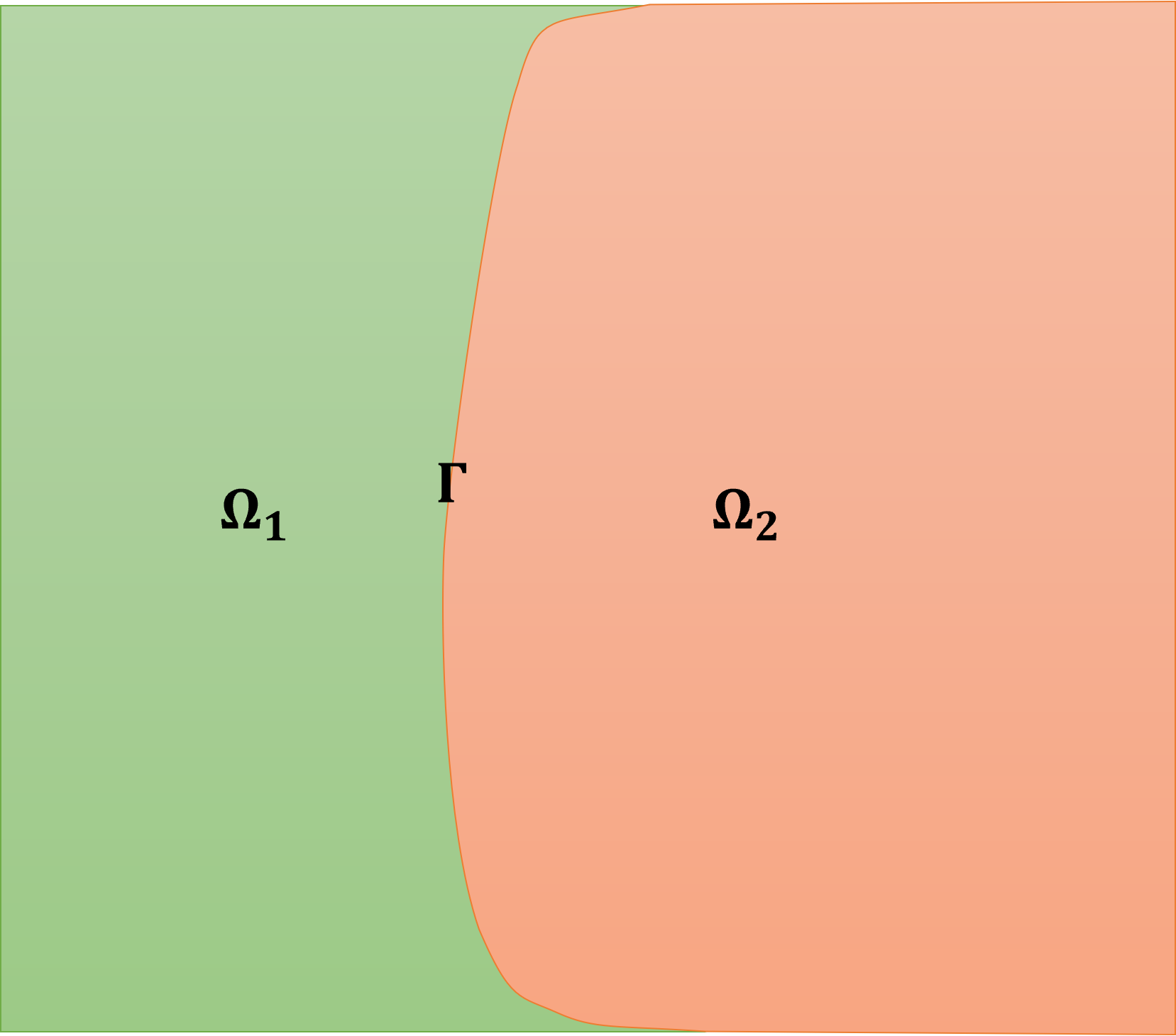}}
    \caption{Illustration of the computational domain. The interface $\Gamma$ divides the domain into two sub-domains $\Omega_1$ and $\Omega_2$.}
    \label{computationaldomain}
\end{figure}

We begin by considering the following second-order scalar elliptic interface problem:
\begin{equation}\label{modelpextenddnn}
\begin{aligned}
-\nabla \cdot \bigl(\beta(\boldsymbol{x}) \nabla u\bigr) & = f, & & \text{in } \Omega_1 \cup \Omega_2, \\
\llbracket u \rrbracket & = g_1, & & \text{on } \Gamma, \\
\llbracket \beta(\boldsymbol{x}) \nabla u \cdot \boldsymbol{n} \rrbracket & = g_2, & & \text{on } \Gamma, \\
u & = g_D, & & \text{on } \partial \Omega,
\end{aligned}
\end{equation}
where $f \in L^2(\Omega)$, $g_D \in H^{1/2}(\partial \Omega)$, and $\boldsymbol{n}$ denotes the unit normal vector on $\Gamma$ pointing outward from $\Omega_1$. The function $\beta(\boldsymbol{x}) \ge \beta_0 > 0$ is piecewise constant:
\[
\beta(\boldsymbol{x}) = 
\begin{cases}
\beta_1, & \text{if } \boldsymbol{x} \in \Omega_1,\\
\beta_2, & \text{if } \boldsymbol{x} \in \Omega_2,
\end{cases}
\]
thus exhibiting a finite jump across the interface $\Gamma$. The notation $\llbracket \cdot \rrbracket$ indicates the jump in a function across $\Gamma$. In particular, for any scalar function $u$, we write
\[
\left.\llbracket u(\boldsymbol{x}) \rrbracket\right|_{\Gamma}
= \left.u_1(\boldsymbol{x})\right|_{\Gamma} - \left.u_2(\boldsymbol{x})\right|_{\Gamma},
\]
where
\[
u(\boldsymbol{x})=
\begin{cases}
u_1(\boldsymbol{x}), & \text{if } \boldsymbol{x} \in \Omega_1,\\
u_2(\boldsymbol{x}), & \text{if } \boldsymbol{x} \in \Omega_2.
\end{cases}
\]

Problems of this type arise in a variety of applications, including composite materials, fluid dynamics with interfaces, and heat transfer in layered media. The jump conditions on $\Gamma$ reflect discontinuities in the solution and/or its flux due to contrasting material properties. Our goal is to accurately capture these jumps without resorting to body-fitted meshes or complicated mesh refinement near $\Gamma$.

\medskip
Next, we consider the corresponding parabolic interface problem, which extends the elliptic model above to include time dependence. Let $T>0$ be a fixed final time, and let $\Omega_1(t)$ and $\Omega_2(t)$ denote the time-dependent sub-domains as determined by the interface $\Gamma(t)$. The parabolic interface problem takes the form:
\begin{equation}\label{modelpextenddnnpaowu}
\begin{aligned}
u_t - \nabla \cdot \bigl(\beta(\boldsymbol{x}) \nabla u\bigr) & = f, & & \text{in } \Omega_1(t) \cup \Omega_2(t), \\
\llbracket u \rrbracket & = g_1, & & \text{on } \Gamma(t), \\
\llbracket \beta(\boldsymbol{x}) \nabla u \cdot \boldsymbol{n} \rrbracket & = g_2, & & \text{on } \Gamma(t), \\
u_0 & = g_0, & & \text{on } \Omega_1(0) \cup \Omega_2(0),\\
u & = g_D, & & \text{on } \partial \Omega(t),
\end{aligned}
\end{equation}
where $f \in L^2\bigl(0,T; L^2(\Omega)\bigr)$, $g_D \in L^2\bigl(0,T; H^{1/2}(\partial \Omega)\bigr)$, and $\boldsymbol{n}$ is the unit normal vector on $\Gamma$ pointing outward from $\Omega_1$. As in the elliptic problem, the diffusion coefficient $\beta(\boldsymbol{x}) \ge \beta_0 > 0$ remains piecewise constant:
\[
\beta(\boldsymbol{x}) = 
\begin{cases}
\beta_1, & \text{if } \boldsymbol{x} \in \Omega_1,\\
\beta_2, & \text{if } \boldsymbol{x} \in \Omega_2,
\end{cases}
\]
and the notation for the jump condition is given by
\[
\left.\llbracket u(\boldsymbol{x}, t) \rrbracket\right|_{\Gamma}
= \left.u_1(\boldsymbol{x}, t)\right|_{\Gamma} - \left.u_2(\boldsymbol{x}, t)\right|_{\Gamma},
\]
with
\[
u(\boldsymbol{x}, t)=
\begin{cases}
u_1(\boldsymbol{x}, t), & \text{if } \boldsymbol{x} \in \Omega_1,\\
u_2(\boldsymbol{x}, t), & \text{if } \boldsymbol{x} \in \Omega_2.
\end{cases}
\]

These parabolic problems model time-dependent processes, such as heat diffusion or mass transfer, where different regions of the domain possess distinct material properties. The additional jump conditions on $\Gamma$ are analogous to those in the elliptic case, reflecting the discontinuities in the solution and flux across the interface. The initial condition $u_0 = g_0$ prescribes the state of the system at $t = 0$, while $u = g_D$ on the boundary $\partial \Omega_2(t)$ (or equivalently on $\partial \Omega$ if the interface does not intersect $\partial \Omega$) provides a time-dependent boundary condition. 

In the subsequent sections, we present a novel deep learning framework that employs multiple activation functions to address both the elliptic interface problem \eqref{modelpextenddnn} and the parabolic interface problem \eqref{modelpextenddnnpaowu}. Our proposed methodology circumvents the traditional challenges associated with interface problems by introducing a specialized neural network architecture that naturally accommodates jump conditions. This approach eliminates the need for body-fitted meshes or complex geometric adaptations, offering a more flexible and computationally efficient solution strategy. The framework's ability to handle both stationary and time-dependent interfaces, coupled with its capacity to capture sharp discontinuities in the solution and its derivatives, makes it particularly well-suited for a wide range of practical applications.
%in materials science, fluid dynamics, and heat transfer problems.

\section{Multi-Activation Function Method}
\label{Multi-ActivationFunction}

In this section, we describe our multi-activation function approach for solving interface problems using deep neural networks (DNNs). Our method adopts a domain decomposition strategy where two independent neural networks are employed, one for each subdomain $\Omega_1$ and $\Omega_2$. These networks share the same architecture but have independently trained parameters, and they are coupled through interface conditions in the loss function. Each network incorporates a multi-activation function mechanism to enhance learning efficiency, particularly for satisfying the interface constraints.

Throughout this section, we use the following notation conventions: $u_{\boldsymbol{\theta}_i}$ denotes the neural network function for subdomain $\Omega_i$ as a mapping, while $u_{\boldsymbol{\theta}_i}(\boldsymbol{x})$ denotes its evaluation at a point $\boldsymbol{x} \in \overline{\Omega}_i$. When discussing function properties (e.g., approximation, hypothesis space), we use $u_{\boldsymbol{\theta}_i}$; when discussing pointwise values (e.g., loss function evaluations), we use $u_{\boldsymbol{\theta}_i}(\boldsymbol{x})$.

\subsection{Deep Neural Network Architecture}

We employ two neural networks $u_{\boldsymbol{\theta}_1}$ and $u_{\boldsymbol{\theta}_2}$ to approximate the solution in $\Omega_1$ and $\Omega_2$, respectively. Each network has $(L-1)$ hidden layers with general activation functions. We treat the $d$-dimensional input as the 0-th layer and the one-dimensional output as the $L$-th layer. Let $m_l$ be the number of neurons in the $l$-th layer. In particular, we set $m_0 = d$ and $m_L = 1$. For any $k \in \mathbb{N}$, we denote $[k] := \{1,2,\cdots, k\}$.

\paragraph{Parameterization.} 
For each subdomain $\Omega_i$ ($i=1,2$), we collect all the parameters of the corresponding network into a vector 
\[
\boldsymbol{\theta}_i = \bigl(\operatorname{vec}(\boldsymbol{W}_i^{[1]}),\, \operatorname{vec}(\boldsymbol{b}_i^{[1]}),\, \cdots,\,
\operatorname{vec}(\boldsymbol{W}_i^{[L]}),\, \operatorname{vec}(\boldsymbol{b}_i^{[L]})\bigr),
\]
where for each layer $l \in [L]$, we have $\boldsymbol{W}_i^{[l]} \in \mathbb{R}^{m_l \times m_{l-1}}$ and $\boldsymbol{b}_i^{[l]} \in \mathbb{R}^{m_l}$. The size $M$ of each network is the total number of parameters:
\begin{equation}
   M \;=\; \sum_{l=0}^{L-1} \bigl(m_l+1\bigr)\, m_{l+1}.
\end{equation}
We call these parameters $W_{i,jk}^{[l]}$ (weights) and $b_{i,j}^{[l]}$ (biases). Note that $\boldsymbol{\theta}_1$ and $\boldsymbol{\theta}_2$ are independently trainable, allowing each network to specialize in approximating the solution within its respective subdomain. The hypothesis space for each subdomain is $\mathcal{H}_i = \{u_{\boldsymbol{\theta}_i} : \boldsymbol{\theta}_i \in \mathbb{R}^M\}$.

\paragraph{Activation Functions.}
For each subdomain network $u_{\boldsymbol{\theta}_i}$ ($i=1,2$), we introduce activation functions of the form
\[
\sigma_j^{[l]}: \mathbb{R} \;\to\; \mathbb{R}, 
\quad\text{for } l \in [L-1],\; j \in [m_l].
\]
We interpret the 0-th layer output as the identity map on $\mathbb{R}^d$, i.e., 
\[
u_{\boldsymbol{\theta}_i}^{[0]}(\boldsymbol{x}) = \boldsymbol{x}
\quad \text{for all } \boldsymbol{x} \in \mathbb{R}^d.
\]
Then for $l \in [L-1]$, the layer output $u_{\boldsymbol{\theta}_i}^{[l]}: \mathbb{R}^d \to \mathbb{R}^{m_l}$ is defined componentwise by
\begin{equation}
    \begin{aligned}
    \bigl(u_{\boldsymbol{\theta}_i}^{[l]}(\boldsymbol{x})\bigr)_j
    \;=\;
    \sigma_j^{[l]}
    \Bigl(\bigl(\boldsymbol{W}_i^{[l]}\,u_{\boldsymbol{\theta}_i}^{[l-1]}(\boldsymbol{x}) + \boldsymbol{b}_i^{[l]}\bigr)_j\Bigr),
    \quad j \in [m_l].
    \end{aligned}
\end{equation}
For the final layer $l = L$, we set
\begin{equation}
u_{\boldsymbol{\theta}_i}^{[L]}(\boldsymbol{x}) 
\;=\;
\boldsymbol{W}_i^{[L]} \,\bigl(u_{\boldsymbol{\theta}_i}^{[L-1]}(\boldsymbol{x})\bigr) \;+\; b_i^{[L]},
\end{equation}
yielding a scalar output $u_{\boldsymbol{\theta}_i}(\boldsymbol{x}) := u_{\boldsymbol{\theta}_i}^{[L]}(\boldsymbol{x})$ for subdomain $\Omega_i$. Although both networks share the same architecture, their parameters $\boldsymbol{\theta}_1$ and $\boldsymbol{\theta}_2$ are trained independently, allowing each network to specialize in its respective subdomain. In many applications, the activation functions $\sigma_j^{[l]}$ are chosen to be the same (e.g., $\tanh$ or Sigmoid) for simplicity. Note that for second-order PDEs, smooth activation functions such as $\tanh$ are preferred over non-smooth ones like ReLU, since the PDE residual requires computing second-order derivatives of the network output.

\subsection{Discrete Loss Function for the Elliptic Interface Problem}

We now introduce the loss function used in our MAF formulation. Our goal is to approximate the interface problem \eqref{modelpextenddnn} by enforcing the PDE and jump conditions in a collocation sense. Specifically, we sample points in the sub-domains $\Omega_1$ and $\Omega_2$, on the boundary $\partial \Omega$, and on the interface $\Gamma$, then define a discrete loss that couples the two subdomain networks $u_{\boldsymbol{\theta}_1}$ and $u_{\boldsymbol{\theta}_2}$.

\paragraph{Interior Losses.}
Let $\{\boldsymbol{x}_k^{\Omega_i}\}_{k=1}^{M_i} \subset \Omega_i$, $i=1,2$, be sets of collocation points in each sub-domain. For each subdomain, we approximate the PDE residual term using the corresponding network:
\begin{equation}
\mathcal{L}_i(\boldsymbol{\theta}_i)
\;:=\;
\frac{1}{M_i} \sum_{k=1}^{M_i}
\Bigl|\,-\nabla \cdot \bigl(\beta_i \nabla u_{\boldsymbol{\theta}_i}\bigr)(\boldsymbol{x}_k^{\Omega_i})
\;-\;f(\boldsymbol{x}_k^{\Omega_i})\Bigr|^2,
\end{equation}
where $\beta_i$ is the constant diffusion coefficient in $\Omega_i$. The derivatives here can be computed by automatic differentiation.

\paragraph{Boundary Loss.}
To enforce the Dirichlet boundary condition on $\partial \Omega$, we sample $M_{\partial \Omega_i}$ points $\{\boldsymbol{x}_k^{\partial \Omega_i}\}_{k=1}^{M_{\partial \Omega_i}} \subset \partial \Omega_i \setminus \Gamma$ for each subdomain and define
\begin{equation}
\mathcal{L}_{\partial \Omega_i}(\boldsymbol{\theta}_i)
\;:=\;
\frac{1}{M_{\partial \Omega_i}} 
\sum_{k=1}^{M_{\partial \Omega_i}}
\Bigl|\,u_{\boldsymbol{\theta}_i}\bigl(\boldsymbol{x}_k^{\partial \Omega_i}\bigr)
\;-\;
g_D\bigl(\boldsymbol{x}_k^{\partial \Omega_i}\bigr)\Bigr|^2.
\end{equation}

\paragraph{Interface Loss.}
The interface loss is crucial as it couples the two independent networks by enforcing the jump conditions. We sample $M_{\Gamma}$ points 
\(\{\boldsymbol{x}_k^{\Gamma}\}_{k=1}^{M_{\Gamma}} \subset \Gamma\)
and define:
\[
\mathcal{L}_{\Gamma}(\boldsymbol{\theta}_1, \boldsymbol{\theta}_2)
\;:=\;
\sum_{k=1}^{M_{\Gamma}}
\Bigl(\frac{\gamma_1}{M_{\Gamma}} \,\bigl|u_{\boldsymbol{\theta}_1}\bigl(\boldsymbol{x}_k^{\Gamma}\bigr) - u_{\boldsymbol{\theta}_2}\bigl(\boldsymbol{x}_k^{\Gamma}\bigr)
- g_1\bigl(\boldsymbol{x}_k^{\Gamma}\bigr)\bigr|^2
\]
\[
\;+\;
\frac{\gamma_2}{M_{\Gamma}}\bigl|\beta_1\nabla u_{\boldsymbol{\theta}_1}\bigl(\boldsymbol{x}_k^{\Gamma}\bigr) \cdot \boldsymbol{n}
- \beta_2\nabla u_{\boldsymbol{\theta}_2}\bigl(\boldsymbol{x}_k^{\Gamma}\bigr) \cdot \boldsymbol{n}
- g_2\bigl(\boldsymbol{x}_k^{\Gamma}\bigr)\bigr|^2\Bigr).
\]
Here, $\gamma_1$ and $\gamma_2$ are weights to balance the importance of each jump condition. Note that $u_{\boldsymbol{\theta}_1}$ and $u_{\boldsymbol{\theta}_2}$ are evaluated at the same interface points but represent the solution limits from each subdomain.

\paragraph{Total Loss.}
Combining the above terms, the total loss for our elliptic interface problem couples both networks:
\begin{equation}\label{totelloss}
\mathcal{L}_{\text{total}}(\boldsymbol{\theta}_1, \boldsymbol{\theta}_2)
\;:=\;
\mathcal{L}_1(\boldsymbol{\theta}_1)
\;+\;
\mathcal{L}_2(\boldsymbol{\theta}_2)
\;+\;
\mathcal{L}_{\Gamma}(\boldsymbol{\theta}_1, \boldsymbol{\theta}_2)
\;+\;
\mathcal{L}_{\partial \Omega_1}(\boldsymbol{\theta}_1)
\;+\;
\mathcal{L}_{\partial \Omega_2}(\boldsymbol{\theta}_2).
\end{equation}
We seek to find
\[
\min_{(\boldsymbol{\theta}_1, \boldsymbol{\theta}_2) \in \mathbb{R}^M \times \mathbb{R}^M} 
\mathcal{L}_{\text{total}}(\boldsymbol{\theta}_1, \boldsymbol{\theta}_2).
\]
If $(\boldsymbol{\theta}_1^*, \boldsymbol{\theta}_2^*)$ denotes the minimizer, then the corresponding neural network solutions are $u_{\boldsymbol{\theta}_1^*}$ for $\Omega_1$ and $u_{\boldsymbol{\theta}_2^*}$ for $\Omega_2$.

\paragraph{Training Dynamics.}
In practice, gradient-based methods are used to update both parameter sets $\boldsymbol{\theta}_1$ and $\boldsymbol{\theta}_2$ jointly according to
\begin{equation}\label{trainingode}
\begin{cases}
\dfrac{\mathrm{d}\boldsymbol{\theta}_i}{\mathrm{d}t}
\;=\;
-\,\nabla_{\boldsymbol{\theta}_i}\,
\mathcal{L}_{\text{total}}\bigl(\boldsymbol{\theta}_1, \boldsymbol{\theta}_2\bigr), \quad i = 1, 2,
\\[6pt]
\boldsymbol{\theta}_i(0) \;=\;\boldsymbol{\theta}_{i,0},
\end{cases}
\end{equation}
during training. Variants of (stochastic) gradient descent (SGD), ADAM or second-order methods can be employed, and automatic differentiation helps compute the gradients efficiently. Note that the interface loss $\mathcal{L}_{\Gamma}$ introduces coupling between the two networks, ensuring consistency of the solution and flux across the interface.

\subsection{Extension to Parabolic Interface Problems}
For parabolic interface problems, we extend our approach by treating time as an additional dimension in the neural network architecture, effectively working in $\mathbb{R}^{d+1}$. The domain $\Omega$ is divided into two subdomains $\Omega_1$ and $\Omega_2$ separated by an interface $\Gamma(t)$. This formulation allows us to handle various interface configurations:
This formulation allows us to handle static interfaces where $\Gamma(t) = \Gamma$ remains fixed, moving interfaces where $\Gamma(t)$ translates according to a prescribed velocity field, and rotating interfaces where $\Gamma(t)$ undergoes rotational motion about a fixed axis.
Each subdomain network now accepts input $(\mathbf{x}, t) \in \mathbb{R}^{d+1}$. The loss function for the parabolic case in each subdomain $\Omega_i$ ($i=1,2$) includes an additional temporal derivative term:
\begin{equation}
\mathcal{L}_i^p(\boldsymbol{\theta}_i) := \frac{1}{M_i} \sum_{k=1}^{M_i}
\Bigl|\partial_t u_{\boldsymbol{\theta}_i} - \nabla \cdot \bigl(\beta_i \nabla u_{\boldsymbol{\theta}_i}(\boldsymbol{x}_k^{\Omega_i}, t_k)\bigr)
- f(\boldsymbol{x}_k^{\Omega_i}, t_k)\Bigr|^2,
\end{equation}
with the interface conditions coupling the two networks:
\begin{equation}
\begin{aligned}
    \mathcal{L}_{\Gamma}^p(\boldsymbol{\theta}_1, \boldsymbol{\theta}_2) &:= \sum_{k=1}^{M_{\Gamma}}
\Bigl(\frac{\gamma_1}{M_{\Gamma}} \bigl|u_{\boldsymbol{\theta}_1}\bigl(\boldsymbol{x}_k^{\Gamma}, t\bigr) - u_{\boldsymbol{\theta}_2}\bigl(\boldsymbol{x}_k^{\Gamma}, t\bigr)
- g_1\bigl(\boldsymbol{x}_k^{\Gamma}, t\bigr)\bigr|^2 \\
&+ \frac{\gamma_2}{M_{\Gamma}}\bigl|\beta_1 \nabla u_{\boldsymbol{\theta}_1}\bigl(\boldsymbol{x}_k^{\Gamma}, t\bigr) \cdot \boldsymbol{n} - \beta_2 \nabla u_{\boldsymbol{\theta}_2}\bigl(\boldsymbol{x}_k^{\Gamma}, t\bigr) \cdot \boldsymbol{n}
- g_2\bigl(\boldsymbol{x}_k^{\Gamma}, t\bigr)\bigr|^2\Bigr),
\end{aligned}
\end{equation}
where $M_1$, $M_2$, $M_{\Gamma}$ are the numbers of sampling points for the interior domains and interface condition.

The initial condition loss is defined as:
\begin{equation}
\mathcal{L}_0(\boldsymbol{\theta}_1, \boldsymbol{\theta}_2) := \frac{1}{M_0} \sum_{k=1}^{M_0}
\bigl|u_{\boldsymbol{\theta}_i}(\boldsymbol{x}_k, 0) - g_0(\boldsymbol{x}_k)\bigr|^2, \quad \text{where } \boldsymbol{x}_k \in \Omega_i(0),
\end{equation}
where $M_0$ is the number of sampling points for the initial condition.

For the Dirichlet boundary condition on $\partial\Omega(t)$, we define:
\begin{equation}
\mathcal{L}_{\partial \Omega}(\boldsymbol{\theta}_1, \boldsymbol{\theta}_2) := \frac{1}{M_{\partial \Omega}} \sum_{k=1}^{M_{\partial \Omega}}
\bigl|u_{\boldsymbol{\theta}_i}(\boldsymbol{x}_k^{\partial\Omega}, t_k) - g_D(\boldsymbol{x}_k^{\partial\Omega}, t_k)\bigr|^2, \quad \text{where } \boldsymbol{x}_k \in \Omega_i(t_k).
\end{equation}

The total loss function combines all these components:
\begin{equation}
\mathcal{L}(\boldsymbol{\theta}_1, \boldsymbol{\theta}_2) := \mathcal{L}_1^p(\boldsymbol{\theta}_1) + \mathcal{L}_2^p(\boldsymbol{\theta}_2) + \mathcal{L}_{\Gamma}^p(\boldsymbol{\theta}_1, \boldsymbol{\theta}_2) + \mathcal{L}_0(\boldsymbol{\theta}_1, \boldsymbol{\theta}_2) + \mathcal{L}_{\partial \Omega}(\boldsymbol{\theta}_1, \boldsymbol{\theta}_2).
\end{equation}
%\begin{equation}
%\mathcal{L}(\boldsymbol{\theta}) := \lambda_1\mathcal{L}_1^p(\boldsymbol{\theta}) + \lambda_2\mathcal{L}_2^p(\boldsymbol{\theta}) + \lambda_{\Gamma}\mathcal{L}_{\Gamma}^p(\boldsymbol{\theta}) + \lambda_0\mathcal{L}_0(\boldsymbol{\theta}) + \lambda_D\mathcal{L}_D(\boldsymbol{\theta}).
%\end{equation}
%where $\lambda_1$, $\lambda_2$, $\lambda_{\Gamma}$, $\lambda_0$, and $\lambda_D$ are positive weighting parameters.

\subsection{Multi-Activation for Enhanced Learning Near Interfaces}\label{Multi-Activation}

Although our domain decomposition approach employs separate networks for each subdomain---where each network only needs to approximate a smooth solution within its domain---the \emph{interface loss term} $\mathcal{L}_{\Gamma}(\boldsymbol{\theta}_1, \boldsymbol{\theta}_2)$ imposes strong constraints that require both networks to rapidly adapt their behavior near the interface $\Gamma$. To facilitate efficient learning under these interface constraints, we propose incorporating \emph{multi-activation functions} within each subdomain network.

While the true solution is smooth (e.g., $C^2(\overline{\Omega}_i)$) within each subdomain, the interface conditions create a challenging optimization landscape where networks must satisfy jump conditions exactly at the interface while maintaining smooth approximations in the bulk. Standard single-activation networks using a single smooth activation (e.g., $\sigma^{[l]} = \tanh$) may require excessive parameters to achieve this dual objective efficiently.

We propose a multi-activation mechanism that allows each layer $l$ to blend multiple candidate activations. Concretely, for each subdomain network, we define
    \[
    \sigma^{[l]}(\mathbf{x})
    \;=\;
    \omega_1(\mathbf{x})
\,\tanh\bigl(W_1^{[l]}\,\mathbf{x} + \mathbf{b}_1^{[l]}\bigr)
    \;+\;
    \omega_2(\mathbf{x})
\,\exp\!\Bigl(-\frac{(W_2^{[l]}\,\mathbf{x})^2}{\gamma}\Bigr),
\]
where $\gamma$ is a positive constant, $W_1^{[l]}$ and $W_2^{[l]}$ are learnable weights, and $\omega_1(\mathbf{x})$, $\omega_2(\mathbf{x})$ are interface-aware weighting functions satisfying
\[
\omega_1(\mathbf{x}) \to 0 \text{ as } \mathbf{x} \to \Gamma, \quad
\omega_1(\mathbf{x}) \to 1 \text{ as } \mathbf{x} \to \partial \Omega,
    \]
    and
    \[
\omega_2(\mathbf{x}) \to 1 \text{ as } \mathbf{x} \to \Gamma, \quad
\omega_2(\mathbf{x}) \to 0 \text{ as } \mathbf{x} \to \partial \Omega.
\]
These weighting functions can be derived from a level set or distance function indicating proximity to $\Gamma$. This design allocates specialized activation behavior near the interface where satisfying the jump conditions is most critical.

The multi-activation design enhances each subdomain network's ability to efficiently satisfy the interface constraints while maintaining accurate approximations in the bulk. By combining different activation functions and weighting them adaptively based on location, each network can dedicate appropriate representational capacity to regions where the interface loss imposes the strongest gradients during training.

The combination of $\tanh$ and Gaussian-type activation functions provides several advantages for learning near interfaces:
\begin{itemize}
    \item The $\tanh$ function is infinitely differentiable ($C^\infty$), ensuring that second-order derivatives required for evaluating the PDE residual are well-defined throughout the domain.
    \item The Gaussian-type activation $\exp\!\bigl(-(W_2^{[l]}\,\mathbf{x})^2/\gamma\bigr)$ introduces \emph{localized sensitivity} near the interface, enabling efficient satisfaction of jump conditions without disrupting the solution approximation far from $\Gamma$.  
    \item The parameter $\gamma$ provides an additional degree of freedom to control the spatial extent of interface-focused learning.
    \item The bounded outputs of both activations contribute to numerical stability during training, particularly when the interface conditions involve flux matching with high-contrast coefficients.
    \item The Gaussian-like decay characteristics complement the global smoothness of $\tanh$, creating a more expressive function space that can efficiently represent solutions satisfying both bulk PDE residuals and interface constraints.
\end{itemize}

In summary, while each subdomain network approximates a smooth solution, the multi-activation mechanism \emph{improves training efficiency} by providing specialized representational capacity near the interface where the coupling constraints are enforced. This reduces the total number of parameters and training iterations needed to achieve a given accuracy compared to standard single-activation architectures.

%\begin{remark}
%To improve the effectiveness of our multi-activation approach, we implement several key enhancements:
%\begin{equation}
%\sigma_{\text{combined}}^{[l]}(\mathbf{x}) = \sum_{j=1}^2 \omega_j(\mathbf{x}) \sigma_j^{[l]}(\mathbf{x}),
%\end{equation}
%where $\{\sigma_j^{[l]}\}_{j=1}^2$ represents a collection of candidate activation functions, and $\{\omega_j(\mathbf{x})\}_{j=1}^2$ are corresponding weight functions satisfying $\sum_{j=1}^2 \omega_j(\mathbf{x}) = 1$. The weight functions are constructed using distance-based metrics:
%\begin{equation}
%\omega_j(\mathbf{x}) = \frac{\exp(-d_j(\mathbf{x})/\epsilon)}{\sum_{k=1}^2 \exp(-d_k(\mathbf{x})/\epsilon)},
%\end{equation}
%where $d_j(\mathbf{x})$ represents the distance to relevant geometric features (interface, boundary, etc.), and $\epsilon$ is a smoothing parameter.
%\end{remark}

In the subsequent section, we present a comprehensive error analysis of our multi-activation function method. We establish theoretical convergence rates and estimate a priori error in appropriate function spaces. The analysis considers both the neural network architecture's approximation capabilities and the impact of the discrete loss function minimization. We derive bounds for the solution error regarding the network width, depth, and the number of training points, accounting for the presence of interface discontinuities and the effectiveness of our multi-activation strategy.

\section{Error Estimation for Multi-Activation Function Method for Elliptic Interface Problem and Parabolic Interface Problem}
\label{sec:error-estimation}

In this section, we present an error analysis framework for the proposed multi-activation function method. Our analysis is based on the generalization error estimation framework for physics-informed neural networks developed by Mishra and Molinaro \cite{mishra2023estimates}, which we adapt to interface problems with jump conditions. We also draw upon related error analysis techniques from \cite{de2022error, zhang2024priori}.

Throughout this section, we use the following notation:
\begin{itemize}
    \item $\mathbb{X}(\overline{\Omega}_i)$ denotes the function space for the solution in subdomain $\Omega_i$. For elliptic problems with classical solutions, we take $\mathbb{X}(\overline{\Omega}_i) = C^2(\overline{\Omega}_i)$, the space of twice continuously differentiable functions on $\overline{\Omega}_i$. More generally, $\mathbb{X}(\overline{\Omega}_i) = H^s(\Omega_i)$ for appropriate Sobolev regularity $s \geq 2$.
    \item $\mathbb{X}(\overline{\Omega}_i(t) \times [0, T])$ denotes the corresponding space-time function space for parabolic problems.
    \item $\|\cdot\|_{\mathbb{X}(\overline{\Omega}_i)}$ denotes the norm in the space $\mathbb{X}(\overline{\Omega}_i)$.
\end{itemize}

\begin{remark}[Scope and Limitations of the Error Analysis]
The error bounds presented in this section provide \emph{upper bounds} on the generalization error, relating the solution error to the trained loss values and quadrature errors. These bounds are conditional on several assumptions:
\begin{itemize}
    \item The neural network has sufficient approximation capacity to represent the solution with small residuals;
    \item The optimization algorithm successfully finds parameters that minimize the loss function;
    \item The sampling points are distributed appropriately to approximate the relevant integrals.
\end{itemize}
The actual error in practice depends on factors not captured by these bounds, including the choice of network architecture, optimization algorithm, initialization, and hyperparameters. Therefore, while these theoretical results provide useful guidance, the numerical experiments in Section~\ref{NumericalResults} serve as the primary validation of our method's effectiveness.
\end{remark}

\subsection{The Least-Squares DNN Approach for Elliptic Interface Problems}\label{Least-SquaresE}

It is natural to rewrite the elliptic interface problem \eqref{modelpextenddnn} as a minimization problem in order to apply the DNN approach. There are different ways to do this, following similar approaches in the literature. In this work, we adopt the LS formulation which minimizes residuals by defining an LS functional that incorporates all equations introduced in \eqref{modelpextenddnn} together. This approach has been widely used in physics-informed neural networks for various PDEs, including elliptic interface problems.

In particular, for \eqref{modelpextenddnn}, we define a least-squares formulation as follows:

\begin{equation}
\begin{aligned}
    \mathcal{R}(\tilde{u}) &:= \sum_{i=1}^{2} \left(\omega_{\mathcal{L}_i} \|\mathcal{L}_i(\tilde{u})\|_{0,\Omega_i}^2 + \omega_{\mathcal{B}_i} \|\mathcal{B}_i(\tilde{u})\|_{0,\partial\Omega_i / \Gamma}^2\right)\\
    &+ \omega_{\Gamma_1} \|\mathcal{I}_1(\tilde{u})\|_{0,\Gamma}^2 + \omega_{\Gamma_2} \|\mathcal{I}_2(\tilde{u})\|_{0,\Gamma}^2,
\end{aligned}
\end{equation}
where $\omega_{\mathcal{L}_i}$, $\omega_{\mathcal{B}}$, $\omega_{\Gamma_1}$, and $\omega_{\Gamma_2}$ are prescribed weight coefficients for each corresponding term. The operators are defined as:
\begin{equation}
    \begin{aligned}
\mathcal{L}_i(\tilde{u}) &= -\nabla \cdot (\beta(\boldsymbol{x}) \nabla \tilde{u}) - f, \quad \text{in } \Omega_i, \quad i=1,2, \\
\mathcal{B}_i(\tilde{u}) &= \tilde{u} - g_D, \quad \text{on } \partial\Omega_i / \Gamma, \\
\mathcal{I}_1(\tilde{u}) &= \llbracket \tilde{u} \rrbracket - g_1, \quad \text{on } \Gamma, \\
\mathcal{I}_2(\tilde{u}) &= \llbracket \beta(\boldsymbol{x}) \nabla \tilde{u} \cdot \boldsymbol{n} \rrbracket - g_2, \quad \text{on } \Gamma.
\end{aligned}
\end{equation}

Then, the LS problem is to find $u \in \mathbb{V}$ such that
\begin{equation}
    \mathcal{R}(u) = \min_{\tilde{u} \in \mathbb{V}} \mathcal{R}(\tilde{u}),
\end{equation}
where $\mathbb{V}$ is a suitable space for the LS solution.

As shown in Section \ref{Multi-ActivationFunction}, we employ two independent neural networks to approximate the exact solution $u$ in \eqref{modelpextenddnn}:
$$
u_i \approx u_{\boldsymbol{\theta}_i}(\boldsymbol{x}), \quad \boldsymbol{x} \in \Omega_i, \quad i = 1, 2,
$$
where $\boldsymbol{x} \in \mathbb{R}^d$, and $d$ is the spatial dimension. Each network $u_{\boldsymbol{\theta}_i}$ has the same architecture but with independently trainable parameters $\boldsymbol{\theta}_i$.

Based on the LS formulation and the DNN approximation, we seek to find the parameter pairs $(\boldsymbol{\theta}_1, \boldsymbol{\theta}_2)$ such that:

$$
\min_{(\boldsymbol{\theta}_1, \boldsymbol{\theta}_2) \in \mathbb{R}^M \times \mathbb{R}^M} \mathcal{R}(\boldsymbol{x};\boldsymbol{\theta}_1, \boldsymbol{\theta}_2),
$$
and the total loss functional, $\mathcal{R}(\boldsymbol{x};\boldsymbol{\theta}_1, \boldsymbol{\theta}_2)$, according to the LS functional, is defined as:

\begin{equation}
    \mathcal{R}(\boldsymbol{x};\boldsymbol{\theta}_1, \boldsymbol{\theta}_2) := \mathcal{R}(u_{\boldsymbol{\theta}_1}(\boldsymbol{x}), u_{\boldsymbol{\theta}_2}(\boldsymbol{x})).
\end{equation}

Since the loss function contains integrals that cannot be computed exactly, we approximate them by the Monte Carlo integration based on a set of sampling points. More precisely, we introduce the following mean square error (MSE) formulations to approximate corresponding integrals:

\begin{equation}
    \begin{aligned}
\mathcal{F}_{\mathcal{L}_i}(\boldsymbol{\theta}_i) &:= \frac{1}{M_{i}} \sum_{k=1}^{M_{i}} |\mathcal{L}_i(\boldsymbol{x}_k; \boldsymbol{\theta}_i)|^2, \quad i=1,2, \\
\mathcal{F}_{\mathcal{B}_i}(\boldsymbol{\theta}_i) &:= \frac{1}{M_{\mathcal{B}_i}} \sum_{k=1}^{M_{\mathcal{B}_i}} |\mathcal{B}_i(\boldsymbol{x}_k; \boldsymbol{\theta}_i)|^2, \\
\mathcal{F}_{\mathcal{I}_1}(\boldsymbol{\theta}_1, \boldsymbol{\theta}_2) &:= \frac{1}{M_{\Gamma}} \sum_{k=1}^{M_{\Gamma}} |\mathcal{I}_1(\boldsymbol{x}_k; \boldsymbol{\theta}_1, \boldsymbol{\theta}_2)|^2, \\
\mathcal{F}_{\mathcal{I}_2}(\boldsymbol{\theta}_1, \boldsymbol{\theta}_2) &:= \frac{1}{M_{\Gamma}} \sum_{k=1}^{M_{\Gamma}} |\mathcal{I}_2(\boldsymbol{x}_k; \boldsymbol{\theta}_1, \boldsymbol{\theta}_2)|^2,
\end{aligned}
\end{equation}
where $M_{i}$, $M_{\mathcal{B}_i}$, and $M_{\Gamma}$ are the number of sampling points used for the interior part in each subdomain, the boundary part ($M_{\mathcal{B}_1} + M_{\mathcal{B}_2} = M_{\partial \Omega}$), and the two interface conditions, respectively. If one of them equals 0, then the corresponding loss functional or the MSE is dropped. By slightly abusing the notation, we use $\boldsymbol{x}_k$ $(k = 1, \cdots)$ to represent sampling points and drop the dependence of $\boldsymbol{x}$ on the discrete level.

Thus, the total discrete loss is defined as:
$$
\mathcal{F}(\boldsymbol{\theta}_1, \boldsymbol{\theta}_2) := \sum_{i=1}^{2} \left(\omega_{\mathcal{L}_i} \mathcal{F}_{\mathcal{L}_i}(\boldsymbol{\theta}_i) + \omega_{\mathcal{B}_i} \mathcal{F}_{\mathcal{B}_i}(\boldsymbol{\theta}_i)\right) + \omega_{\Gamma_1} \mathcal{F}_{\mathcal{I}_1}(\boldsymbol{\theta}_1, \boldsymbol{\theta}_2) + \omega_{\Gamma_2} \mathcal{F}_{\mathcal{I}_2}(\boldsymbol{\theta}_1, \boldsymbol{\theta}_2),
$$
and our DNN method for (1) solves the following minimization problem:
$$
\min_{(\boldsymbol{\theta}_1, \boldsymbol{\theta}_2) \in \mathbb{R}^M \times \mathbb{R}^M} \mathcal{F}(\boldsymbol{\theta}_1, \boldsymbol{\theta}_2) = \min_{(\boldsymbol{\theta}_1, \boldsymbol{\theta}_2)} \left[ \sum_{i=1}^{2} \left(\omega_{\mathcal{L}_i} \mathcal{F}_{\mathcal{L}_i}(\boldsymbol{\theta}_i) + \omega_{\mathcal{B}_i} \mathcal{F}_{\mathcal{B}_i}(\boldsymbol{\theta}_i)\right) + \omega_{\Gamma_1} \mathcal{F}_{\mathcal{I}_1}(\boldsymbol{\theta}_1, \boldsymbol{\theta}_2) + \omega_{\Gamma_2} \mathcal{F}_{\mathcal{I}_2}(\boldsymbol{\theta}_1, \boldsymbol{\theta}_2) \right].
$$

%To solve this minimization problem, we use standard optimization algorithms such as the SGD method. In practice, we use a variant of SGD, the so-called adaptive moment estimation (Adam). The gradients used in the gradient descent method are computed via backward propagation and automatic differentiation. When the approximated minimizer $\boldsymbol{\theta}^*$ is reached, we attain the desired DNN result, $u_{\boldsymbol{\theta}}(\boldsymbol{x}; \boldsymbol{\theta}^*)$, which is the numerical solution of the elliptic interface problem \eqref{modelpextenddnn}.

To solve the elliptic interface problem using the proposed method, we formulate the following total LS formulation for $\tilde{u} \in H^1(\Omega_1) \cup H^1(\Omega_2)$:
\begin{equation}
    \begin{aligned}
\mathcal{R}(\tilde{u}) := &\sum_{i=1}^{2} \left(\omega_{\mathcal{L}_i} \int_{\Omega_i} \left|-\nabla \cdot (\beta(\boldsymbol{x}) \nabla \tilde{u}_{i}) - f\right|^2 \, d\boldsymbol{x} \right.\\
&\left.+ \omega_{\mathcal{B}_i} \int_{\partial \Omega_i / \Gamma} \left|\tilde{u}_i - g_D\right|^2 \, ds\right) \\
&+ \omega_{\Gamma_1} \int_{\Gamma} \left|\tilde{u}_1 - \tilde{u}_2 - g_1\right|^2 \, ds \\
&+ \omega_{\Gamma_2} \int_{\Gamma} \left|\beta_1 \nabla \tilde{u}_1 \cdot \boldsymbol{n}_1 - \beta_2 \nabla \tilde{u}_2 \cdot \boldsymbol{n}_2 - g_2\right|^2 \, ds,
\end{aligned}
\end{equation}
where $\tilde{u}_i = \tilde{u}|_{\Omega_i}$ denotes the restriction of $\tilde{u}$ to subdomain $\Omega_i$. 
%The weight coefficients $\omega_{\mathcal{L}_i}$, $\omega_{\mathcal{B}}$, $\omega_{\Gamma_1}$, and $\omega_{\Gamma_2}$ balance the contribution of each term to the overall loss.
The DNN approximation of this problem defines $u_{\boldsymbol{\theta}_i}(\boldsymbol{x})$ as an approximation of $u$ in subdomain $\Omega_i$, leading to the discretized loss function:
\begin{equation}
    \begin{aligned}
\mathcal{F}(\boldsymbol{\theta}_1, \boldsymbol{\theta}_2) &= \sum_{i=1}^{2} \left(\omega_{\mathcal{L}_i} \mathcal{F}_{\mathcal{L}_i}(\boldsymbol{\theta}_i) + \omega_{\mathcal{B}_i} \mathcal{F}_{\mathcal{B}_i}(\boldsymbol{\theta}_i) \right)+ \omega_{\Gamma_1} \mathcal{F}_{\mathcal{I}_1}(\boldsymbol{\theta}_1, \boldsymbol{\theta}_2) + \omega_{\Gamma_2} \mathcal{F}_{\mathcal{I}_2}(\boldsymbol{\theta}_1, \boldsymbol{\theta}_2) \\
&= \sum_{i=1}^{2} \left(\omega_{\mathcal{L}_i} \frac{1}{M_{i}} \sum_{k=1}^{M_{i}} \left|-\nabla \cdot (\beta(\boldsymbol{x}_k) \nabla u_{\boldsymbol{\theta}_i}(\boldsymbol{x}_k)) - f(\boldsymbol{x}_k)\right|^2 \right.\\
&\left.+ \omega_{\mathcal{B}_i} \frac{1}{M_{\mathcal{B}_i}} \sum_{k=1}^{M_{\mathcal{B}_i}} \left|u_{\boldsymbol{\theta}_i}(\boldsymbol{x}_k) - g_D(\boldsymbol{x}_k)\right|^2 \right)\\
&+ \omega_{\Gamma_1} \frac{1}{M_{\Gamma}} \sum_{k=1}^{M_{\Gamma}} \left|u_{\boldsymbol{\theta}_1}(\boldsymbol{x}_k) - u_{\boldsymbol{\theta}_2}(\boldsymbol{x}_k) - g_1(\boldsymbol{x}_k)\right|^2 \\
&+ \omega_{\Gamma_2} \frac{1}{M_{\Gamma}} \sum_{k=1}^{M_{\Gamma}} \left|\beta_1 \nabla u_{\boldsymbol{\theta}_1}(\boldsymbol{x}_k) \cdot \boldsymbol{n} - \beta_2 \nabla u_{\boldsymbol{\theta}_2}(\boldsymbol{x}_k) \cdot \boldsymbol{n} - g_2(\boldsymbol{x}_k)\right|^2.
\end{aligned}
\end{equation}

The optimization problem is then solved using the Adam optimizer, with gradients computed through automatic differentiation. The optimal parameter sets $(\boldsymbol{\theta}_1^*, \boldsymbol{\theta}_2^*)$ give the approximate solutions $u_{\boldsymbol{\theta}_1^*}(\boldsymbol{x})$ and $u_{\boldsymbol{\theta}_2^*}(\boldsymbol{x})$ for the elliptic interface problem in $\Omega_1$ and $\Omega_2$, respectively.

\subsection{Error Estimation for the Elliptic Interface Problem}

\subsubsection{Abstract Framework for Error Analysis}
Following the abstract framework developed by Mishra and Molinaro \cite{mishra2023estimates} for PINNs error analysis, we establish two major assumptions necessary for the error analysis of our two-network domain decomposition method applied to the elliptic interface problem \eqref{modelpextenddnn}. The key insight from \cite{mishra2023estimates} is to formulate stability-type conditions that relate the solution error to the PDE residuals. In our setting, we have two independent networks $u_{\boldsymbol{\theta}_1}$ and $u_{\boldsymbol{\theta}_2}$ approximating the solution in $\Omega_1$ and $\Omega_2$ respectively, coupled through interface conditions.
%[PDE Stability Condition for Two-Network Approximation]
\begin{assumption}\label{e1}
We assume that the elliptic interface problem \eqref{modelpextenddnn} satisfies a \emph{conditional stability estimate}: for the exact solution $u_i$ in subdomain $\Omega_i$ and any sufficiently regular approximation $v_i \in \mathbb{X}(\overline{\Omega}_i)$ (in our case, $v_i = u_{\boldsymbol{\theta}_i}$ from network $i$), $i = 1, 2$, there exist exponents $0 < \beta_{\mathcal{L}_i}, \beta_{\mathcal{B}_i}, \beta_{\Gamma} \leq 1$ such that
\begin{equation}
\begin{aligned}
    \sum_{i=1}^{2} \|u_i - v_i\|^2_{\mathbb{X}(\overline{\Omega}_i)} &\leq C_{pde} \left\{ \sum_{i=1}^{2} \left[ \left(\|-\nabla \cdot (\beta(\boldsymbol{x})\nabla u_i) - (-\nabla \cdot (\beta(\boldsymbol{x})\nabla v_i))\|^2_{L^2(\Omega_i)}\right)^{\beta_{\mathcal{L}_i}} \right. \right.\\
&\left. + \left(\|u_i - v_i\|^2_{L^2(\partial\Omega_i \setminus \Gamma)}\right)^{\beta_{\mathcal{B}_i}}  \right] + \left(\|\llbracket u \rrbracket - \llbracket v \rrbracket\|^2_{L^2(\Gamma)}\right)^{\beta_{\Gamma}}\\
&\left.  + \left(\|\llbracket \beta(\boldsymbol{x}) \nabla u \cdot \boldsymbol{n} \rrbracket - \llbracket \beta(\boldsymbol{x}) \nabla v \cdot \boldsymbol{n} \rrbracket\|^2_{L^2(\Gamma)}\right)^{\beta_{\Gamma}} \right\},
\end{aligned}
\end{equation}
where the constant $C_{pde}$ depends on $\|u_i\|_{\mathbb{X}(\overline{\Omega}_i)}$, $\|v_i\|_{\mathbb{X}(\overline{\Omega}_i)}$, the coefficient $\beta(\boldsymbol{x})$, the domain geometry, and the regularity of the underlying elliptic interface problem.

This assumption encapsulates a stability property of the PDE: small residuals in the PDE, boundary conditions, and interface conditions imply small errors in the solution. The exponents $\beta_{\mathcal{L}_i}, \beta_{\mathcal{B}_i}, \beta_{\Gamma}$ depend on the specific problem and may be less than 1 for ill-posed or weakly stable problems. For typical well-posed elliptic interface problems with smooth data, such stability estimates hold with $\beta_{\mathcal{L}_i} = \beta_{\mathcal{B}_i} = \beta_{\Gamma} = 1$; see, e.g., \cite{gilbarg2001elliptic, chen1998finite} for classical regularity theory of elliptic equations with discontinuous coefficients. Note that the constant $C_{pde}$ depends on the coefficient contrast ratio $\beta_1/\beta_2$; for high-contrast problems (e.g., $\beta_1/\beta_2 \gg 1$ or $\ll 1$), this constant may become large, which should be considered when interpreting the error bounds.
\end{assumption}

\begin{assumption}\label{e2}
Consider an integral $I(g) := \int_{\mathcal{D}} g(\boldsymbol{x}) \, d\boldsymbol{x}$ and a quadrature rule using $N$ quadrature points $\boldsymbol{x}_i \in \mathcal{D}$ $(1 \leq i \leq N)$, i.e., $I_N(g) := \sum_{i=1}^{N} w_i g(\boldsymbol{x}_i)$, $w_i \in \mathbb{R}^+$, $1 \leq i \leq N$. We assume the quadrature error is
\begin{equation}
|I(g) - I_N(g)| \leq C_{quad} N^{-\alpha}, \quad \alpha > 0,
\end{equation}
where the constant $C_{quad}$ depends on the dimension of the domain and the property of the integrand $g(\boldsymbol{x})$.
\end{assumption}

Based on the above two assumptions, we have the following framework of the error estimation for the proposed two-network domain decomposition method applied to the elliptic interface problem. The key feature is that the error bound explicitly accounts for the separate approximations $u_{\boldsymbol{\theta}_1}$ and $u_{\boldsymbol{\theta}_2}$ in each subdomain and their coupling through the interface loss terms.
%[Abstract Error Framework for Two-Network Domain Decomposition]
\begin{theorem}
Let $u_i \in \mathbb{X}(\overline{\Omega}_i)$, $i = 1, 2$, be the classical solution of the elliptic interface problem in subdomain $\Omega_i$. Let $u_{\boldsymbol{\theta}_i}^{*} \in \mathbb{X}(\overline{\Omega}_i)$ be the numerical approximation obtained by the neural network with parameters $\boldsymbol{\theta}_i^*$ for subdomain $\Omega_i$. Under Assumptions \ref{e1} and \ref{e2}, we have the following generalization error estimation:
\begin{equation}
\begin{aligned}
    \sum_{i=1}^{2} \|u_i - u_{\boldsymbol{\theta}_i}^{*}\|^2_{\mathbb{X}(\overline{\Omega}_i)} &\leq C_{\text{pde}} \left\{ \sum_{i=1}^{2} \left[ (\mathcal{F}_{\mathcal{L}_i}(\boldsymbol{\theta}_i^*))^{\beta_{\mathcal{L}_i}} + (\mathcal{F}_{\mathcal{B}_i}(\boldsymbol{\theta}_i^*))^{\beta_{\mathcal{B}_i}} \right] \right.\\
&\left. + (\mathcal{F}_{\mathcal{I}_1}(\boldsymbol{\theta}_1^*, \boldsymbol{\theta}_2^*))^{\beta_{\Gamma}} + (\mathcal{F}_{\mathcal{I}_2}(\boldsymbol{\theta}_1^*, \boldsymbol{\theta}_2^*))^{\beta_{\Gamma}}\right.\\
&\left. + \sum_{i=1}^{2} \left( C_{\text{quad}}^{\mathcal{L}_i}M_{i}^{-\alpha_{\mathcal{L}_i}\beta_{\mathcal{L}_i}} + C_{\text{quad}}^{\mathcal{B}_i}M_{\mathcal{B}_i}^{-\alpha_{\mathcal{B}_i}\beta_{\mathcal{B}_i}}  \right) \right.\\
&\left. + C_{\text{quad}}^{\mathcal{I}_1}M_{\Gamma}^{-\alpha_{\mathcal{I}_1}\beta_{\Gamma}} + C_{\text{quad}}^{{\mathcal{I}_2}}M_{\Gamma}^{-\alpha_{\mathcal{I}_2}\beta_{\Gamma}} \right\},
\end{aligned}
\end{equation}
where $C_{\text{pde}}$ is a positive constant depending on $\|u_i\|_{\mathbb{X}(\overline{\Omega}_i)}$, $\|u_{\boldsymbol{\theta}_i}^{*}\|_{\mathbb{X}(\overline{\Omega}_i)}$, $\beta(\boldsymbol{x})$, and the domains $\Omega_i$ $(i = 1, 2)$. Note that the interior and boundary loss terms $\mathcal{F}_{\mathcal{L}_i}$ and $\mathcal{F}_{\mathcal{B}_i}$ depend only on the parameters $\boldsymbol{\theta}_i$ of network $i$, while the interface loss terms $\mathcal{F}_{\mathcal{I}_1}$ and $\mathcal{F}_{\mathcal{I}_2}$ depend on both parameter sets $(\boldsymbol{\theta}_1, \boldsymbol{\theta}_2)$, reflecting the coupling between the two networks.
\end{theorem}

\begin{proof}
Let $e^u_i := u_i - u_{\boldsymbol{\theta}_i}^{*}$ be the error in subdomain $\Omega_i$, where $u_{\boldsymbol{\theta}_i}^{*}$ is the approximation from network $i$ with optimal parameters $\boldsymbol{\theta}_i^*$. By a direct calculation, we have, for $i = 1, 2$,
\begin{equation}
\begin{aligned}
-\nabla \cdot \bigl(\beta(\boldsymbol{x}) \nabla e^u_i\bigr) &= \mathcal{R}^f_i \quad \text{in } \Omega_i, \\
 e^u_1 - e^u_2&= \mathcal{R}^{g_1} \quad \text{on } \Gamma, \\
 \beta_1 \nabla e^u_1 \cdot \boldsymbol{n} - \beta_2 \nabla e^u_2 \cdot \boldsymbol{n}&= \mathcal{R}^{g_2} \quad \text{on } \Gamma, \\
e^u_i &= \mathcal{R}^{g_D}_i \quad \text{on } \partial \Omega_i \setminus \Gamma,
\end{aligned}
\end{equation}
where $\mathcal{R}^f_i := \nabla \cdot \bigl(\beta_i \nabla u_{\boldsymbol{\theta}_i}^{*}\bigr) + f$ is the PDE residual for network $i$, $\mathcal{R}^{g_1} := u_{\boldsymbol{\theta}_1}^{*} - u_{\boldsymbol{\theta}_2}^{*} - g_1$ is the solution jump residual (coupling networks 1 and 2), $\mathcal{R}^{g_2} := \beta_1 \nabla u_{\boldsymbol{\theta}_1}^{*} \cdot \boldsymbol{n} - \beta_2 \nabla u_{\boldsymbol{\theta}_2}^{*} \cdot \boldsymbol{n} - g_2$ is the flux jump residual (also coupling both networks), and $\mathcal{R}^{g_D}_i := u_{\boldsymbol{\theta}_i}^{*} - g_D$ is the boundary residual for network $i$.

Using Assumption \ref{e1}, we have
\begin{equation}
\begin{aligned}
\sum_{i=1}^{2} \|u_i - u_{\boldsymbol{\theta}_i}^{*}\|^2_{\mathbb{X}(\overline{\Omega}_i)} &\leq C_{pde} \left\{ \sum_{i=1}^{2} \left[ \left(\|\mathcal{R}^f_i\|^2_{L^2(\Omega_i)}\right)^{\beta_{\mathcal{L}_i}} + \left(\|\mathcal{R}^{g_D}_i\|^2_{L^2(\partial\Omega_i \setminus \Gamma)}\right)^{\beta_{\mathcal{B}_i}}\right] \right. \\
&\quad \left. + \left(\|\mathcal{R}^{g_1}\|^2_{L^2(\Gamma)}\right)^{\beta_{\Gamma}} + \left(\|\mathcal{R}^{g_2}\|^2_{L^2(\Gamma)}\right)^{\beta_{\Gamma}} \right\}.
\end{aligned}
\end{equation}

To further estimate the terms on the right-hand side, we use Assumption \ref{e2}. For the interior residual of network $i$, we have
\begin{equation}
\left(\|\mathcal{R}^f_i\|^2_{L^2(\Omega_i)}\right)^{\beta_{\mathcal{L}_i}} \leq (\mathcal{F}_{\mathcal{L}_i}(\boldsymbol{\theta}_i^*))^{\beta_{\mathcal{L}_i}} + C_{\text{quad}}^{\mathcal{L}_i}M_{i}^{-\alpha_{\mathcal{L}_i}\beta_{\mathcal{L}_i}}.
\end{equation}

Similarly, for the boundary residual (depending on $\boldsymbol{\theta}_i$) and the interface residuals (depending on both $\boldsymbol{\theta}_1$ and $\boldsymbol{\theta}_2$):
\begin{equation}
\begin{aligned}
\left(\|\mathcal{R}^{g_D}_i\|^2_{L^2(\partial\Omega_i \setminus \Gamma)}\right)^{\beta_{\mathcal{B}_i}} &\leq (\mathcal{F}_{\mathcal{B}_i}(\boldsymbol{\theta}_i^*))^{\beta_{\mathcal{B}_i}} + C_{\text{quad}}^{\mathcal{B}_i}M_{\mathcal{B}_i}^{-\alpha_{\mathcal{B}_i}\beta_{\mathcal{B}_i}},  \\
\left(\|\mathcal{R}^{g_1}\|^2_{L^2(\Gamma)}\right)^{\beta_{\Gamma}} &\leq (\mathcal{F}_{\mathcal{I}_1}(\boldsymbol{\theta}_1^*, \boldsymbol{\theta}_2^*))^{\beta_{\Gamma}} + C_{\text{quad}}^{\mathcal{I}_1}M_{\Gamma}^{-\alpha_{\mathcal{I}_1}\beta_{\Gamma}},\\
\left(\|\mathcal{R}^{g_2}\|^2_{L^2(\Gamma)}\right)^{\beta_{\Gamma}} &\leq (\mathcal{F}_{\mathcal{I}_2}(\boldsymbol{\theta}_1^*, \boldsymbol{\theta}_2^*))^{\beta_{\Gamma}} + C_{\text{quad}}^{\mathcal{I}_2}M_{\Gamma}^{-\alpha_{\mathcal{I}_2}\beta_{\Gamma}}.
\end{aligned}
\end{equation}

Substituting these estimates into the previous inequality completes the proof. The key observation is that the interface loss terms couple both networks through the parameters $(\boldsymbol{\theta}_1^*, \boldsymbol{\theta}_2^*)$, while the interior and boundary terms depend only on the respective subdomain network.
\end{proof}

\subsubsection{Application to the Elliptic Interface Problem}

Following the abstract error analysis framework established in the previous section, we now apply our theoretical results to the specific elliptic interface problem \eqref{modelpextenddnn}. Let $e_i^u := u_i - u_{\boldsymbol{\theta}}^{i,*}$, $i = 1, 2$, be the errors in each subdomain. 

The following theorem provides an error bound under the stability assumption. We note that similar results have been established in \cite{mishra2023estimates} for standard PINNs; here we adapt the framework to interface problems.
%[Abstract Error Framework for Two-Network Domain Decomposition]
\begin{theorem}\label{thm:errorEstimation}
Let $u_i \in C^2(\overline{\Omega}_i)$, $i = 1, 2$, be the classical solution in each subdomain of the elliptic interface problem \eqref{modelpextenddnn}. Let $u_{\boldsymbol{\theta}_i}^{*} \in C^2(\overline{\Omega}_i)$ be the numerical approximations obtained by the DNN method. Under Assumptions~\ref{e1} and~\ref{e2}, with the stability exponents $\beta_{\mathcal{L}_i} = \beta_{\mathcal{B}_i} = \beta_{\Gamma} = 1$ (well-posed case), the following error bound holds:
\begin{equation}
\begin{aligned}
\sum_{i=1}^{2} \|e_i^u\|^2_{\mathbb{X}(\overline{\Omega}_i)} \leq C_{pde} \left[ \sum_{i=1}^{2} \mathcal{F}_{\mathcal{L}_i}(\boldsymbol{\theta}_i^*) + \sum_{i=1}^{2} \mathcal{F}_{\mathcal{B}_i}(\boldsymbol{\theta}_i^*) + \mathcal{F}_{\mathcal{I}_1}(\boldsymbol{\theta}_1^*, \boldsymbol{\theta}_2^*) + \mathcal{F}_{\mathcal{I}_2}(\boldsymbol{\theta}_1^*, \boldsymbol{\theta}_2^*) \right.\\
\left. + \sum_{i=1}^{2} C_{\text{quad}}^{\mathcal{L}_i}M_{i}^{-\alpha_{\mathcal{L}_i}} + \sum_{i=1}^{2} C_{\text{quad}}^{\mathcal{B}_i}M_{\mathcal{B}_i}^{-\alpha_{\mathcal{B}_i}} + C_{\text{quad}}^{\mathcal{I}_1}M_{\Gamma}^{-\alpha_{\mathcal{I}_1}} + C_{\text{quad}}^{\mathcal{I}_2}M_{\Gamma}^{-\alpha_{\mathcal{I}_2}} \right],
\end{aligned}
\end{equation}
where $C_{pde}$ is the stability constant from Assumption~\ref{e1}, depending on $\|u_i\|_{\mathbb{X}(\overline{\Omega}_i)}$, $\|u_{\boldsymbol{\theta}_i}^{*}\|_{\mathbb{X}(\overline{\Omega}_i)}$, the coefficient $\beta(\boldsymbol{x})$, and the domain geometry.
\end{theorem}

\begin{proof}
The proof follows directly from applying Assumptions~\ref{e1} and~\ref{e2}. By Assumption~\ref{e1}, the solution error is bounded by the residuals:
\begin{equation}
\begin{aligned}
\sum_{i=1}^{2} \|e_i^u\|^2_{\mathbb{X}(\overline{\Omega}_i)} &\leq C_{pde} \left\{ \sum_{i=1}^{2} \|\mathcal{R}_i^f\|^2_{L^2(\Omega_i)} + \sum_{i=1}^{2} \|\mathcal{R}_i^{g_D}\|^2_{L^2(\partial\Omega_i \setminus \Gamma)} \right.\\
&\quad \left. + \|\mathcal{R}^{g_1}\|^2_{L^2(\Gamma)} + \|\mathcal{R}^{g_2}\|^2_{L^2(\Gamma)} \right\},
\end{aligned}
\end{equation}
where $\mathcal{R}_i^f$, $\mathcal{R}_i^{g_D}$, $\mathcal{R}^{g_1}$, and $\mathcal{R}^{g_2}$ are the PDE, boundary, and interface residuals of the neural network approximation.

By Assumption~\ref{e2}, the $L^2$ norms of the residuals can be related to the discrete loss functions plus quadrature errors. For example,
\[
\|\mathcal{R}_i^f\|^2_{L^2(\Omega_i)} \leq \mathcal{F}_{\mathcal{L}_i}(\boldsymbol{\theta}_i^*) + C_{\text{quad}}^{\mathcal{L}_i}M_{i}^{-\alpha_{\mathcal{L}_i}}.
\]
Applying similar bounds to all residual terms and combining yields the stated result.
\end{proof}

\begin{remark}[Interpretation and Limitations]
The error bound in Theorem~\ref{thm:errorEstimation} should be interpreted carefully. The theorem guarantees that \emph{if} a small-loss parameter pair $(\boldsymbol{\theta}_1^*, \boldsymbol{\theta}_2^*)$ is found, \emph{then} the generalization error is small, provided the stability constant $C_{pde}$ is not too large. Importantly, \emph{it does not guarantee that the optimization process will find such a parameter pair}. For interface problems with high-contrast coefficients (e.g., $\beta_1/\beta_2 \gg 1$ or $\ll 1$), the constant $C_{pde}$ may be large, potentially degrading the error bound. The theorem assumes that $(\boldsymbol{\theta}_1^*, \boldsymbol{\theta}_2^*)$ is a (local) minimizer of the joint loss function $\mathcal{F}(\boldsymbol{\theta}_1, \boldsymbol{\theta}_2)$; in practice, optimization algorithms may not find the global minimum, introducing an additional optimization error not captured by this bound. Note that since the two networks are coupled through the interface loss terms, the optimization must jointly minimize over both parameter sets. The bound also assumes both neural networks can achieve small residuals in their respective subdomains, requiring sufficient network width, depth, and appropriate activation functions---conditions that the multi-activation strategy in Section~\ref{Multi-Activation} is designed to facilitate.
\end{remark}

\begin{remark}[Role of Multi-Activation Functions]
The error bound in Theorem~\ref{thm:errorEstimation} requires activation functions that are sufficiently smooth to allow evaluation of second-order derivatives in the PDE residual. The $\tanh$ and Gaussian-type activations used in our multi-activation strategy (Section~\ref{Multi-Activation}) satisfy this requirement, being $C^\infty$ functions. In contrast, non-smooth activations such as ReLU are not suitable for second-order PDEs since their second derivatives are not well-defined (ReLU has a discontinuous first derivative at zero). The multi-activation strategy contributes to achieving smaller loss values $\mathcal{F}(\boldsymbol{\theta}^*)$ in practice by providing specialized activation behavior near the interface, where satisfying the jump conditions is most challenging. This practical benefit is demonstrated in the numerical experiments of Section~\ref{NumericalResults}.
\end{remark}

\subsection{The Least-Squares DNN Approach for Parabolic Interface Problems}\label{Least-SquaresP}

We extend the least-squares (LS) formulation approach to the time-dependent parabolic interface problem \eqref{modelpextenddnnpaowu}. This extension introduces additional complexity due to the temporal dimension and the potentially moving interface $\Gamma(t)$.

Following the approach presented for elliptic problems, we rewrite the parabolic interface problem as a minimization problem suitable for the DNN approach. We define an LS functional that incorporates all equations from \eqref{modelpextenddnnpaowu} as follows:

\begin{equation}
\begin{aligned}
    \mathcal{R}(\tilde{u}) &:= \sum_{i=1}^{2} \left(\omega_{\mathcal{L}_i} \|\mathcal{L}_i(\tilde{u})\|_{0,\Omega_i(t)}^2 + \omega_{\mathcal{B}_i} \|\mathcal{B}_i(\tilde{u})\|_{0,\partial\Omega_i(t) / \Gamma(t)}^2 + \omega_{\mathcal{T}_i} \|\mathcal{T}_i(\tilde{u})\|_{0,\Omega_i(0) }^2\right)  \\
    & + \omega_{\Gamma_1} \|\mathcal{I}_1(\tilde{u})\|_{0,\Gamma(t)}^2+ \omega_{\Gamma_2} \|\mathcal{I}_2(\tilde{u})\|_{0,\Gamma(t)}^2 ,
\end{aligned}
\end{equation}
where $\omega_{\mathcal{L}_i}$, $\omega_{\mathcal{B}_i}$, $\omega_{\Gamma_1}$, $\omega_{\Gamma_2}$, and $\omega_{\mathcal{T}_i}$ are prescribed weight coefficients for each corresponding term. The operators are defined as:
\begin{equation}
    \begin{aligned}
\mathcal{L}_i(\tilde{u}) &= \tilde{u}_t - \nabla \cdot (\beta(\boldsymbol{x}) \nabla \tilde{u}) - f, \quad \text{in } \Omega_i(t), \quad i=1,2, \\
\mathcal{B}_i(\tilde{u}) &= \tilde{u} - g_D, \quad \text{on } \partial\Omega_i(t) / \Gamma(t), \\
\mathcal{I}_1(\tilde{u}) &= \llbracket \tilde{u} \rrbracket - g_1, \quad \text{on } \Gamma(t), \\
\mathcal{I}_2(\tilde{u}) &= \llbracket \beta(\boldsymbol{x}) \nabla \tilde{u} \cdot \boldsymbol{n} \rrbracket - g_2, \quad \text{on } \Gamma(t), \\
\mathcal{T}_i(\tilde{u}) &= \tilde{u}(\boldsymbol{x},0) - g_0, \quad \text{on } \Omega_i(0).
\end{aligned}
\end{equation}

The LS problem is to find $u \in \mathbb{V}$ such that
\begin{equation}
    \mathcal{R}(u) = \min_{\tilde{u} \in \mathbb{V}} \mathcal{R}(\tilde{u}),
\end{equation}
where $\mathbb{V}$ is a suitable space to which the LS solution belongs, considering both spatial and temporal regularity requirements.

To approximate the exact solution $u$ in \eqref{modelpextenddnnpaowu}, we employ two independent neural networks:
$$
u_i \approx u_{\boldsymbol{\theta}_i}(\boldsymbol{x},t), \quad \boldsymbol{x} \in \Omega_i(t), \quad i = 1, 2,
$$
where $\boldsymbol{x} \in \mathbb{R}^d$, $t \in [0,T]$, and $d$ is the spatial dimension. Each network $u_{\boldsymbol{\theta}_i}$ has the same architecture but with independently trainable parameters $\boldsymbol{\theta}_i$.

Based on the LS formulation and the DNN approximation, our parabolic DNN method is defined as finding $(\boldsymbol{\theta}_1, \boldsymbol{\theta}_2)$ such that:

$$
\min_{(\boldsymbol{\theta}_1, \boldsymbol{\theta}_2) \in \mathbb{R}^M \times \mathbb{R}^M} \mathcal{R}(\boldsymbol{x},t;\boldsymbol{\theta}_1, \boldsymbol{\theta}_2),
$$
and the total loss functional is defined as:

\begin{equation}
    \mathcal{R}(\boldsymbol{x},t;\boldsymbol{\theta}_1, \boldsymbol{\theta}_2) := \mathcal{R}(u_{\boldsymbol{\theta}_1}(\boldsymbol{x},t), u_{\boldsymbol{\theta}_2}(\boldsymbol{x},t)).
\end{equation}

Since the loss function contains space-time integrals that cannot be computed exactly, we approximate them using Monte Carlo integration based on sampling points. We introduce the following mean square error (MSE) formulations:

\begin{equation}
    \begin{aligned}
\mathcal{F}_{\mathcal{L}_i}(\boldsymbol{\theta}_i) &:= \frac{1}{M_{i}} \sum_{k=1}^{M_{i}} |\mathcal{L}_i(\boldsymbol{x}_k,t_k; \boldsymbol{\theta}_i)|^2, \quad i=1,2, \\
\mathcal{F}_{\mathcal{B}_i}(\boldsymbol{\theta}_i) &:= \frac{1}{M_{\mathcal{B}_i}} \sum_{k=1}^{M_{\mathcal{B}_i}} |\mathcal{B}_i(\boldsymbol{x}_k,t_k; \boldsymbol{\theta}_i)|^2, \\
\mathcal{F}_{\mathcal{I}_1}(\boldsymbol{\theta}_1, \boldsymbol{\theta}_2) &:= \frac{1}{M_{\Gamma}} \sum_{k=1}^{M_{\Gamma}} |\mathcal{I}_1(\boldsymbol{x}_k,t_k; \boldsymbol{\theta}_1, \boldsymbol{\theta}_2)|^2, \\
\mathcal{F}_{\mathcal{I}_2}(\boldsymbol{\theta}_1, \boldsymbol{\theta}_2) &:= \frac{1}{M_{\Gamma}} \sum_{k=1}^{M_{\Gamma}} |\mathcal{I}_2(\boldsymbol{x}_k,t_k; \boldsymbol{\theta}_1, \boldsymbol{\theta}_2)|^2, \\
\mathcal{F}_{\mathcal{T}_i}(\boldsymbol{\theta}_i) &:= \frac{1}{M_{\mathcal{T}_i}} \sum_{k=1}^{M_{\mathcal{T}_i}} |\mathcal{T}_i(\boldsymbol{x}_k,0; \boldsymbol{\theta}_i)|^2,
\end{aligned}
\end{equation}
where $M_{i}$, $M_{\mathcal{B}_i}$, $M_{\Gamma}$, and $M_{\mathcal{T}_i}$ are the number of sampling points used for the interior part in each subdomain, the boundary part ($M_{\mathcal{B}_1} + M_{\mathcal{B}_2} = M_{\partial \Omega}$), the interface conditions, and the initial condition($M_{\mathcal{T}_1} + M_{\mathcal{T}_2} = M_{0}$), respectively.

Thus, the total discrete loss is defined as:
$$
\mathcal{F}(\boldsymbol{\theta}_1, \boldsymbol{\theta}_2) := \sum_{i=1}^{2} \left(\omega_{\mathcal{L}_i} \mathcal{F}_{\mathcal{L}_i}(\boldsymbol{\theta}_i) + \omega_{\mathcal{B}_i} \mathcal{F}_{\mathcal{B}_i}(\boldsymbol{\theta}_i) + \omega_{\mathcal{T}_i} \mathcal{F}_{\mathcal{T}_i}(\boldsymbol{\theta}_i)\right) + \omega_{\Gamma_1} \mathcal{F}_{\mathcal{I}_1}(\boldsymbol{\theta}_1, \boldsymbol{\theta}_2) + \omega_{\Gamma_2} \mathcal{F}_{\mathcal{I}_2}(\boldsymbol{\theta}_1, \boldsymbol{\theta}_2),
$$
and our DNN method for \eqref{modelpextenddnnpaowu} solves the following minimization problem:
$$
\min_{(\boldsymbol{\theta}_1, \boldsymbol{\theta}_2) \in \mathbb{R}^M \times \mathbb{R}^M} \mathcal{F}(\boldsymbol{\theta}_1, \boldsymbol{\theta}_2).
$$

To provide a more explicit representation, we express the total LS formulation for $\tilde{u} \in H^1(\Omega_1(t)) \cup H^1(\Omega_2(t))$ in the space-time domain:
\begin{equation}
    \begin{aligned}
\mathcal{R}(\tilde{u}) := &\sum_{i=1}^{2} \left(\omega_{\mathcal{L}_i} \int_{0}^{T}\int_{\Omega_i(t)} \left|\tilde{u}_{it} - \nabla \cdot (\beta(\boldsymbol{x}) \nabla \tilde{u}_{i}) - f\right|^2 \, d\boldsymbol{x}dt \right.\\
&\left.+ \omega_{\mathcal{B}_i} \int_{0}^{T}\int_{\partial \Omega_i(t) / \Gamma(t)} \left|\tilde{u}_i - g_D\right|^2 \, dsdt \right.\\
&\left. + \omega_{\mathcal{T}_i} \int_{\Omega_i(0)} \left|\tilde{u}_i(\boldsymbol{x},0) - g_0\right|^2 \, d\boldsymbol{x}\right) \\
&+ \omega_{\Gamma_1} \int_{0}^{T}\int_{\Gamma(t)} \left|\tilde{u}_1 - \tilde{u}_2 - g_1\right|^2 \, dsdt \\
&+ \omega_{\Gamma_2} \int_{0}^{T}\int_{\Gamma(t)} \left|\beta_1 \nabla \tilde{u}_1 \cdot \boldsymbol{n}_1 - \beta_2 \nabla \tilde{u}_2 \cdot \boldsymbol{n}_2 - g_2\right|^2 \, dsdt,
\end{aligned}
\end{equation}
where $\tilde{u}_i = \tilde{u}|_{\Omega_i(t)}$ denotes the restriction of $\tilde{u}$ to subdomain $\Omega_i(t)$ at time $t$.

The DNN approximation of this problem leads to the discretized loss function:
\begin{equation}
    \begin{aligned}
\mathcal{F}(\boldsymbol{\theta}_1, \boldsymbol{\theta}_2) &= \sum_{i=1}^{2} \left(\omega_{\mathcal{L}_i} \mathcal{F}_{\mathcal{L}_i}(\boldsymbol{\theta}_i) + \omega_{\mathcal{B}_i} \mathcal{F}_{\mathcal{B}_i}(\boldsymbol{\theta}_i) + \omega_{\mathcal{T}_i} \mathcal{F}_{\mathcal{T}_i}(\boldsymbol{\theta}_i) \right)\\
&\quad + \omega_{\Gamma_1} \mathcal{F}_{\mathcal{I}_1}(\boldsymbol{\theta}_1, \boldsymbol{\theta}_2) + \omega_{\Gamma_2} \mathcal{F}_{\mathcal{I}_2}(\boldsymbol{\theta}_1, \boldsymbol{\theta}_2) \\
&= \sum_{i=1}^{2} \left(\omega_{\mathcal{L}_i} \frac{1}{M_{i}} \sum_{k=1}^{M_{i}} \left|\frac{\partial u_{\boldsymbol{\theta}_i}}{\partial t}(\boldsymbol{x}_k,t_k) - \nabla \cdot (\beta_i \nabla u_{\boldsymbol{\theta}_i}(\boldsymbol{x}_k,t_k)) - f(\boldsymbol{x}_k,t_k)\right|^2 \right.\\
&\left.+ \omega_{\mathcal{B}_i} \frac{1}{M_{\mathcal{B}_i}} \sum_{k=1}^{M_{\mathcal{B}_i}} \left|u_{\boldsymbol{\theta}_i}(\boldsymbol{x}_k,t_k) - g_D(\boldsymbol{x}_k,t_k)\right|^2 + \omega_{\mathcal{T}_i} \frac{1}{M_{\mathcal{T}_i}} \sum_{k=1}^{M_{\mathcal{T}_i}} \left|u_{\boldsymbol{\theta}_i}(\boldsymbol{x}_k,0) - g_0(\boldsymbol{x}_k)\right|^2\right)\\
&+ \omega_{\Gamma_1} \frac{1}{M_{\Gamma}} \sum_{k=1}^{M_{\Gamma}} \left|u_{\boldsymbol{\theta}_1}(\boldsymbol{x}_k,t_k) - u_{\boldsymbol{\theta}_2}(\boldsymbol{x}_k,t_k) - g_1(\boldsymbol{x}_k,t_k)\right|^2 \\
&+ \omega_{\Gamma_2} \frac{1}{M_{\Gamma}} \sum_{k=1}^{M_{\Gamma}} \left|\beta_1 \nabla u_{\boldsymbol{\theta}_1}(\boldsymbol{x}_k,t_k) \cdot \boldsymbol{n} - \beta_2 \nabla u_{\boldsymbol{\theta}_2}(\boldsymbol{x}_k,t_k) \cdot \boldsymbol{n} - g_2(\boldsymbol{x}_k,t_k)\right|^2.
\end{aligned}
\end{equation}

\subsection{Error Estimation for the Parabolic Interface Problem}

\subsubsection{Abstract Framework for Error Analysis}
Following the abstract framework presented in Section \ref{Least-SquaresE}, we establish two major assumptions necessary for the error analysis of the DNN method applied to the parabolic interface problem \eqref{modelpextenddnnpaowu}.
%[Abstract Error Framework for Two-Network Domain Decomposition]
\begin{assumption}\label{p1}
For the exact solution $u_i$ in subdomain $\Omega_i(t)$ and any sufficiently regular approximation $v_i \in \mathbb{X}(\overline{\Omega}_i(t) \times [0, T])$ (in our case, $v_i = u_{\boldsymbol{\theta}_i}$ from network $i$), $i = 1, 2$, the differential operators in the parabolic interface model problem satisfy, for $0 < \beta_{\mathcal{L}_i}, \beta_{\mathcal{B}_i}, \beta_{\Gamma}, \beta_{\mathcal{T}} \leq 1$,
\begin{equation}
\begin{aligned}
    \sum_{i=1}^{2} \|u_i - v_i\|^2_{\mathbb{X}(\overline{\Omega}_i(t) \times [0,T])} &\leq C_{pde} \left\{ \sum_{i=1}^{2} \left[ \left(\|u_{i,t} - \nabla \cdot (\beta(\boldsymbol{x})\nabla u_i) - (v_{i,t} - \nabla \cdot (\beta(\boldsymbol{x})\nabla v_i))\|^2_{L^2(\Omega_i(t))}\right)^{\beta_{\mathcal{L}_i}} \right. \right.\\
&\left. + \left(\|u_i - v_i\|^2_{L^2(\partial\Omega_i(t) \setminus \Gamma(t))}\right)^{\beta_{\mathcal{B}_i}}  \right] + \left(\|\llbracket u \rrbracket - \llbracket v \rrbracket\|^2_{L^2(\Gamma(t))}\right)^{\beta_{\Gamma}}\\
&\left.  + \left(\|\llbracket \beta(\boldsymbol{x}) \nabla u \cdot \boldsymbol{n} \rrbracket - \llbracket \beta(\boldsymbol{x}) \nabla v \cdot \boldsymbol{n} \rrbracket\|^2_{L^2(\Gamma(t))}\right)^{\beta_{\Gamma}} \right.\\
&\left. + \sum_{i=1}^{2}\left(\|u_{i0} - v_{i0}\|^2_{L^2(\Omega_i(0))}\right)^{\beta_{\mathcal{T}_i}} \right\},
\end{aligned}
\end{equation}
where the constant $C_{pde}$ depends on $\|u_i\|_{\mathbb{X}(\overline{\Omega}_i(t) \times [0, T])}$ as well as the regularity property of the underlying parabolic interface model problem.
\end{assumption}

\begin{assumption}\label{p2}
Consider an integral $I(g) := \int_{\mathcal{D}} g(\boldsymbol{x},t) \, d\boldsymbol{x}dt$ and a quadrature rule using $N$ quadrature points $(\boldsymbol{x}_i,t_i) \in \mathcal{D}$ $(1 \leq i \leq N)$, i.e., $I_N(g) := \sum_{i=1}^{N} w_i g(\boldsymbol{x}_i,t_i)$, $w_i \in \mathbb{R}^+$, $1 \leq i \leq N$. We assume the quadrature error is
\begin{equation}
|I(g) - I_N(g)| \leq C_{quad} N^{-\alpha}, \quad \alpha > 0,
\end{equation}
where the constant $C_{quad}$ depends on the dimension of the domain and the property of the integrand $g(\boldsymbol{x},t)$.
\end{assumption}

Based on the above two assumptions, we have the following error estimation framework for the proposed two-network domain decomposition method applied to the parabolic interface problem.

%[Abstract Error Framework for Two-Network Parabolic Domain Decomposition]
\begin{theorem}
Let $u_i \in \mathbb{X}(\overline{\Omega}_i(t) \times [0, T])$, $i = 1, 2$, be the classical solution of the parabolic interface problem in subdomain $\Omega_i(t)$. Let $u_{\boldsymbol{\theta}_i}^{*} \in \mathbb{X}(\overline{\Omega}_i(t) \times [0, T])$ be the numerical approximation obtained by the neural network with parameters $\boldsymbol{\theta}_i^*$ for subdomain $\Omega_i(t)$. Under Assumptions \ref{p1} and \ref{p2}, we have the following generalization error estimation:
\begin{equation}
\begin{aligned}
    \sum_{i=1}^{2} \|u_i - u_{\boldsymbol{\theta}_i}^{*}\|^2_{\mathbb{X}(\overline{\Omega}_i(t) \times [0,T])} &\leq C_{\text{pde}} \left\{ \sum_{i=1}^{2} \left[ (\mathcal{F}_{\mathcal{L}_i}(\boldsymbol{\theta}_i^*))^{\beta_{\mathcal{L}_i}} + (\mathcal{F}_{\mathcal{B}_i}(\boldsymbol{\theta}_i^*))^{\beta_{\mathcal{B}_i}} + (\mathcal{F}_{\mathcal{T}_i}(\boldsymbol{\theta}_i^*))^{\beta_{\mathcal{T}_i}} \right] \right.\\
&\left.+ (\mathcal{F}_{\mathcal{I}_1}(\boldsymbol{\theta}_1^*, \boldsymbol{\theta}_2^*))^{\beta_{\Gamma}}  + (\mathcal{F}_{\mathcal{I}_2}(\boldsymbol{\theta}_1^*, \boldsymbol{\theta}_2^*))^{\beta_{\Gamma}} \right.\\
&\left. + \sum_{i=1}^{2} \left( C_{\text{quad}}^{\mathcal{L}_i}M_{i}^{-\alpha_{\mathcal{L}_i}\beta_{\mathcal{L}_i}} + C_{\text{quad}}^{\mathcal{B}_i}M_{\mathcal{B}_i}^{-\alpha_{\mathcal{B}_i}\beta_{\mathcal{B}_i}} + C_{\text{quad}}^{\mathcal{T}_i}M_{\mathcal{T}_i}^{-\alpha_{\mathcal{T}_i}\beta_{\mathcal{T}_i}}  \right) \right.\\
&\left. + C_{\text{quad}}^{\mathcal{I}_1}M_{\Gamma}^{-\alpha_{\mathcal{I}_1}\beta_{\Gamma}} + C_{\text{quad}}^{{\mathcal{I}_2}}M_{\Gamma}^{-\alpha_{\mathcal{I}_2}\beta_{\Gamma}} \right\},
\end{aligned}
\end{equation}
where $C_{\text{pde}}$ is a positive constant depending on $\|u_i\|_{\mathbb{X}(\overline{\Omega}_i(t) \times [0,T])}$, $\|u_{\boldsymbol{\theta}_i}^{*}\|_{\mathbb{X}(\overline{\Omega}_i(t) \times [0,T])}$, $\beta(\boldsymbol{x})$, and the domains $\Omega_i(t)$ $(i = 1, 2)$. Note that the interior, boundary, and initial loss terms depend only on the parameters $\boldsymbol{\theta}_i$ of network $i$, while the interface loss terms depend on both parameter sets $(\boldsymbol{\theta}_1, \boldsymbol{\theta}_2)$.
\end{theorem}

\begin{proof}
Let $e^u_i := u_i - u_{\boldsymbol{\theta}_i}^{*}$ be the error in subdomain $\Omega_i(t)$, where $u_{\boldsymbol{\theta}_i}^{*}$ is the approximation from network $i$ with optimal parameters $\boldsymbol{\theta}_i^*$. By a direct calculation, we have, for $i = 1, 2$,
\begin{equation}
\begin{aligned}
e^u_{i,t} - \nabla \cdot \bigl(\beta_i \nabla e^u_i\bigr) &= \mathcal{R}^f_i \quad \text{in } \Omega_i(t), \\
 e^u_1 - e^u_2&= \mathcal{R}^{g_1} \quad \text{on } \Gamma(t), \\
 \beta_1 \nabla e^u_1 \cdot \boldsymbol{n} - \beta_2 \nabla e^u_2 \cdot \boldsymbol{n}&= \mathcal{R}^{g_2} \quad \text{on } \Gamma(t), \\
e^u_i(0) &= \mathcal{R}^{g_0}_i \quad \text{on } \Omega_i(0), \\
e^u_i &= \mathcal{R}^{g_D}_i \quad \text{on } \partial \Omega_i(t) \setminus \Gamma(t),
\end{aligned}
\end{equation}
where $\mathcal{R}^f_i := \partial_t u_{\boldsymbol{\theta}_i}^{*} - \nabla \cdot \bigl(\beta_i \nabla u_{\boldsymbol{\theta}_i}^{*}\bigr) - f$ is the PDE residual for network $i$, $\mathcal{R}^{g_1} := u_{\boldsymbol{\theta}_1}^{*} - u_{\boldsymbol{\theta}_2}^{*} - g_1$ is the solution jump residual (coupling both networks), $\mathcal{R}^{g_2} := \beta_1 \nabla u_{\boldsymbol{\theta}_1}^{*} \cdot \boldsymbol{n} - \beta_2 \nabla u_{\boldsymbol{\theta}_2}^{*} \cdot \boldsymbol{n} - g_2$ is the flux jump residual (also coupling both networks), $\mathcal{R}^{g_0}_i := u_{\boldsymbol{\theta}_i}^{*}(\cdot,0) - g_0$ is the initial condition residual for network $i$, and $\mathcal{R}^{g_D}_i := u_{\boldsymbol{\theta}_i}^{*} - g_D$ is the boundary residual for network $i$.

Using Assumption \ref{p1}, we have
\begin{equation}
\begin{aligned}
\sum_{i=1}^{2} \|u_i - u_{\boldsymbol{\theta}_i}^{*}\|^2_{\mathbb{X}(\overline{\Omega}_i(t) \times [0,T])} &\leq C_{pde} \left\{ \sum_{i=1}^{2} \left[ \left(\|\mathcal{R}^f_i\|^2_{L^2(\Omega_i(t))}\right)^{\beta_{\mathcal{L}_i}} + \left(\|\mathcal{R}^{g_D}_i\|^2_{L^2(\partial\Omega_i(t) \setminus \Gamma(t))}\right)^{\beta_{\mathcal{B}_i}}\right] \right. \\
&\quad \left. + \left(\|\mathcal{R}^{g_1}\|^2_{L^2(\Gamma(t))}\right)^{\beta_{\Gamma}} + \left(\|\mathcal{R}^{g_2}\|^2_{L^2(\Gamma(t))}\right)^{\beta_{\Gamma}}\right. \\
&\quad \left.+\sum_{i=1}^{2} \left(\|\mathcal{R}^{g_0}_i\|^2_{L^2(\Omega_i(0))}\right)^{\beta_{\mathcal{T}_i}} \right\}.
\end{aligned}
\end{equation}

To further estimate the terms on the right-hand side, we use Assumption \ref{p2}. For the interior residual of network $i$, we have
\begin{equation}
\left(\|\mathcal{R}^f_i\|^2_{L^2(\Omega_i(t))}\right)^{\beta_{\mathcal{L}_i}} \leq (\mathcal{F}_{\mathcal{L}_i}(\boldsymbol{\theta}_i^*))^{\beta_{\mathcal{L}_i}} + C_{\text{quad}}^{\mathcal{L}_i}M_{i}^{-\alpha_{\mathcal{L}_i}\beta_{\mathcal{L}_i}}.
\end{equation}

Similarly, for the boundary and initial condition residuals (depending on $\boldsymbol{\theta}_i$) and the interface residuals (depending on both $\boldsymbol{\theta}_1$ and $\boldsymbol{\theta}_2$):
\begin{equation}
\begin{aligned}
\left(\|\mathcal{R}^{g_D}_i\|^2_{L^2(\partial\Omega_i(t) \setminus \Gamma(t))}\right)^{\beta_{\mathcal{B}_i}} &\leq (\mathcal{F}_{\mathcal{B}_i}(\boldsymbol{\theta}_i^*))^{\beta_{\mathcal{B}_i}} + C_{\text{quad}}^{\mathcal{B}_i}M_{\mathcal{B}_i}^{-\alpha_{\mathcal{B}_i}\beta_{\mathcal{B}_i}},  \\
\left(\|\mathcal{R}^{g_1}\|^2_{L^2(\Gamma(t))}\right)^{\beta_{\Gamma}} &\leq (\mathcal{F}_{\mathcal{I}_1}(\boldsymbol{\theta}_1^*, \boldsymbol{\theta}_2^*))^{\beta_{\Gamma}} + C_{\text{quad}}^{\mathcal{I}_1}M_{\Gamma}^{-\alpha_{\mathcal{I}_1}\beta_{\Gamma}},\\
\left(\|\mathcal{R}^{g_2}\|^2_{L^2(\Gamma(t))}\right)^{\beta_{\Gamma}} &\leq (\mathcal{F}_{\mathcal{I}_2}(\boldsymbol{\theta}_1^*, \boldsymbol{\theta}_2^*))^{\beta_{\Gamma}} + C_{\text{quad}}^{\mathcal{I}_2}M_{\Gamma}^{-\alpha_{\mathcal{I}_2}\beta_{\Gamma}},\\
\left(\|\mathcal{R}^{g_0}_i\|^2_{L^2(\Omega_i(0) )}\right)^{\beta_{\mathcal{T}_i}} &\leq (\mathcal{F}_{\mathcal{T}_i}(\boldsymbol{\theta}_i^*))^{\beta_{\mathcal{T}_i}} + C_{\text{quad}}^{\mathcal{T}_i}M_{\mathcal{T}}^{-\alpha_{\mathcal{T}_i}\beta_{\mathcal{T}_i}}.
\end{aligned}
\end{equation}

Substituting these estimates into the previous inequality completes the proof. As in the elliptic case, the interface loss terms couple both networks through the parameters $(\boldsymbol{\theta}_1^*, \boldsymbol{\theta}_2^*)$.
\end{proof}

\subsubsection{Application to the parabolic interface problem}
Let $e^u_i := u_i - u_{\boldsymbol{\theta}_i}^{*}$, $i = 1, 2$, be the errors of the parabolic interface problem \eqref{modelpextenddnnpaowu}. Following the abstract analysis presented in subsection \ref{Least-SquaresP} and adapting the framework from \cite{mishra2023estimates}, we state the following conditional error bound.

%[Conditional Error Bound for Parabolic Problems]
\begin{theorem}
Let $u_i \in C^2(\overline{\Omega}_i(t) \times [0, T])$, $i = 1, 2$, be the classical solution of the presented parabolic interface problem; correspondingly, let $u_{\boldsymbol{\theta}_i}^{*} \in C^2(\overline{\Omega}_i(t) \times [0, T])$ be the numerical approximations obtained by the DNN method. Under Assumptions~\ref{p1} and~\ref{p2}, we have the following conditional error bound:
\begin{equation}
\int_{0}^{T} \sum_{i=1}^{2} \int_{\Omega_i} |e_i^u(\boldsymbol{x}_i, t)|^2 d\boldsymbol{x}dt
\end{equation}

\begin{align}
\leq C_1 T(1+C_2Te^{C_2T}) \left[ \sum_{i=1}^{2} \mathcal{F}_{\mathcal{T}_i}(\boldsymbol{\Theta}^*) + \sum_{i=1}^{2} C_{quad}^{\mathcal{T}_i}M_{\mathcal{T}_i}^{-\alpha_{\mathcal{T}_i}} \right]
\end{align}

\begin{align}
+ C_1 T^{\frac{1}{2}}(1+C_2Te^{C_2T}) \left[ \sum_{i=1}^{2} \mathcal{F}_{\mathcal{L}_i}(\boldsymbol{\Theta}^*) + \sum_{i=1}^{2} \mathcal{F}_{\mathcal{B}_i}(\boldsymbol{\Theta}^*) + \mathcal{F}_{\Gamma}(\boldsymbol{\Theta}^*) \right.\\
\left. + \sum_{i=1}^{2} C_{quad}^{\mathcal{L}_i}M_{i}^{-\alpha_{\mathcal{L}_i}} + \sum_{i=1}^{2} C_{quad}^{\mathcal{B}_i}M_{\mathcal{B}_i}^{-\alpha_{\mathcal{B}_i}} + C_{quad}^{\Gamma}M_{\Gamma}^{-\alpha_{\Gamma}} \right]^{\frac{1}{2}},
\end{align}

where $C_1$ and $C_2$ are positive constants depending on $\|u_i\|_{C^2(\overline{\Omega}_i(t) \times [0,T])}$, $\|u_{\boldsymbol{\theta}}^{i,*}\|_{C^2(\overline{\Omega}_i(t) \times [0,T])}$, $\beta(\boldsymbol{x})$, and the domains $\Omega_i(t)$ $(i = 1, 2)$.
\end{theorem}

\begin{proof}
Let $e^u_i := u_i - u_{\boldsymbol{\theta}}^{i,*}$, $i = 1, 2$, be the errors of the parabolic interface problem \eqref{modelpextenddnnpaowu}. By a direct calculation, we have, for $i = 1, 2$,
\begin{equation}
\begin{aligned}
\frac{\partial e^u_i}{\partial t} - \nabla \cdot \bigl(\beta(\boldsymbol{x}) \nabla e^u_i\bigr) &= \mathcal{R}^f_i \quad \text{in } \Omega_i(t) \times (0, T], \\
\nabla \cdot e^u_i &= \mathcal{R}^{\nabla}_i \quad \text{in } \Omega_i(t) \times (0, T], \\
e^u_1 - e^u_2 &= \mathcal{R}^v_{\Gamma} \quad \text{on } \Gamma(t) \times [0, T], \\
\sigma_1(e^u_1)\boldsymbol{n}_1 + \sigma_2(e^u_2)\boldsymbol{n}_2 &= \mathcal{R}^{\sigma}_{\Gamma} \quad \text{on } \Gamma(t) \times [0, T], \\
e^u_i &= \mathcal{R}^{B}_i \quad \text{on } \partial\Omega_i(t) \setminus \Gamma(t) \times [0, T], \\
e^u_i(\boldsymbol{x}_i, 0) &= \mathcal{R}^{\mathcal{T}}_i \quad \text{in } \Omega_i(0),
\end{aligned}
\end{equation}
where $\mathcal{R}^f_i := \frac{\partial u_{\boldsymbol{\theta}}^{i,*}}{\partial t} - \nabla \cdot \bigl(\beta(\boldsymbol{x}) \nabla u_{\boldsymbol{\theta}}^{i,*}\bigr) - f$, $\mathcal{R}^{\nabla}_i := \nabla \cdot u_{\boldsymbol{\theta}}^{i,*}$, $\mathcal{R}^v_{\Gamma} := u_{\boldsymbol{\theta}}^{2,*} - u_{\boldsymbol{\theta}}^{1,*} - g_1$, $\mathcal{R}^{\sigma}_{\Gamma} := \llbracket \beta(\boldsymbol{x}) \nabla u_{\boldsymbol{\theta}}^{*} \cdot \boldsymbol{n} \rrbracket - g_2$, $\mathcal{R}^{B}_i := u_{\boldsymbol{\theta}}^{i,*} - g_D$, and $\mathcal{R}^{\mathcal{T}}_i := u_{\boldsymbol{\theta}}^{i,*}(\boldsymbol{x}_i, 0) - g_0$.

Multiplying the first equation by $e^u_i$ and integrating over $\Omega_i(t)$, $i = 1, 2$, we have
\begin{equation}
\begin{aligned}
\frac{d}{dt} \int_{\Omega_i} \frac{|e^u_i|^2}{2} d\boldsymbol{x} + \int_{\Omega_i} \beta(\boldsymbol{x})|\nabla e^u_i|^2 d\boldsymbol{x} = \int_{\Omega_i} \mathcal{R}^f_i \cdot e^u_i d\boldsymbol{x} - \int_{\partial \Omega_i} \beta(\boldsymbol{x}) \nabla e^u_i \cdot \boldsymbol{n} \cdot e^u_i ds,
\end{aligned}
\end{equation}
where we applied integration by parts to the diffusion term. 

Summing up for $i = 1, 2$ and considering the interface and boundary conditions, we arrive at
\begin{equation}
\begin{aligned}
\frac{d}{dt}\left[\sum_{i=1}^{2} \int_{\Omega_i} \frac{|e^u_i|^2}{2} d\boldsymbol{x}\right] 
&= \sum_{i=1}^{2} \int_{\Omega_i} \mathcal{R}^f_i \cdot e^u_i d\boldsymbol{x} + \sum_{i=1}^{2} \int_{\Omega_i} \mathcal{R}^{\nabla}_i \cdot e^u_i d\boldsymbol{x} \\
&+ \sum_{i=1}^{2} \int_{\partial\Omega_i \setminus \Gamma} \mathcal{R}^{B}_i \cdot \sigma_i(e^u_i) \boldsymbol{n}_i ds \\
&+ \int_{\Gamma} \mathcal{R}^v_{\Gamma} \cdot \left(\frac{\sigma_1(e^u_1)\boldsymbol{n}_1 - \sigma_2(e^u_2)\boldsymbol{n}_2}{2}\right) ds \\
&+ \int_{\Gamma} \mathcal{R}^{\sigma}_{\Gamma} \cdot \left(\frac{e^u_1 + e^u_2}{2}\right) ds,
\end{aligned}
\end{equation}
where we use the algebraic identity $ac + bd = \frac{1}{2}[(a-b)(c-d) + (a+b)(c+d)]$. 

To estimate the terms on the right-hand side, we apply the Cauchy-Schwarz inequality, Young's inequality, and appropriate trace theorems to obtain
\begin{equation}
\begin{aligned}
\frac{d}{dt}\left[\sum_{i=1}^{2} \int_{\Omega_i} |e^u_i|^2 d\boldsymbol{x}\right] \leq C_1 \left[\sum_{i=1}^{2} \int_{\Omega_i} |\mathcal{R}^f_i|^2 d\boldsymbol{x} + \sum_{i=1}^{2} \int_{\Omega_i} |\mathcal{R}^{\nabla}_i|^2 d\boldsymbol{x} \right. \\
\left. + \sum_{i=1}^{2} \int_{\partial\Omega_i \setminus \Gamma} |\mathcal{R}^{B}_i|^2 ds + \int_{\Gamma} |\mathcal{R}^v_{\Gamma}|^2 ds + \int_{\Gamma} |\mathcal{R}^{\sigma}_{\Gamma}|^2 ds\right]^{\frac{1}{2}} \\
+ C_2 \left(\sum_{i=1}^{2} \int_{\Omega_i} |e^u_i|^2 d\boldsymbol{x}\right),
\end{aligned}
\end{equation}
where $C_1$ and $C_2$ are positive constants depending on $\|u_i\|_{C^2(\overline{\Omega}_i(t) \times [0,T])}$, $\|u_{\boldsymbol{\theta}}^{i,*}\|_{C^2(\overline{\Omega}_i(t) \times [0,T])}$, $\beta(\boldsymbol{x})$, and the domains $\Omega_i(t)$ $(i = 1, 2)$. 

For any $0 \leq \tau \leq T$, integrating the above estimation over time, we obtain
\begin{equation}
\begin{aligned}
\sum_{i=1}^{2} \int_{\Omega_i} |e^u_i(\boldsymbol{x}_i, \tau)|^2 d\boldsymbol{x} 
&\leq \sum_{i=1}^{2} \int_{\Omega_i} |\mathcal{R}^{\mathcal{T}}_i|^2 d\boldsymbol{x} + C_1 T^{\frac{1}{2}} \left[ \int_0^T \left( \sum_{i=1}^{2} \int_{\Omega_i} |\mathcal{R}^f_i|^2 d\boldsymbol{x} + \sum_{i=1}^{2} \int_{\Omega_i} |\mathcal{R}^{\nabla}_i|^2 d\boldsymbol{x} \right. \right. \\
&\left. \left. + \sum_{i=1}^{2} \int_{\partial\Omega_i \setminus \Gamma} |\mathcal{R}^{B}_i|^2 ds + \int_{\Gamma} (|\mathcal{R}^v_{\Gamma}|^2 + |\mathcal{R}^{\sigma}_{\Gamma}|^2) ds \right) dt \right]^{\frac{1}{2}} \\
&+ C_2 \int_0^T \left( \sum_{i=1}^{2} \int_{\Omega_i} |e^u_i(\boldsymbol{x}_i, t)|^2 d\boldsymbol{x} \right) dt \\
&=: \mathcal{R} + C_2 \int_0^T \left( \sum_{i=1}^{2} \int_{\Omega_i} |e^u_i(\boldsymbol{x}_i, t)|^2 d\boldsymbol{x} \right) dt.
\end{aligned}
\end{equation}

Now applying Grönwall's inequality yields $\sum_{i=1}^{2} \int_{\Omega_i} |e^u_i(\boldsymbol{x}_i, \tau)|^2 d\boldsymbol{x} \leq (1 + C_2 Te^{C_2 T})\mathcal{R}$. Integrating again over $[0, T]$ results in
\begin{equation}\label{gronwall}
\int_0^T \sum_{i=1}^{2} \int_{\Omega_i} |e^u_i(\boldsymbol{x}_i, t)|^2 d\boldsymbol{x}dt \leq (T + C_2 T^2 e^{C_2 T})\mathcal{R}.
\end{equation}

By the definition of $\mathcal{R}$, this equation essentially verifies Assumption \ref{p1} for the parabolic interface problem with $\beta_{\mathcal{L}_i} = \beta_{\mathcal{B}_i} = \beta_{\Gamma} = \frac{1}{2}$ and $\beta_{\mathcal{T}_i} = 1$, $i = 1, 2$. Following the abstract analysis presented in subsection \ref{Least-SquaresP} and using Assumption \ref{p1}, we have
\begin{equation}
\begin{aligned}
\mathcal{R} &\leq \sum_{i=1}^{2} \mathcal{F}_{\mathcal{T}_i}(\boldsymbol{\Theta}^*) + \sum_{i=1}^{2} C_{quad}^{\mathcal{T}_i}M_{\mathcal{T}_i}^{-\alpha_{\mathcal{T}_i}} \\
&+ C_1 T^{\frac{1}{2}} \left[ \sum_{i=1}^{2} \mathcal{F}_{\mathcal{L}_i}(\boldsymbol{\Theta}^*) + \sum_{i=1}^{2} \mathcal{F}_{\mathcal{B}_i}(\boldsymbol{\Theta}^*) + \mathcal{F}_{\Gamma}(\boldsymbol{\Theta}^*) \right.\\
& \left. + \sum_{i=1}^{2} C_{quad}^{\mathcal{L}_i}M_{i}^{-\alpha_{\mathcal{L}_i}} + \sum_{i=1}^{2} C_{quad}^{\mathcal{B}_i}M_{\mathcal{B}_i}^{-\alpha_{\mathcal{B}_i}} + C_{quad}^{\Gamma}M_{\Gamma}^{-\alpha_{\Gamma}} \right]^{\frac{1}{2}}
\end{aligned}
\end{equation}

Substituting the above estimation back into \eqref{gronwall}, we complete the proof.

\end{proof}

The error bounds presented in this section provide a theoretical foundation for understanding the proposed method. The bounds are conditional on stability assumptions (Assumptions~\ref{e1} and~\ref{p1}) that encapsulate the well-posedness of the underlying PDEs; for interface problems with high-contrast coefficients, these stability constants may be large. The bounds suggest that achieving small training loss values is necessary (though not sufficient) for obtaining accurate solutions, justifying loss monitoring during training as a practical convergence indicator. The sampling density affects the error through quadrature error terms, which decrease as the number of collocation points increases.

The bounds do not explicitly account for the approximation capacity of specific neural network architectures, the convergence behavior of optimization algorithms, hyperparameter choices (learning rate, batch size), or the effect of different activation function choices. Given these limitations, the numerical experiments in the following section serve as the primary validation of the method's effectiveness across various interface configurations and dimensions.

\section{Numerical Results}
\label{NumericalResults}
\noindent
In this section, we present a series of numerical experiments to illustrate the effectiveness and accuracy of our proposed method for solving elliptic interface problems and parabolic interface problems. We compare the performance of our multi-activation function approach (denoted as \emph{MAF}) against other representative methods, including extended PINN (\emph{XPINN}) \cite{jagtap2020extended}, mesh-free method (\emph{MFM}) \cite{he2022mesh}, Discontinuity Capturing Shallow Neural Network (\emph{DCSNN}) \cite{hu2022discontinuity}, Multi-domain PINN (\emph{M-PINN}) \cite{zhang2022multi}, Interface PINN (\emph{I-PINN}) \cite{sarma2024interface}, and Adaptive Interface PINN (\emph{AdaI-PINN}) \cite{roy2024adaptive}.

% \paragraph{Sampling Strategy.}
Collocation points are generated using either uniform grid sampling or quasi-random (Latin Hypercube) sampling within each subdomain $\Omega_1$ and $\Omega_2$, and on the interface $\Gamma$ and boundaries $\partial\Omega$. For 2D problems, we primarily use uniform grids for reproducibility, while for higher-dimensional problems (3D--10D), we employ Latin Hypercube sampling to mitigate the curse of dimensionality. The loss function integrals are approximated via Monte Carlo-type quadrature using these collocation points, with equal weights $w_k = |\Omega_i|/M_i$ for interior points and similar formulas for boundary/interface terms. This is consistent with the quadrature assumption (Assumption~\ref{e2}) in our error analysis.

We employ distinct network architectures for each method to ensure fair comparison. For our approach (MAF), we use two independent neural networks (one for each subdomain $\Omega_1$ and $\Omega_2$), each with 3 hidden layers containing 50 neurons per layer. XPINN uses an ensemble of 5 neural networks with 3 hidden layers $\times$ 20 neurons each. MFM, M-PINN, I-PINN, and AdaI-PINN all use two neural networks with 3 hidden layers $\times$ 50 neurons per subdomain. DCSNN employs a single-hidden-layer architecture with 300 neurons due to method constraints. Table~\ref{tab:computational_cost} summarizes the parameter counts for each method. All methods use the $\tanh$ activation function (with additional multi-activation for MAF), and the Adam optimizer is employed for 50,000 training steps with a learning rate of 0.001. All codes are implemented in \texttt{PyTorch} and run on a GPU workstation equipped with an Nvidia A100 GPU.

\begin{table}[H]
    \centering
    \caption{Network architecture comparison for 2D elliptic interface problems}
    \label{tab:computational_cost}
    \begin{tabular}{|l|c|c|}
        \hline
        \textbf{Method} & \textbf{Total Parameters} & \textbf{Networks} \\
        \hline
        MAF (Ours) & 10,802 & 2 \\
        XPINN & 12,505 & 5 \\
        MFM & 10,802 & 2 \\
        DCSNN & 901 & 1 \\
        M-PINN & 10,802 & 2 \\
        I-PINN & 11,502 & 2 \\
        AdaI-PINN & 11,502 & 2 \\
        \hline
    \end{tabular}
\end{table}

To quantify accuracy, we compute the relative $L_2$ error for elliptic interface problems as
\begin{equation}\label{errore}
    \text{Error}_{e}
    \;=\; 
    \frac{ \|u_{\text{NN}}(\boldsymbol{x}) - u(\boldsymbol{x})\|_{L^2(\Omega)}}{\|u(\boldsymbol{x})\|_{L^2(\Omega)}},
\end{equation}
where $u_{\text{NN}}(\boldsymbol{x}) = u_{\boldsymbol{\theta}_1^*}(\boldsymbol{x})$ for $\boldsymbol{x} \in \Omega_1$ and $u_{\text{NN}}(\boldsymbol{x}) = u_{\boldsymbol{\theta}_2^*}(\boldsymbol{x})$ for $\boldsymbol{x} \in \Omega_2$. For parabolic interface problems, we compute the space-time error $\text{Error}_{p} = \|u_{\text{NN}} - u\|_{L^2(0,T;L^2(\Omega))} / \|u\|_{L^2(0,T;L^2(\Omega))}$.

For the interface-aware weighting functions in our multi-activation mechanism, we use $\omega_2(\mathbf{x}) = \exp(-\gamma \cdot d(\mathbf{x},\Gamma))$ and $\omega_1(\mathbf{x}) = 1 - \omega_2(\mathbf{x})$, where $\gamma = 10$ and $d(\mathbf{x},\Gamma)$ is the Euclidean distance to the interface. The distance function is computed analytically: for straight lines $d = |x - x_0|$, for circles $d = |\sqrt{(x-x_c)^2 + (y-y_c)^2} - r|$, for ellipses via Newton iteration, and for parametric curves via numerical minimization. The weight function $\omega_2$ concentrates near the interface $\Gamma$, enabling the multi-activation mechanism to focus learning capacity in this critical region; specific illustrations are provided in each test case.

\subsection{Elliptic Interface Problem}
\subsubsection{Test 1: Piecewise Constant Coefficients in a Two-Domain Problem}

We consider a rectangular domain $\Omega = [-1,1] \times [-1,1]$, partitioned into two subdomains by a vertical interface at $x = 0$: $\Omega_1 = (0,1) \times (-1,1)$, $\Omega_2 = (-1,0) \times (-1,1)$, and $\Gamma = \{0\} \times (-1,1)$. The elliptic interface problem is formulated as:
\begin{equation}\label{test1_pde}
\begin{aligned}
-\nabla \cdot (\beta(\boldsymbol{x}) \nabla u) &= f(\boldsymbol{x}), & & \text{in } \Omega_1 \cup \Omega_2, \\
\llbracket u \rrbracket &= g_1, & & \text{on } \Gamma, \\
\llbracket \beta \nabla u \cdot \boldsymbol{n} \rrbracket &= g_2, & & \text{on } \Gamma, \\
u &= g_D, & & \text{on } \partial \Omega,
\end{aligned}
\end{equation}
where $\beta(\boldsymbol{x}) = -1$ in $\Omega_1$ and $\beta(\boldsymbol{x}) = +1$ in $\Omega_2$. The manufactured exact solution is $u(x,y) = -\sin(2\pi x)\sin(2\pi y) - 1$ in $\Omega_1$ and $u(x,y) = +\sin(2\pi x)\sin(2\pi y) + 1$ in $\Omega_2$. The source term is $f(x,y) = 8\pi^2 \sin(2\pi x)\sin(2\pi y)$ in $\Omega_1$ and $f(x,y) = -8\pi^2 \sin(2\pi x)\sin(2\pi y)$ in $\Omega_2$. The jump conditions are $g_1 = \llbracket u \rrbracket|_{\Gamma} = -2$ and $g_2 = \llbracket \beta \nabla u \cdot \boldsymbol{n} \rrbracket|_{\Gamma} = 0$. For this straight-line interface, the distance function is $d(\mathbf{x}, \Gamma) = |x|$.

We distribute the collocation points uniformly in each subdomain and along the interface. The training and validation points are shown in Figure~\ref{pointstest1}, and the numerical solutions with absolute errors are presented in Figure~\ref{solutiontest1}. Figure~\ref{fig:weight_functions_test1} illustrates the interface-aware weight functions $\omega_1(\mathbf{x})$ and $\omega_2(\mathbf{x})$ along a line crossing the interface.

\begin{figure}
    \centering
    \subfigure[Training points]{\includegraphics[width=5cm, height=4cm]{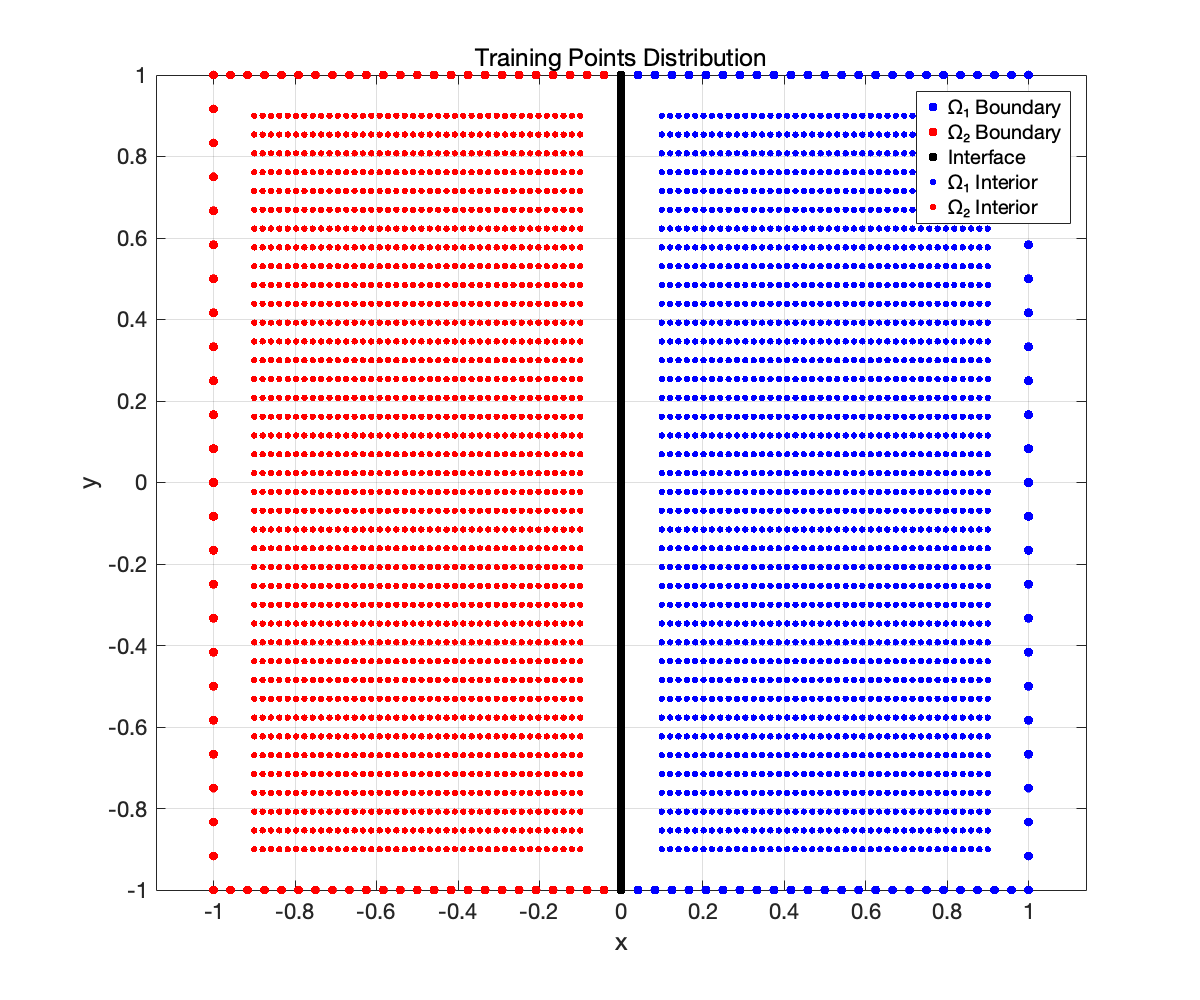}}
    \subfigure[Validation points]{\includegraphics[width=5cm, height=4cm]{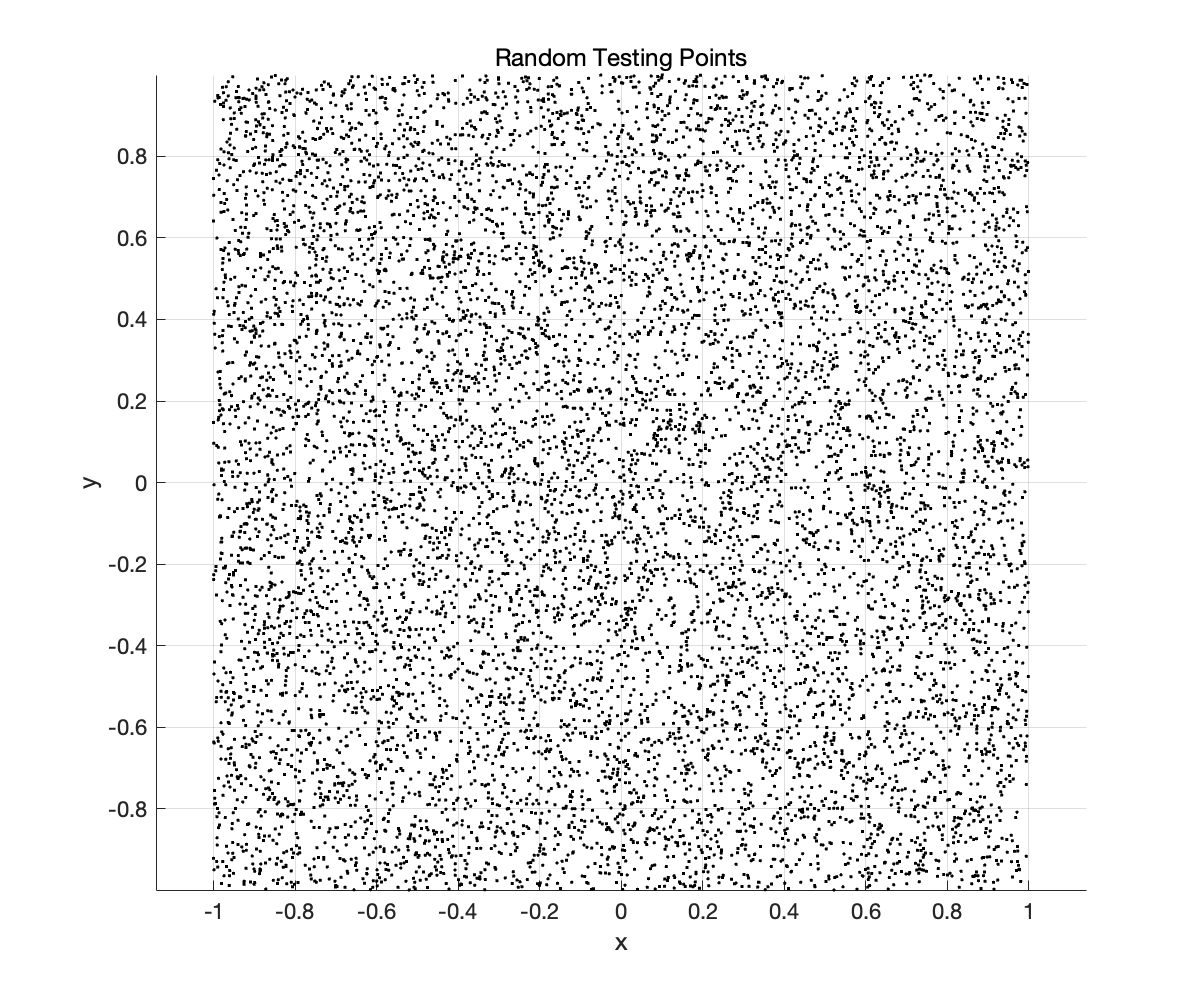}}
    \caption{Training and validation points for Test 1 when $M_{\Omega} = 1600$, $M_{\partial \Omega}=160$, and $M_{\Gamma}=100$.}
    \label{pointstest1}
\end{figure}

\begin{figure}
    \centering
    \subfigure[$u_{\text{Exact}}$]{\includegraphics[width=3.5cm, height=3cm]{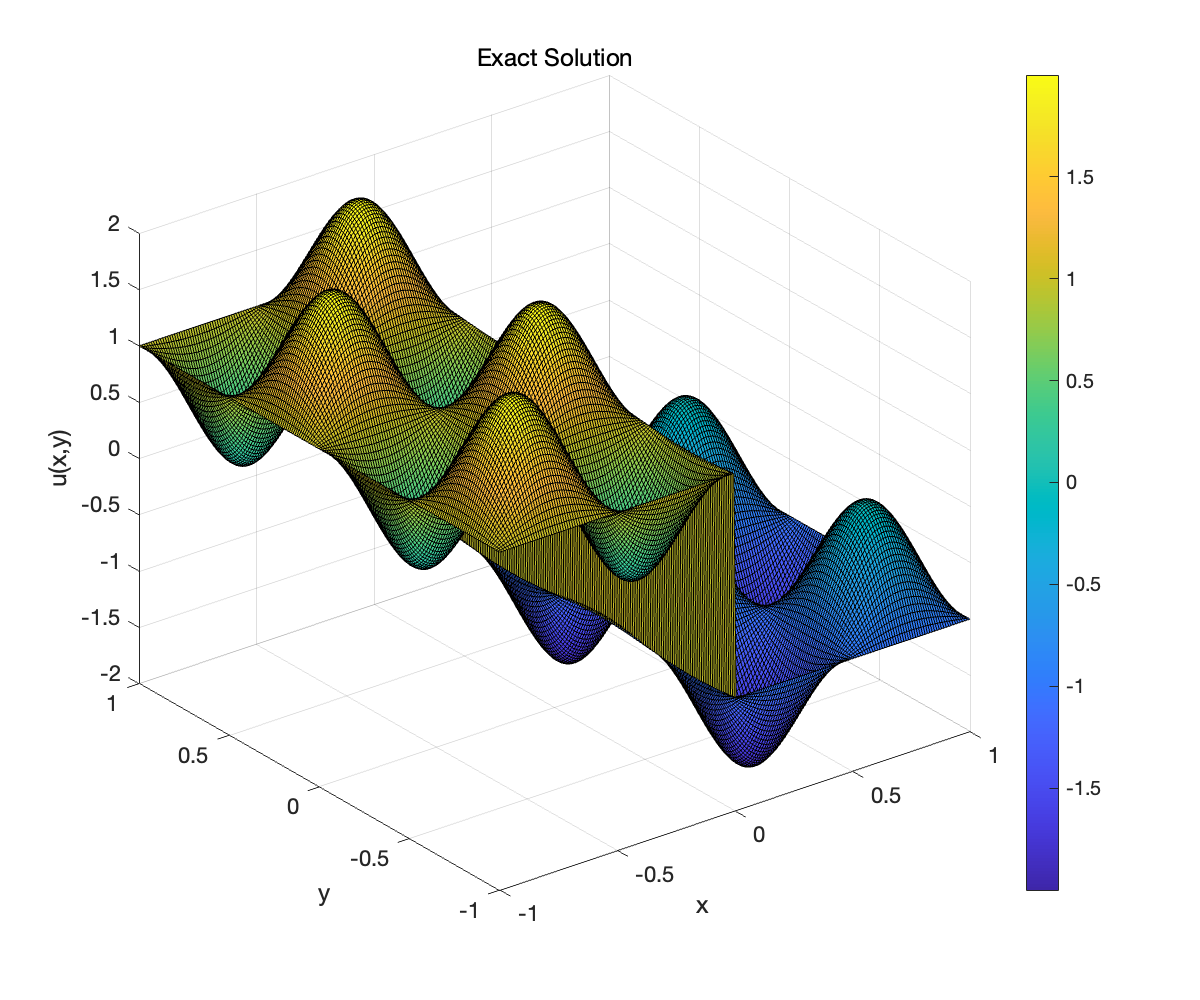}}
    \subfigure[$u_{\text{DNN}}$]{\includegraphics[width=3.5cm, height=3cm]{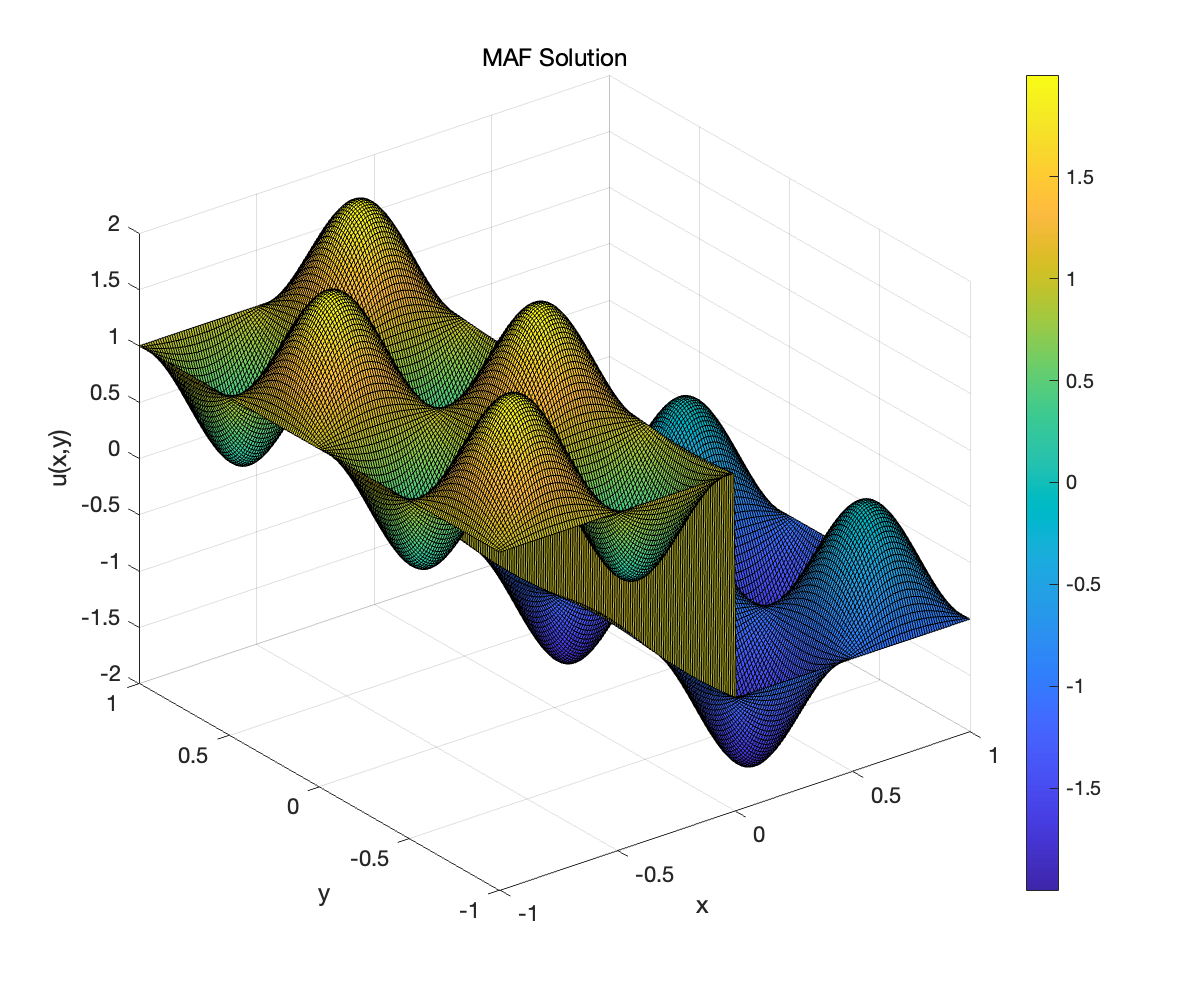}}
    \subfigure[Absolute Error]{\includegraphics[width=3.5cm, height=3cm]{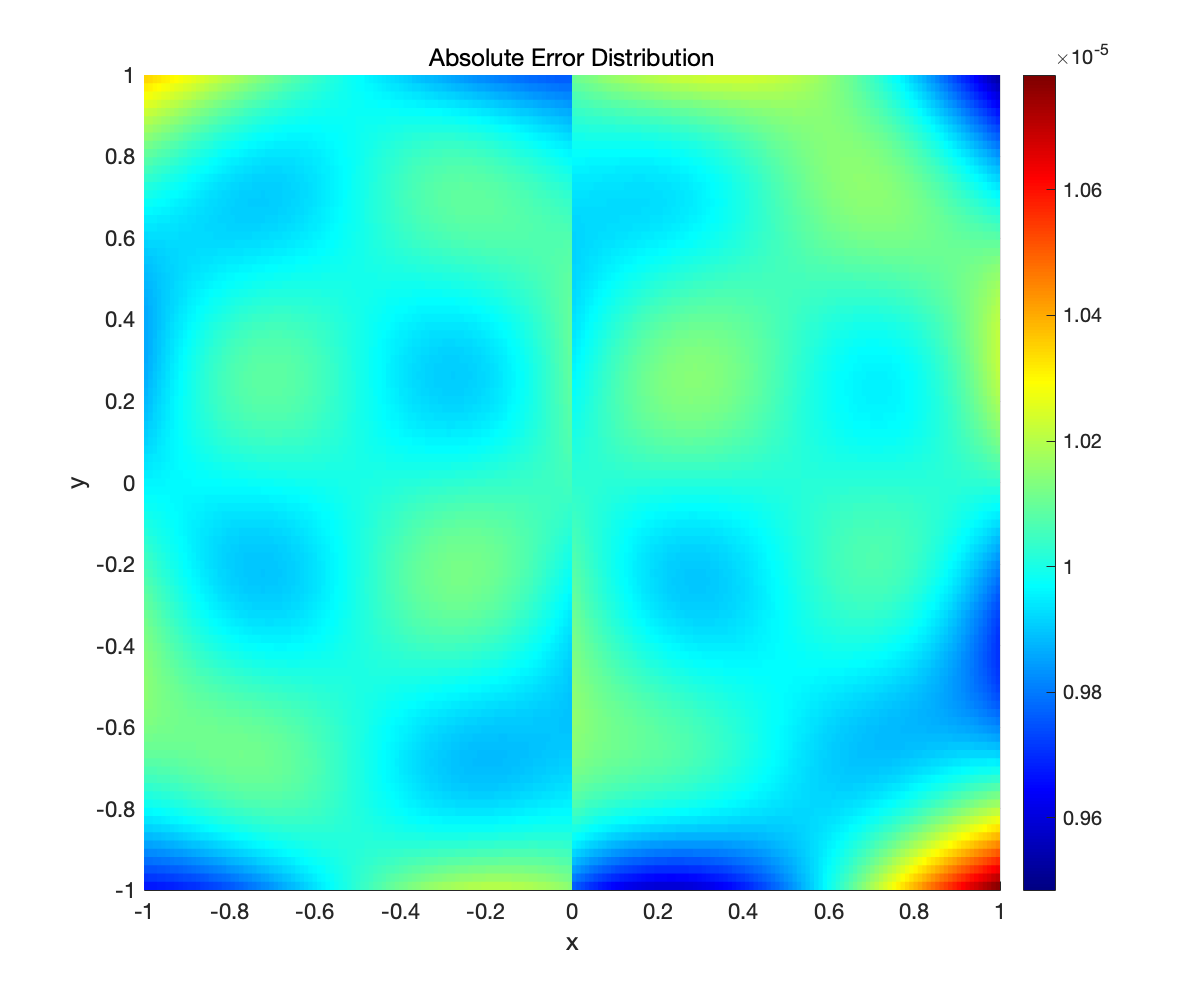}}
    \caption{$u_{\text{DNN}}$, $u_{\text{Exact}}$ and their absolute error for Test 1.}
    \label{solutiontest1}
\end{figure}

\begin{figure}
    \centering
    \includegraphics[width=7cm, height=4cm]{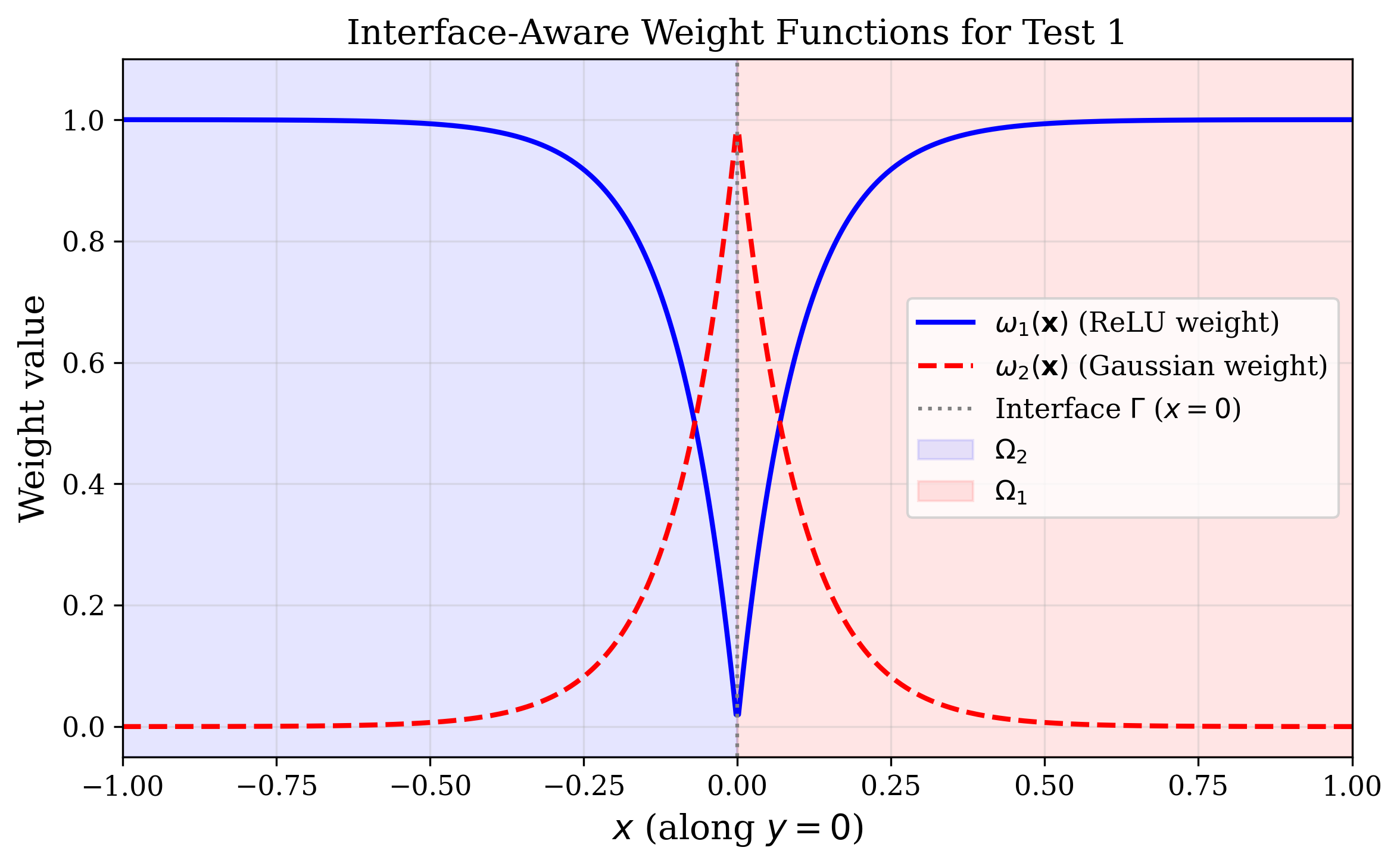}
    \caption{Interface-aware weight functions $\omega_1(\mathbf{x})$ and $\omega_2(\mathbf{x})$ along a line crossing the interface at $x=0$ for Test 1.}
    \label{fig:weight_functions_test1}
\end{figure}

\begin{table}
    \centering
    \caption{Relative $L^2$ error comparison for Test 1.}
    \label{errortabletest1}
    \small
    \begin{tabular}{|c|c|c|c|c|c|c|c|}
        \hline
        $(M_\Omega, M_{\partial\Omega}, M_\Gamma)$ & MAF & DCSNN & AdaI & I-PINN & M-PINN & MFM & XPINN \\
        \hline
        (100, 40, 25) & $6.33e{-}4$ & $1.02e{-}3$ & $1.85e{-}3$ & $3.21e{-}3$ & $5.67e{-}3$ & $2.15e{-}2$ & $3.96e{-}2$ \\
        \hline
        (400, 80, 50) & $8.14e{-}5$ & $2.15e{-}4$ & $3.56e{-}4$ & $6.12e{-}4$ & $1.15e{-}3$ & $5.15e{-}3$ & $7.96e{-}3$ \\
        \hline
        (1600, 160, 100) & $9.12e{-}6$ & $6.15e{-}5$ & $8.91e{-}5$ & $1.53e{-}4$ & $2.87e{-}4$ & $1.10e{-}3$ & $4.38e{-}3$ \\
        \hline
    \end{tabular}
\end{table}

\begin{figure}
    \centering
    \includegraphics[width=12cm, height=4cm]{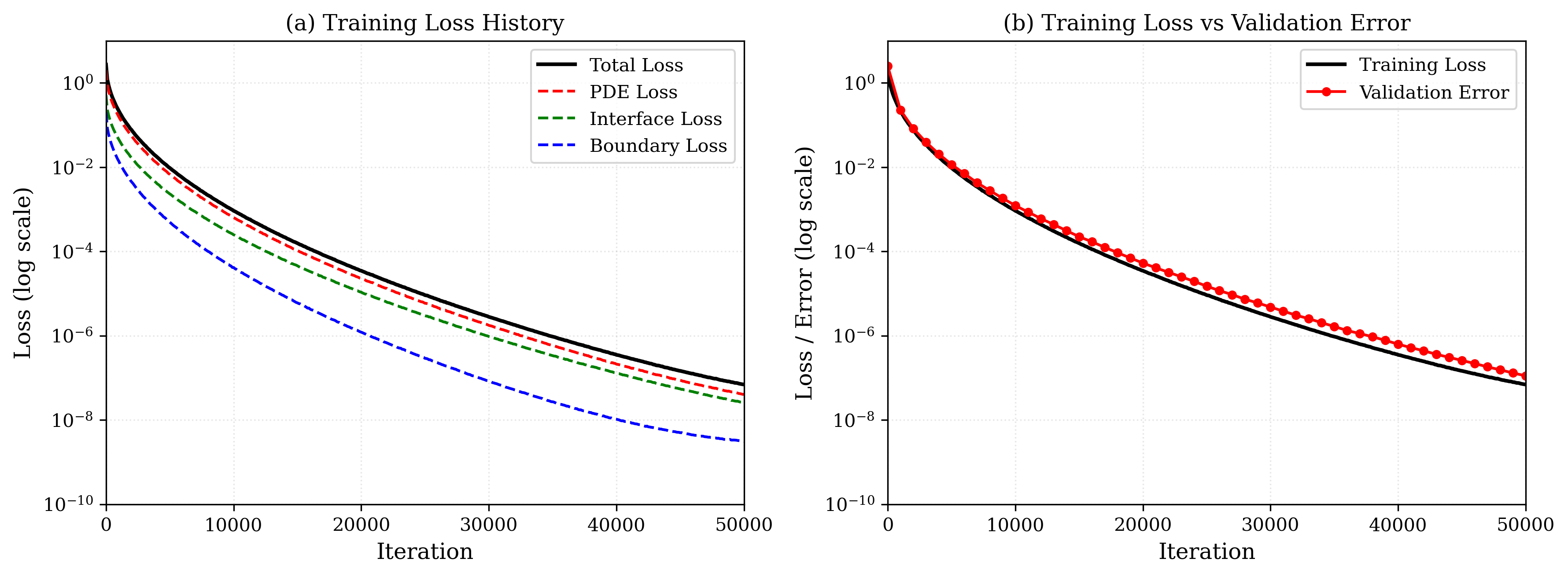}
    \caption{Training convergence for Test 1: (a) Component-wise loss decomposition showing PDE, boundary, and interface loss terms; (b) Comparison of training loss and validation error demonstrating no overfitting.}
    \label{fig:loss_test1}
\end{figure}

Table \ref{errortabletest1} presents the relative $L^2$ error comparison for all methods under different sampling densities. The MAF method achieves the best accuracy across all configurations, reaching $9.12 \times 10^{-6}$ with the finest sampling, which outperforms DCSNN by 6.7$\times$, AdaI-PINN by 9.8$\times$, and I-PINN by 16.8$\times$. Traditional methods (XPINN, MFM) show errors 2--3 orders of magnitude larger. As sampling density increases, MAF demonstrates consistent improvement with an approximate convergence rate of $\mathcal{O}(M^{-1})$. The superior performance of MAF over M-PINN (which uses the same domain decomposition strategy but without multi-activation) confirms that the interface-aware activation weighting provides tangible benefits for learning near discontinuities. Figure~\ref{solutiontest1} (corresponding to the finest sampling configuration with $M_\Omega=1600$, $M_{\partial\Omega}=160$, $M_\Gamma=100$) reveals that errors are concentrated near the interface $\Gamma$ where the solution jump occurs ($\llbracket u \rrbracket = -2$), which aligns with the challenges imposed by the jump conditions.

Figure~\ref{fig:loss_test1} shows the training loss and validation error evolution during the 50,000 training iterations. Both training loss and validation error decrease monotonically, with the validation error closely tracking the training loss throughout the optimization process. This parallel decay demonstrates that no overfitting occurs: the network generalizes well to unseen validation points rather than merely memorizing the training data. The final training loss reaches approximately $10^{-6}$, consistent with the achieved relative $L^2$ error of $9.12 \times 10^{-6}$.

\subsubsection{Test 2: Sunflower-Shaped Interface}

We consider a domain $\Omega = [-1,1] \times [-1,1]$ with a sunflower-shaped interface defined parametrically as $x(\theta) = r(\theta) \cos(\theta) + x_c$, $y(\theta) = r(\theta) \sin(\theta) + y_c$, where $r(\theta) = r_0 + r_1 \sin(\omega\theta)$ with $r_0 = 0.4$, $r_1 = 0.2$, $\omega=20$, and $(x_c,y_c) = (0.02\sqrt{5},0.02\sqrt{5})$. The subdomain $\Omega_1$ is the interior of the sunflower curve and $\Omega_2 = \Omega \setminus \overline{\Omega}_1$. The diffusion coefficients are $\beta_1 = 1$ in $\Omega_1$ and $\beta_2 = 10$ in $\Omega_2$. The exact solution is $u_1 = r^2/\beta_1$ in $\Omega_1$ and $u_2 = (r^4 - 0.1\ln(2r))/\beta_2$ in $\Omega_2$, where $r=\sqrt{x^2+y^2}$. The source terms are $f_1 = -4$ in $\Omega_1$ and $f_2 = -(12r^2 + 0.1/r^2)/\beta_2$ in $\Omega_2$. The jump conditions are $g_1 = u_1|_\Gamma - u_2|_\Gamma$ and $g_2 = \beta_1 \nabla u_1 \cdot \boldsymbol{n} - \beta_2 \nabla u_2 \cdot \boldsymbol{n}$, computed from the exact solution. The Dirichlet boundary condition is $g_D = u_2|_{\partial\Omega}$. For this parametric interface, the distance function is computed numerically as $d(\mathbf{x}, \Gamma) = \min_{\theta} \|\mathbf{x} - \mathbf{r}(\theta)\|$.

We distribute collocation points randomly in each subdomain and along the interface. The training and validation points are shown in Figure~\ref{pointstest2}, and the numerical solutions with absolute errors are presented in Figure~\ref{solutiontest2}. Figure~\ref{fig:weight_functions_test2} illustrates the interface-aware weight functions along a radial line crossing the sunflower interface.

\begin{figure}
    \centering
    \subfigure[Training points]{\includegraphics[width=5cm, height=4cm]{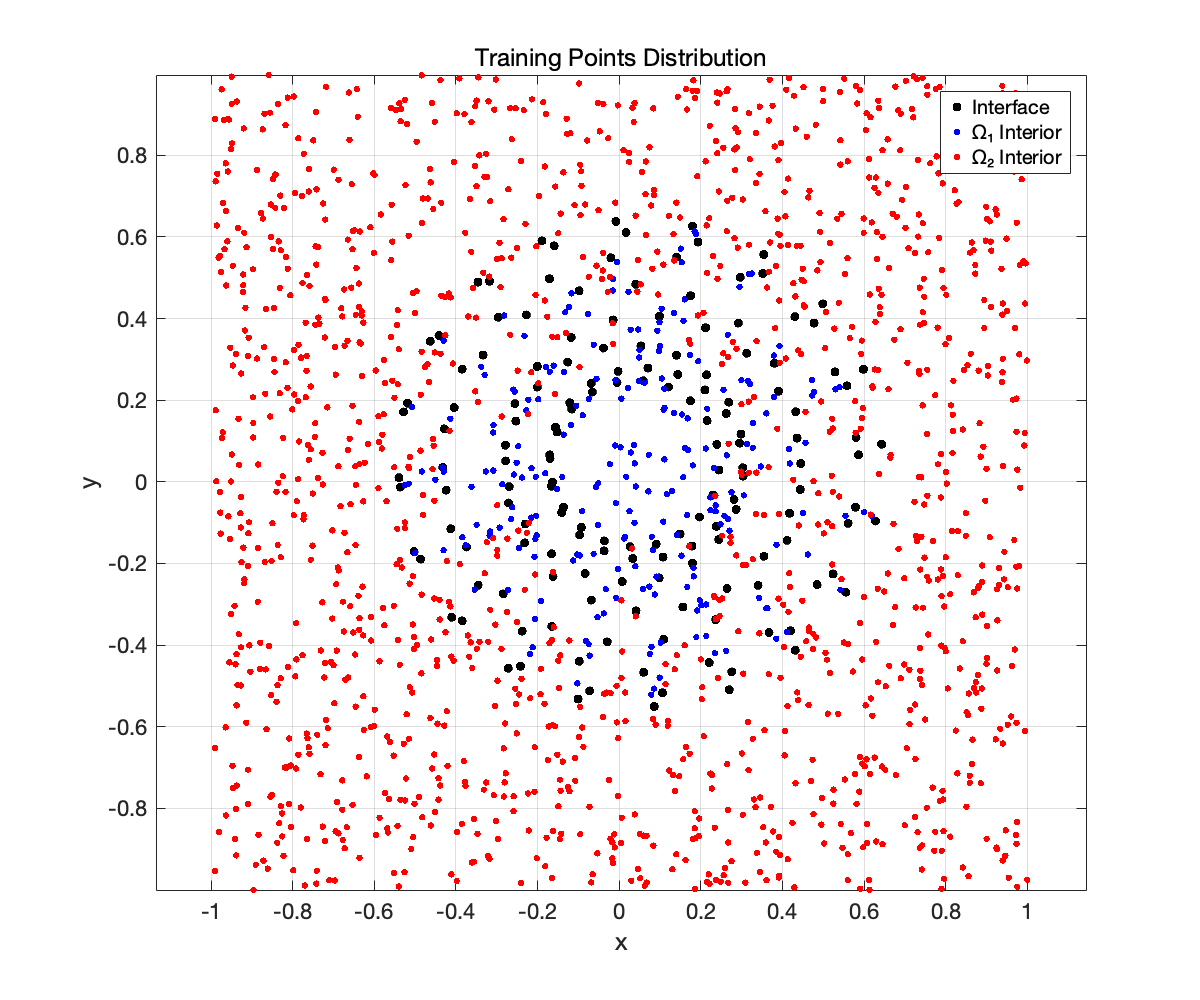}}
    \subfigure[Validation points]{\includegraphics[width=5cm, height=4cm]{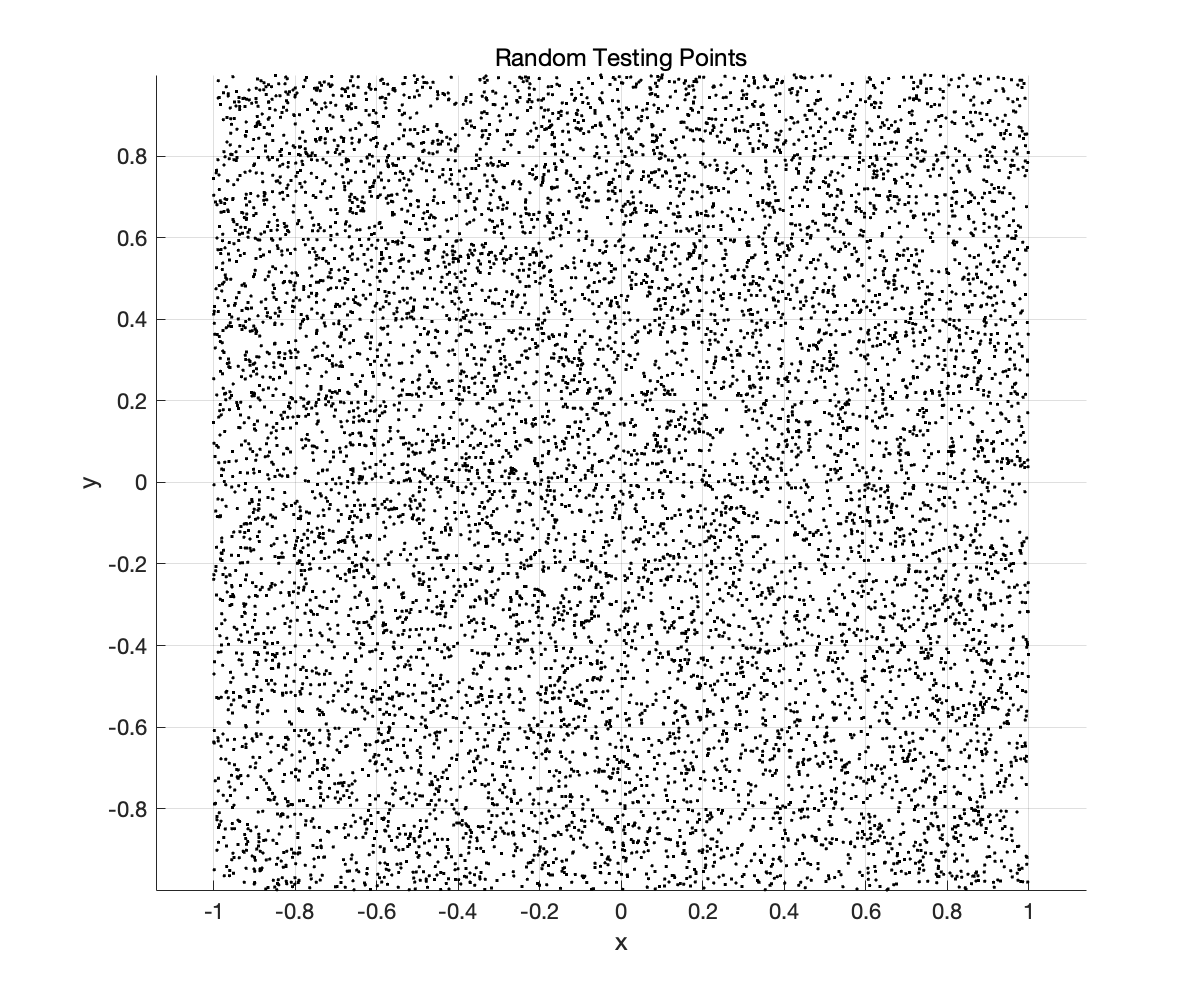}}
    \caption{Training and validation points for Test 2 ($M_{\Omega} = 1600$, $M_{\Gamma}=100$).}
    \label{pointstest2}
\end{figure}

\begin{figure}
    \centering
    \subfigure[$u_{\text{Exact}}$]{\includegraphics[width=3.5cm, height=3cm]{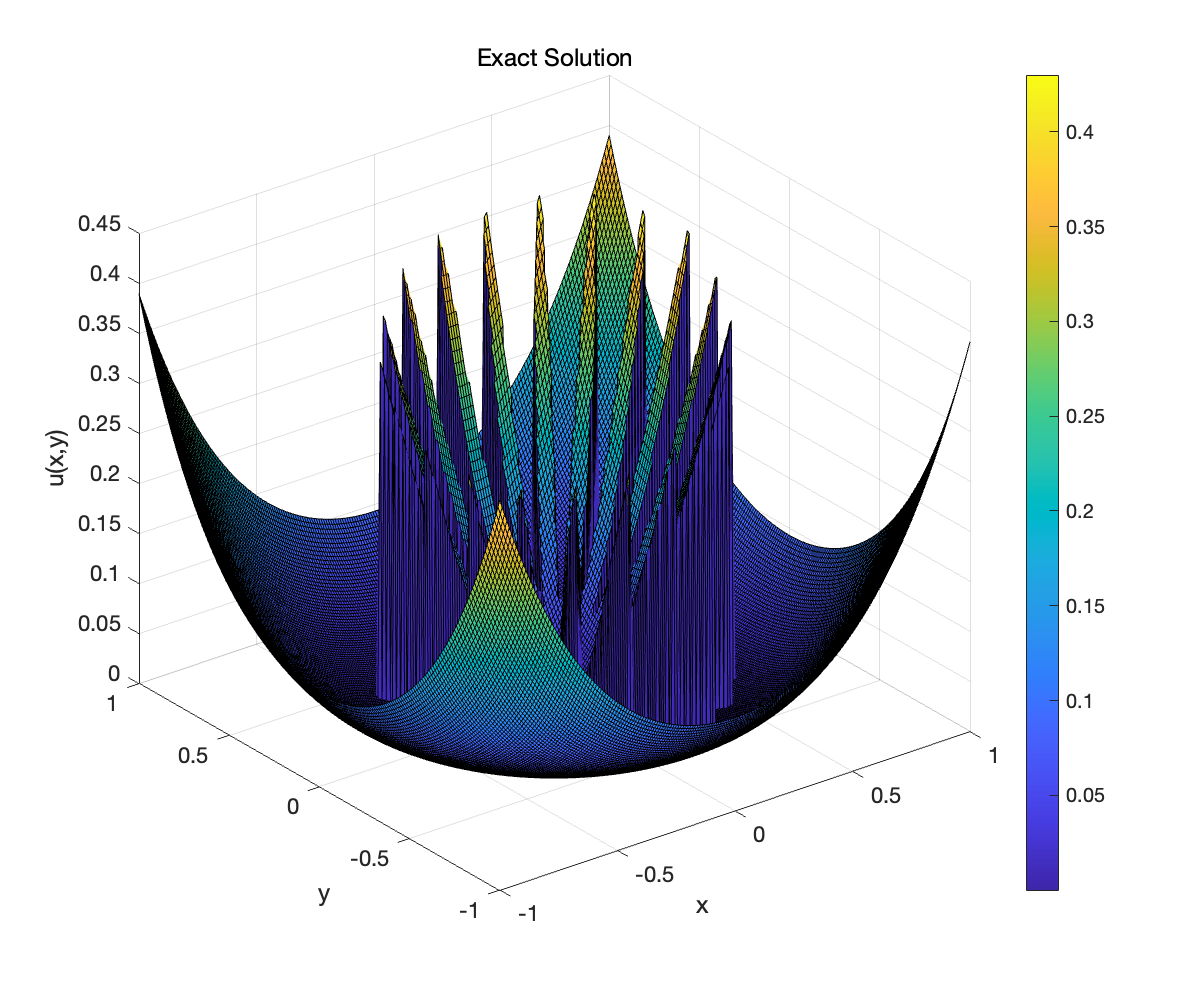}}
    \subfigure[$u_{\text{DNN}}$]{\includegraphics[width=3.5cm, height=3cm]{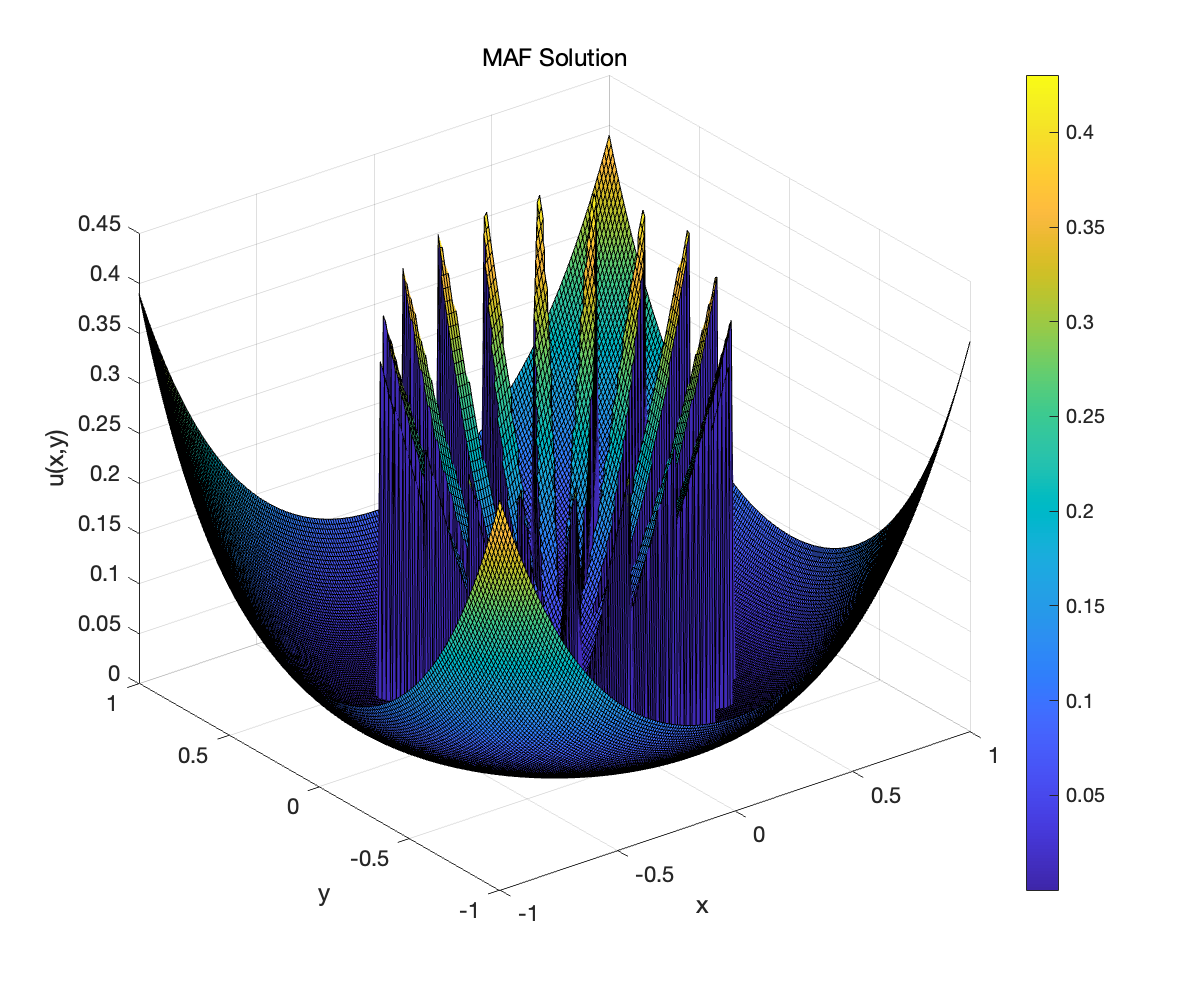}}
    \subfigure[Absolute Error]{\includegraphics[width=3.5cm, height=3cm]{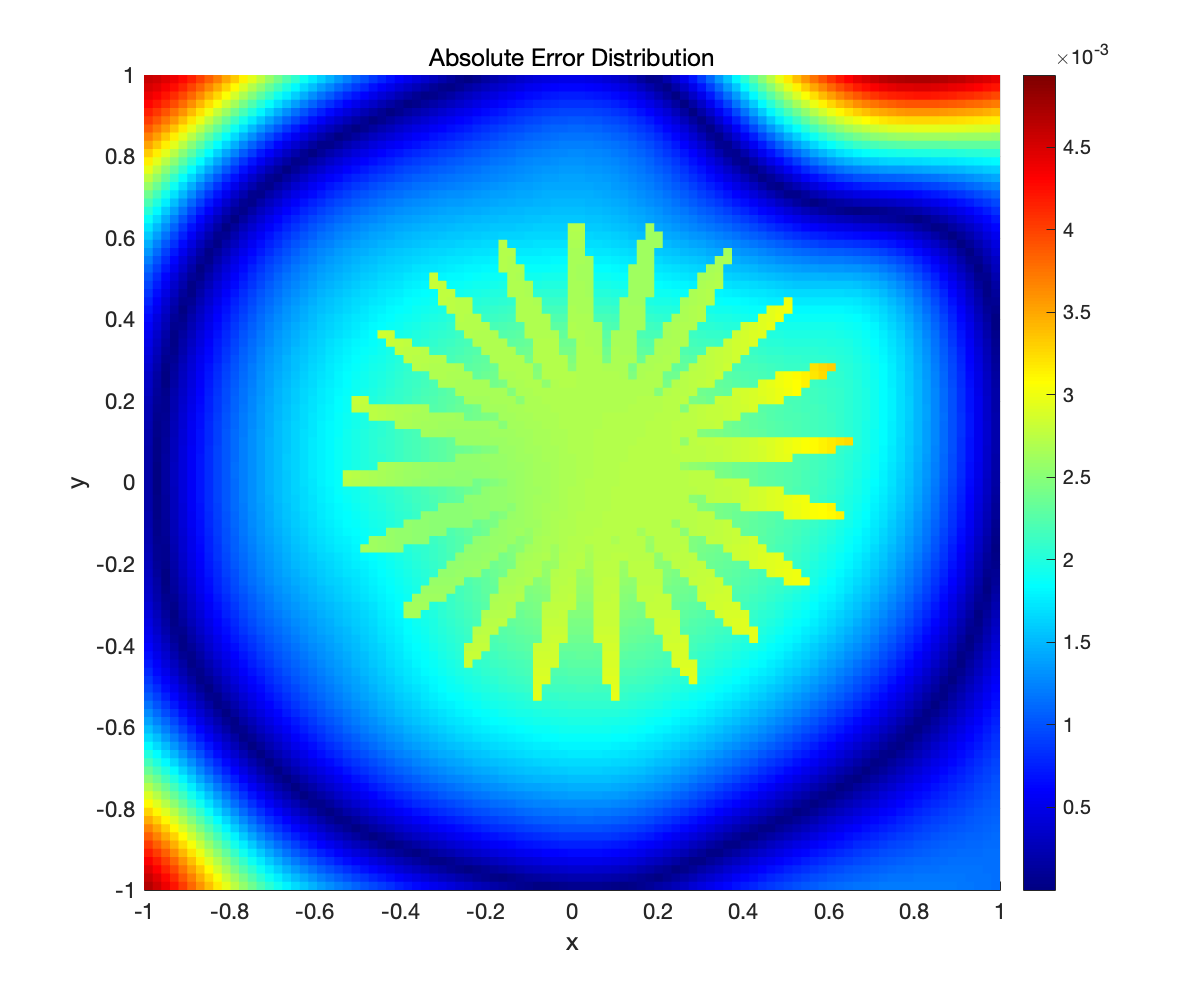}}
    \caption{$u_{\text{DNN}}$, $u_{\text{Exact}}$ and absolute error for Test 2.}
    \label{solutiontest2}
\end{figure}

\begin{figure}
    \centering
    \includegraphics[width=8cm, height=5cm]{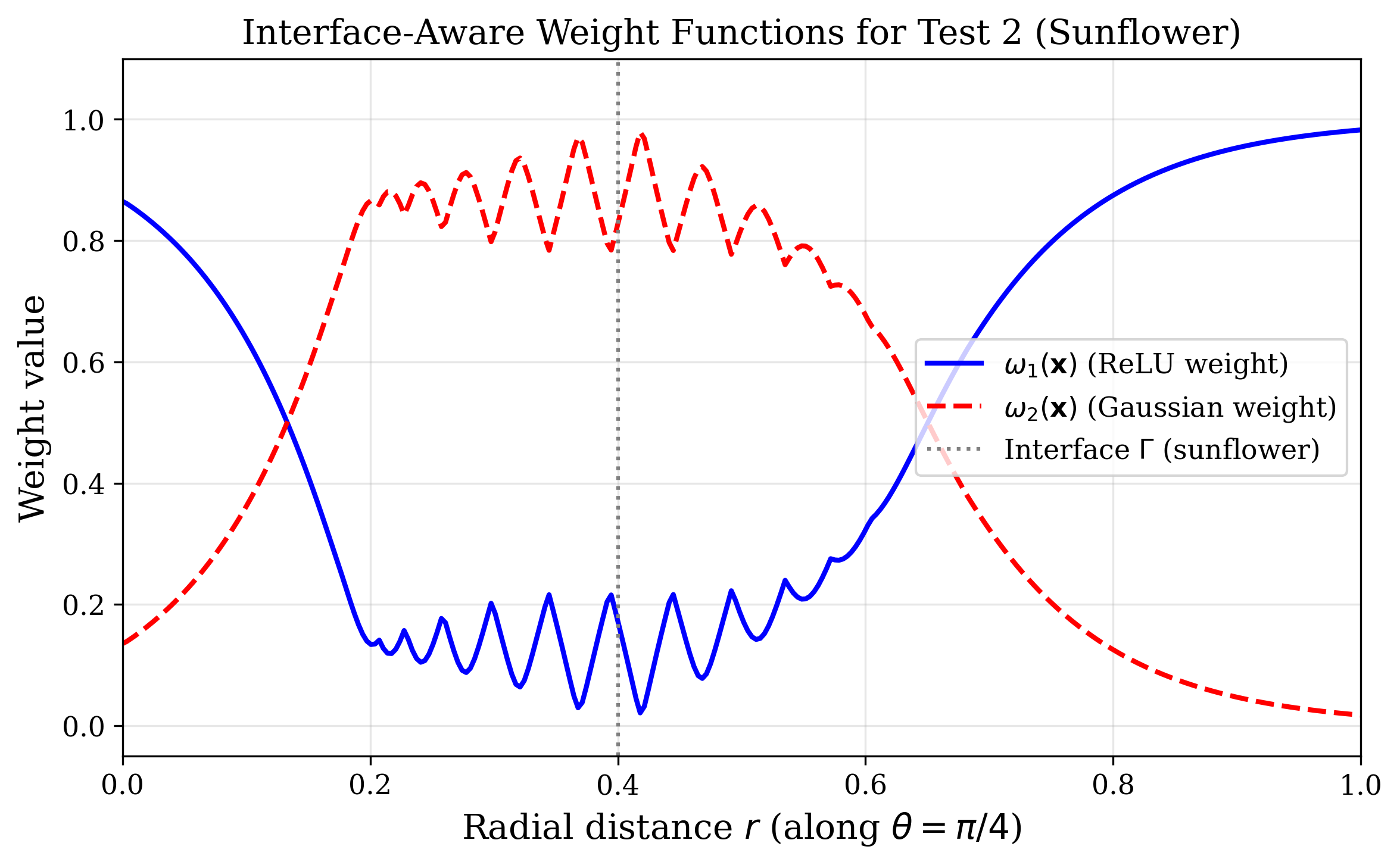}
    \caption{Interface-aware weight functions $\omega_1(\mathbf{x})$ and $\omega_2(\mathbf{x})$ along a radial line crossing the sunflower interface for Test 2.}
    \label{fig:weight_functions_test2}
\end{figure}

Table~\ref{errortabletest2} presents the relative $L^2$ error comparison for all methods. The MAF method achieves the best accuracy ($9.43 \times 10^{-4}$) with the finest sampling, comparable to DCSNN and significantly outperforming XPINN and MFM. The complex sunflower geometry with 20 petals presents a challenging test case, yet MAF demonstrates robust performance in capturing the solution behavior near this intricate interface. Figure~\ref{solutiontest2} (corresponding to the finest sampling configuration with $M_\Omega=1600$, $M_{\partial\Omega}=160$, $M_\Gamma=100$) shows that errors are primarily concentrated near the interface region where the solution exhibits both geometric complexity and discontinuity, which aligns with the challenges imposed by the intricate sunflower-shaped jump conditions.

\begin{table}[H]
    \centering
    \caption{Relative $L^2$ error comparison for Test 2.}
    \label{errortabletest2}
    \small
    \begin{tabular}{|c|c|c|c|c|c|c|c|}
        \hline
        $(M_\Omega, M_{\partial\Omega}, M_\Gamma)$ & MAF & DCSNN & AdaI & I-PINN & M-PINN & MFM & XPINN \\
        \hline
        (100, 40, 25) & $2.09e{-}2$ & $4.33e{-}2$ & $5.12e{-}2$ & $5.89e{-}2$ & $6.01e{-}2$ & $6.25e{-}2$ & $8.32e{-}2$ \\
        \hline
        (400, 80, 50) & $6.89e{-}3$ & $6.05e{-}3$ & $1.25e{-}2$ & $1.87e{-}2$ & $2.15e{-}2$ & $3.33e{-}2$ & $4.73e{-}2$ \\
        \hline
        (1600, 160, 100) & $9.43e{-}4$ & $9.02e{-}4$ & $2.15e{-}3$ & $4.56e{-}3$ & $6.78e{-}3$ & $1.12e{-}2$ & $3.38e{-}2$ \\
        \hline
    \end{tabular}
\end{table}

\subsubsection{Test 3: Elliptic-Shaped Interface}

We consider a domain $\Omega = [-1,1] \times [-1,1]$ with an elliptical interface $\Gamma: (x/0.2)^2 + (y/0.5)^2 = 1$. The subdomain $\Omega_1$ is the interior of the ellipse and $\Omega_2 = \Omega \setminus \overline{\Omega}_1$. This test case features a high-contrast diffusion coefficient: $\beta_1 = 10^{-3}$ in $\Omega_1$ and $\beta_2 = 1$ in $\Omega_2$. The exact solution is $u_1 = e^x e^y$ in $\Omega_1$ and $u_2 = \sin(x)\sin(y)$ in $\Omega_2$. The source terms are $f_1 = -2 \times 10^{-3} e^{x+y}$ in $\Omega_1$ and $f_2 = 2\sin(x)\sin(y)$ in $\Omega_2$. The jump conditions are $g_1 = u_1|_\Gamma - u_2|_\Gamma = e^{x+y} - \sin(x)\sin(y)$ on $\Gamma$ and $g_2 = \beta_1 \nabla u_1 \cdot \boldsymbol{n} - \beta_2 \nabla u_2 \cdot \boldsymbol{n}$, computed analytically. The Dirichlet boundary condition is $g_D = \sin(x)\sin(y)$ on $\partial\Omega$. For this elliptical interface, the distance function is computed via Newton iteration to find the closest point on the ellipse.

We distribute collocation points randomly in each subdomain and along the interface. The training and validation points are shown in Figure~\ref{pointstest3}, and the numerical solutions with absolute errors are presented in Figure~\ref{solutiontest3}. Figure~\ref{fig:weight_functions_test3} illustrates the interface-aware weight functions along the $x$-axis crossing the ellipse interface.

\begin{figure}[H]
    \centering
    \subfigure[Training points]{\includegraphics[width=5cm, height=4cm]{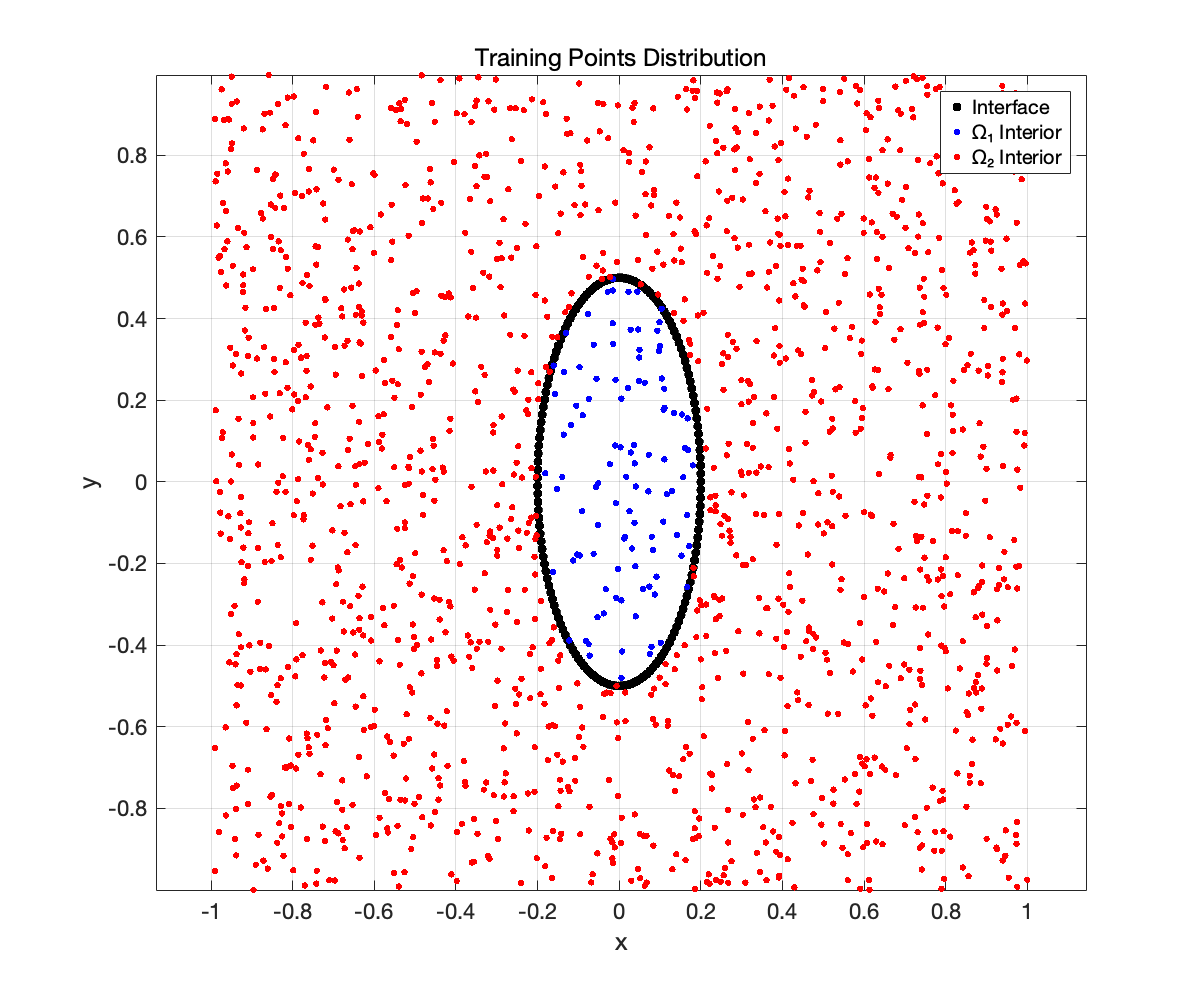}}
    \subfigure[Validation points]{\includegraphics[width=5cm, height=4cm]{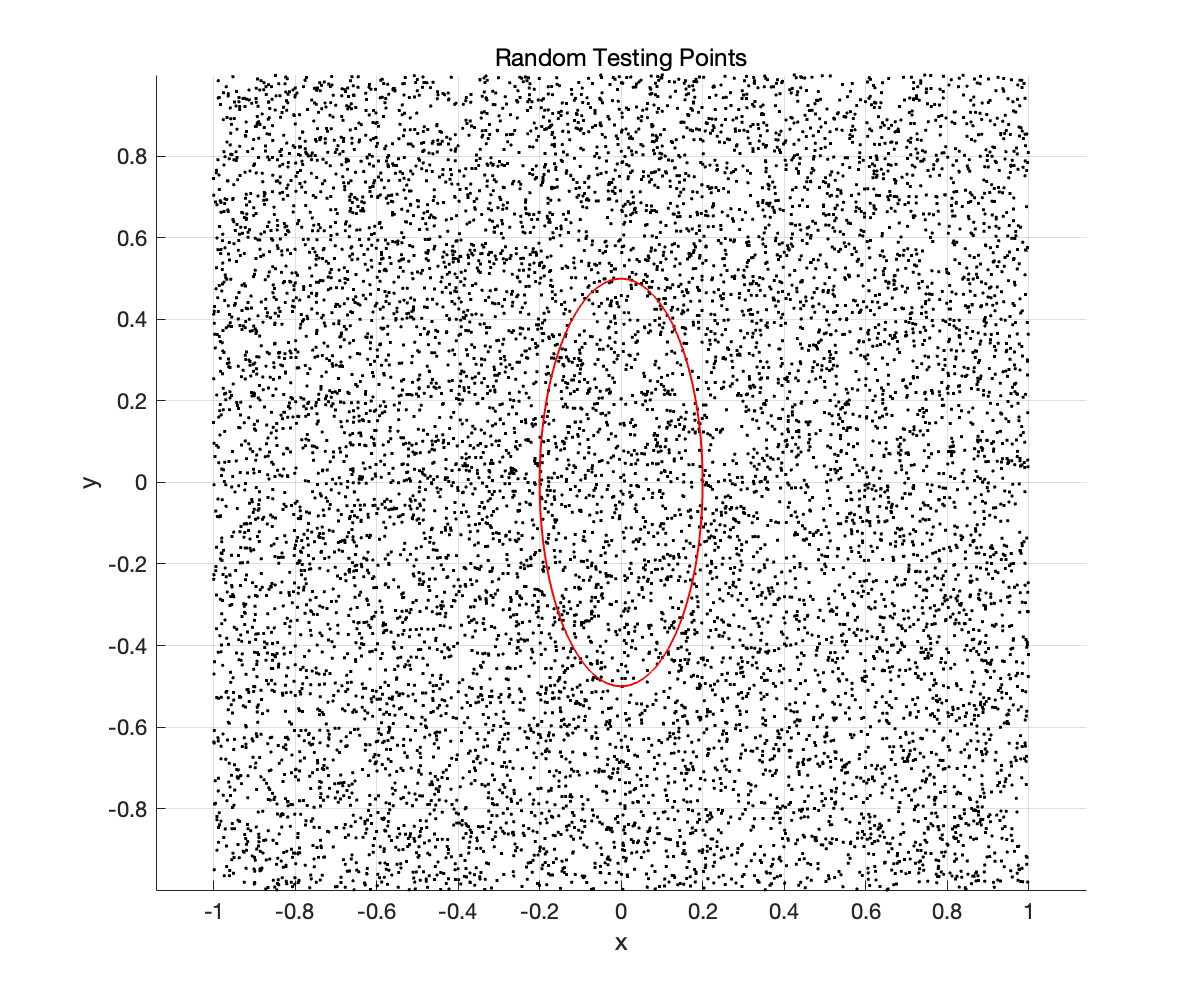}}
    \caption{Training and validation points for Test 3 ($M_{\Omega} = 1600$, $M_{\Gamma}=100$).}
    \label{pointstest3}
\end{figure}

\begin{figure}[H]
    \centering
    \subfigure[$u_{\text{Exact}}$]{\includegraphics[width=5cm, height=4cm]{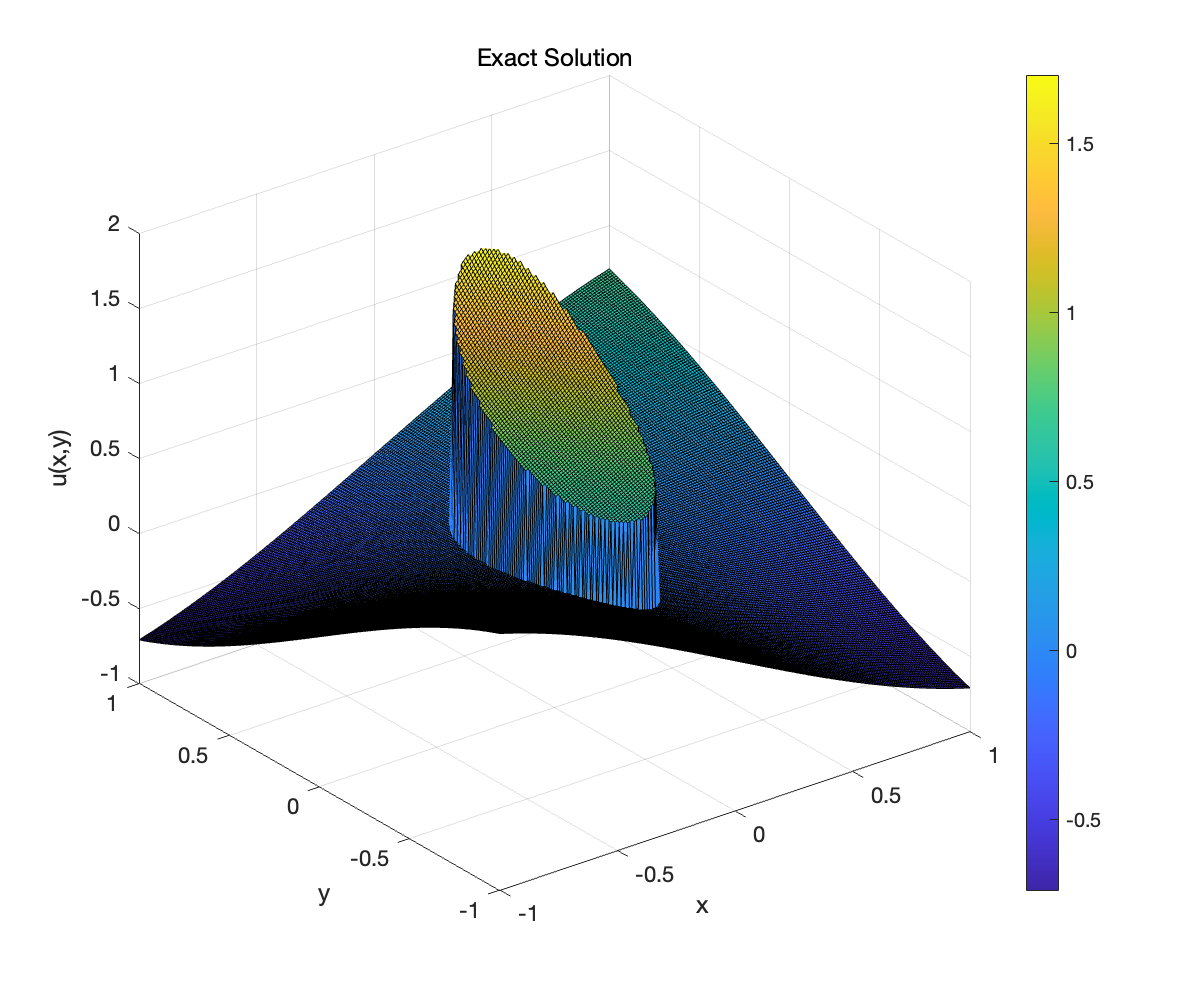}}
    \subfigure[$u_{\text{DNN}}$]{\includegraphics[width=5cm, height=4cm]{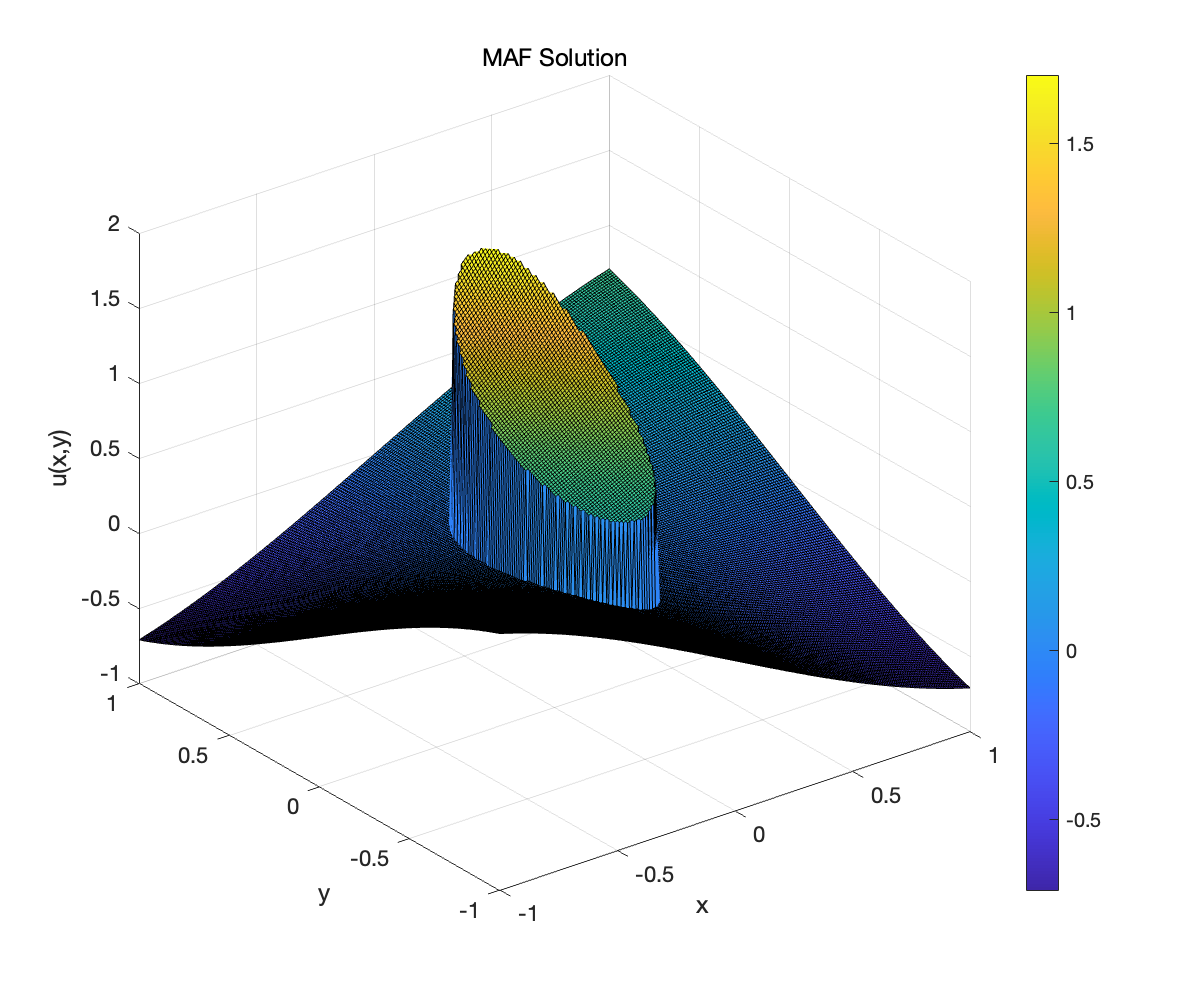}}
    \subfigure[Absolute Error]{\includegraphics[width=5cm, height=4cm]{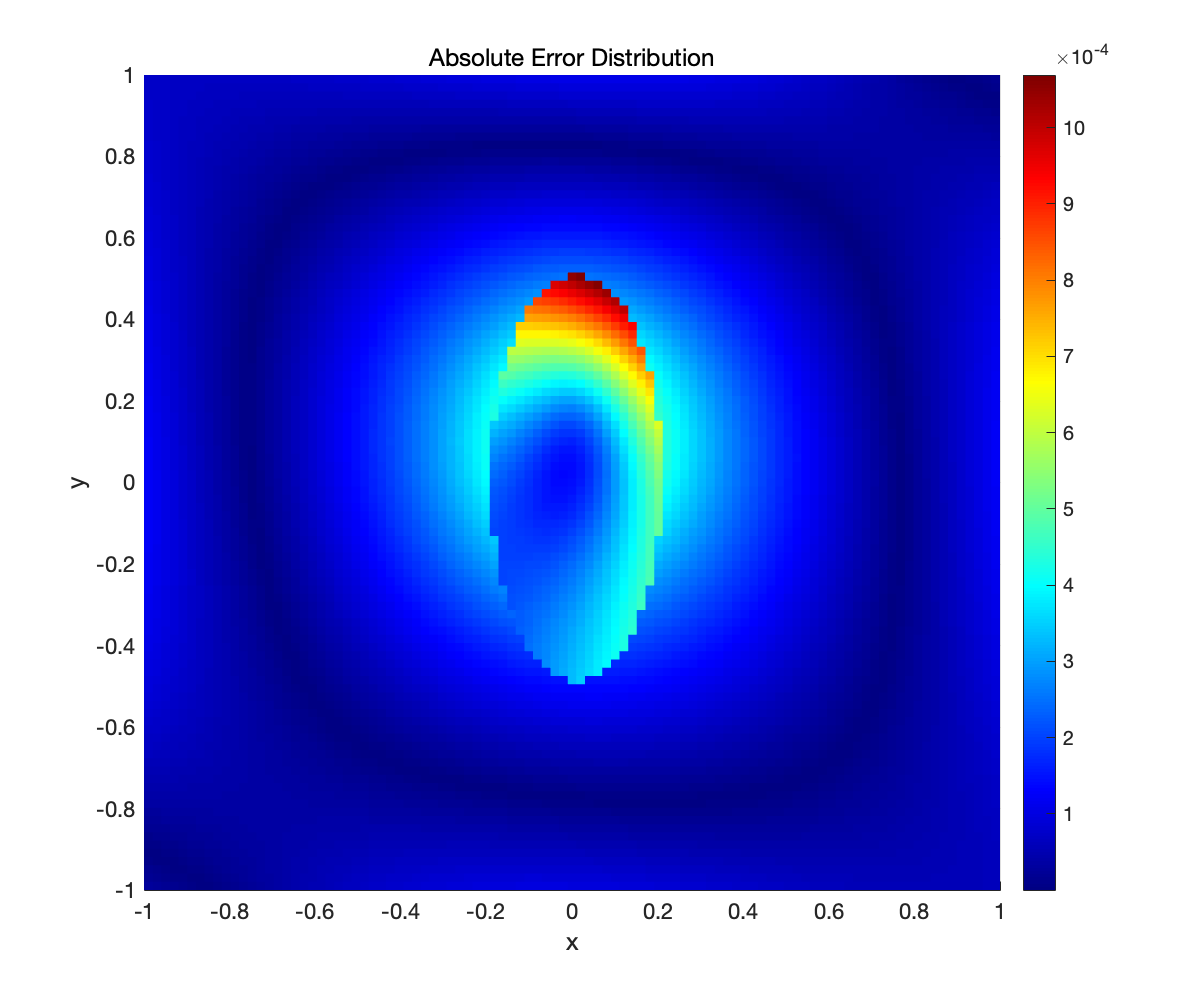}}
    \caption{$u_{\text{DNN}}$, $u_{\text{Exact}}$ and absolute error for Test 3.}
    \label{solutiontest3}
\end{figure}

\begin{figure}[H]
    \centering
    \includegraphics[width=8cm, height=5cm]{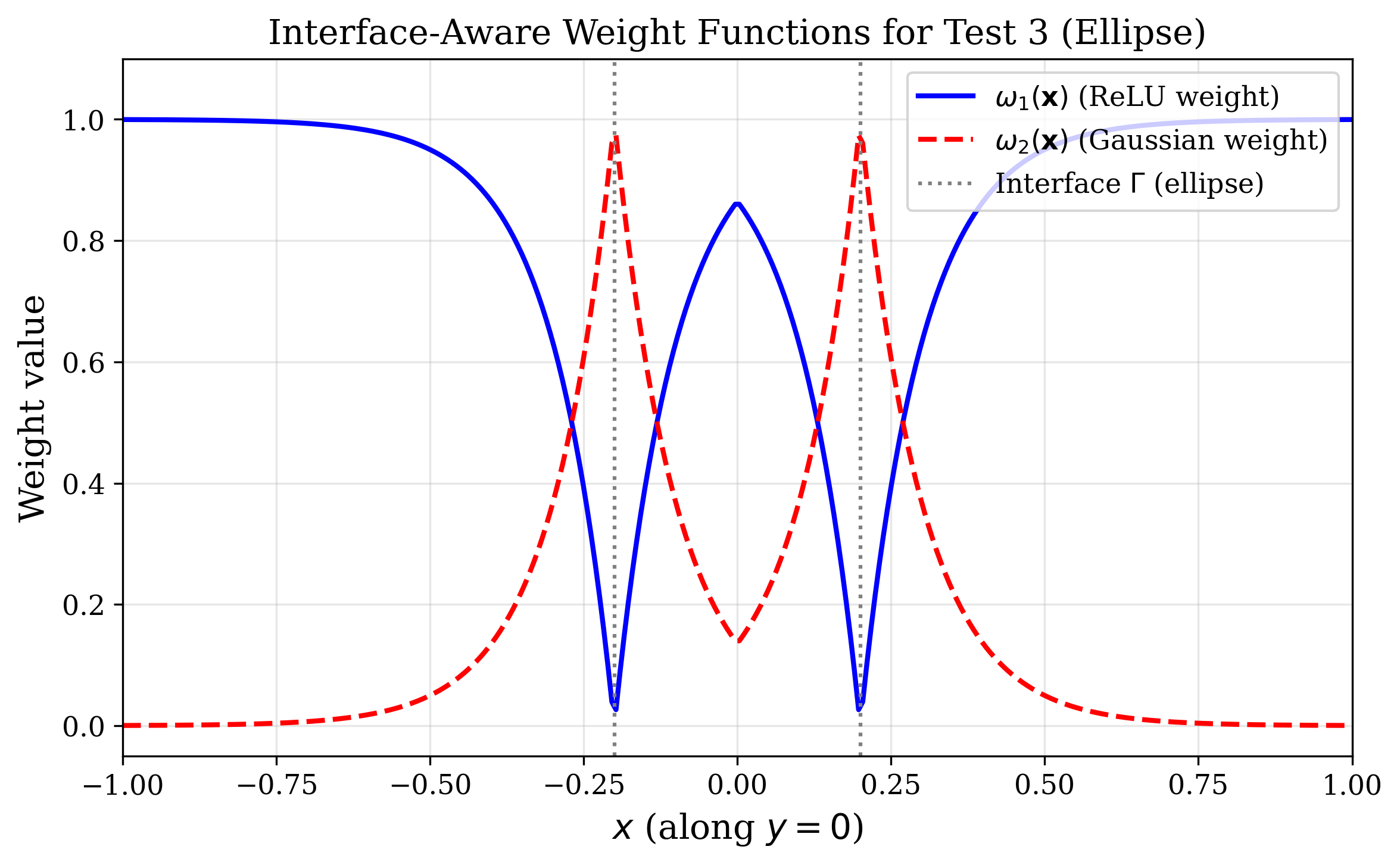}
    \caption{Interface-aware weight functions $\omega_1(\mathbf{x})$ and $\omega_2(\mathbf{x})$ along the $x$-axis crossing the ellipse interface for Test 3.}
    \label{fig:weight_functions_test3}
\end{figure}

Table~\ref{errortabletest3} presents the relative $L^2$ error comparison for all methods. This test case presents a particularly challenging scenario due to the high contrast in diffusion coefficients ($10^{-3}$ versus $1$) and the distinct functional forms of the solution in each subdomain. The MAF method achieves remarkable accuracy ($1.71 \times 10^{-4}$) with the finest sampling, outperforming XPINN by 38$\times$ and MFM by 25$\times$, while maintaining comparable accuracy to DCSNN. Figure~\ref{solutiontest3} (corresponding to the finest sampling configuration with $M_\Omega=1600$, $M_{\partial\Omega}=160$, $M_\Gamma=100$) shows that errors are primarily concentrated near the elliptical interface, which aligns with the challenges imposed by the high-contrast coefficient jump conditions ($\beta_1/\beta_2 = 10^{-3}$).

\begin{table}[H]
    \centering
    \caption{Relative $L^2$ error comparison for Test 3.}
    \label{errortabletest3}
    \small
    \begin{tabular}{|c|c|c|c|c|c|c|c|}
        \hline
        $(M_\Omega, M_{\partial\Omega}, M_\Gamma)$ & MAF & DCSNN & AdaI & I-PINN & M-PINN & MFM & XPINN \\
        \hline
        (100, 40, 25) & $3.32e{-}3$ & $4.11e{-}3$ & $6.78e{-}3$ & $1.12e{-}2$ & $1.89e{-}2$ & $2.53e{-}2$ & $4.84e{-}2$ \\
        \hline
        (400, 80, 50) & $7.06e{-}4$ & $9.07e{-}4$ & $1.56e{-}3$ & $2.89e{-}3$ & $5.12e{-}3$ & $7.43e{-}3$ & $8.49e{-}3$ \\
        \hline
        (1600, 160, 100) & $1.71e{-}4$ & $5.85e{-}4$ & $4.23e{-}4$ & $7.56e{-}4$ & $1.35e{-}3$ & $4.29e{-}3$ & $6.55e{-}3$ \\
        \hline
    \end{tabular}
\end{table}

\subsubsection{Test 4: Flower-Shaped Interface}

We consider a flower-shaped domain $\Omega$ with boundary $r(\theta) = 1 - 0.3\cos(5\theta)$ and an interior flower-shaped interface $\Gamma: r(\theta) = 0.4 - 0.2\cos(5\theta)$. The subdomain $\Omega_1$ is the interior of the flower curve and $\Omega_2 = \Omega \setminus \overline{\Omega}_1$. The diffusion coefficient is $\beta_1 = 1$ in $\Omega_1$ and $\beta_2 = 10$ in $\Omega_2$. The exact solution is $u_1 = e^x e^y$ in $\Omega_1$ and $u_2 = \sin(x)\sin(y)$ in $\Omega_2$. The source terms are $f_1 = -2e^{x+y}$ in $\Omega_1$ and $f_2 = 20\sin(x)\sin(y)$ in $\Omega_2$. The jump conditions are $g_1 = u_1|_\Gamma - u_2|_\Gamma = e^{x+y} - \sin(x)\sin(y)$ on $\Gamma$ and $g_2 = \beta_1 \nabla u_1 \cdot \boldsymbol{n} - \beta_2 \nabla u_2 \cdot \boldsymbol{n}$, computed analytically from the exact solution. The Dirichlet boundary condition is $g_D = \sin(x)\sin(y)$ on $\partial\Omega$. For this parametric interface, the distance function is computed numerically as $d(\mathbf{x}, \Gamma) = \min_{\theta} \|\mathbf{x} - \mathbf{r}(\theta)\|$.

We distribute collocation points randomly in each subdomain and along the interface. The training and validation points are shown in Figure~\ref{pointstest4}, and the numerical solutions with absolute errors are presented in Figure~\ref{solutiontest4}. Figure~\ref{fig:weight_functions_test4} illustrates the interface-aware weight functions along a radial line crossing the flower interface.

\begin{figure}[H]
    \centering
    \subfigure[Training points]{\includegraphics[width=5cm, height=4cm]{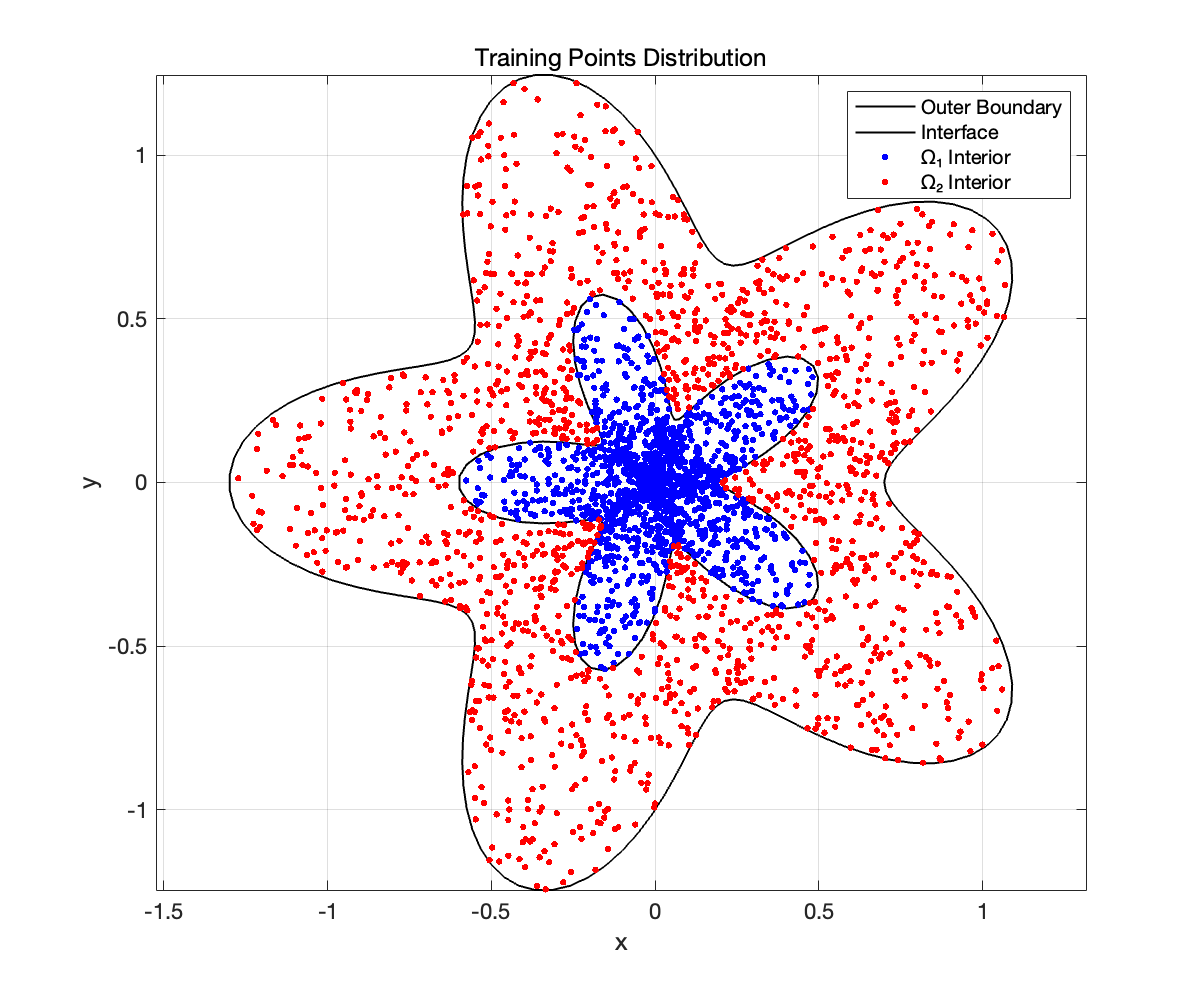}}
    \subfigure[Validation points]{\includegraphics[width=5cm, height=4cm]{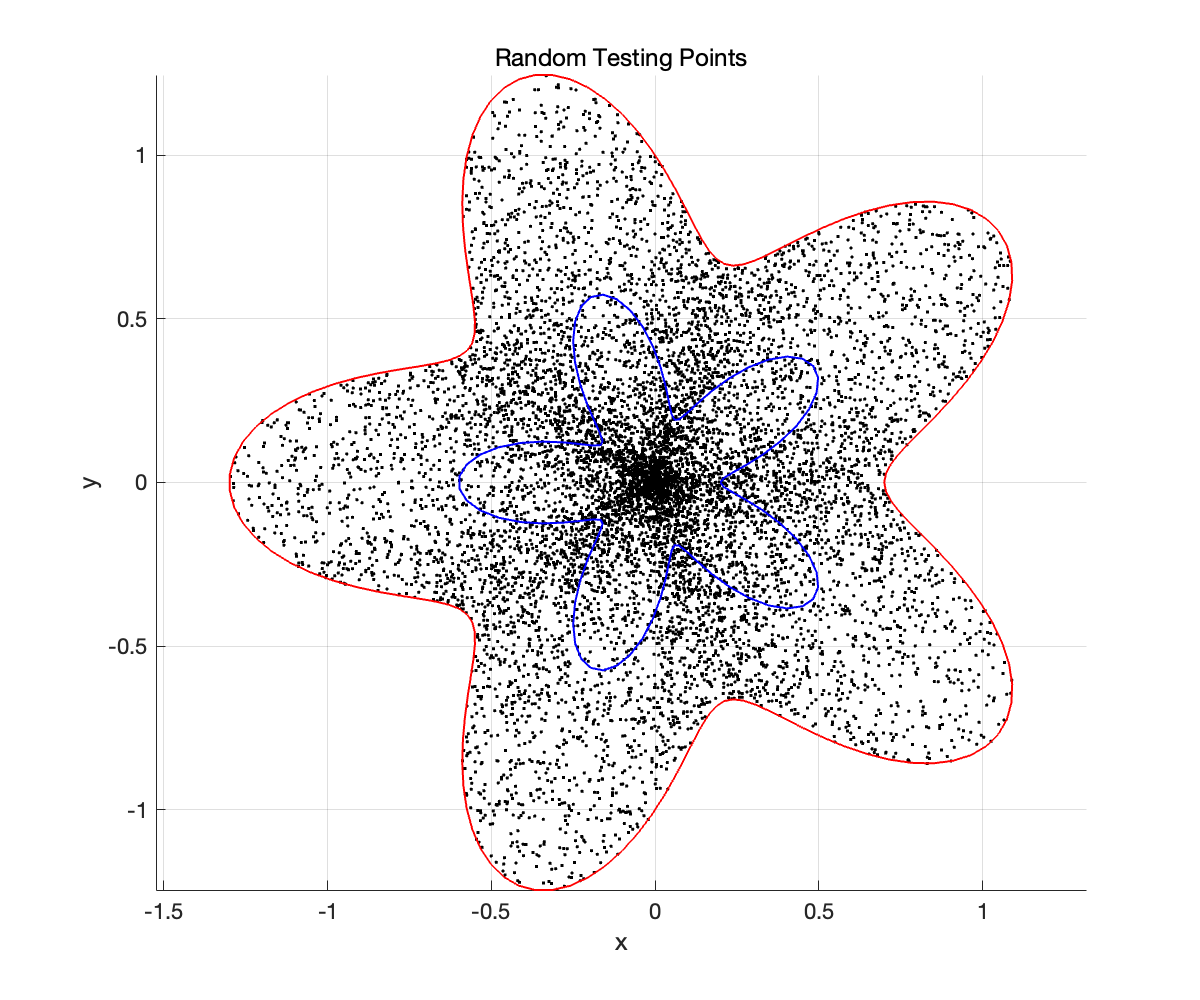}}
    \caption{Training and validation points for Test 4 ($M_{\Omega} = 1600$, $M_{\Gamma}=100$).}
    \label{pointstest4}
\end{figure}

\begin{figure}[H]
    \centering
    \subfigure[$u_{\text{Exact}}$]{\includegraphics[width=5cm, height=4cm]{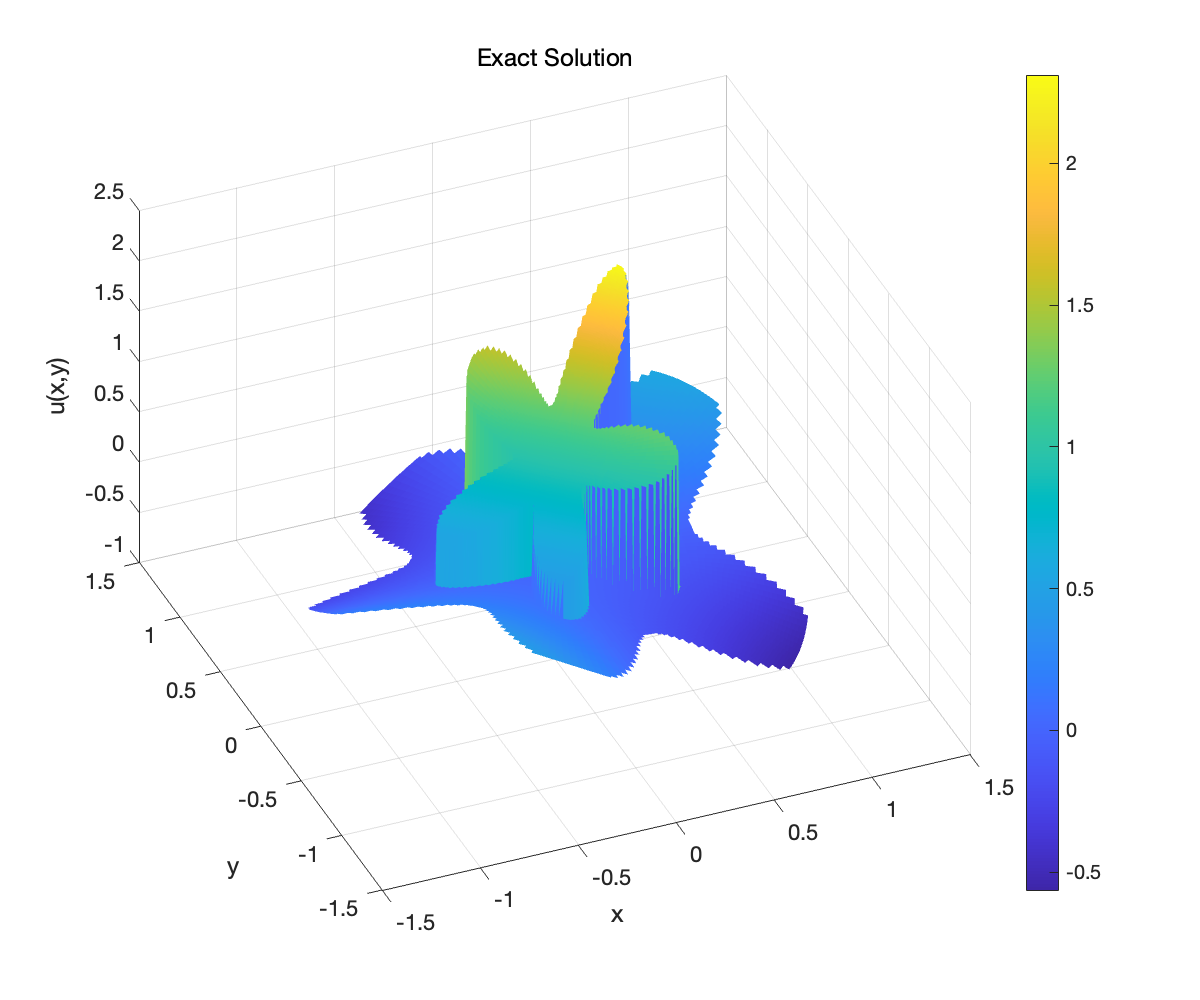}}
    \subfigure[$u_{\text{DNN}}$]{\includegraphics[width=5cm, height=4cm]{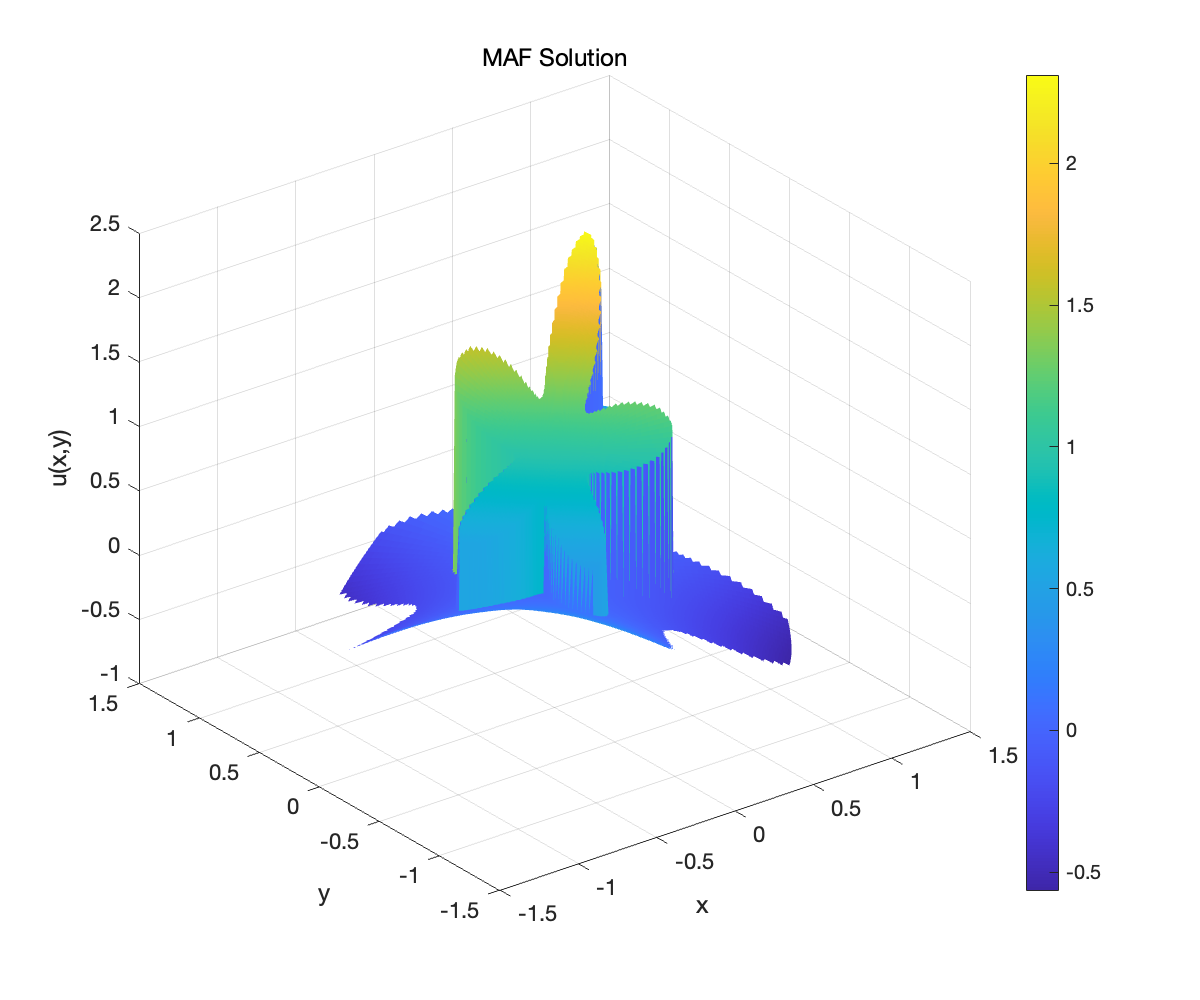}}
    \subfigure[Absolute Error]{\includegraphics[width=5cm, height=4cm]{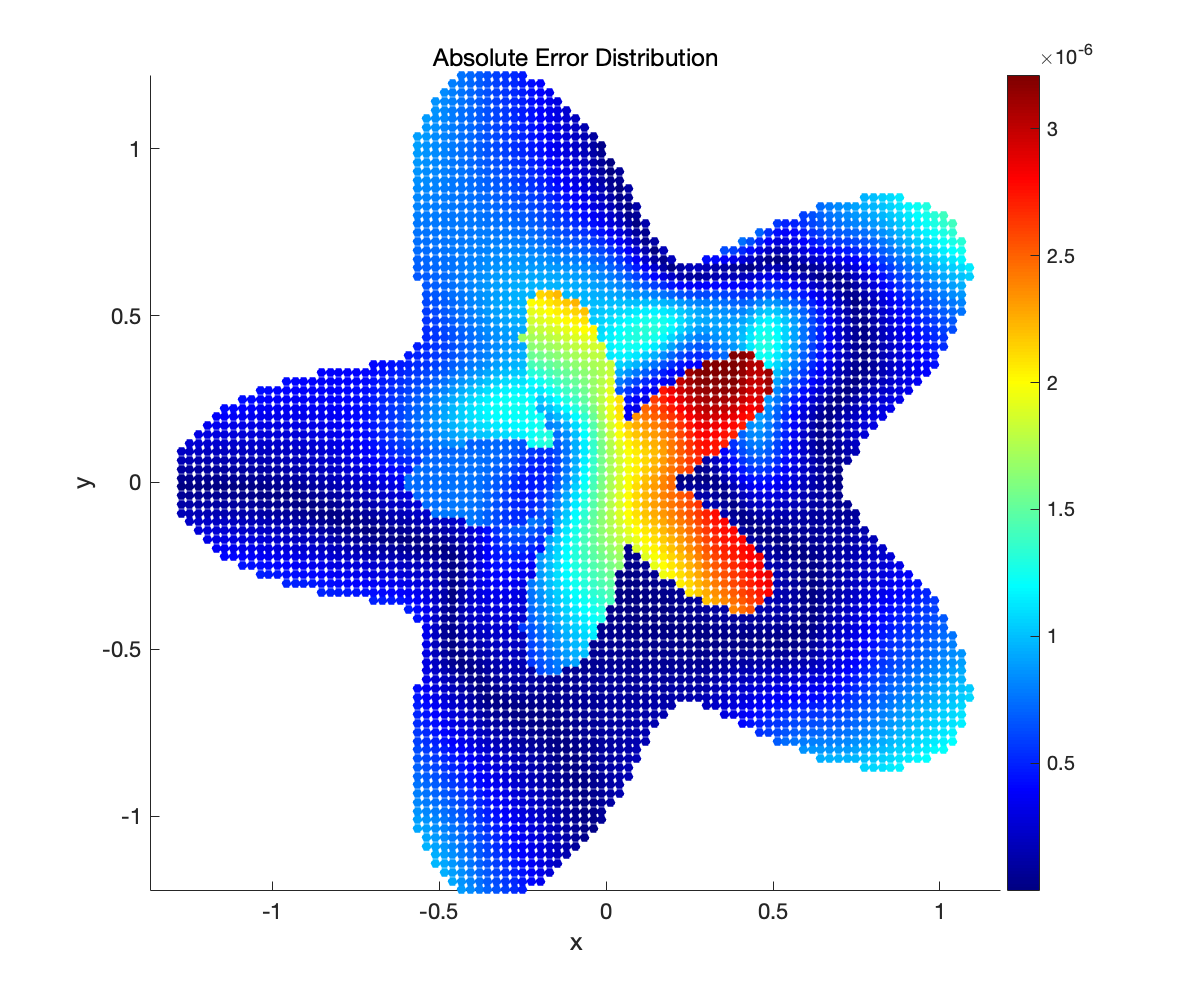}}
    \caption{$u_{\text{DNN}}$, $u_{\text{Exact}}$ and absolute error for Test 4.}
    \label{solutiontest4}
\end{figure}

\begin{figure}[H]
    \centering
    \includegraphics[width=8cm, height=5cm]{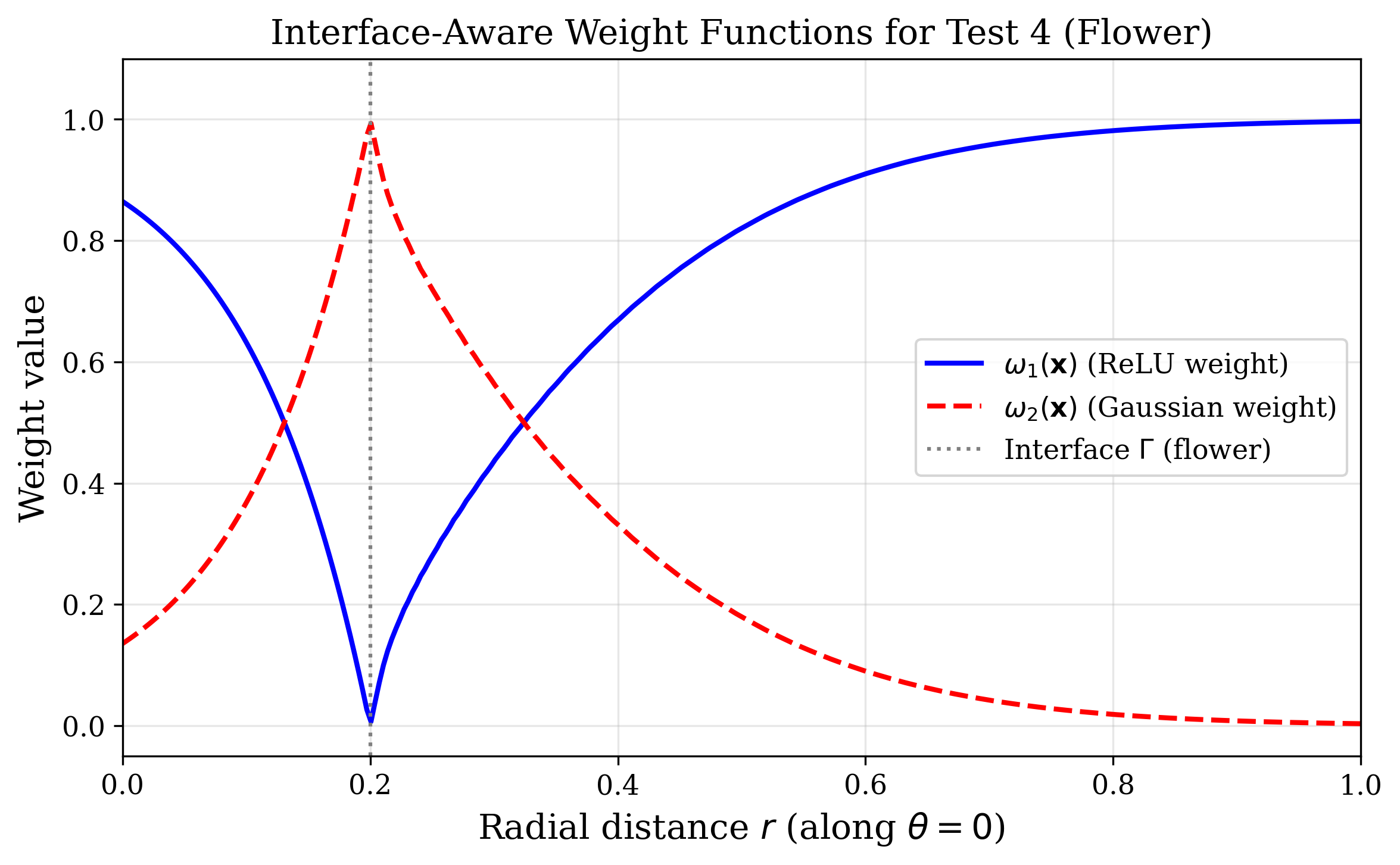}
    \caption{Interface-aware weight functions $\omega_1(\mathbf{x})$ and $\omega_2(\mathbf{x})$ along a radial line crossing the flower interface for Test 4.}
    \label{fig:weight_functions_test4}
\end{figure}

Table~\ref{errortabletest4} presents the relative $L^2$ error comparison for all methods. This test case demonstrates exceptional performance of the MAF method on a problem combining geometric complexity (5-petal flower) with distinct solution behaviors (exponential vs. trigonometric). The MAF method achieves remarkable accuracy ($8.99 \times 10^{-7}$) with the finest sampling, comparable to DCSNN and significantly outperforming XPINN by 611$\times$ and MFM by 112$\times$. Figure~\ref{solutiontest4} (corresponding to the finest sampling configuration with $M_\Omega=1600$, $M_{\partial\Omega}=160$, $M_\Gamma=100$) shows that errors are primarily concentrated near the flower-shaped interface vertices where the curvature is highest, which aligns with the challenges imposed by the complex geometry and distinct solution behaviors ($e^{x+y}$ vs. $\sin(x)\sin(y)$).

\begin{table}[H]
    \centering
    \caption{Relative $L^2$ error comparison for Test 4.}
    \label{errortabletest4}
    \small
    \begin{tabular}{|c|c|c|c|c|c|c|c|}
        \hline
        $(M_\Omega, M_{\partial\Omega}, M_\Gamma)$ & MAF & DCSNN & AdaI & I-PINN & M-PINN & MFM & XPINN \\
        \hline
        (100, 40, 25) & $3.85e{-}5$ & $3.12e{-}5$ & $5.67e{-}5$ & $1.23e{-}4$ & $3.56e{-}4$ & $8.52e{-}4$ & $2.98e{-}3$ \\
        \hline
        (400, 80, 50) & $6.26e{-}6$ & $7.01e{-}6$ & $1.15e{-}5$ & $2.89e{-}5$ & $8.76e{-}5$ & $3.47e{-}4$ & $6.48e{-}4$ \\
        \hline
        (1600, 160, 100) & $8.99e{-}7$ & $8.89e{-}7$ & $2.34e{-}6$ & $6.78e{-}6$ & $2.15e{-}5$ & $1.01e{-}4$ & $5.49e{-}4$ \\
        \hline
    \end{tabular}
\end{table}

\subsubsection{Test 5: 3D Ellipsoidal Interface}

We consider a three-dimensional domain $\Omega = [-1,1]^3$ with an ellipsoidal interface $\Gamma: (x/0.7)^2 + (y/0.5)^2 + (z/0.3)^2 = 1$. The subdomain $\Omega_1$ is the interior of the ellipsoid and $\Omega_2 = \Omega \setminus \overline{\Omega}_1$. This test case features a high-contrast diffusion coefficient: $\beta_1 = 10^{-3}$ in $\Omega_1$ and $\beta_2 = 1$ in $\Omega_2$. The exact solution is $u_1 = e^x e^y e^z$ in $\Omega_1$ and $u_2 = \sin(x)\sin(y)\sin(z)$ in $\Omega_2$. The source terms are $f_1 = -3 \times 10^{-3} e^{x+y+z}$ in $\Omega_1$ and $f_2 = 3\sin(x)\sin(y)\sin(z)$ in $\Omega_2$. The jump conditions are $g_1 = u_1|_\Gamma - u_2|_\Gamma$ and $g_2 = \beta_1 \nabla u_1 \cdot \boldsymbol{n} - \beta_2 \nabla u_2 \cdot \boldsymbol{n}$, computed from the exact solution. The Dirichlet boundary condition is $g_D = \sin(x)\sin(y)\sin(z)$ on $\partial\Omega$. For this ellipsoidal interface, the distance function is computed via iterative projection to find the closest point on the ellipsoid surface.

We distribute collocation points randomly in each subdomain and along the interface. The training and validation points are shown in Figure~\ref{pointstest5}, and the numerical solutions with absolute errors are presented in Figure~\ref{solutiontest5}. Figure~\ref{fig:weight_functions_test5} illustrates the interface-aware weight functions along the $x$-axis crossing the ellipsoid interface.

\begin{figure}[H]
    \centering
    \subfigure[Training points]{\includegraphics[width=4.5cm, height=3cm]{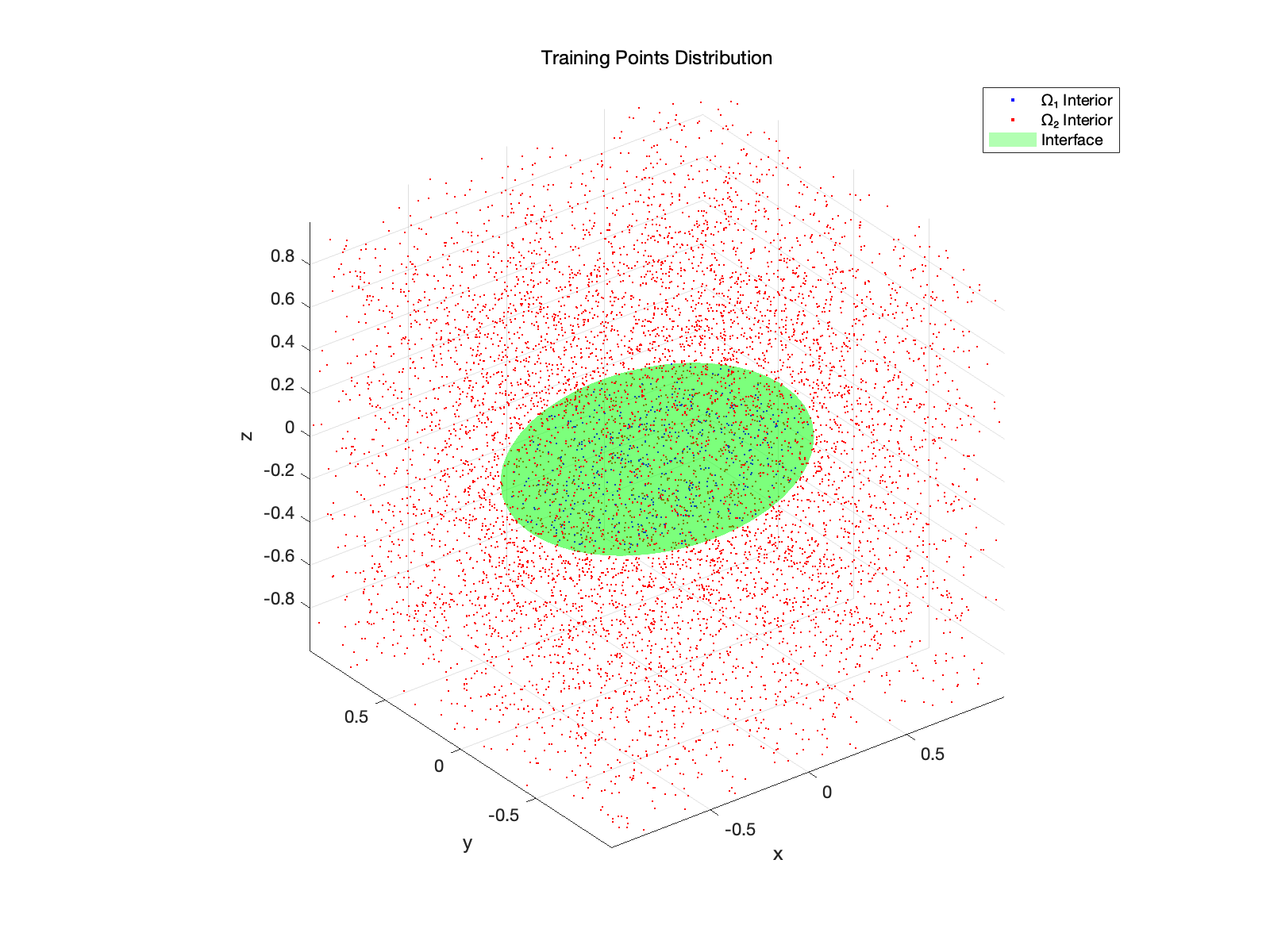}}
    \subfigure[Validation points]{\includegraphics[width=4.5cm, height=3cm]{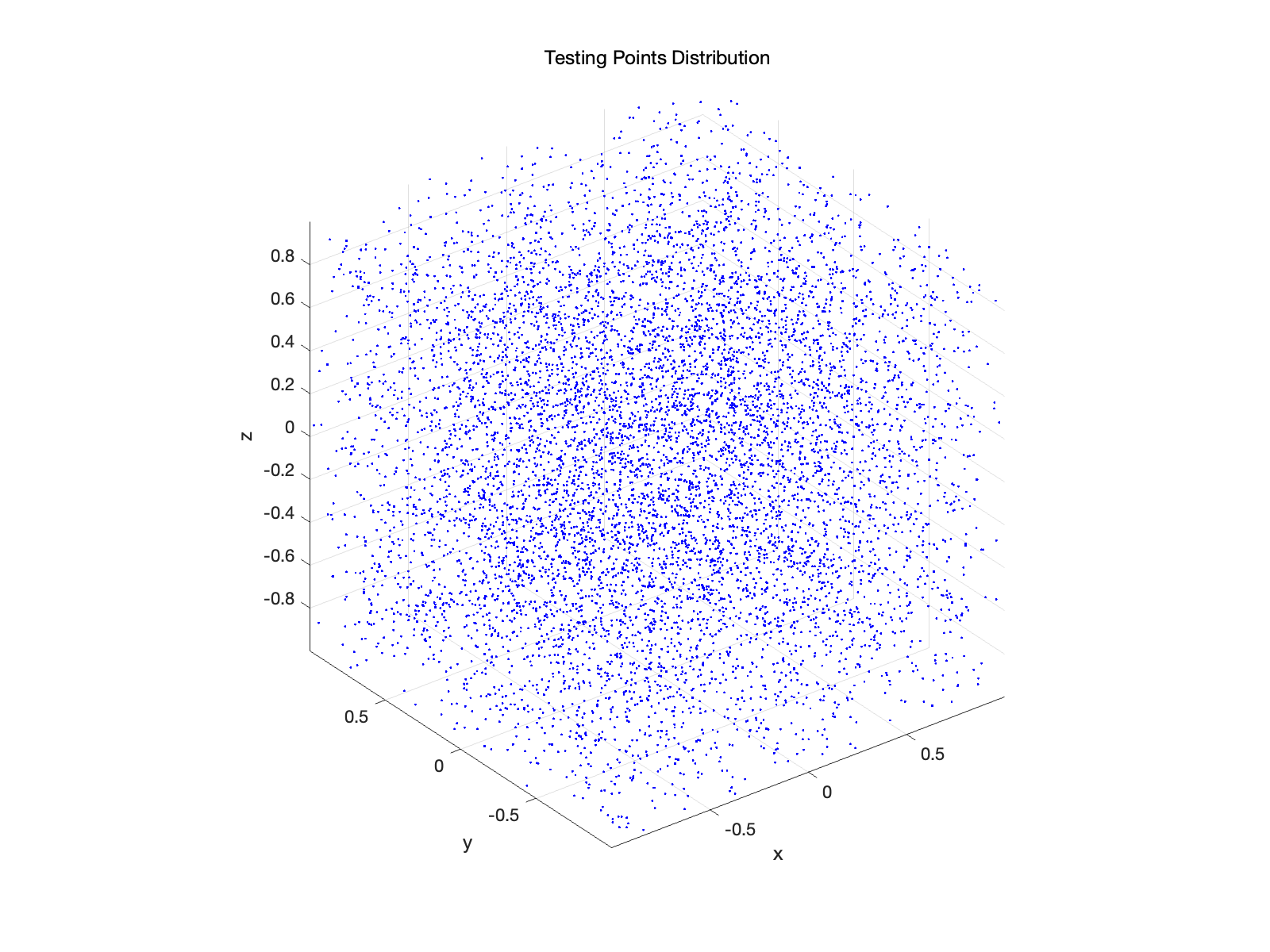}}
    \caption{Training and validation points for Test 5 ($M_{\Omega} = 8000$, $M_{\Gamma}=120$).}
    \label{pointstest5}
\end{figure}

\begin{figure}[H]
    \centering
    \subfigure[$u_{\text{Exact}}$ at $x=0$, $y=0$, $z=0$]{\includegraphics[width=5cm, height=4cm]{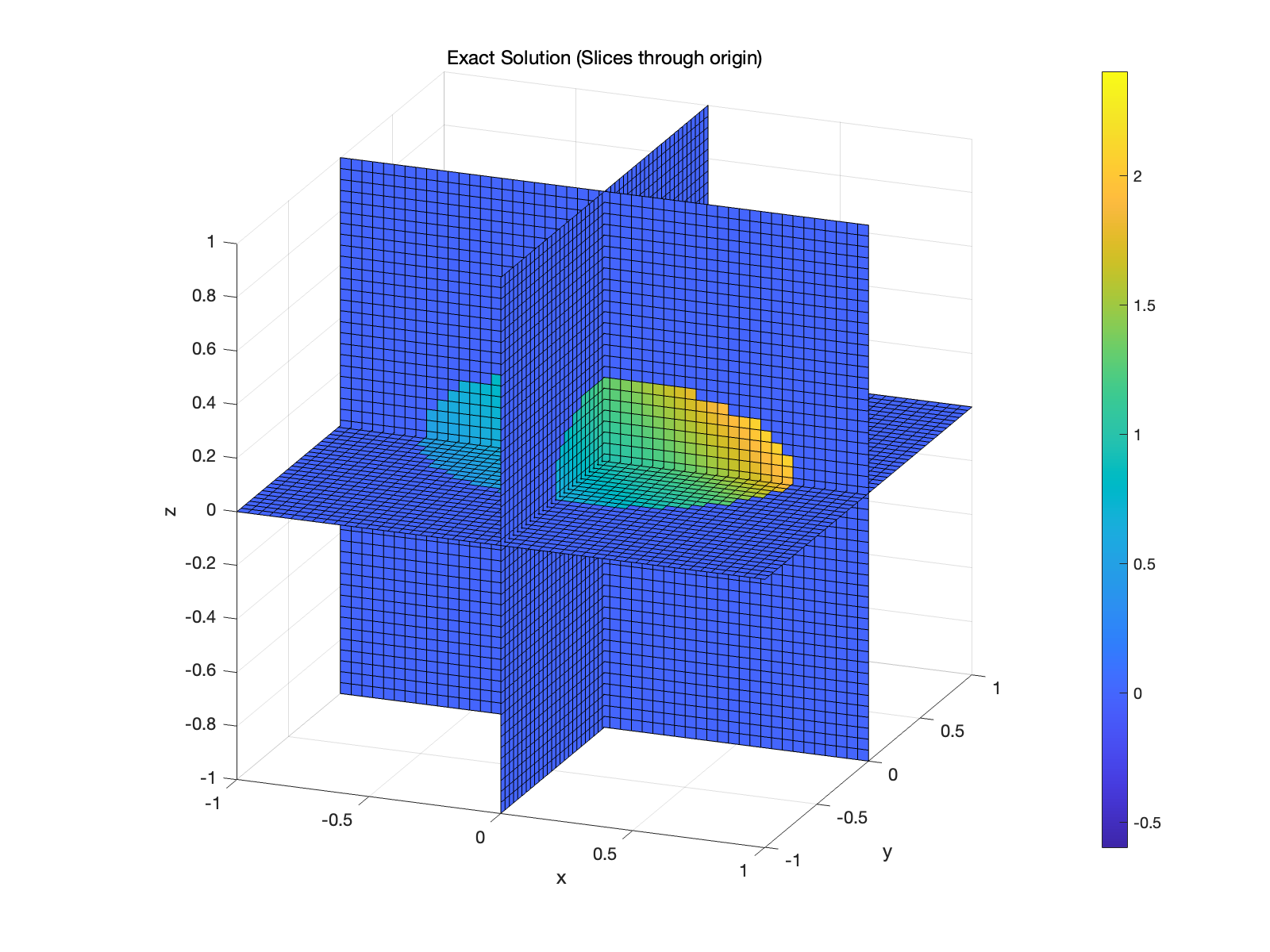}}
    \subfigure[$u_{\text{DNN}}$ at $x=0$, $y=0$, $z=0$]{\includegraphics[width=5cm, height=4cm]{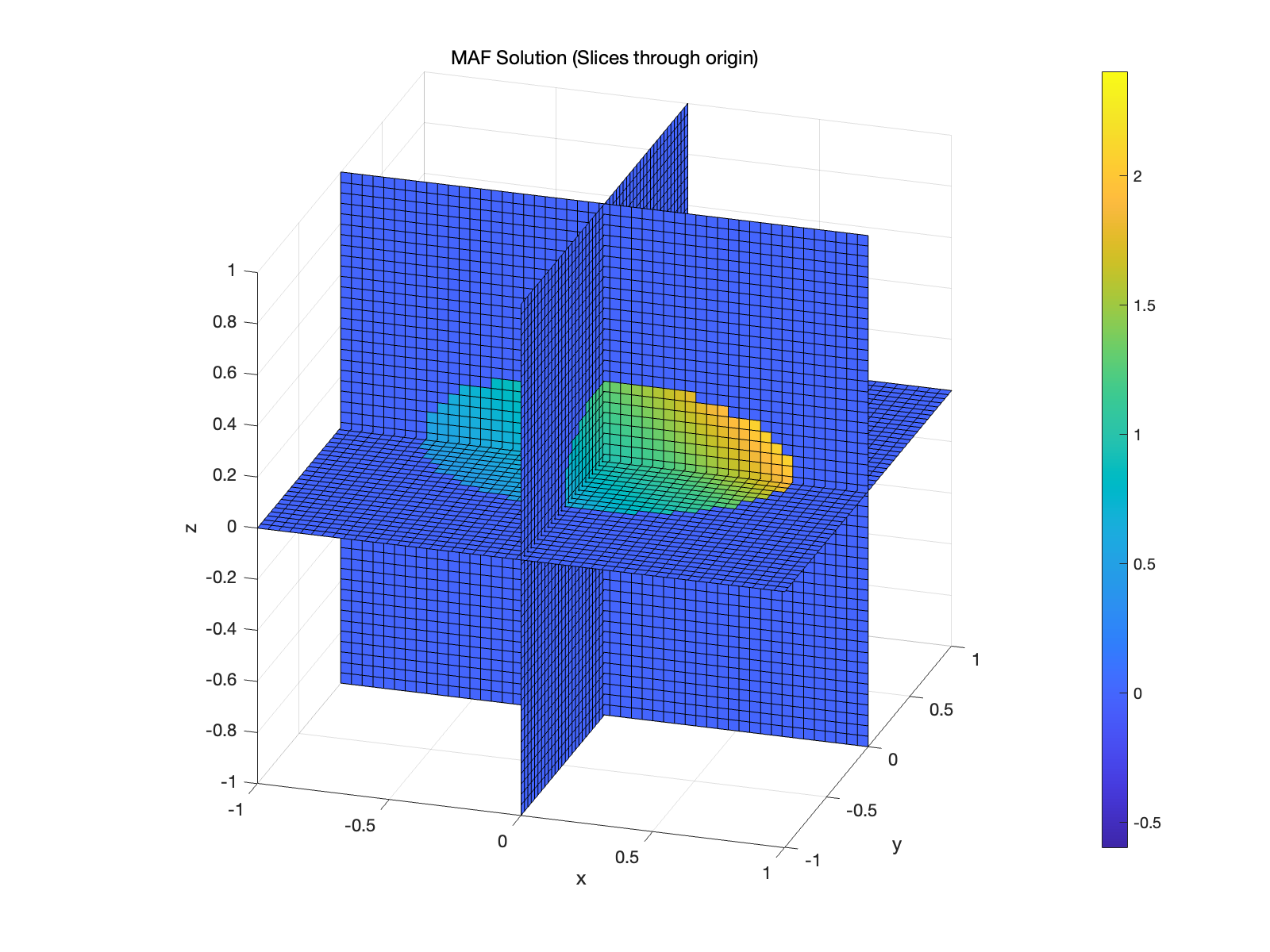}}
    \subfigure[Absolute Error outside]{\includegraphics[width=5cm, height=4cm]{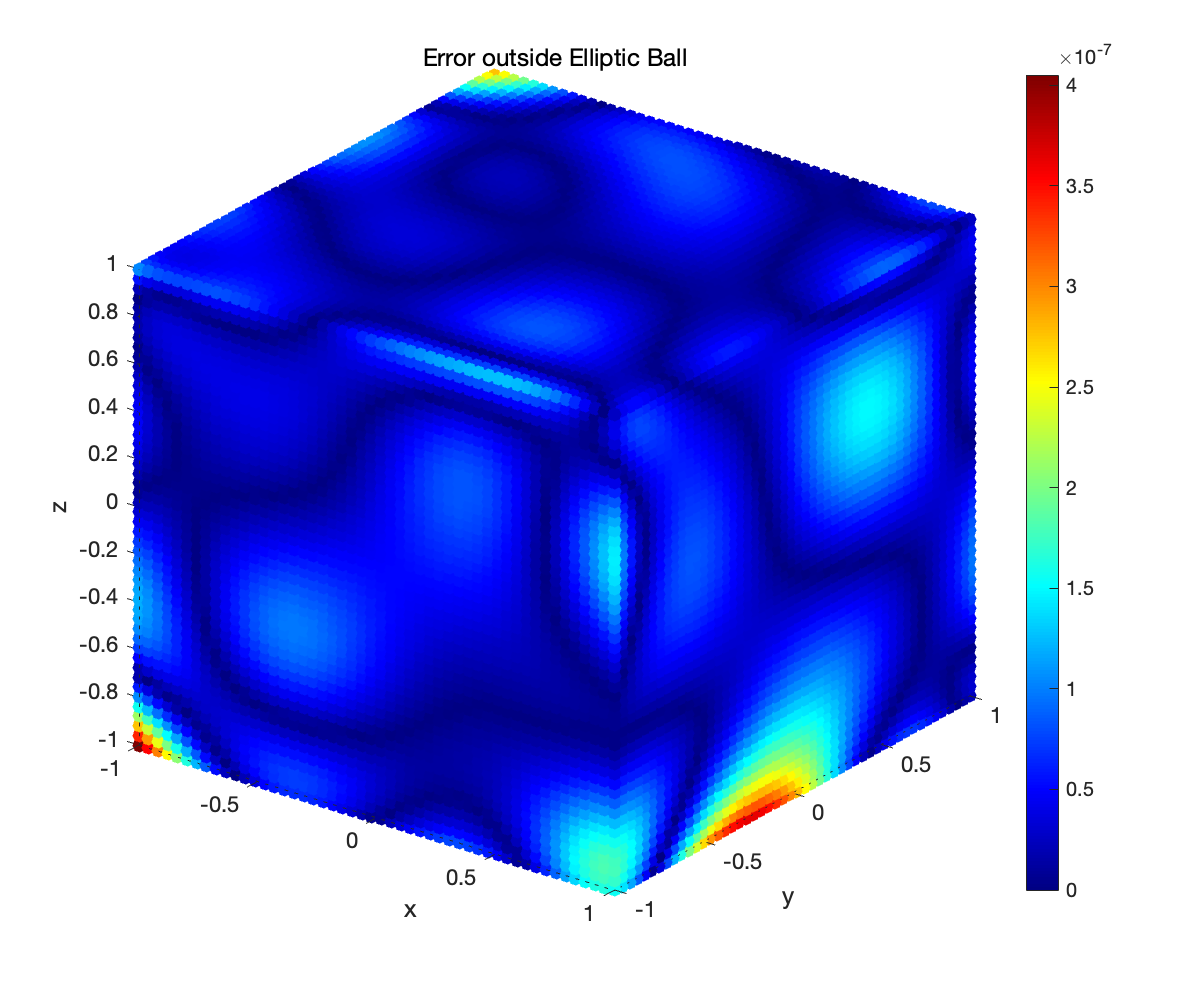}}
    \subfigure[Absolute Error inside]{\includegraphics[width=5cm, height=4cm]{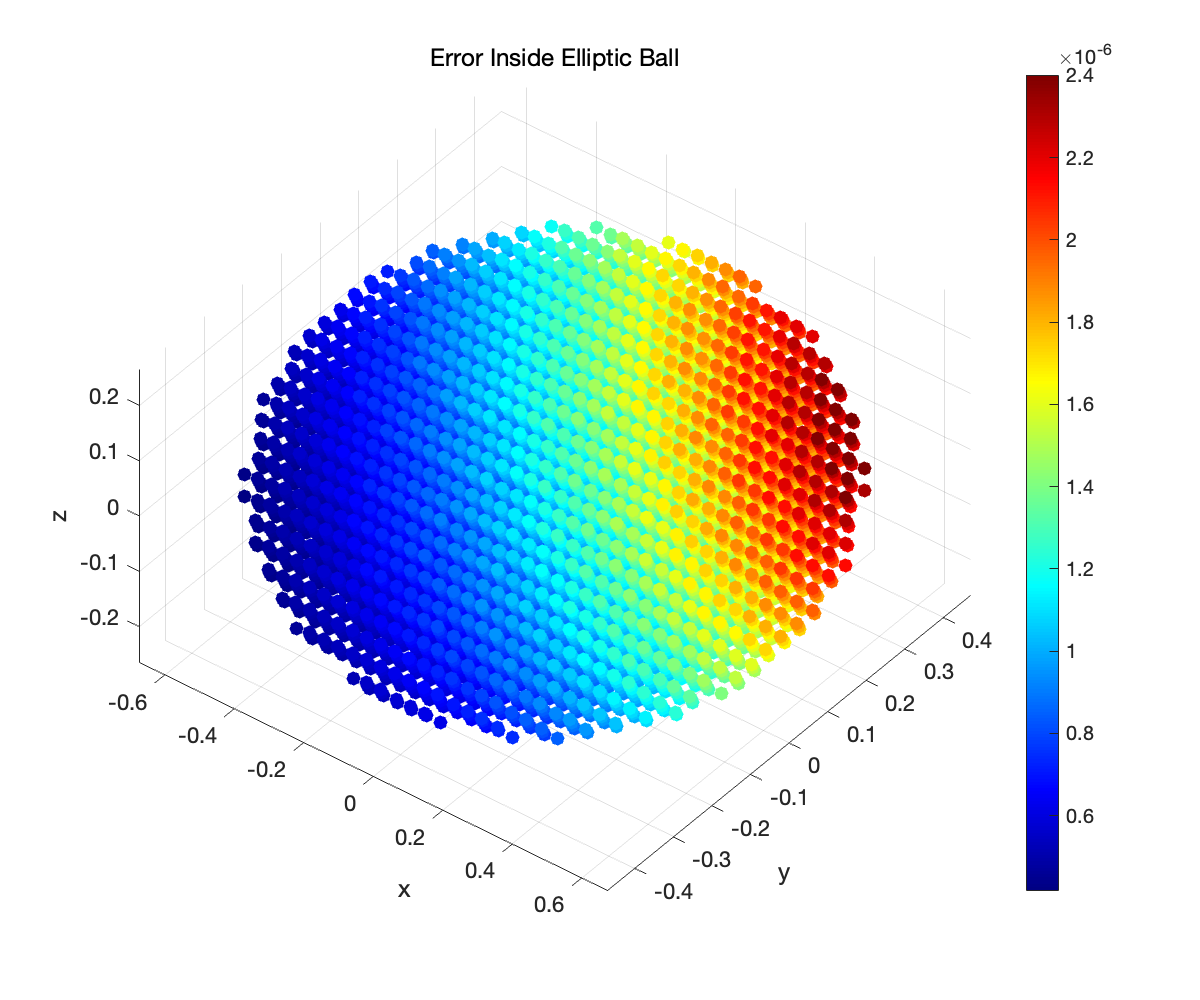}}
    \caption{$u_{\text{DNN}}$, $u_{\text{Exact}}$ and absolute error for Test 5.}
    \label{solutiontest5}
\end{figure}

\begin{figure}[H]
    \centering
    \includegraphics[width=8cm, height=5cm]{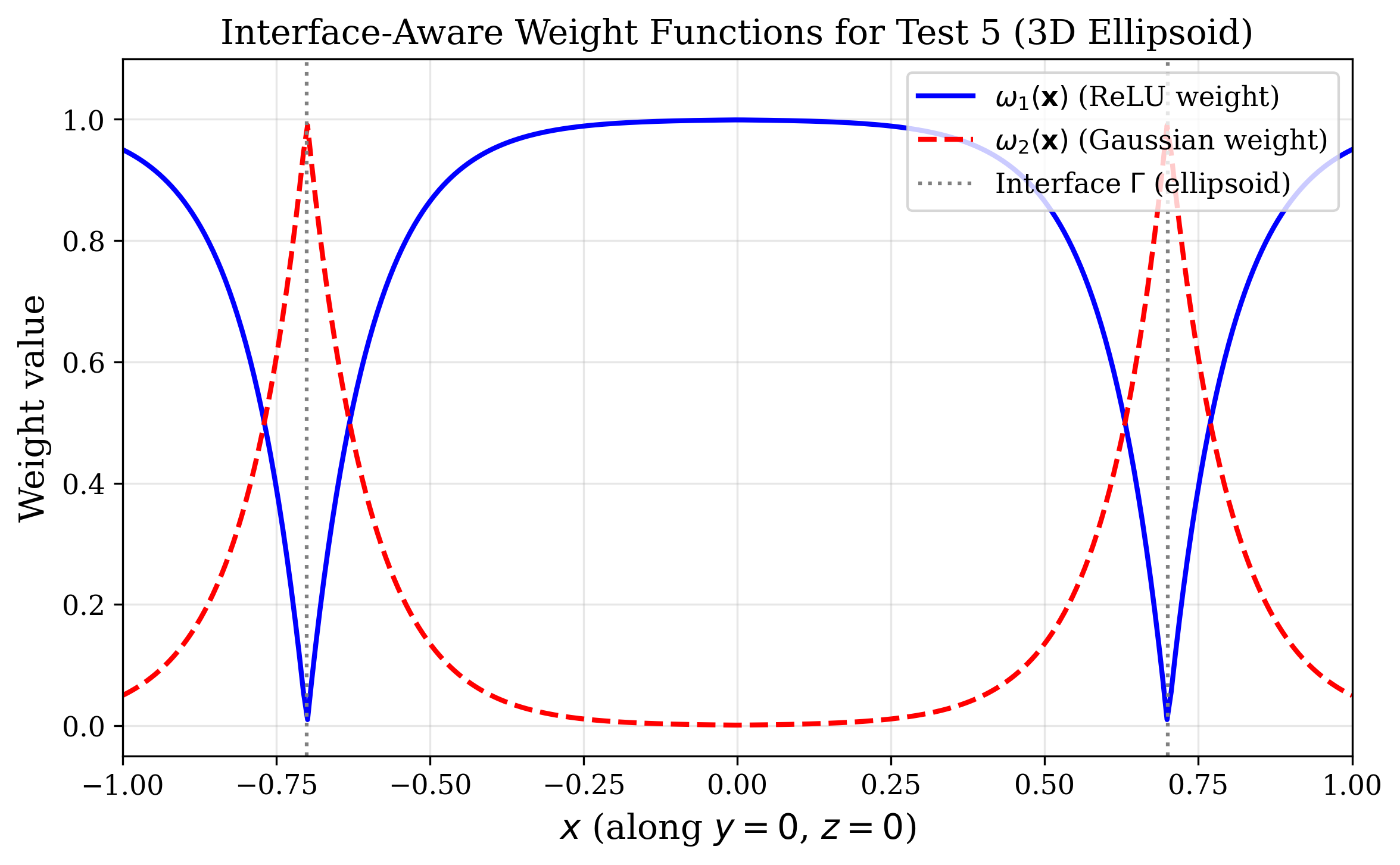}
    \caption{Interface-aware weight functions $\omega_1(\mathbf{x})$ and $\omega_2(\mathbf{x})$ along the $x$-axis crossing the 3D ellipsoid interface for Test 5.}
    \label{fig:weight_functions_test5}
\end{figure}

Table~\ref{errortabletest5} presents the relative $L^2$ error comparison for all methods. This 3D test case demonstrates the effectiveness of our multi-activation approach in handling three-dimensional interface problems with high-contrast coefficients ($10^{-3}$ versus $1$). The MAF method achieves remarkable accuracy ($8.26 \times 10^{-7}$) with the finest sampling, comparable to DCSNN and significantly outperforming XPINN by 785$\times$ and MFM by 420$\times$. Figure~\ref{solutiontest5} (corresponding to the finest sampling configuration with $M_\Omega=8000$, $M_{\partial\Omega}=150$, $M_\Gamma=120$) shows that errors are primarily concentrated near the ellipsoidal interface, which aligns with the challenges imposed by the high-contrast coefficient jump ($\beta_1/\beta_2 = 10^{-3}$) and the 3D geometry.

\begin{table}[H]
    \centering
    \caption{Relative $L^2$ error comparison for Test 5.}
    \label{errortabletest5}
    \small
    \begin{tabular}{|c|c|c|c|c|c|c|c|}
        \hline
        $(M_\Omega, M_{\partial\Omega}, M_\Gamma)$ & MAF & DCSNN & AdaI & I-PINN & M-PINN & MFM & XPINN \\
        \hline
        (125, 50, 30) & $1.43e{-}5$ & $1.02e{-}5$ & $2.56e{-}5$ & $5.89e{-}5$ & $4.23e{-}4$ & $1.12e{-}3$ & $3.38e{-}3$ \\
        \hline
        (1000, 100, 60) & $5.85e{-}6$ & $6.12e{-}6$ & $1.12e{-}5$ & $2.34e{-}5$ & $1.56e{-}4$ & $8.52e{-}4$ & $8.98e{-}4$ \\
        \hline
        (8000, 150, 120) & $8.26e{-}7$ & $8.15e{-}7$ & $2.89e{-}6$ & $7.12e{-}6$ & $4.56e{-}5$ & $3.47e{-}4$ & $6.48e{-}4$ \\
        \hline
    \end{tabular}
\end{table}

\subsubsection{Test 6: 10D Hyperspherical Interface}

We extend our analysis to a high-dimensional setting with $\Omega = [-1,1]^{10}$ and a hyperspherical interface $\Gamma: \sum_{i=1}^{10} (x_i/0.5)^2 = 1$. The subdomain $\Omega_1$ is the interior of the hypersphere and $\Omega_2 = \Omega \setminus \overline{\Omega}_1$. This test case features a high-contrast diffusion coefficient: $\beta_1 = 10^{-3}$ in $\Omega_1$ and $\beta_2 = 1$ in $\Omega_2$. The exact solution is $u_1 = \prod_{i=1}^{10} e^{x_i}$ in $\Omega_1$ and $u_2 = \prod_{i=1}^{10} \sin(x_i)$ in $\Omega_2$. The source terms are $f_1 = -10 \times 10^{-3} \prod_{i=1}^{10} e^{x_i}$ in $\Omega_1$ and $f_2 = 10 \prod_{i=1}^{10} \sin(x_i)$ in $\Omega_2$. The jump conditions are $g_1 = u_1|_\Gamma - u_2|_\Gamma$ and $g_2 = \beta_1 \nabla u_1 \cdot \boldsymbol{n} - \beta_2 \nabla u_2 \cdot \boldsymbol{n}$, computed from the exact solution. The Dirichlet boundary condition is $g_D = u_2|_{\partial\Omega}$. For this hyperspherical interface, the distance function is computed analytically as $d(\mathbf{x}, \Gamma) = |r - 0.5|$ where $r = \sqrt{\sum_{i=1}^{10} x_i^2}$.

We distribute collocation points randomly in each subdomain and along the interface. The error is calculated on 10,000 randomly sampled validation points. Figure~\ref{losstest6} shows the loss evolution during training.

\begin{figure}[H]
    \centering
    \includegraphics[width=6cm, height=4cm]{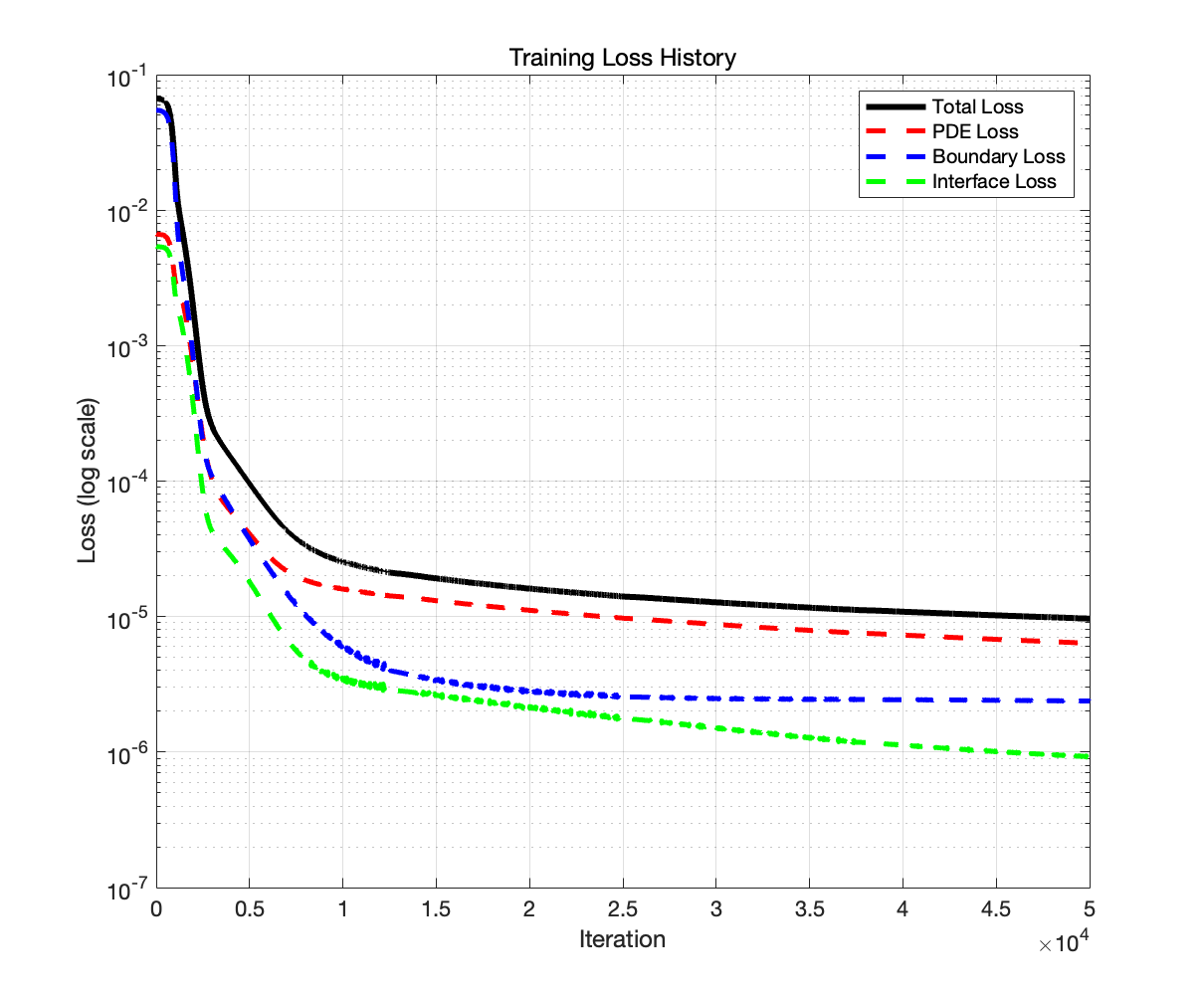}
    \caption{Loss evolution during training for Test 6 (10D).}
    \label{losstest6}
\end{figure}

Table~\ref{errortabletest6} presents the relative $L^2$ error comparison for all methods. This 10D test case demonstrates the robustness of our multi-activation approach in handling high-dimensional interface problems. Despite the curse of dimensionality, the MAF method achieves accuracy on the order of $10^{-3}$ with the finest sampling, comparable to DCSNN and significantly outperforming XPINN by 24$\times$ and MFM by 20$\times$. The performance improves substantially as sampling points increase, with the error reducing from $2.21 \times 10^{-2}$ to $3.34 \times 10^{-3}$. Figure~\ref{losstest6} (corresponding to the finest sampling configuration with $M_\Omega=8000$, $M_{\partial\Omega}=400$, $M_\Gamma=200$) shows that the training loss converges smoothly, demonstrating stable optimization despite the high dimensionality. The remaining error is primarily attributed to the curse of dimensionality and the limited sampling density relative to the 10D domain volume. These comprehensive numerical experiments, ranging from 2D to 10D, validate the effectiveness and versatility of our multi-activation function approach across different dimensionalities and interface geometries.

\begin{table}[H]
    \centering
    \caption{Relative $L^2$ error comparison for Test 6.}
    \label{errortabletest6}
    \small
    \begin{tabular}{|c|c|c|c|c|c|c|c|}
        \hline
        $(M_\Omega, M_{\partial\Omega}, M_\Gamma)$ & MAF & DCSNN & AdaI & I-PINN & M-PINN & MFM & XPINN \\
        \hline
        (2000, 100, 50) & $2.21e{-}2$ & $3.43e{-}2$ & $4.56e{-}2$ & $6.78e{-}2$ & $1.56e{-}1$ & $2.98e{-}1$ & $1.01e{-}1$ \\
        \hline
        (4000, 200, 100) & $7.95e{-}3$ & $7.56e{-}3$ & $1.23e{-}2$ & $2.34e{-}2$ & $5.67e{-}2$ & $8.43e{-}2$ & $8.56e{-}2$ \\
        \hline
        (8000, 400, 200) & $3.34e{-}3$ & $4.01e{-}3$ & $5.89e{-}3$ & $1.12e{-}2$ & $3.45e{-}2$ & $6.75e{-}2$ & $8.01e{-}2$ \\
        \hline
    \end{tabular}
\end{table}

\subsection{Parabolic Interface Problem}

We present numerical studies of parabolic interface problems with both fixed and moving interfaces. While existing methods (MFM, DCSNN, M-PINN, I-PINN, AdaI-PINN) have shown success for elliptic interface problems, they have not been demonstrated for parabolic equations with moving interfaces. Therefore, we compare our MAF method exclusively with XPINN for these challenging scenarios.

We consider model problem \eqref{modelpextenddnnpaowu} with domain $\Omega = [0,3.5] \times [0,3.5]$, $t \in [0,1]$ and four distinct interface configurations. In all cases, the exact solution is $u = 0.01 e^t e^x e^y$ in $\Omega_1(t)$ and $u = e^t \sin(x)\sin(y)$ in $\Omega_2(t)$, all reletive boundary conditions, interface conditions, initial conditions, source terms can be calculated according to the exact solutions. We treat time as an additional dimension, solving the problem in a $(2+1)$-dimensional space-time continuum $[0,3.5] \times [0,3.5] \times [0,1]$. The time interval is partitioned into 10 sub-intervals ($\Delta t = 0.1$). Our sampling strategy uses random collocation points (200, 400, 1000) throughout the space-time domain, with structured sampling (20, 40, 100 points) at temporal interfaces for boundary and interface conditions.

\subsubsection{Test 1: Fixed Circular Interface}

We consider a static circular interface $\Gamma: (x-1.5)^2 + (y-1.5)^2 = 1$ that divides the domain into interior region $\Omega_2(t)$ and exterior region $\Omega_1(t)$. Although the interface remains stationary, this problem is challenging due to the time-dependent solution and discontinuity across the interface. The exact solution is $u_1 = 0.01 e^t e^x e^y$ in $\Omega_1(t)$ (exterior) and $u_2 = e^t \sin(x)\sin(y)$ in $\Omega_2(t)$ (interior). The diffusion coefficients are $\beta_1 = 1$ and $\beta_2 = 10$. The source terms are $f_1 = 0.01 e^t e^{x+y}(1 - 2\beta_1)$ and $f_2 = e^t \sin(x)\sin(y)(1 + 2\beta_2)$. The jump conditions $g_1 = u_1|_\Gamma - u_2|_\Gamma$ and $g_2 = \beta_1 \nabla u_1 \cdot \boldsymbol{n} - \beta_2 \nabla u_2 \cdot \boldsymbol{n}$ are computed from the exact solution. The initial condition is $g_0 = u|_{t=0}$ and the Dirichlet boundary condition is $g_D = u_1|_{\partial\Omega}$. The distance function is $d(\mathbf{x}, \Gamma) = |\sqrt{(x-1.5)^2 + (y-1.5)^2} - 1|$. Since the interface is fixed, the weight functions $\omega_1(\mathbf{x})$ and $\omega_2(\mathbf{x})$ remain constant in time, as illustrated in Figure~\ref{fig:weight_parabolic_test1}.

The training and validation points are shown in Figure~\ref{fig:fixed_interfacepoint}, and the numerical solutions with absolute errors at $t=0.5$ and $t=1.0$ are presented in Figures~\ref{fig:numercialresultst05} and~\ref{fig:numercialresultst1} (corresponding to the finest sampling configuration with $M_\Omega=1000$, $M_{\partial\Omega}=100$, $M_\Gamma=100$). At $t=0.5$, MAF achieves a maximum absolute error of $9.12 \times 10^{-5}$, with errors concentrated near the circular interface. At $t=1.0$, the error grows to $2.54 \times 10^{-4}$, demonstrating temporal stability despite exponential growth in the exact solution. Errors are primarily concentrated near the interface $\Gamma$, which aligns with the challenges imposed by the jump conditions. The MAF method achieves a relative $L^2$ error of $8.45 \times 10^{-5}$, outperforming XPINN ($7.33 \times 10^{-3}$) by 87$\times$.

\begin{figure}[H]
    \centering
    \includegraphics[width=8cm, height=4.5cm]{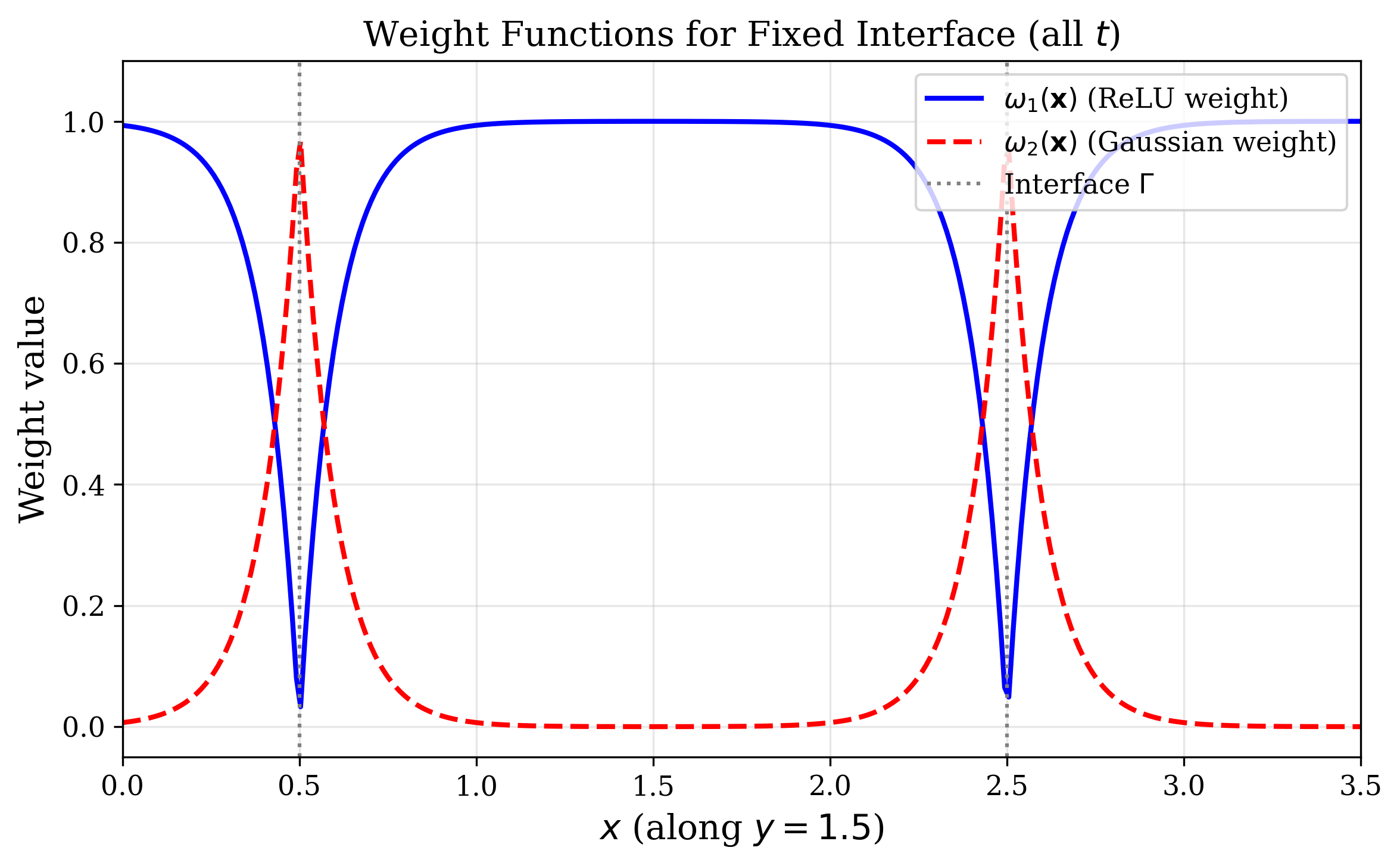}
    \caption{Weight functions $\omega_1(\mathbf{x})$ and $\omega_2(\mathbf{x})$ along $y=1.5$ for Test 1 (fixed interface).}
    \label{fig:weight_parabolic_test1}
\end{figure}

\begin{figure}[H]
    \centering
    \subfigure[Training points at $t=0.5$]{\includegraphics[width=4.5cm, height=3cm]{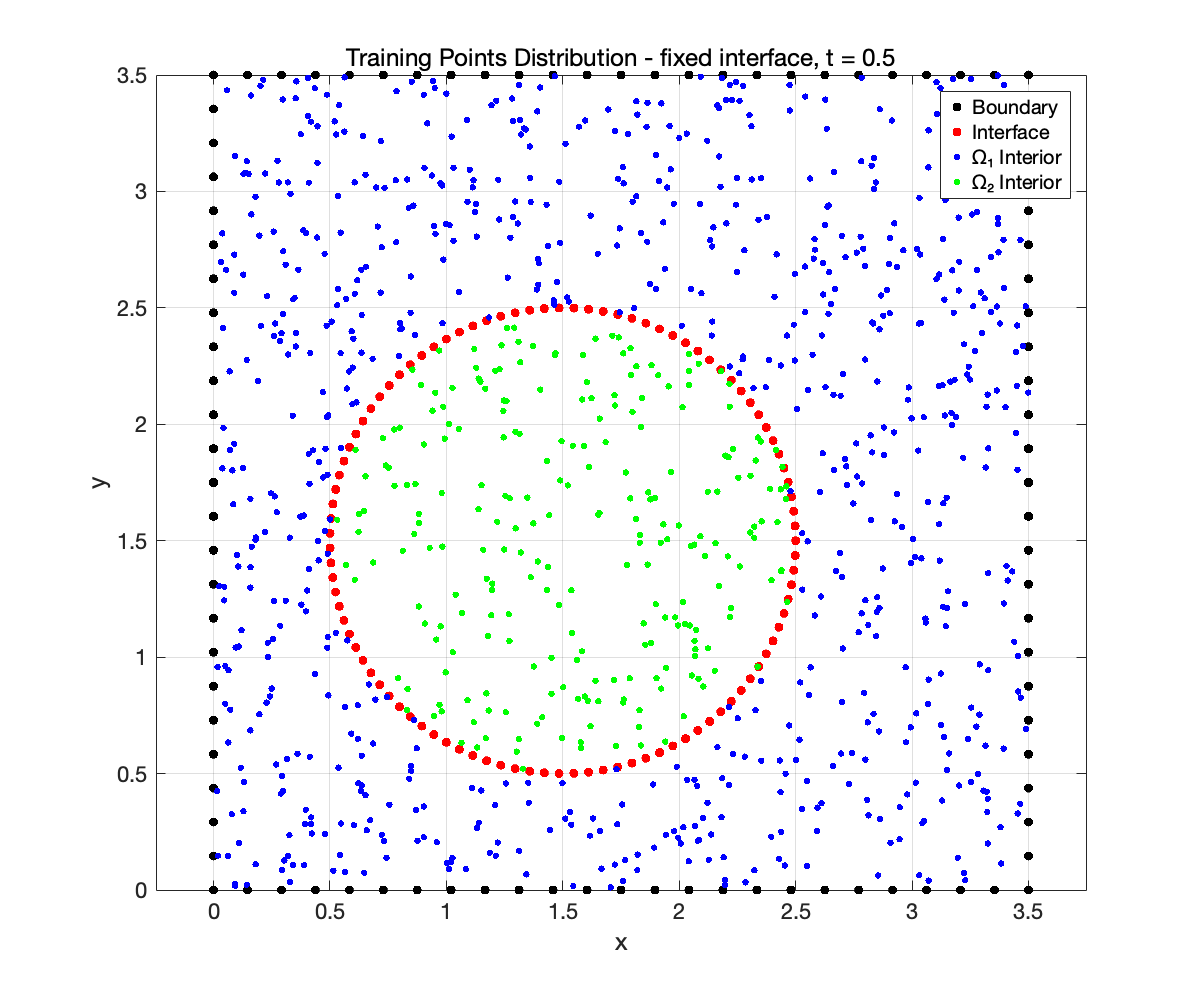}}
    \subfigure[Training points at $t=1$]{\includegraphics[width=4.5cm, height=3cm]{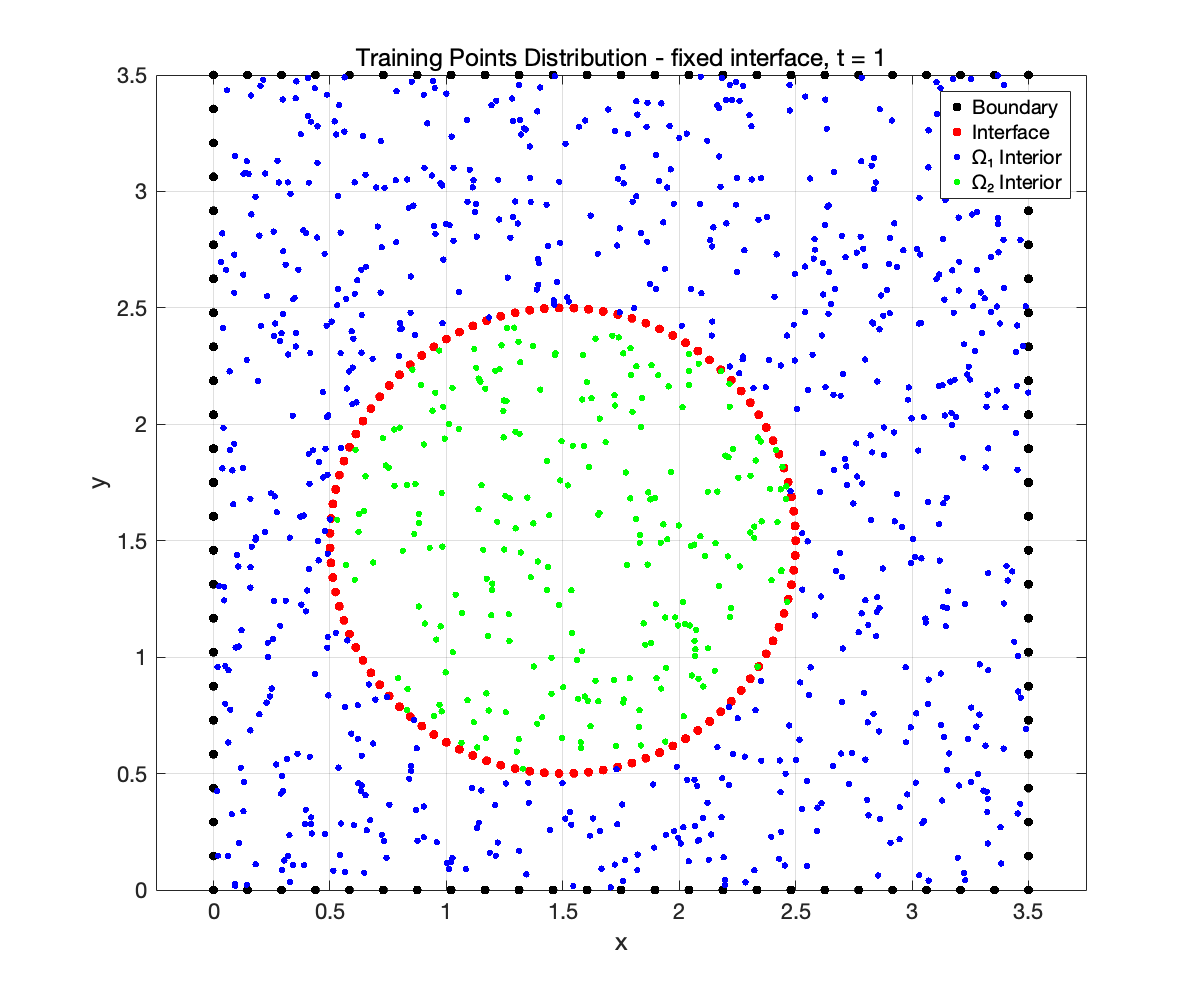}}
    \caption{Training points for Test 1 at $t=0.5$ and $t=1$.}
    \label{fig:fixed_interfacepoint}
\end{figure}

\begin{figure}[H]
    \centering
    \subfigure[Exact at $t=0.5$]{\includegraphics[width=4cm, height=3cm]{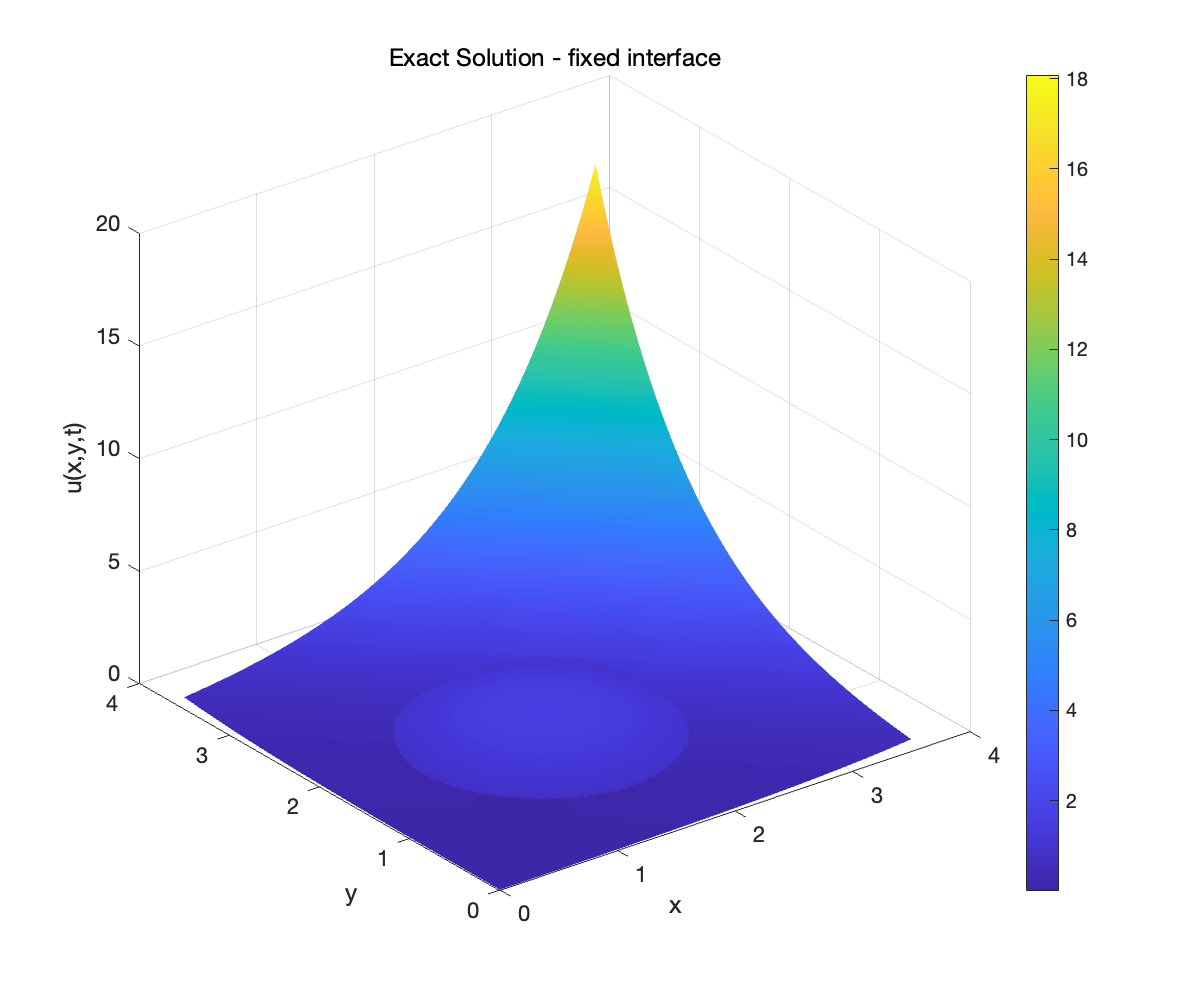}}
    \subfigure[MAF at $t=0.5$]{\includegraphics[width=4cm, height=3cm]{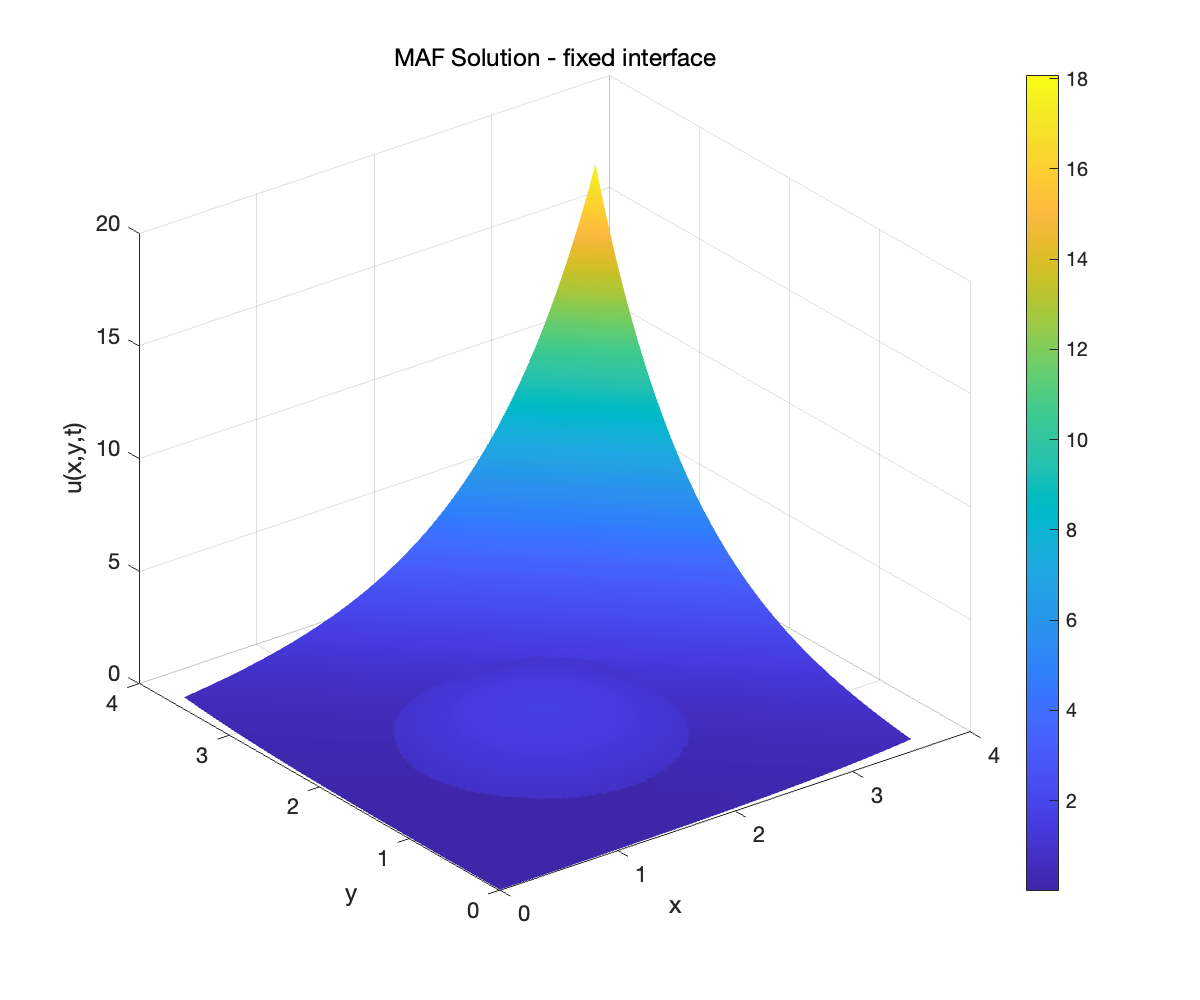}}
    \subfigure[Error at $t=0.5$]{\includegraphics[width=4cm, height=3cm]{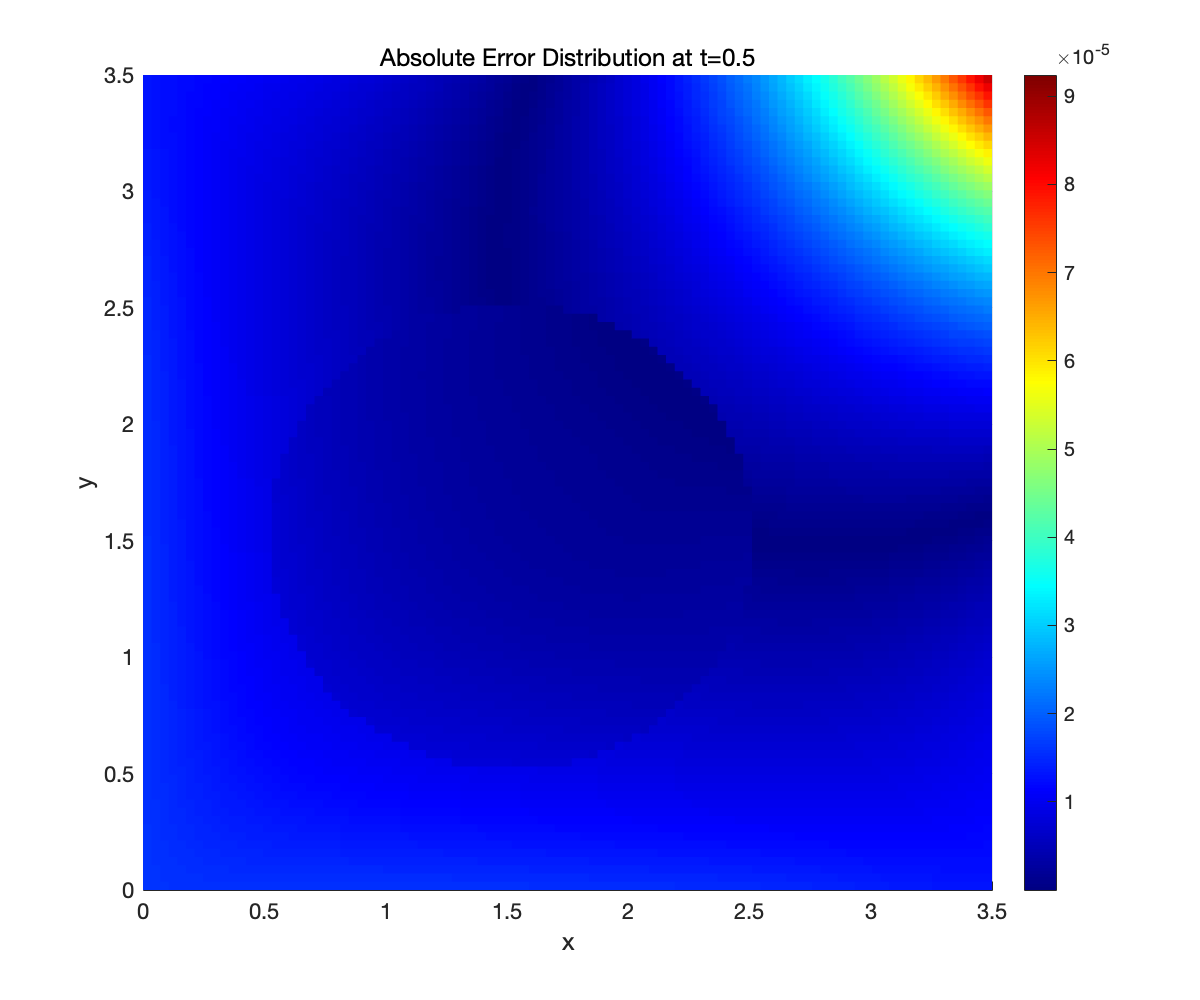}}
    \caption{Numerical results for Test 1 at $t=0.5$.}
    \label{fig:numercialresultst05}
\end{figure}

\begin{figure}[H]
    \centering
    \subfigure[Exact at $t=1$]{\includegraphics[width=4cm, height=3cm]{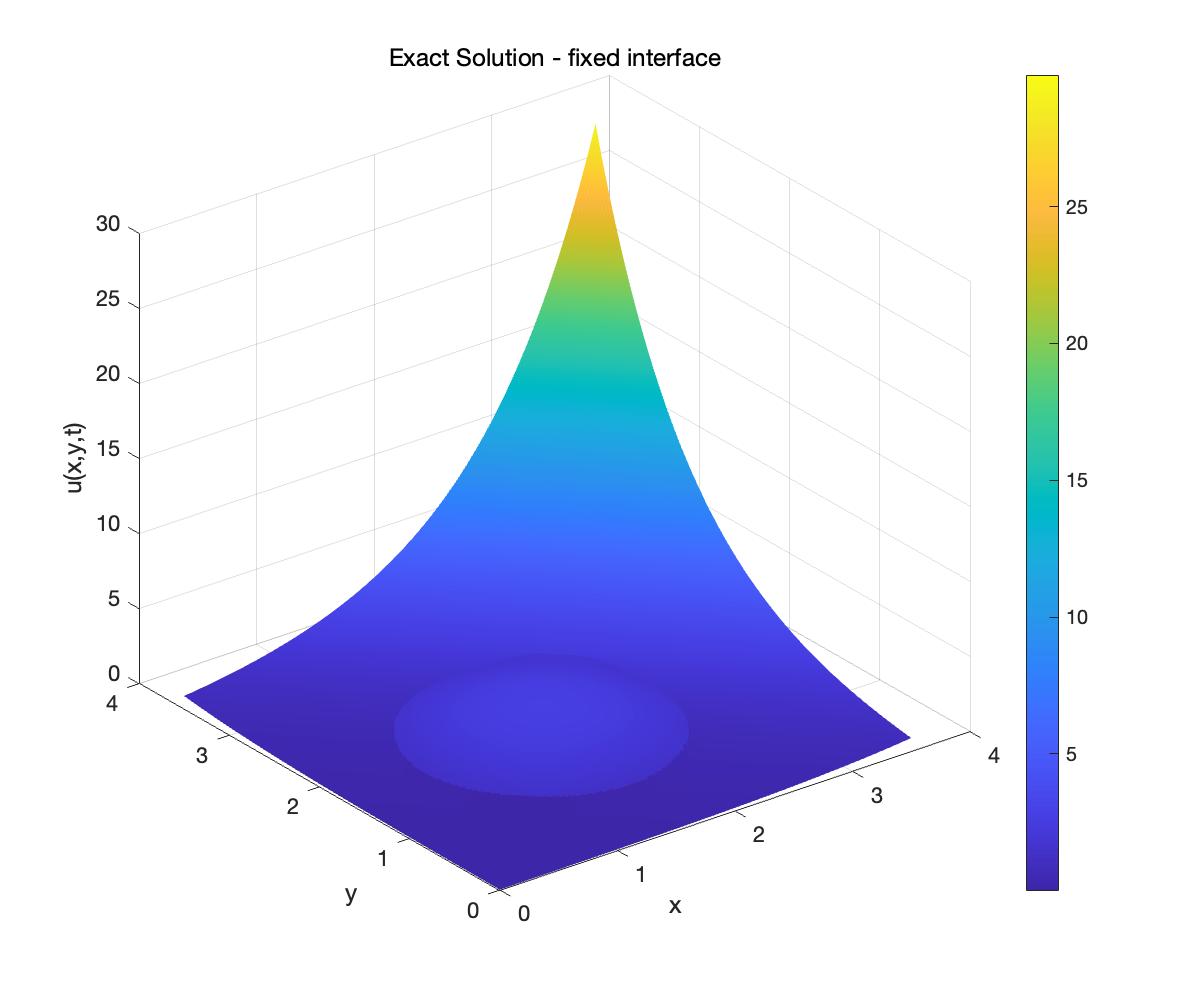}}
    \subfigure[MAF at $t=1$]{\includegraphics[width=4cm, height=3cm]{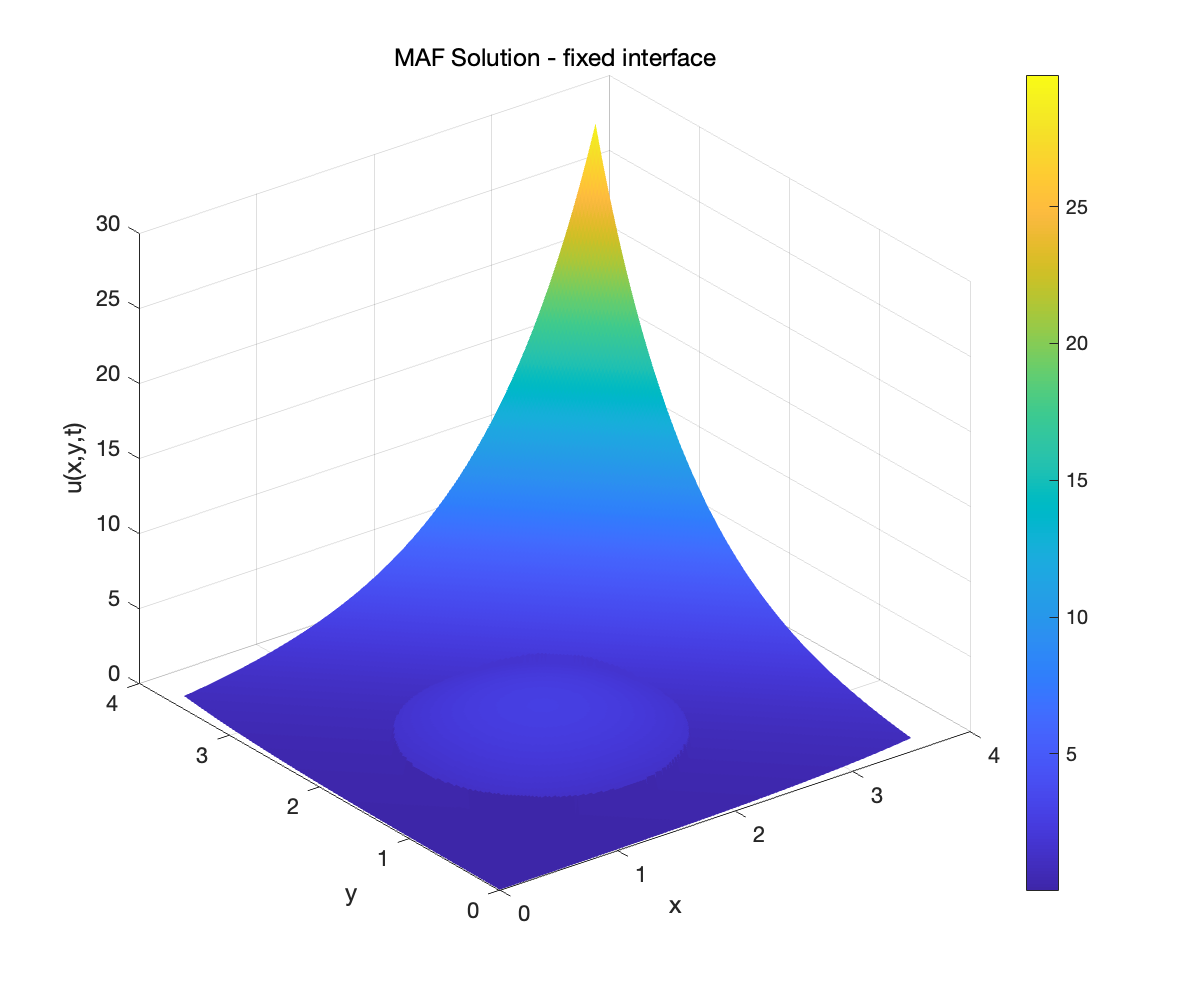}}
    \subfigure[Error at $t=1$]{\includegraphics[width=4cm, height=3cm]{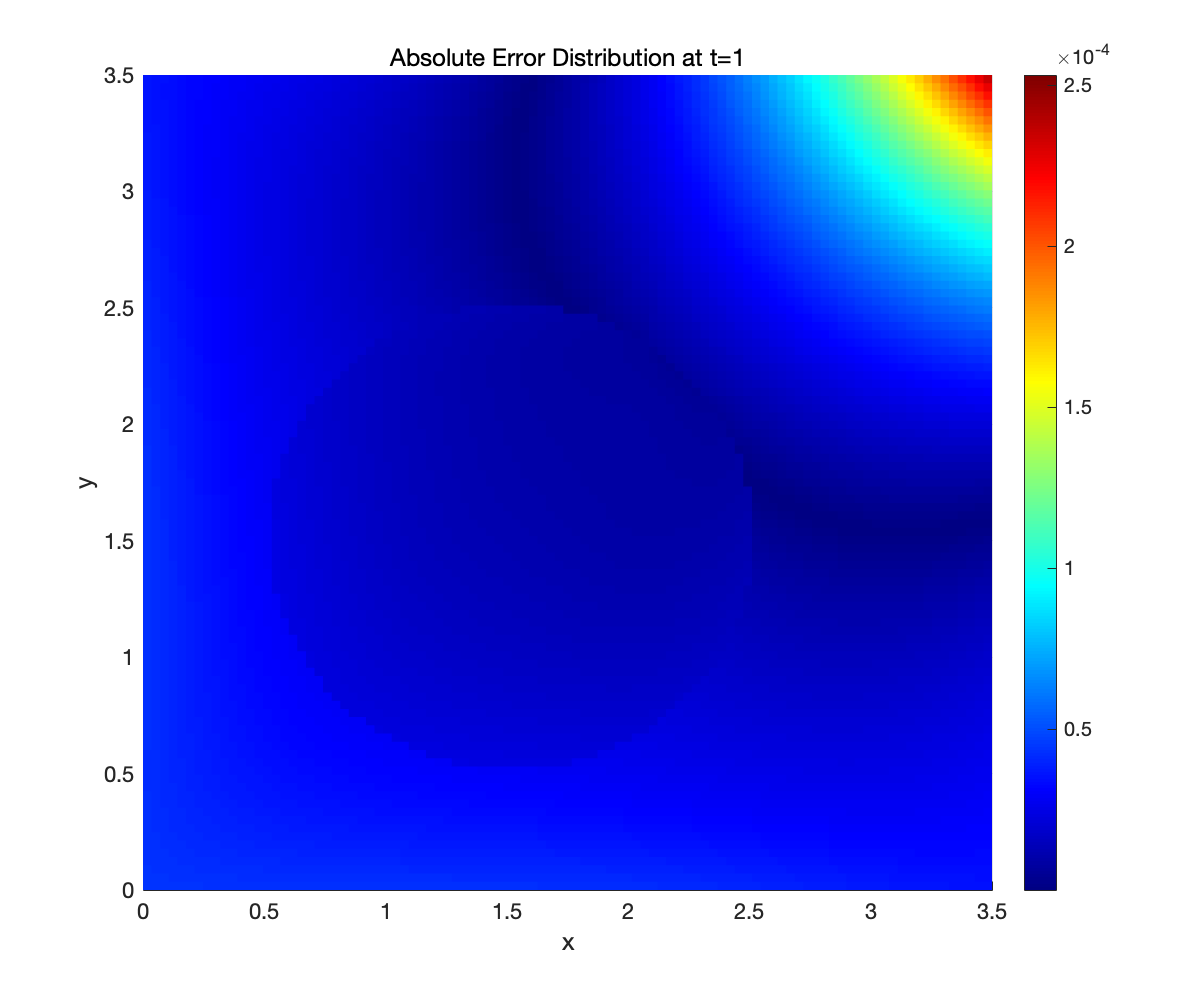}}
    \caption{Numerical results for Test 1 at $t=1$.}
    \label{fig:numercialresultst1}
\end{figure}

\subsubsection{Test 2: Moving Interface (Rigid Body Motion)}

We consider a moving circular interface $\Gamma(t): (x-1.2-t)^2 + (y-1.2-t)^2 = 1$ that translates with velocity $(1,1)$. The exact solution is $u_1 = 0.01 e^t e^x e^y$ in $\Omega_1(t)$ (exterior) and $u_2 = e^t \sin(x)\sin(y)$ in $\Omega_2(t)$ (interior). The diffusion coefficients are $\beta_1 = 1$ and $\beta_2 = 10$. The source terms are $f_1 = 0.01 e^t e^{x+y}(1 - 2\beta_1)$ and $f_2 = e^t \sin(x)\sin(y)(1 + 2\beta_2)$. The jump conditions $g_1$ and $g_2$ are computed from the exact solution. The initial and boundary conditions are $g_0 = u|_{t=0}$ and $g_D = u_1|_{\partial\Omega}$. The distance function evolves as $d(\mathbf{x}, \Gamma(t)) = |\sqrt{(x-1.2-t)^2 + (y-1.2-t)^2} - 1|$, and consequently the weight functions $\omega_1(\mathbf{x}, t)$ and $\omega_2(\mathbf{x}, t)$ change with time. Figure~\ref{fig:weight_parabolic_test2} illustrates how the weight functions shift as the interface moves through the domain.

The training points and numerical results are shown in Figures~\ref{fig:rigid_interfacepoint}--\ref{fig:numercialresultst1rigid} (corresponding to the finest sampling configuration with $M_\Omega=1000$, $M_{\partial\Omega}=100$, $M_\Gamma=100$). At $t=0.5$, the peak error is $7.65 \times 10^{-5}$, concentrated near the evolving interface. At $t=1.0$, the maximum error remains bounded at $3.40 \times 10^{-4}$, demonstrating stable interface tracking. Errors are primarily concentrated near the moving interface $\Gamma(t)$, which aligns with the challenges imposed by tracking the translating circular boundary. The MAF method achieves a relative $L^2$ error of $8.52 \times 10^{-5}$, outperforming XPINN ($7.53 \times 10^{-3}$) by 88$\times$.

\begin{figure}[H]
    \centering
    \includegraphics[width=12cm, height=4cm]{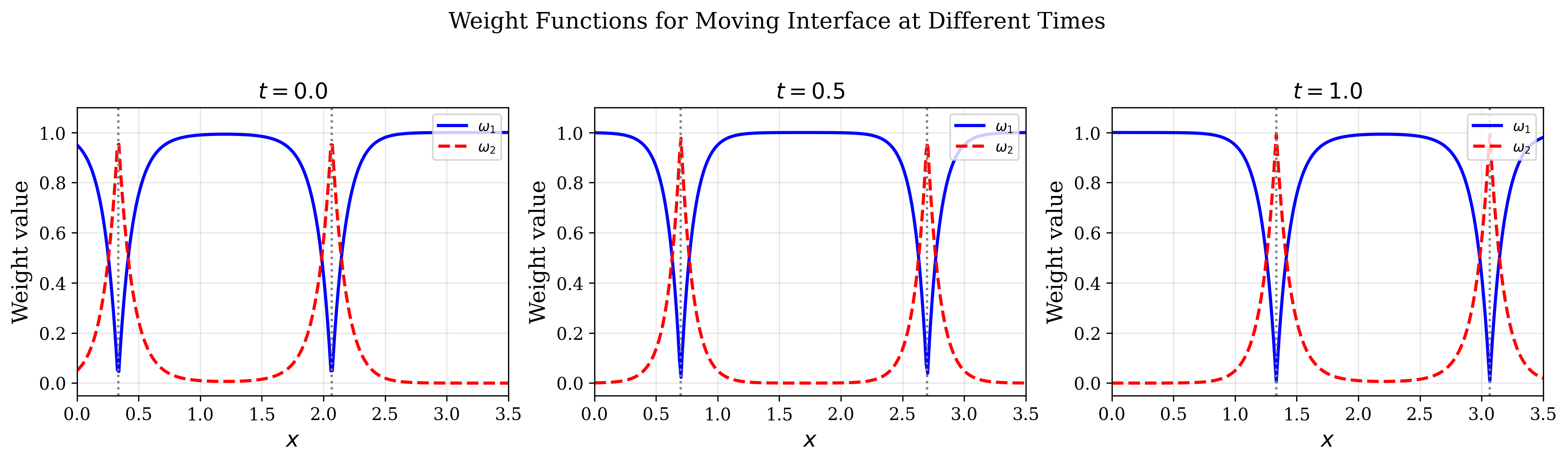}
    \caption{Weight functions at $t=0$, $t=0.5$, and $t=1$ for Test 2 (moving interface).}
    \label{fig:weight_parabolic_test2}
\end{figure}

\begin{figure}[H]
    \centering
    \subfigure[Training points at $t=0.5$]{\includegraphics[width=4.5cm, height=3cm]{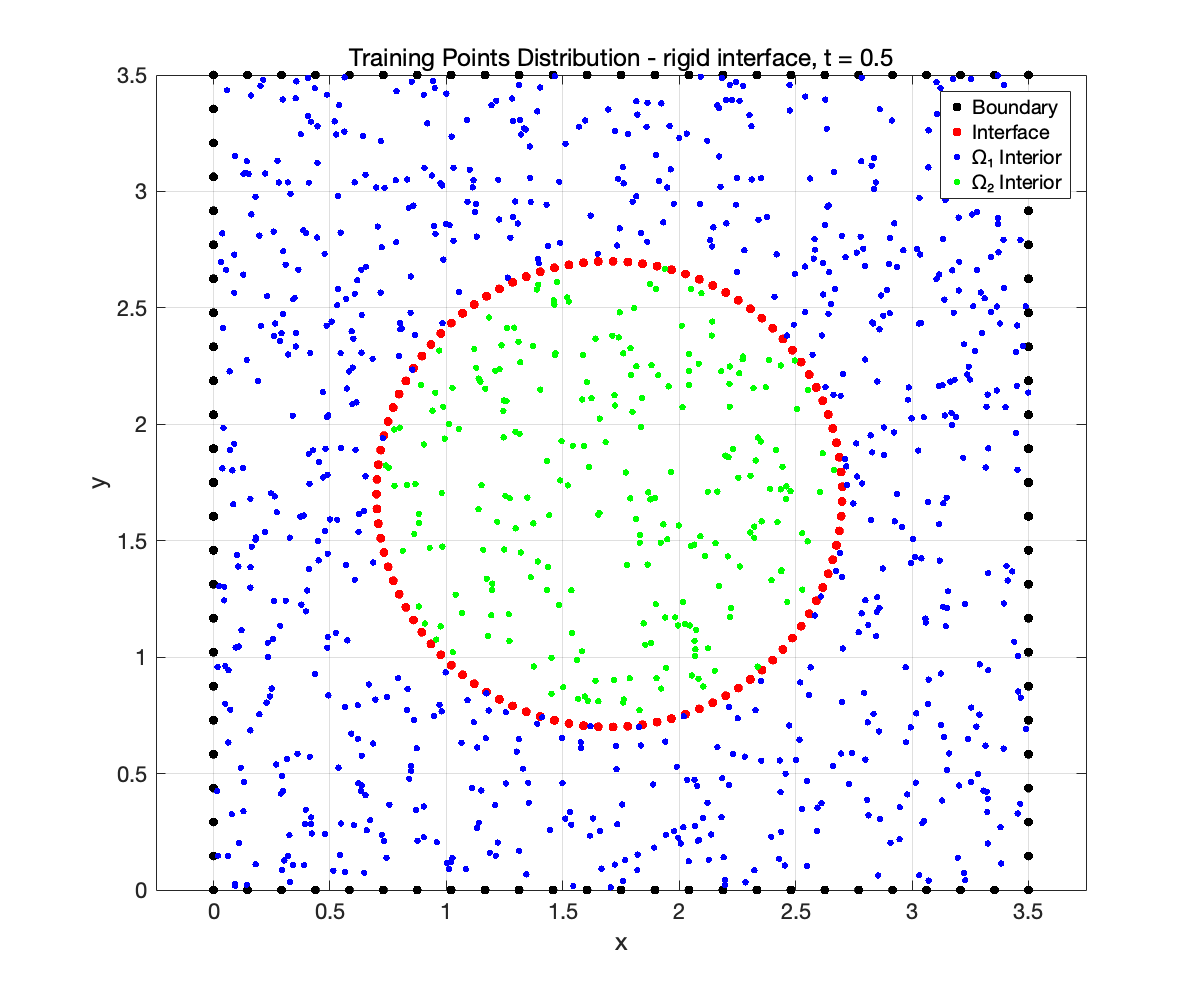}}
    \subfigure[Training points at $t=1$]{\includegraphics[width=4.5cm, height=3cm]{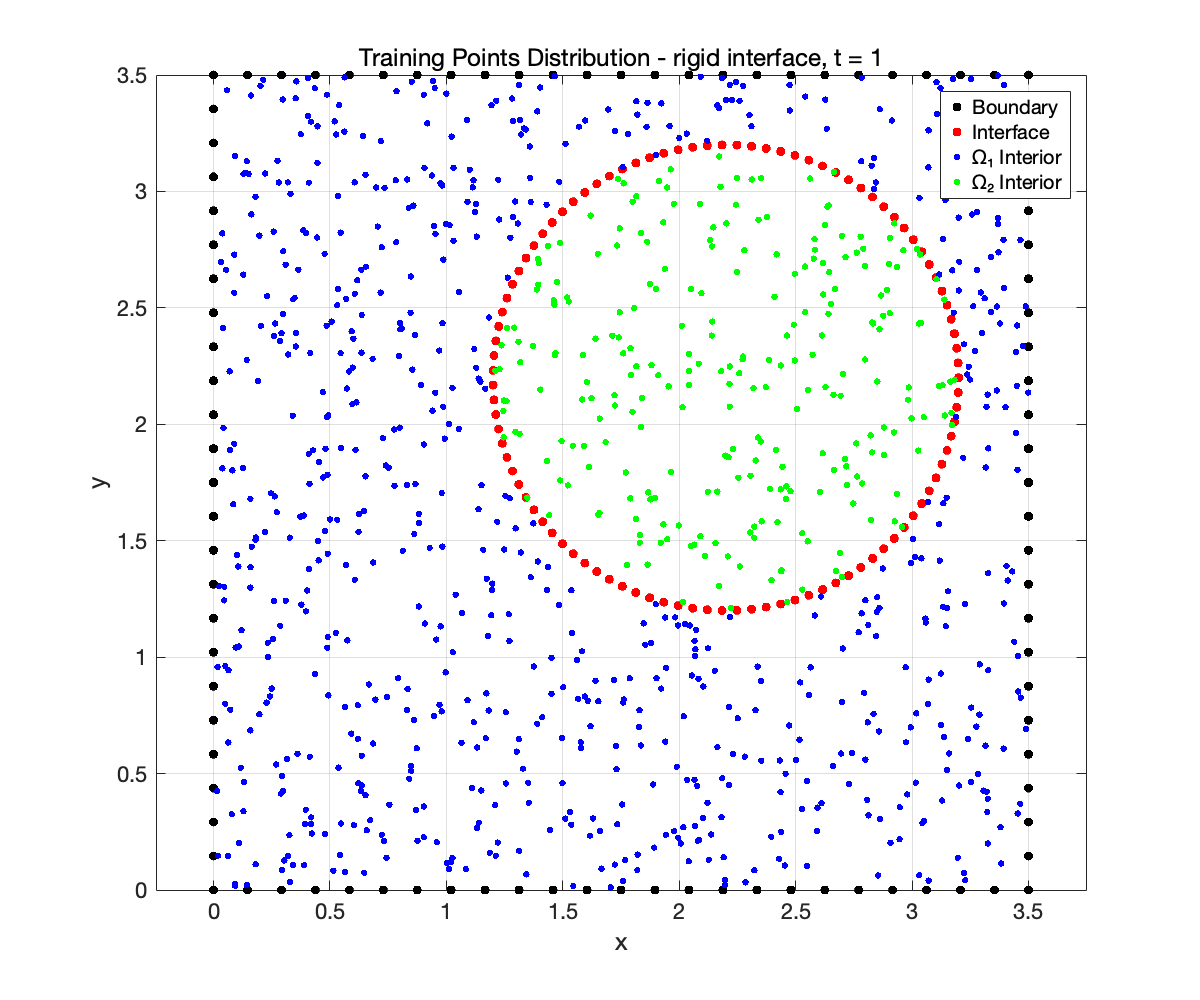}}
    \caption{Training points for Test 2 at $t=0.5$ and $t=1$.}
    \label{fig:rigid_interfacepoint}
\end{figure}

\begin{figure}[H]
    \centering
    \subfigure[Exact at $t=0.5$]{\includegraphics[width=4.5cm, height=3cm]{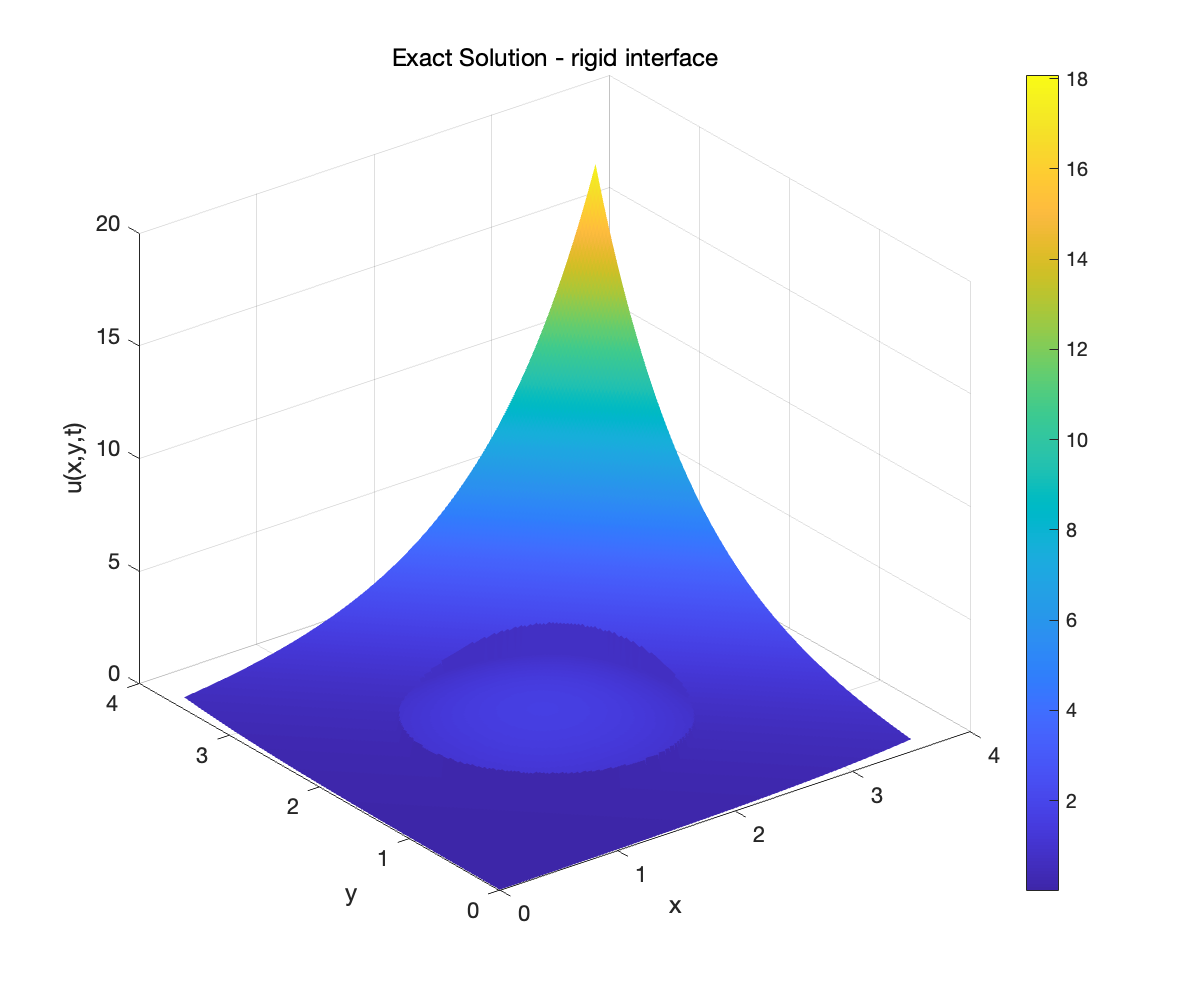}}
    \subfigure[MAF at $t=0.5$]{\includegraphics[width=4.5cm, height=3cm]{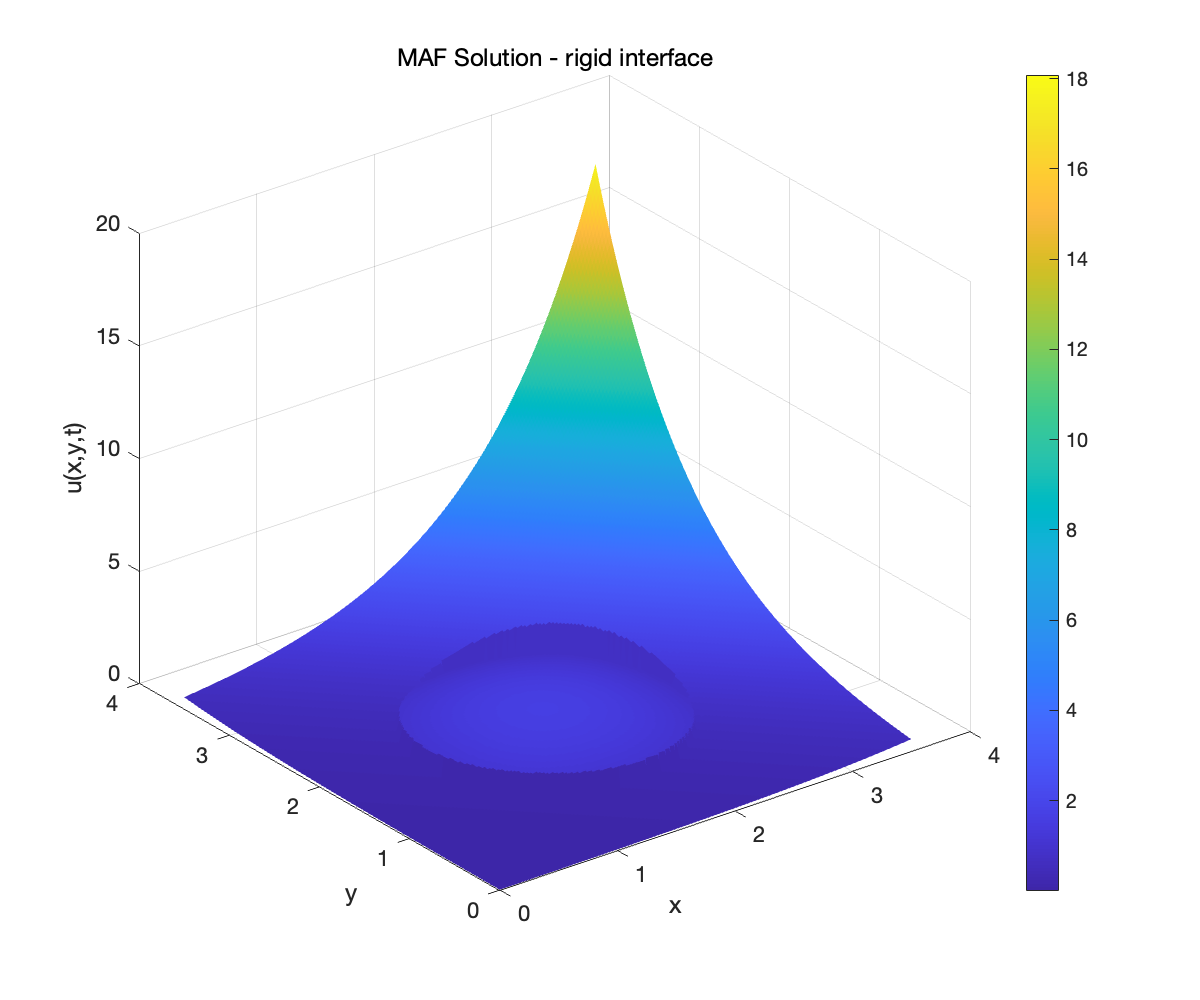}}
    \subfigure[Error at $t=0.5$]{\includegraphics[width=4.5cm, height=3cm]{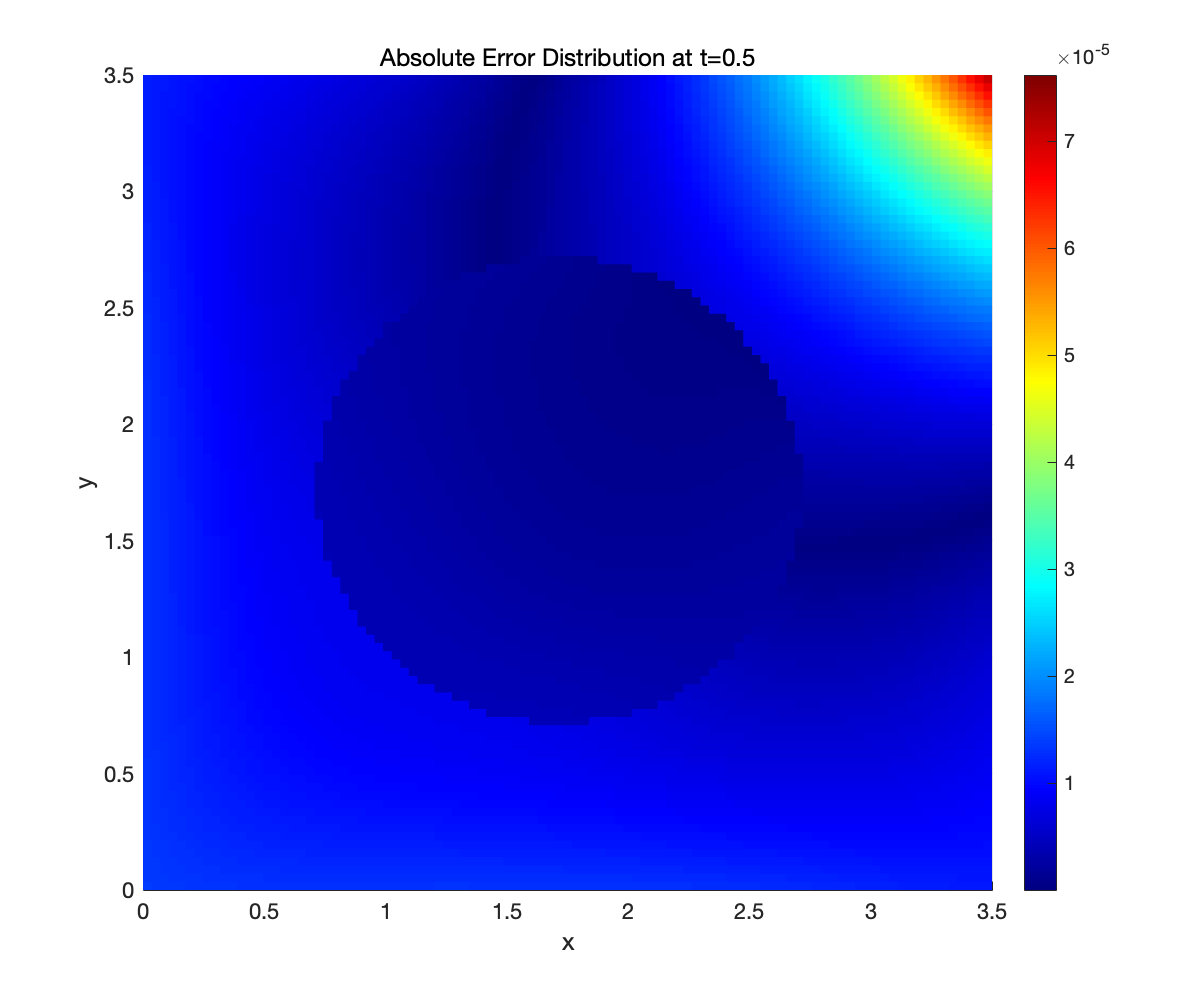}}
    \caption{Numerical results for Test 2 at $t=0.5$.}
    \label{fig:numercialresultst05rigid}
\end{figure}

\begin{figure}[H]
    \centering
    \subfigure[Exact at $t=1$]{\includegraphics[width=4.5cm, height=3cm]{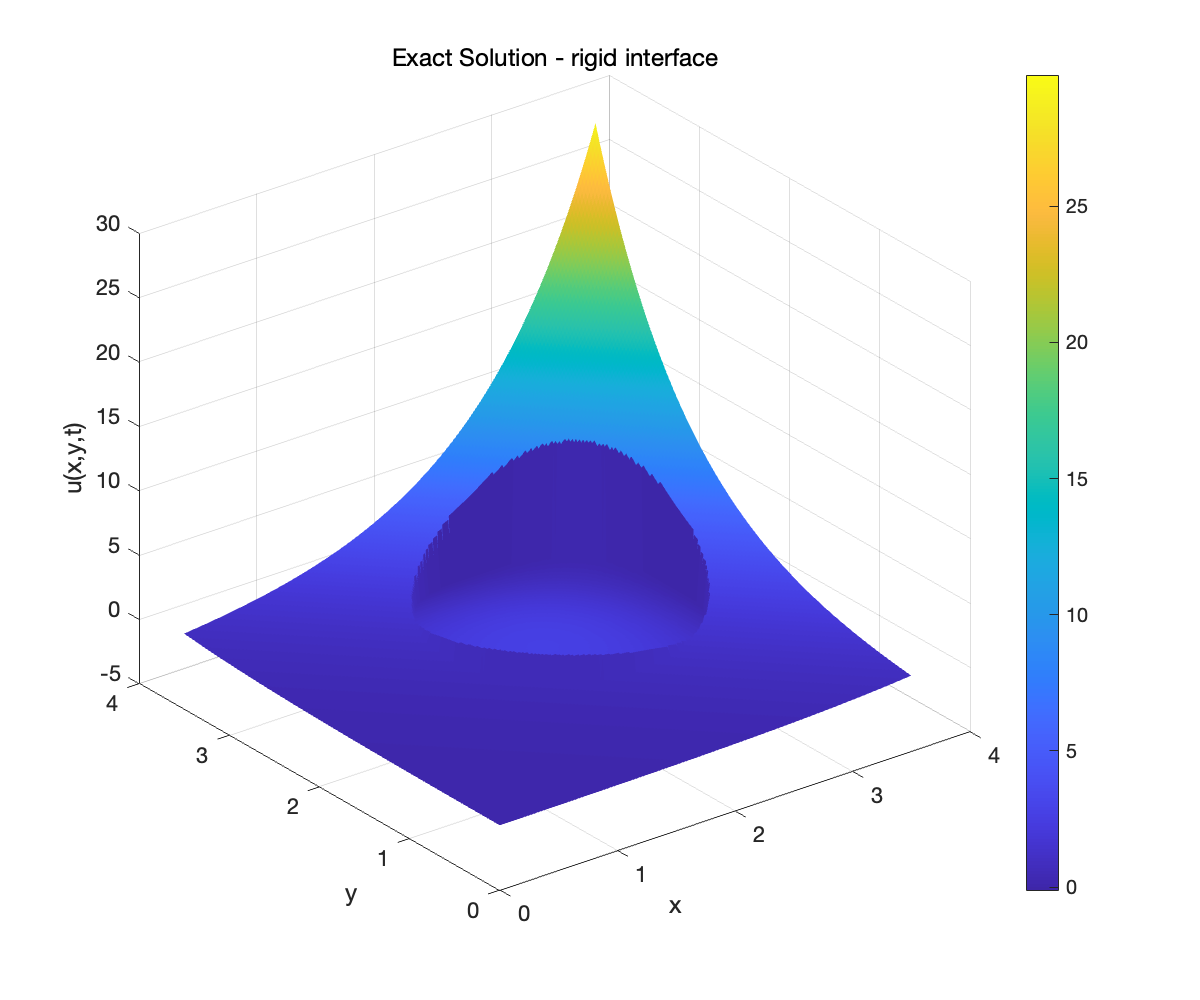}}
    \subfigure[MAF at $t=1$]{\includegraphics[width=4.5cm, height=3cm]{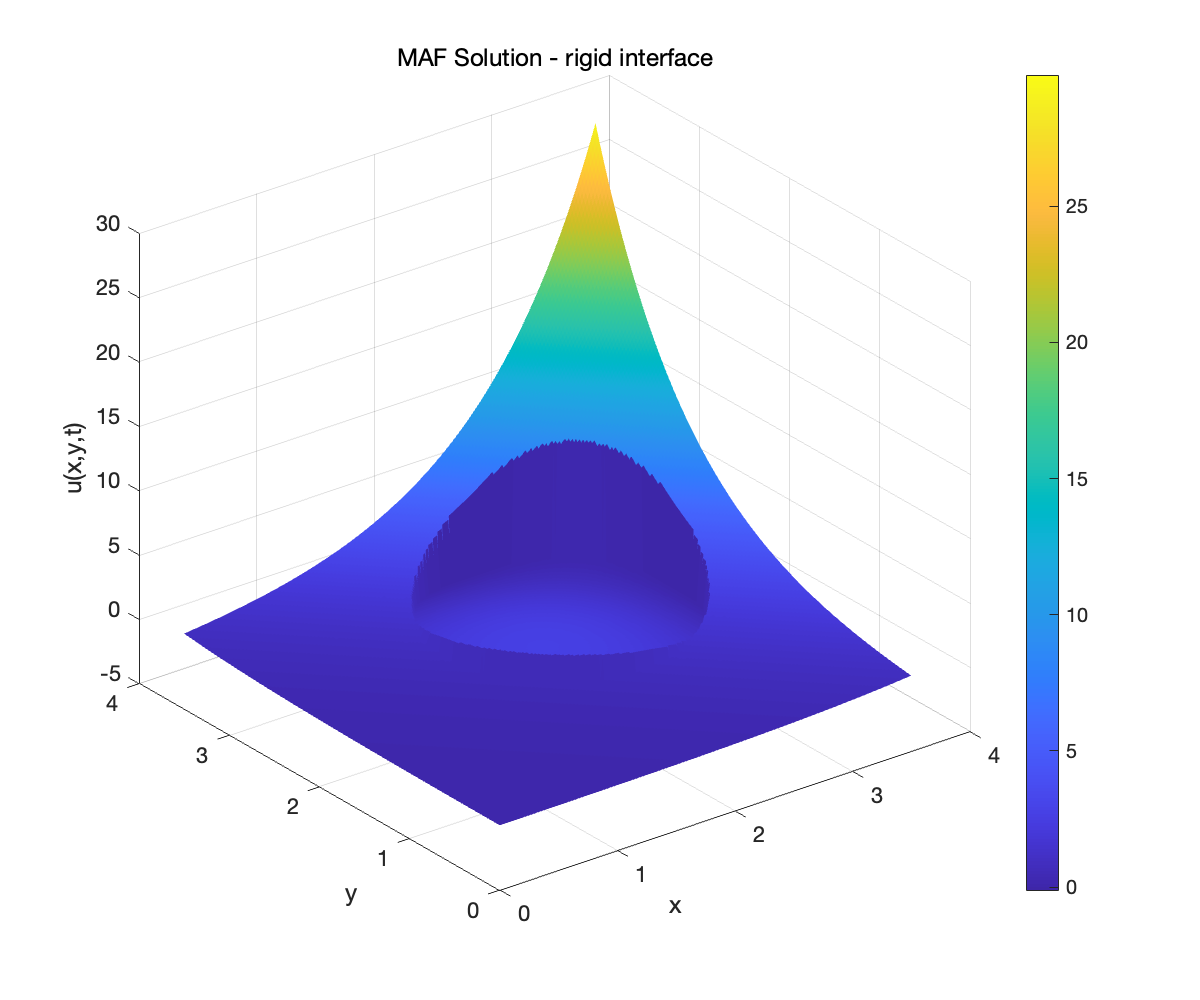}}
    \subfigure[Error at $t=1$]{\includegraphics[width=4.5cm, height=3cm]{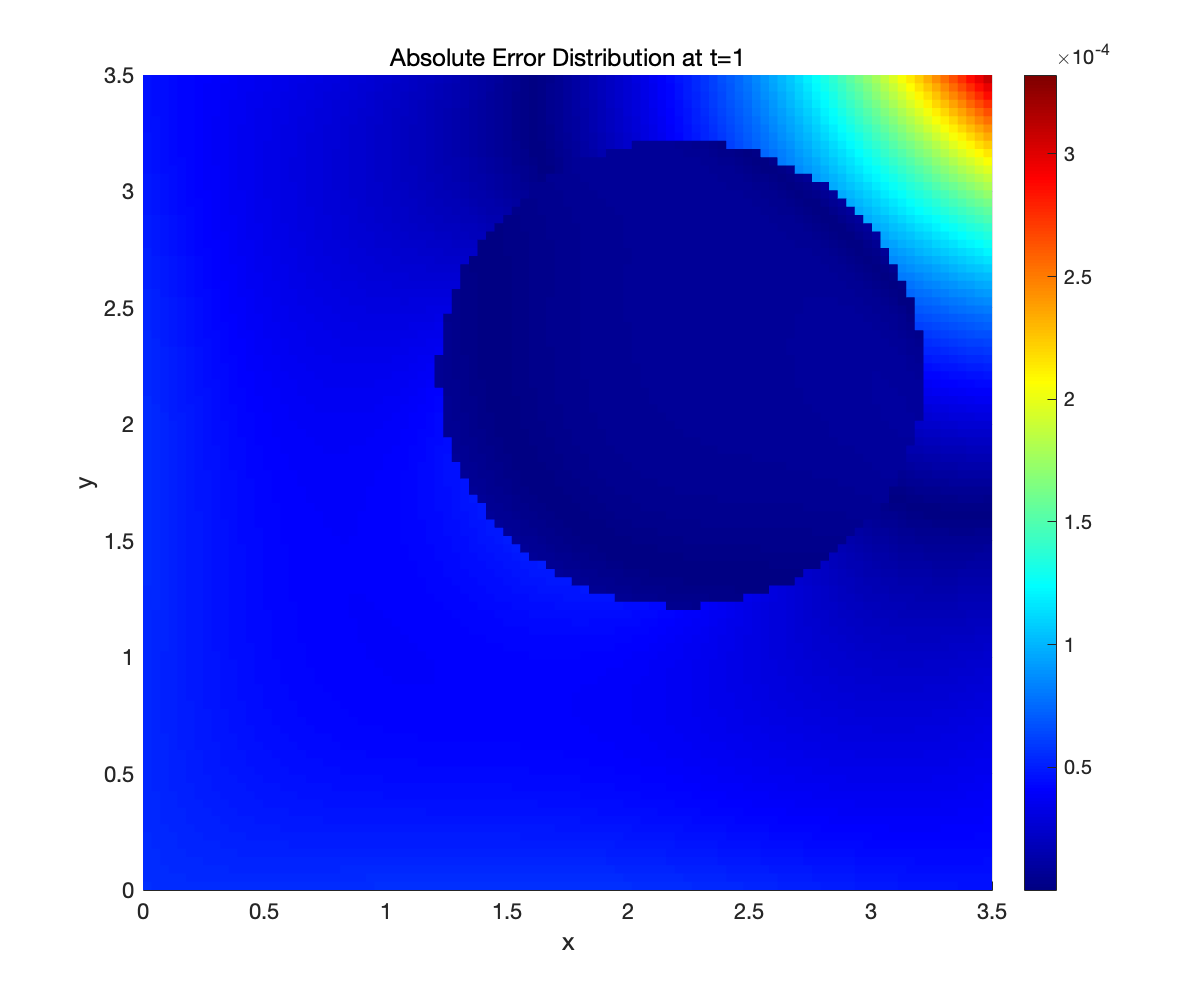}}
    \caption{Numerical results for Test 2 at $t=1$.}
    \label{fig:numercialresultst1rigid}
\end{figure}

\subsubsection{Test 3: Deforming Elliptical Interface}

We consider an elliptical interface $\Gamma(t): (x-x_c(t))^2/a^2(t) + (y-y_c(t))^2/b^2(t) = 1$ with time-dependent center $(x_c(t), y_c(t)) = (1.2 + 0.8t, 1.2 + 0.8t)$ and semi-axes $a(t) = \sqrt{1 + 0.1t}$, $b(t) = \sqrt{1 - 0.1t}$. This combines translation with anisotropic deformation. The exact solution is $u_1 = 0.01 e^t e^x e^y$ in $\Omega_1(t)$ and $u_2 = e^t \sin(x)\sin(y)$ in $\Omega_2(t)$. The diffusion coefficients are $\beta_1 = 1$ and $\beta_2 = 10$. The source terms are $f_1 = 0.01 e^t e^{x+y}(1 - 2\beta_1)$ and $f_2 = e^t \sin(x)\sin(y)(1 + 2\beta_2)$. The jump conditions $g_1$ and $g_2$ are computed from the exact solution. The initial and boundary conditions are $g_0 = u|_{t=0}$ and $g_D = u_1|_{\partial\Omega}$. The distance function and weight functions evolve as both the position and shape of the interface change. Figure~\ref{fig:weight_parabolic_test3} shows the weight function evolution at different times.

The training points and numerical results are shown in Figures~\ref{fig:convex_interfacepoint}--\ref{fig:numercialresultst1convex} (corresponding to the finest sampling configuration with $M_\Omega=1000$, $M_{\partial\Omega}=100$, $M_\Gamma=100$). The peak absolute error is $8.25 \times 10^{-5}$ at $t=0.5$ and grows to $4.52 \times 10^{-4}$ at $t=1.0$. Errors are primarily concentrated near the deforming elliptical interface, which aligns with the challenges imposed by the simultaneous translation and anisotropic deformation. The MAF method achieves a relative $L^2$ error of $8.89 \times 10^{-5}$, outperforming XPINN ($8.68 \times 10^{-3}$) by 98$\times$.

\begin{figure}[H]
    \centering
    \includegraphics[width=12cm, height=4cm]{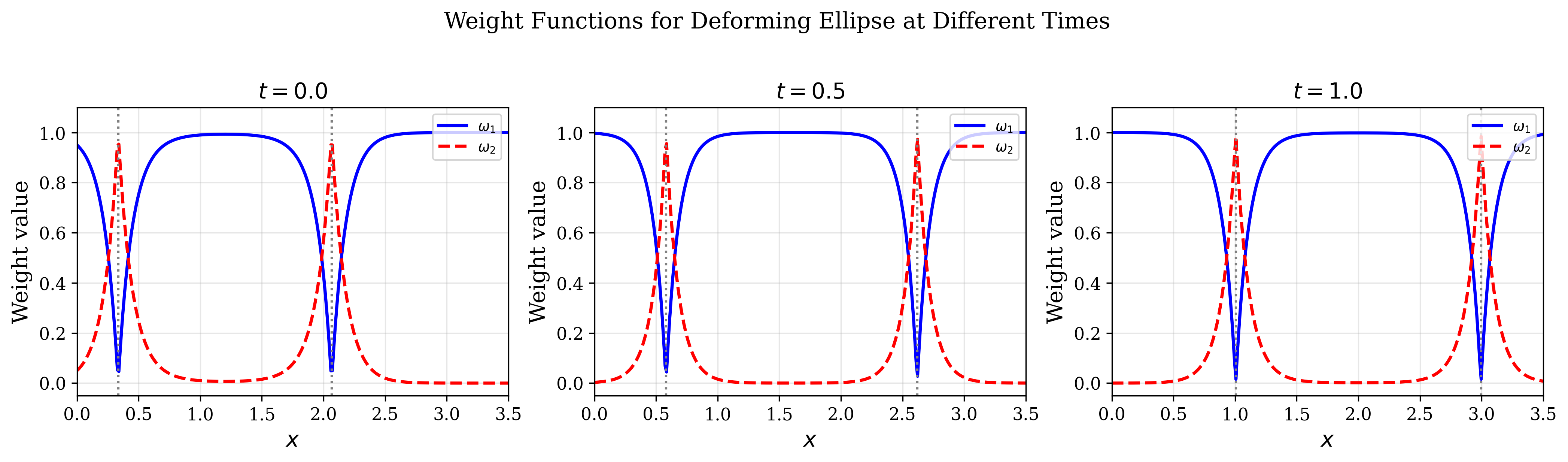}
    \caption{Weight functions at $t=0$, $t=0.5$, and $t=1$ for Test 3 (deforming ellipse).}
    \label{fig:weight_parabolic_test3}
\end{figure}

\begin{figure}[H]
\centering
\subfigure[Training points at $t=0.5$]{\includegraphics[width=4.5cm, height=3cm]{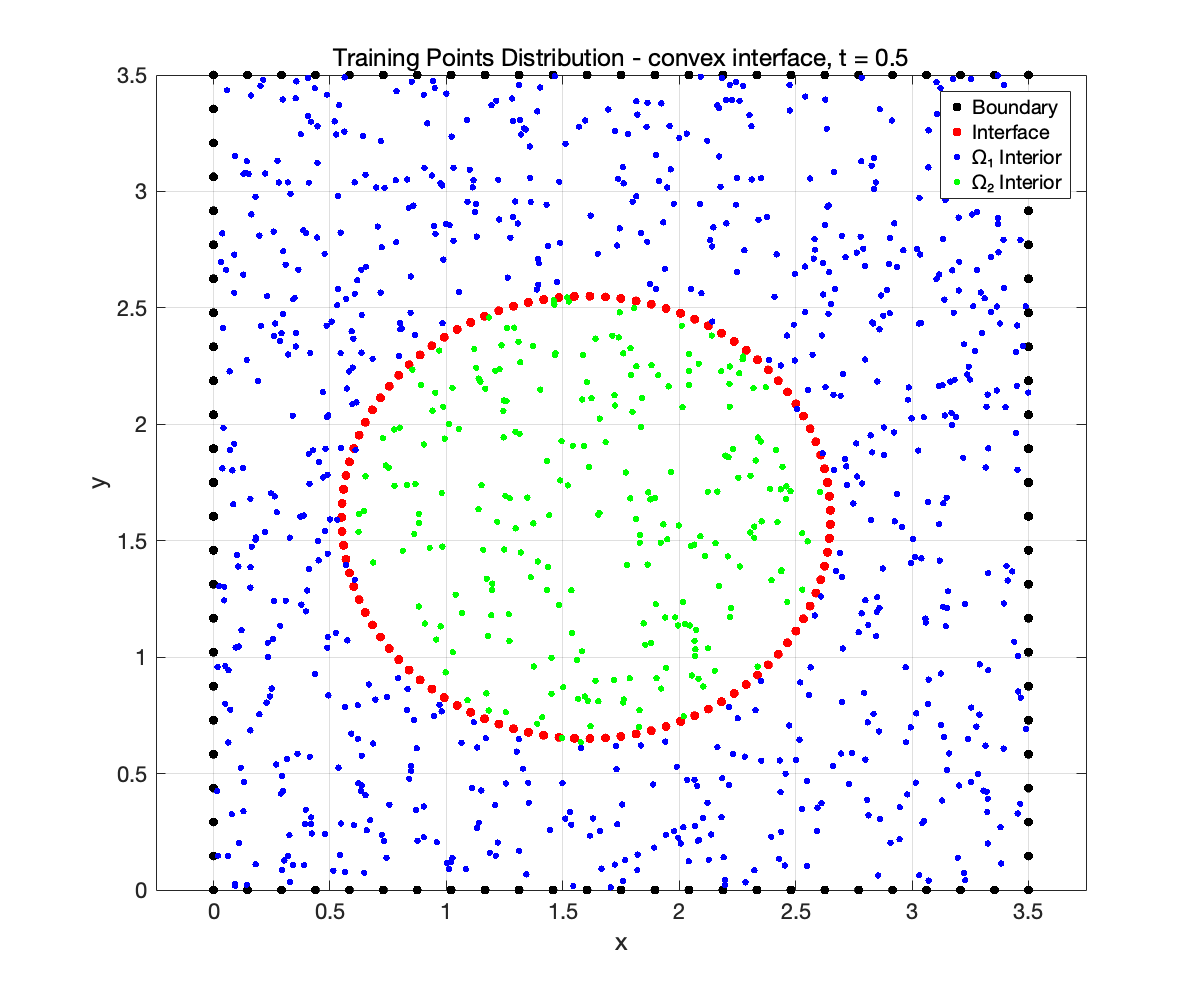}}
\subfigure[Training points at $t=1$]{\includegraphics[width=4.5cm, height=3cm]{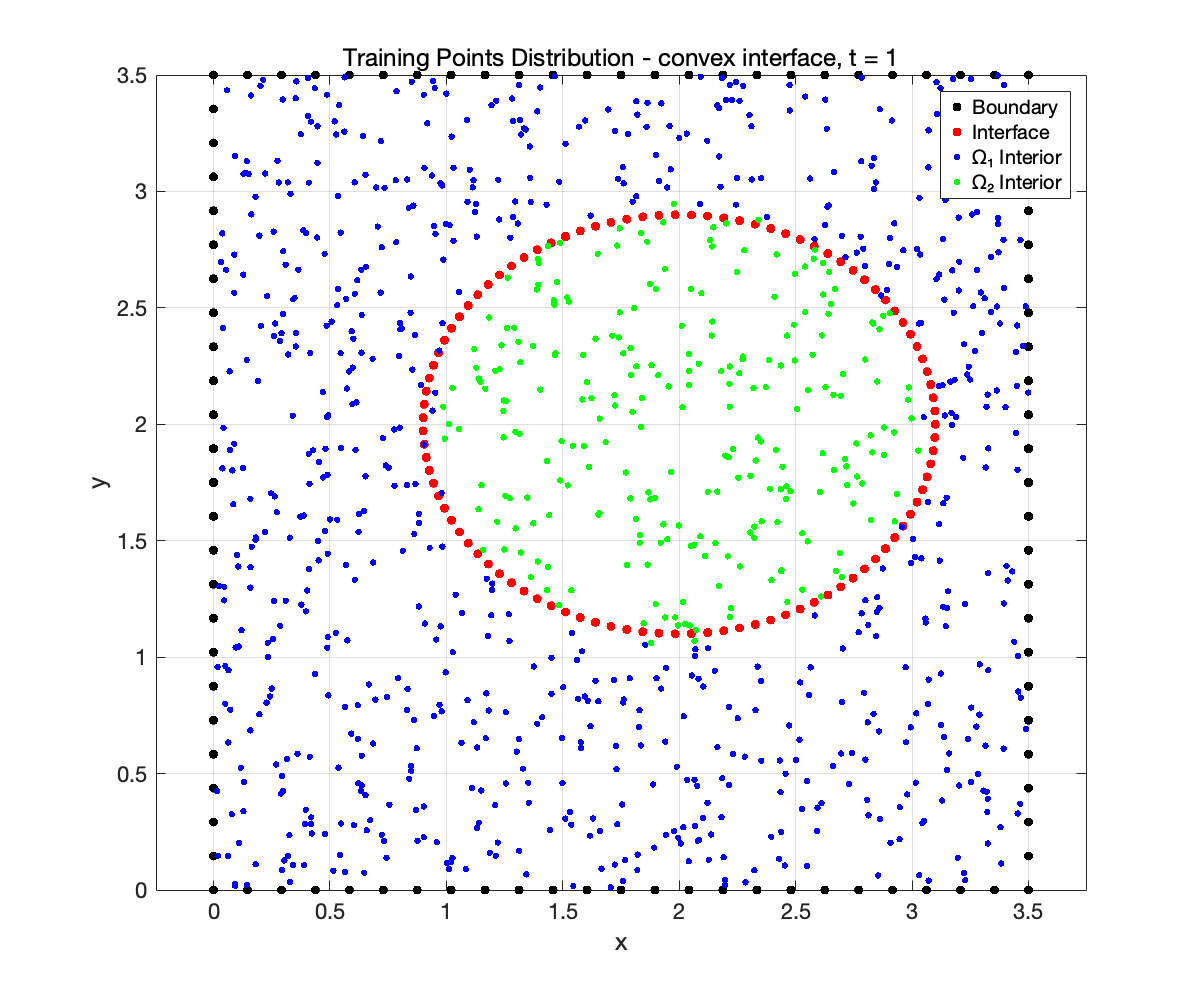}}
    \caption{Training points for Test 3 at $t=0.5$ and $t=1$.}
\label{fig:convex_interfacepoint}
\end{figure}

\begin{figure}[H]
\centering
    \subfigure[Exact at $t=0.5$]{\includegraphics[width=4.5cm, height=3cm]{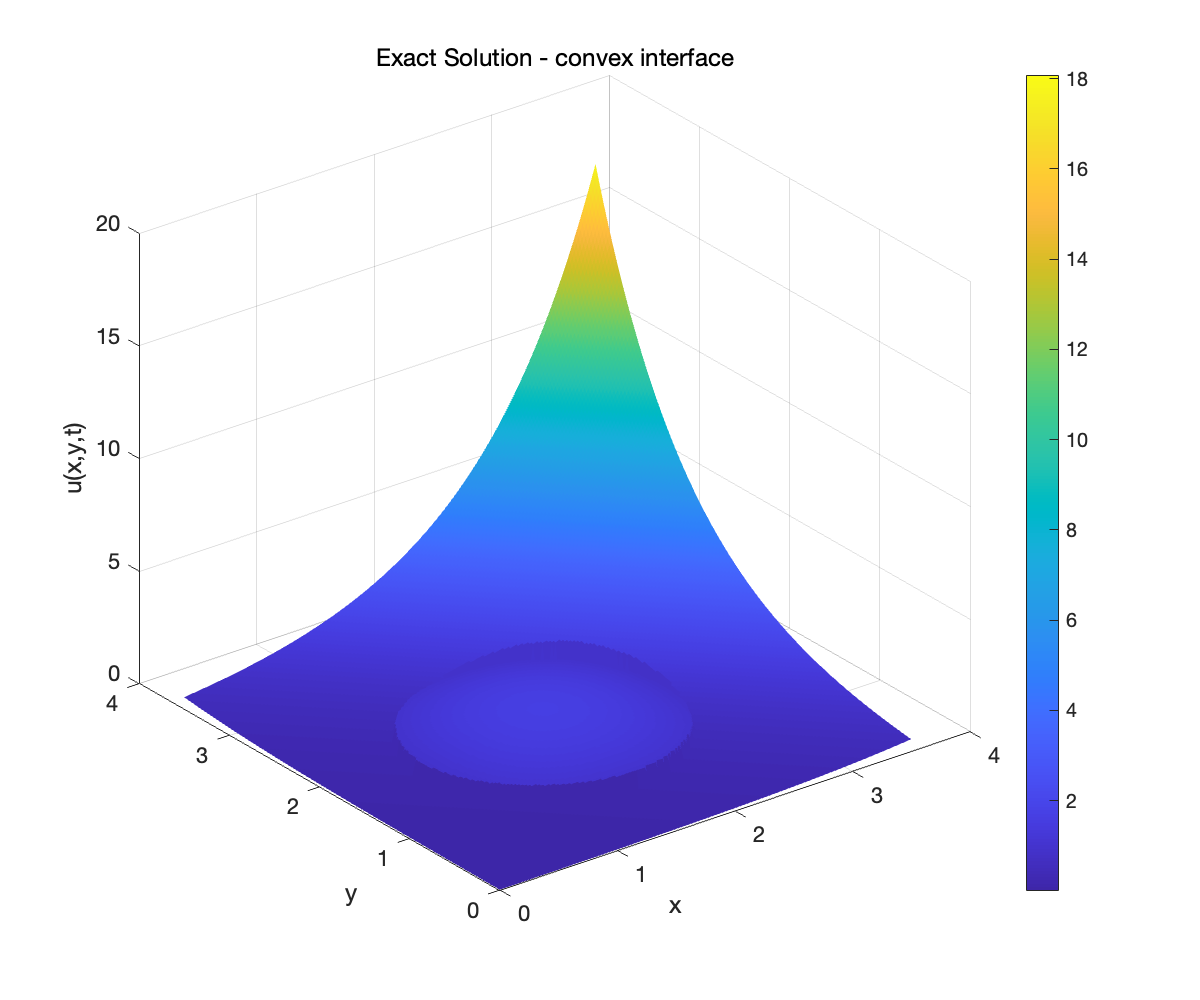}}
    \subfigure[MAF at $t=0.5$]{\includegraphics[width=4.5cm, height=3cm]{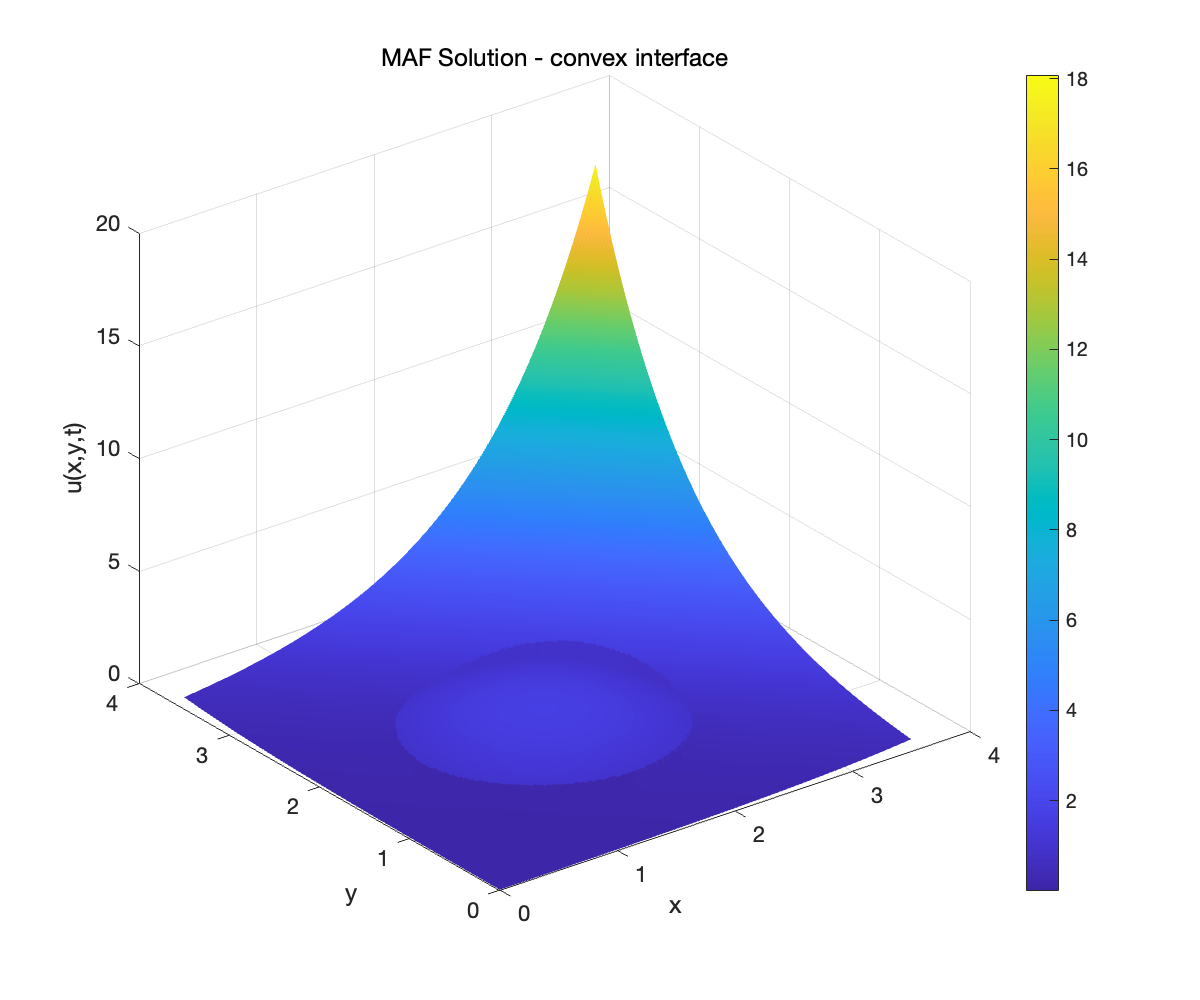}}
    \subfigure[Error at $t=0.5$]{\includegraphics[width=4.5cm, height=3cm]{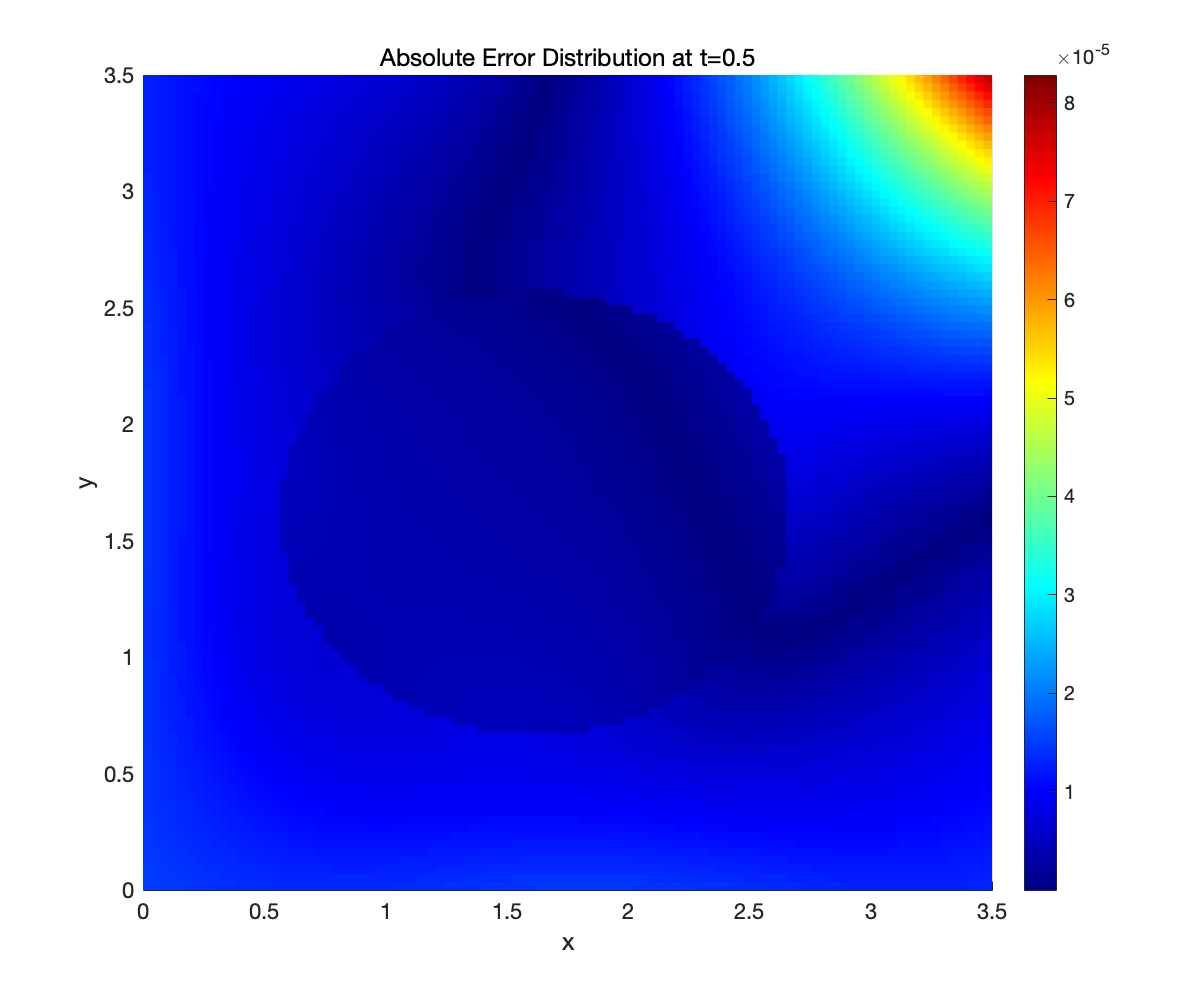}}
    \caption{Numerical results for Test 3 at $t=0.5$.}
\label{fig:numercialresultst05convex}
\end{figure}

\begin{figure}[H]
\centering
    \subfigure[Exact at $t=1$]{\includegraphics[width=4.5cm, height=3cm]{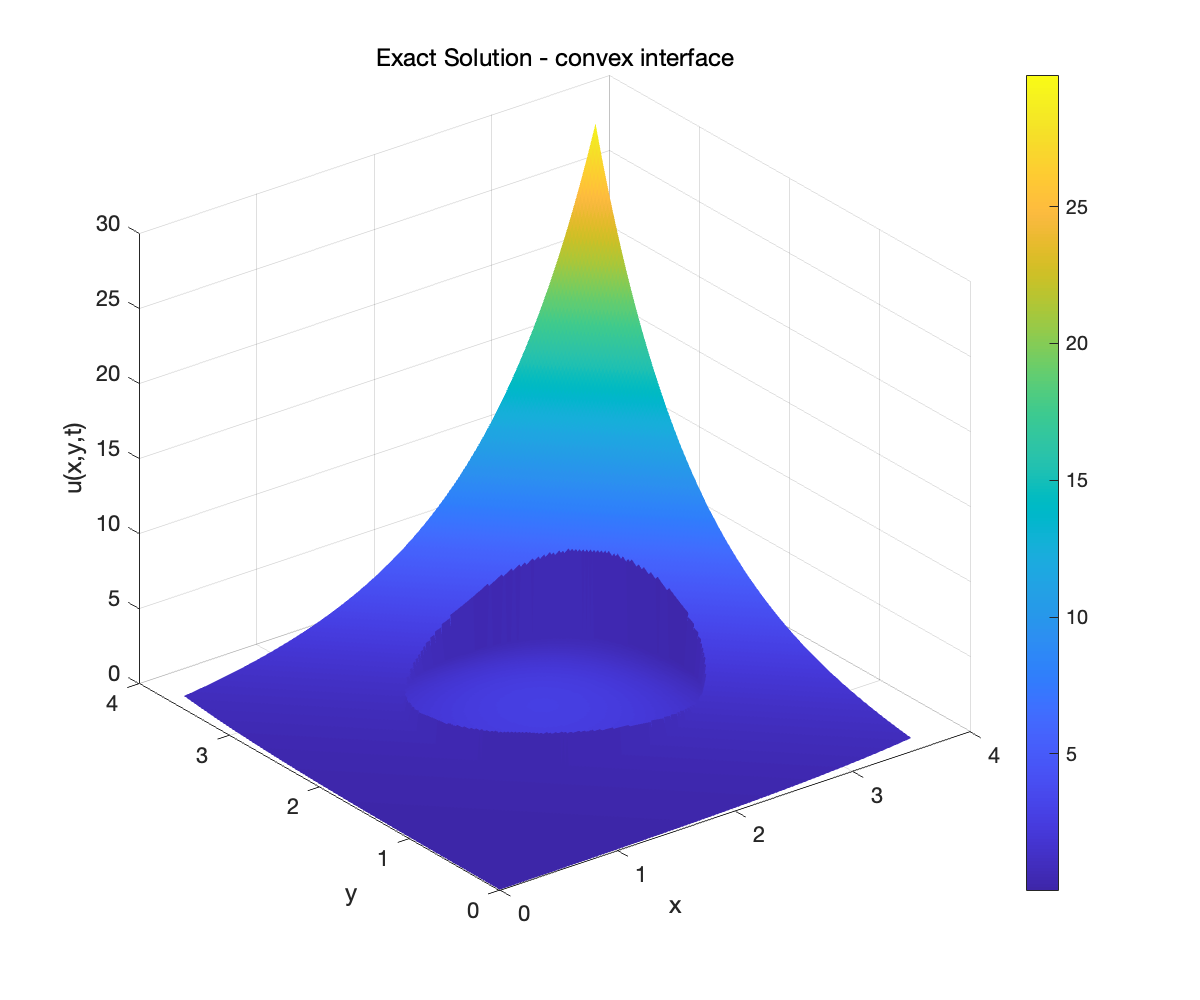}}
    \subfigure[MAF at $t=1$]{\includegraphics[width=4.5cm, height=3cm]{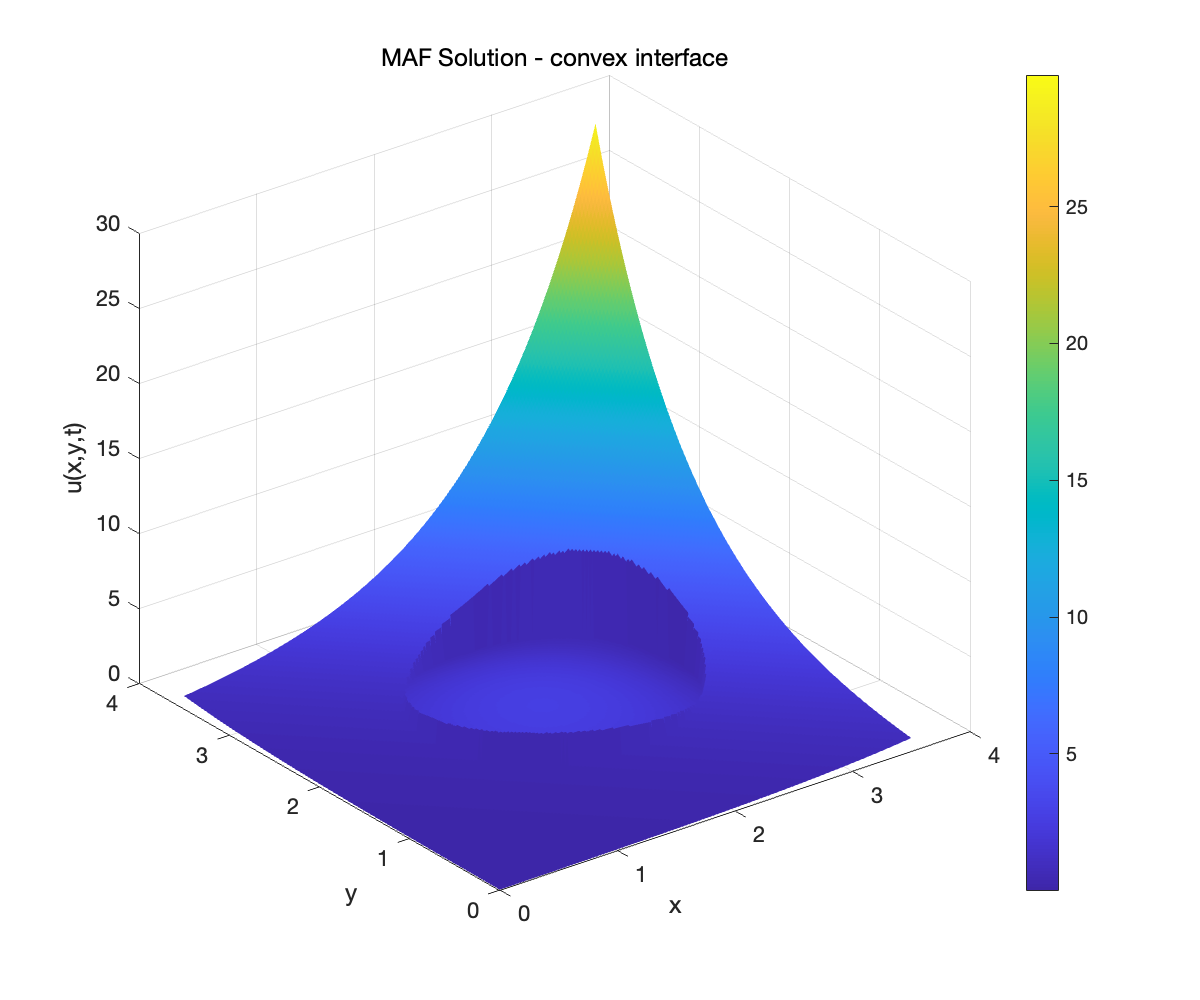}}
    \subfigure[Error at $t=1$]{\includegraphics[width=4.5cm, height=3cm]{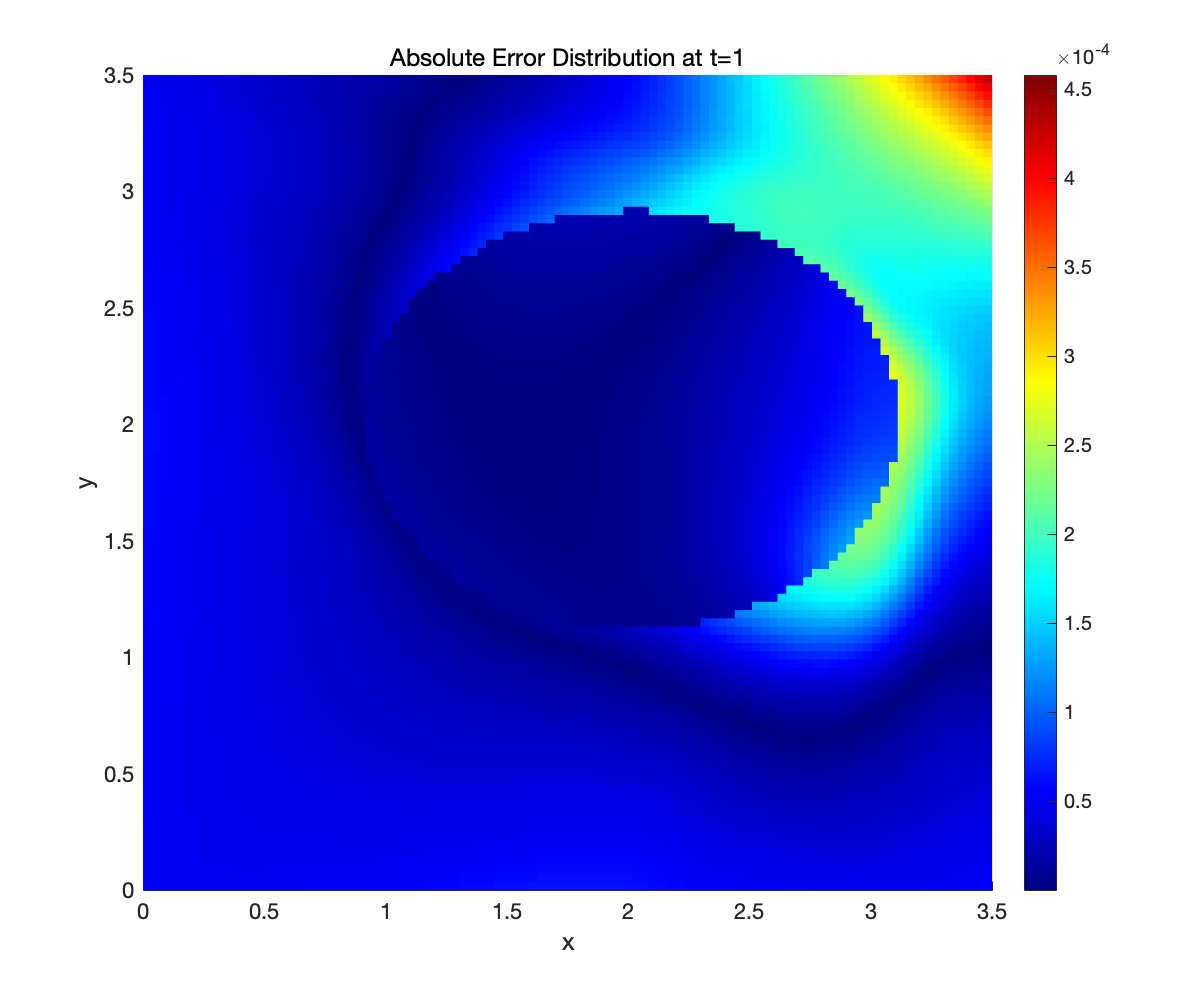}}
    \caption{Numerical results for Test 3 at $t=1$.}
\label{fig:numercialresultst1convex}
\end{figure}

\subsubsection{Test 4: Star-Shaped Interface with Deformation}

We consider a star-shaped interface undergoing translation, rotation, and parametric deformation. The interface is defined by $\mathbf{x}(t) = \mathbf{R}(2\pi t) [r(\theta,t)\cos\theta, r(\theta,t)\sin\theta]^T + \mathbf{c}(t)$, where $r(\theta,t) = 1 - 0.3t\cos(5\theta)$, $\mathbf{R}(\phi)$ is a rotation matrix, and $\mathbf{c}(t) = (1.2 + 0.8t, 1.2 + 0.8t)$. This combines time-varying topology, rotation, translation, and amplitude-modulated radial perturbations. The exact solution is $u_1 = 0.01 e^t e^x e^y$ in $\Omega_1(t)$ and $u_2 = e^t \sin(x)\sin(y)$ in $\Omega_2(t)$. The source terms are $f_1 = 0.01 e^t e^{x+y}(1 - 2\beta_1)$ and $f_2 = e^t \sin(x)\sin(y)(1 + 2\beta_2)$, where $\beta_1 = 1$ and $\beta_2 = 10$. The jump conditions are $g_1 = u_1|_{\Gamma(t)} - u_2|_{\Gamma(t)}$ and $g_2 = \beta_1 \nabla u_1 \cdot \boldsymbol{n} - \beta_2 \nabla u_2 \cdot \boldsymbol{n}$, computed from the exact solution. The initial condition is $g_0 = u|_{t=0}$ and the Dirichlet boundary condition is $g_D = u_2|_{\partial\Omega}$. Figure~\ref{fig:weight_parabolic_test4} shows the weight function evolution.

The training points and numerical results are shown in Figures~\ref{fig:star_interfacepoint}--\ref{fig:star_results_t1} (corresponding to the finest sampling configuration with $M_\Omega=1000$, $M_{\partial\Omega}=100$, $M_\Gamma=100$). The maximum absolute error is $9.10 \times 10^{-5}$ at $t=0.5$ and grows to $5.02 \times 10^{-4}$ at $t=1.0$. Errors are primarily concentrated near regions of maximum curvature (interface vertices), which aligns with the challenges imposed by the time-varying star-shaped geometry and the combined effects of translation, rotation, and deformation. The MAF method achieves a relative $L^2$ error of $9.01 \times 10^{-5}$, outperforming XPINN ($8.99 \times 10^{-3}$) by 100$\times$.

\begin{figure}[H]
    \centering
    \includegraphics[width=11cm, height=3.5cm]{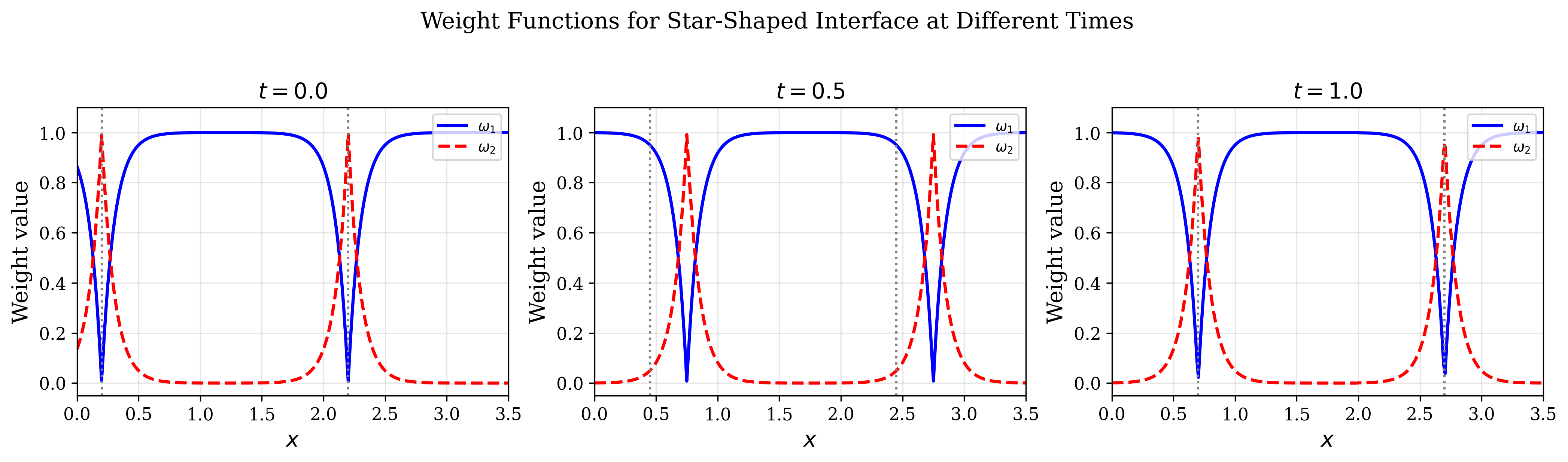}
    \caption{Weight functions at $t=0$, $t=0.5$, and $t=1$ for Test 4 (star-shaped interface).}
    \label{fig:weight_parabolic_test4}
\end{figure}

\begin{figure}[H]
    \centering
    \subfigure[Training points at $t=0.5$]{\includegraphics[width=4.5cm, height=3cm]{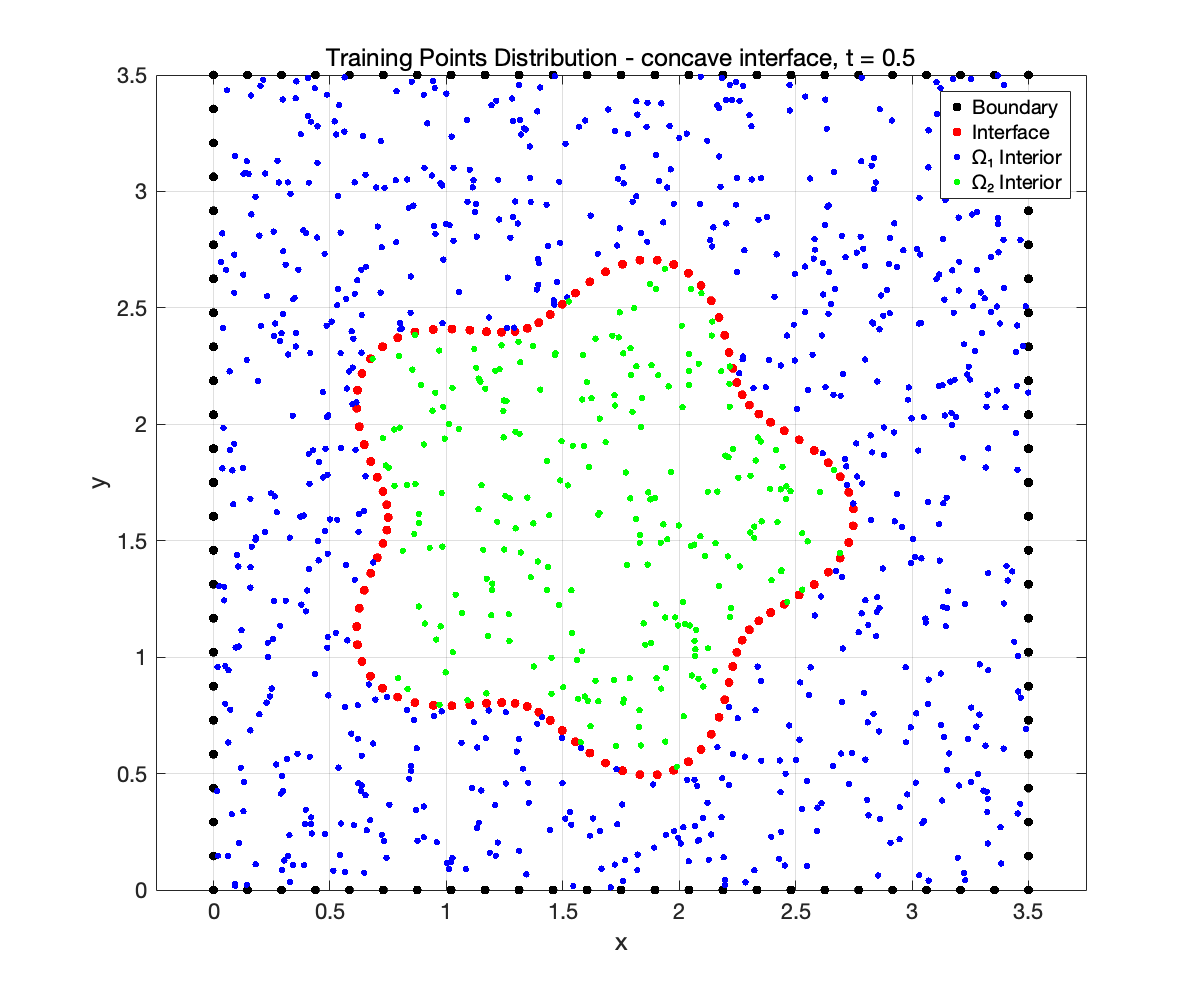}}
    \subfigure[Training points at $t=1$]{\includegraphics[width=4.5cm, height=3cm]{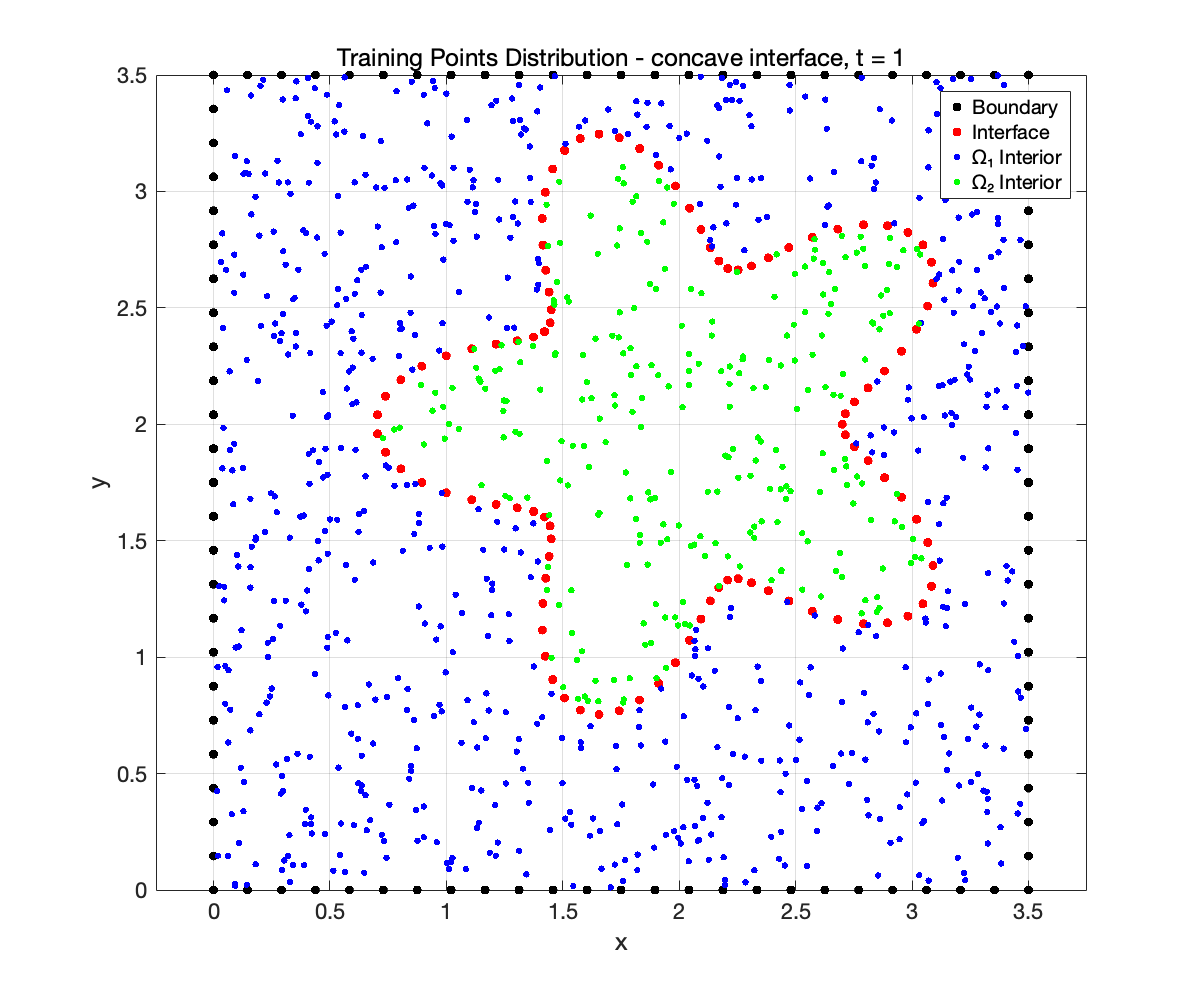}}
    \caption{Training points for Test 4 at $t=0.5$ and $t=1$.}
    \label{fig:star_interfacepoint}
\end{figure}

\begin{figure}[H]
    \centering
    \subfigure[Exact at $t=0.5$]{\includegraphics[width=4.5cm, height=3cm]{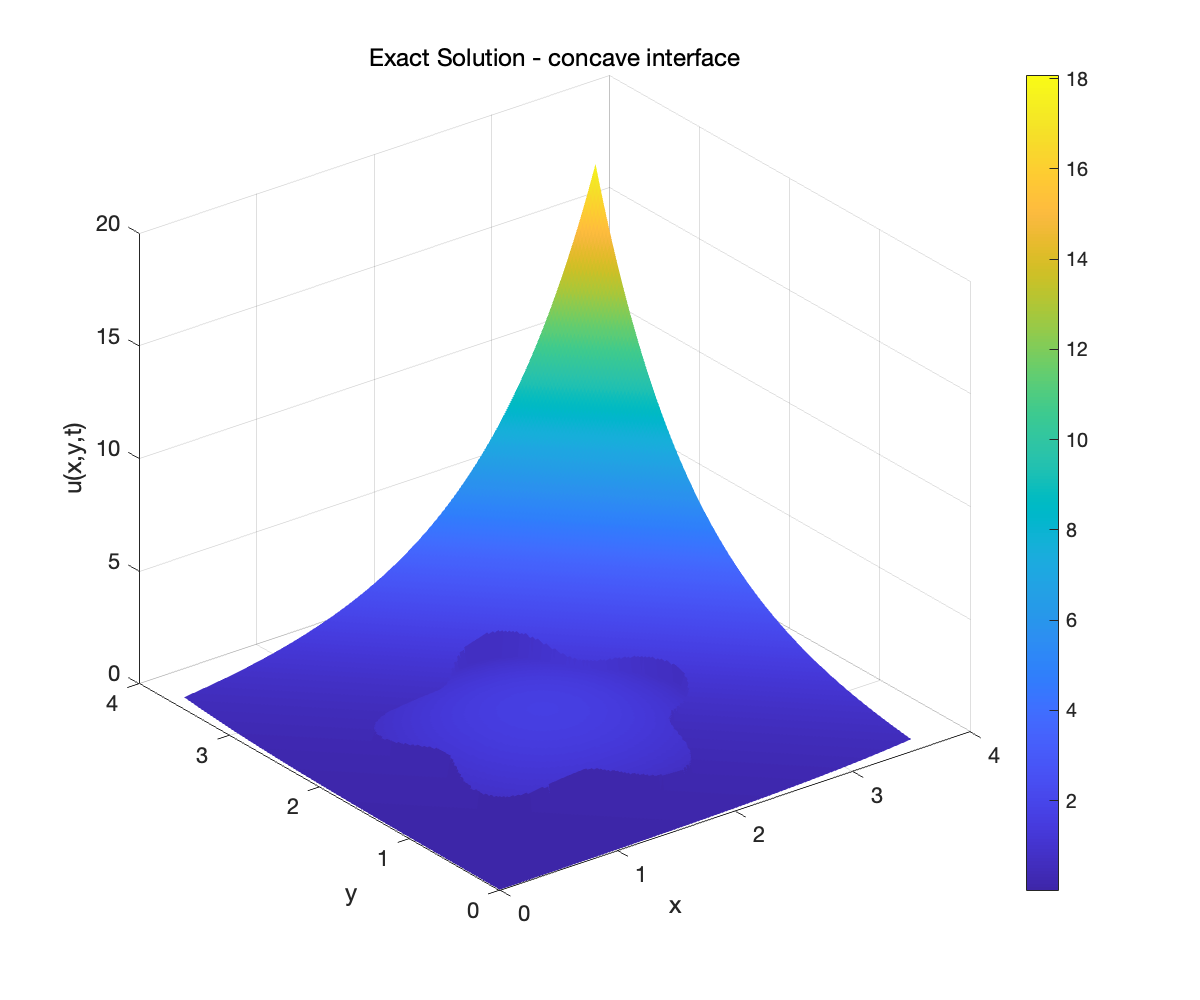}}
    \subfigure[MAF at $t=0.5$]{\includegraphics[width=4.5cm, height=3cm]{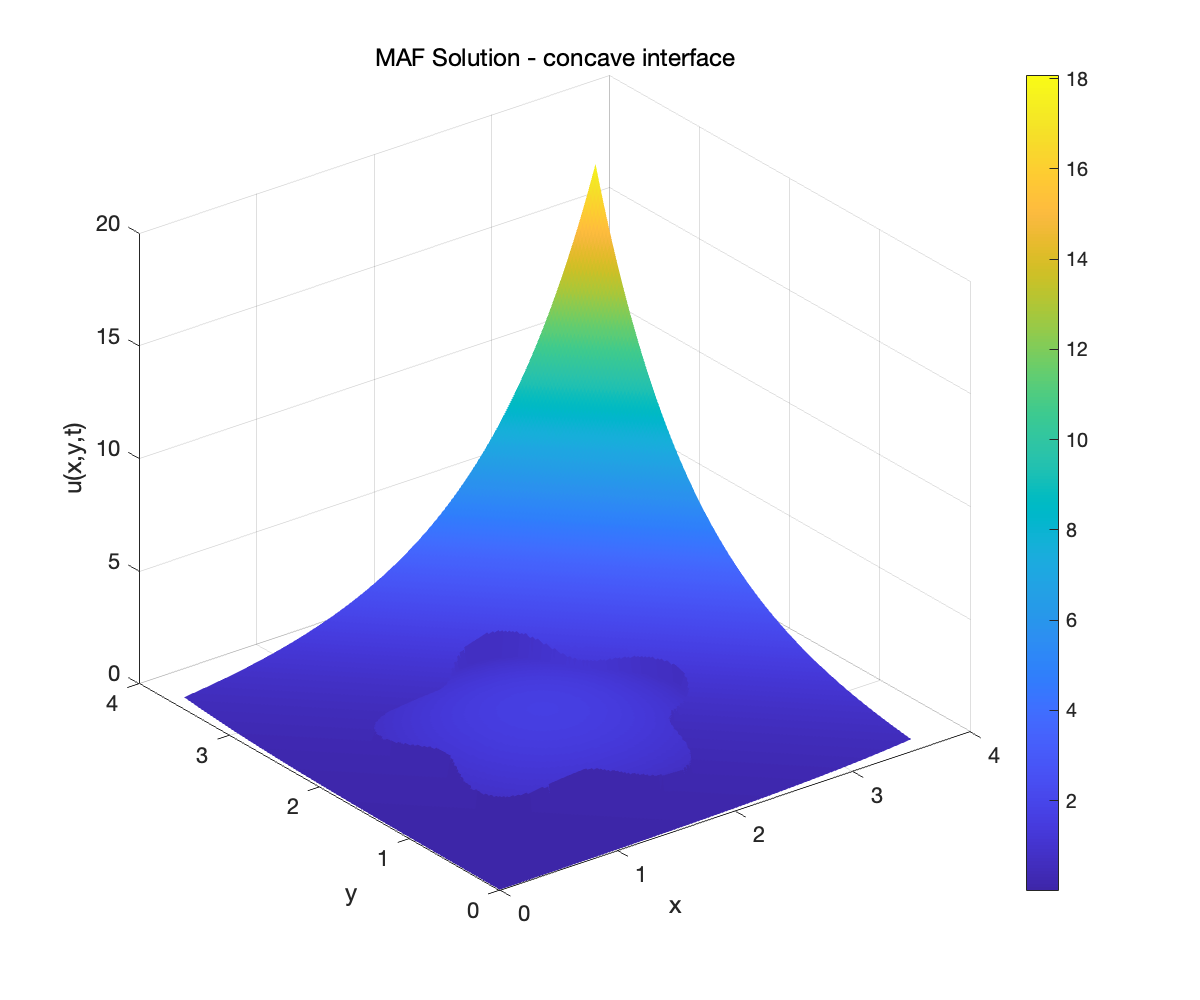}}
    \subfigure[Error at $t=0.5$]{\includegraphics[width=4.5cm, height=3cm]{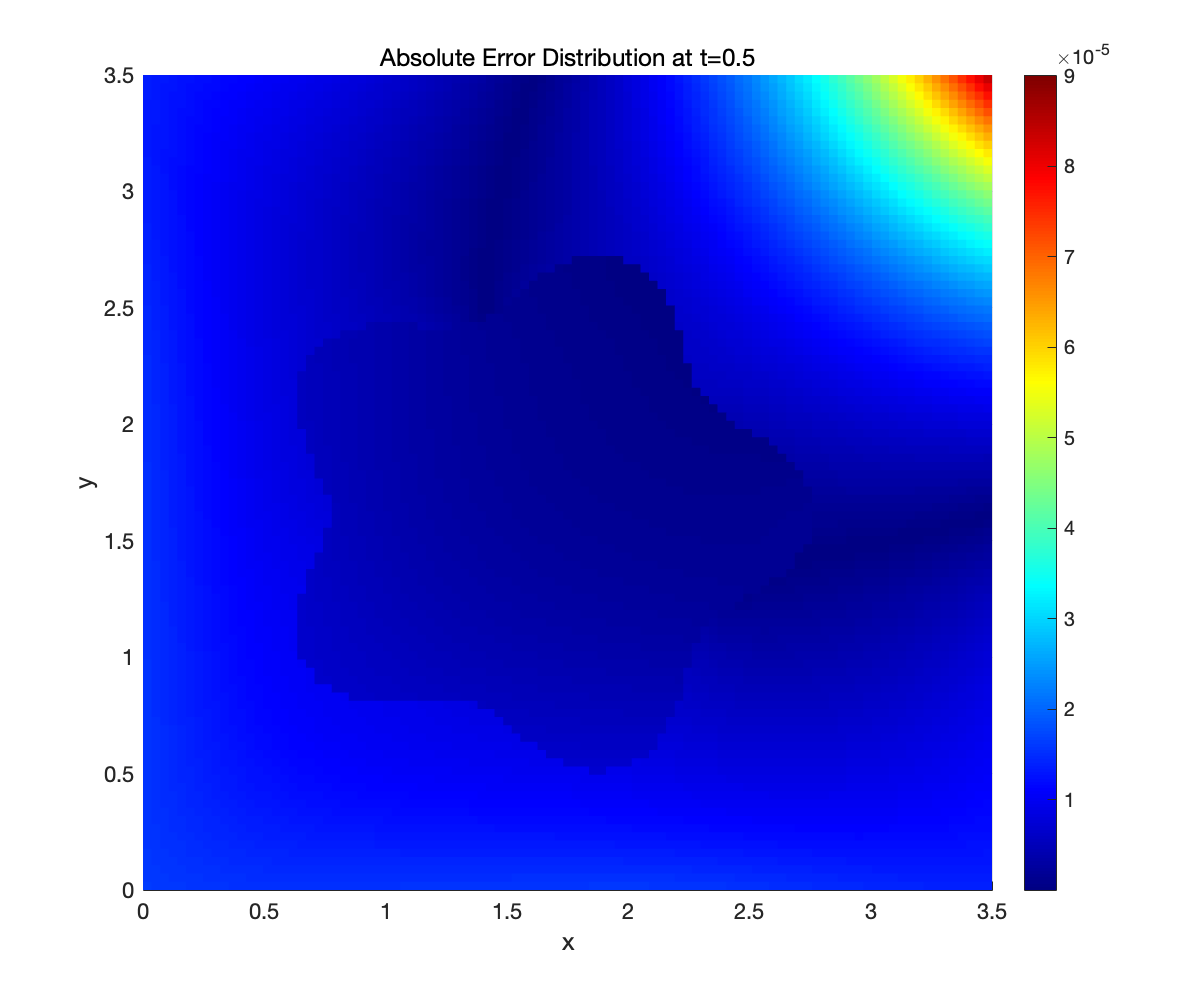}}
    \caption{Numerical results for Test 4 at $t=0.5$.}
    \label{fig:star_results_t05}
\end{figure}

\begin{figure}[H]
    \centering
    \subfigure[Exact at $t=1$]{\includegraphics[width=4.5cm, height=3cm]{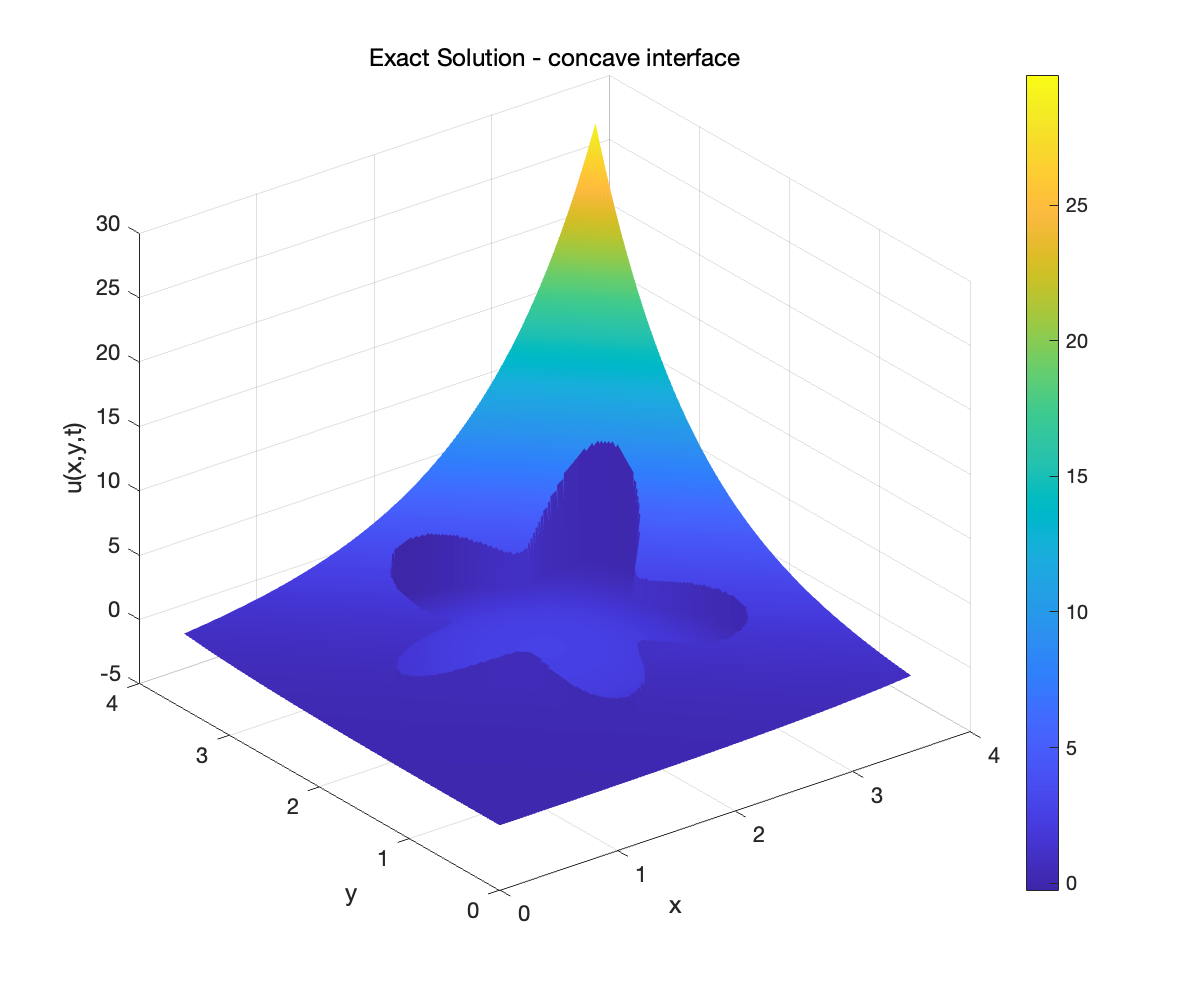}}
    \subfigure[MAF at $t=1$]{\includegraphics[width=4.5cm, height=3cm]{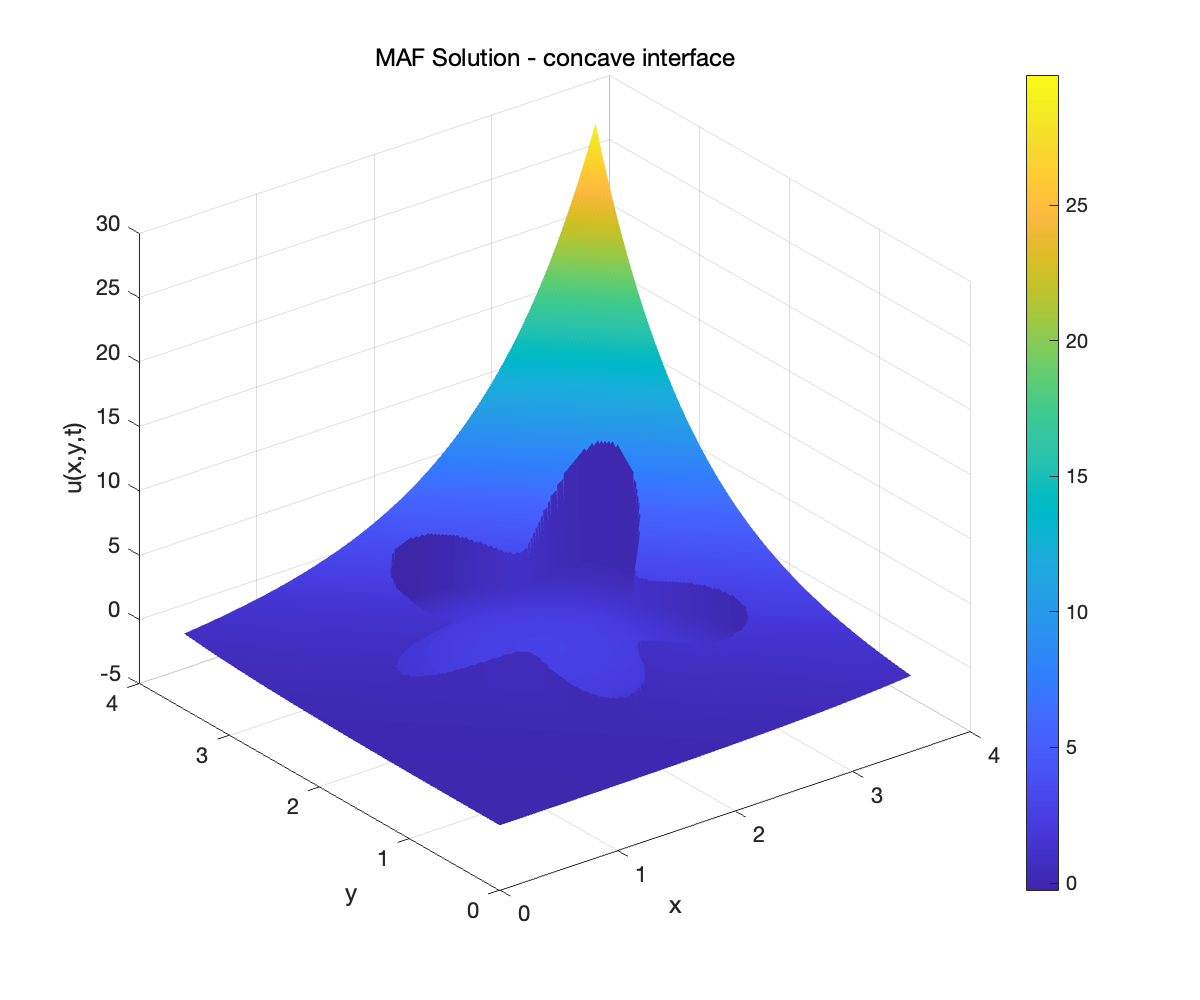}}
    \subfigure[Error at $t=1$]{\includegraphics[width=4.5cm, height=3cm]{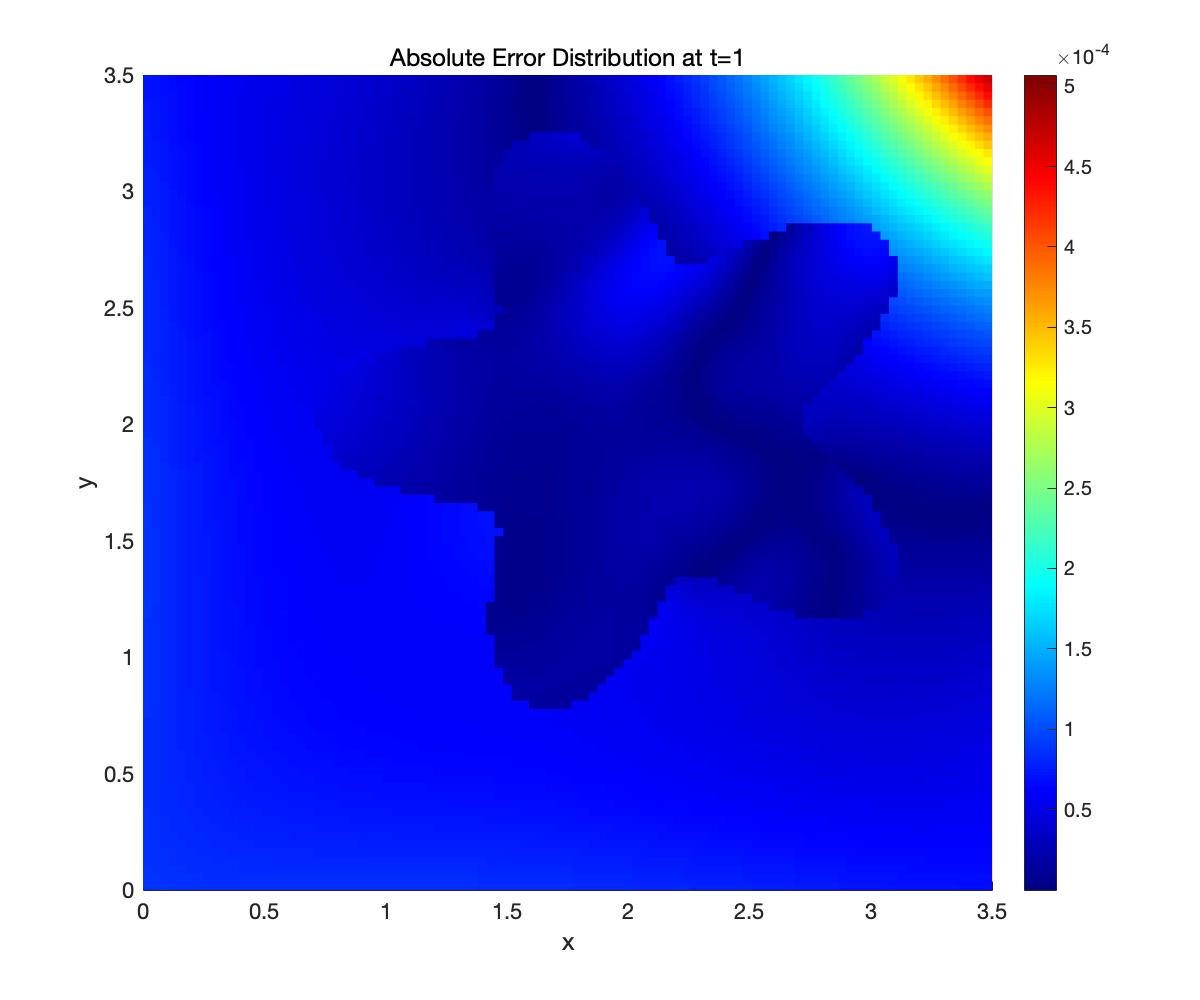}}
    \caption{Numerical results for Test 4 at $t=1$.}
    \label{fig:star_results_t1}
\end{figure}

Table~\ref{table:parabolic_summary} summarizes the relative $L^2$ errors for all parabolic test cases under different sampling densities. The MAF method consistently outperforms XPINN by 1--2 orders of magnitude across all interface configurations. For moving interface scenarios, where no previous deep learning methods (MFM, DCSNN, M-PINN, I-PINN, AdaI-PINN) have been demonstrated, our method maintains robust performance with errors consistently below $10^{-4}$ at the finest sampling. The multi-activation strategy proves particularly effective in handling time-dependent interface dynamics, adapting smoothly to interface motion and deformation.

\begin{table}[H]
    \centering
    \caption{Relative $L^2$ error comparison for parabolic interface problems.}
    \label{table:parabolic_summary}
    \small
    \begin{tabular}{|c|c|c|c|c|c|}
        \hline
        $(M_\Omega, M_{\partial\Omega}, M_\Gamma)$ & Method & Fixed & Moving & Deforming & Star \\
        \hline
        \multirow{2}{*}{(200, 20, 20)} & MAF & $7.47e{-}4$ & $8.56e{-}4$ & $8.12e{-}4$ & $8.84e{-}4$ \\
        & XPINN & $5.10e{-}2$ & $4.14e{-}2$ & $5.33e{-}2$ & $6.34e{-}2$ \\
        \hline
        \multirow{2}{*}{(400, 40, 40)} & MAF & $2.40e{-}4$ & $1.46e{-}4$ & $1.88e{-}4$ & $2.69e{-}4$ \\
        & XPINN & $1.12e{-}2$ & $1.35e{-}2$ & $1.56e{-}2$ & $1.64e{-}2$ \\
        \hline
        \multirow{2}{*}{(1000, 100, 100)} & MAF & $8.45e{-}5$ & $8.52e{-}5$ & $8.89e{-}5$ & $9.01e{-}5$ \\
        & XPINN & $7.33e{-}3$ & $7.53e{-}3$ & $8.68e{-}3$ & $8.99e{-}3$ \\
        \hline
    \end{tabular}
\end{table}

\section{Conclusion}
\label{Conclusion}

We have developed a domain decomposition-based deep neural network framework with multi-activation functions (MAF) for solving elliptic and parabolic interface problems with discontinuous coefficients. Our method employs two independent neural networks, one for each subdomain, coupled through interface conditions in the loss function. Within each network, a multi-activation mechanism with interface-aware weighting enhances training efficiency near interfaces where coupling constraints are most demanding. This approach naturally accommodates solution and flux jumps across material interfaces without body-fitted meshes or explicit interface tracking. We presented conditional error bounds adapted from the generalization error framework of Mishra and Molinaro~\cite{mishra2023estimates}, relating solution accuracy to trained loss values and quadrature errors. Extensive numerical experiments validate the method's effectiveness: for elliptic problems in 2D--10D with various interface geometries (straight lines, ellipses, sunflower shapes, flowers, and hyperspheres), MAF achieves competitive or superior accuracy compared to XPINN, MFM, DCSNN, M-PINN, I-PINN, and AdaI-PINN; for parabolic problems with static and moving interfaces---where existing methods (MFM, DCSNN, M-PINN, I-PINN, AdaI-PINN) have not been demonstrated---MAF outperforms XPINN by 1--2 orders of magnitude across fixed, translating, deforming, and star-shaped interface configurations.

Several limitations should be acknowledged. The theoretical error bounds are conditional on stability assumptions that may not hold uniformly for high-contrast coefficient problems. Practical performance depends significantly on hyperparameter choices, network architecture, and optimization strategies requiring problem-specific tuning. Future work could explore advanced training strategies (adaptive sampling, curriculum learning) for higher-dimensional problems, extend the framework to nonlinear PDEs and coupled multiphysics problems, investigate multi-activation architectures from a theoretical perspective, and incorporate probabilistic approaches for uncertainty quantification.

\section{Acknowledgments}

This work was supported by the National Natural Science Foundation of China (Grant No.~12171340). 
%The author would like to thank the anonymous reviewers for their valuable comments and suggestions, which have significantly improved the quality of this manuscript.

\subsection*{Declaration on the Use of AI-Assisted Tools}

During the preparation of this manuscript, the author used AI-assisted tools (including large language models) to improve the readability, language, and presentation of the text. Specifically, AI tools were employed for:
\begin{itemize}
    \item Language editing and proofreading to enhance clarity and grammar;
    \item Assisting with the organization and formatting of certain sections;
    \item Generating preliminary drafts of some explanatory text that were subsequently reviewed and revised.
\end{itemize}
The author takes full responsibility for the scientific content, methodology, results, and conclusions presented in this work. All mathematical derivations, theoretical analyses, numerical implementations, and experimental results were conducted and verified by the author. The use of AI tools was limited to assisting with writing and presentation, and did not involve the generation of scientific content or research findings.

\bibliographystyle{siam}
\bibliography{jumpref}
\end{document}